\documentclass[11pt,a4paper,reqno]{amsart}
\pdfoutput=1
\usepackage{times}
\usepackage[margin=1.0in,tmargin=0.9in,bmargin=0.8in]{geometry}
\usepackage[dvipsnames]{xcolor}
\usepackage[english]{babel} 

\definecolor{bred}{rgb}{0.8,0,0}

\usepackage{subfigure}
\usepackage{amsmath,amssymb,graphicx,epsfig}
\usepackage{mathtools}
\usepackage{enumerate,amsthm,epstopdf}
\usepackage[flushleft]{threeparttable}
\usepackage{multirow}
\usepackage{longtable}
\usepackage[utf8]{inputenc} 
\usepackage[T1]{fontenc}    
\usepackage{bbm}
\usepackage{xargs}
\usepackage{latexsym}
\usepackage{wasysym}
\usepackage{float}
\usepackage{hyperref}       
\usepackage{url}            
\usepackage{booktabs}       
\usepackage{amsfonts}       
\usepackage{nicefrac}       
\usepackage{microtype}      
\usepackage{cleveref}
\usepackage[textsize=footnotesize]{todonotes}
\hypersetup{colorlinks,linkcolor={blue},citecolor={bred},urlcolor={blue}}
\usepackage[numbers]{natbib}

\usepackage{stmaryrd}
\usepackage{algorithm,algpseudocode}
\usepackage{tabularx}
\usepackage{algorithmicx}
\allowdisplaybreaks

\numberwithin{equation}{section}

\makeatletter
\newcommand{\multiline}[1]{%
  \begin{tabularx}{\dimexpr\linewidth-\ALG@thistlm}[t]{@{}X@{}}
    #1
  \end{tabularx}
}
\makeatother

\newtheorem{theorem}{Theorem}[section]
\newtheorem{proposition}[theorem]{Proposition}
\newtheorem{lemma}[theorem]{Lemma}
\newtheorem{corollary}[theorem]{Corollary}
\newtheorem{remark}[theorem]{Remark}
\newtheorem{definition}[theorem]{Definition}

\newtheorem{assumption}[theorem]{Assumption}

\newcommand{\R}{\mathbb{{R}}}

\newcommand{\N}{\mathbb{{N}}}

\newcommand{\cX}{\mathcal{{X}}}

\newcommand\bE{\mathbb{E}}
\newcommand\bF{\mathbb{F}}

\newcommand\bN{\mathbb{N}}
\newcommand\bR{\mathbb{R}}
\newcommand\bP{\mathbb{P}}
\newcommand\bZ{\mathbb{Z}}
\newcommand\bV{\mathbb{V}}

\newcommand\bX{\mathbb{X}}

\newcommand\bVar{\mathbb{V}ar}
\newcommand\fv{\mathbf{v}}

\newcommand\bS{\mathbb{S}}

\newcommand\cB{\mathcal{B}}

\newcommand\cD{\mathcal{D}}

\newcommand\cF{\mathcal{F}}

\newcommand\cK{\mathcal{K}}
\newcommand\cL{\mathcal{L}}
\newcommand\cM{\mathcal{M}}

\newcommand\cR{\mathcal{R}}

\newcommand\cU{\mathcal{U}}
\newcommand\cV{\mathcal{V}}

\newcommand{\cal}{\curvearrowleft}

\makeatletter
 \newcommand{\sumstar}
 {\operatornamewithlimits{\sum@\kern-.2em\raise1ex\hbox{*}}}
 \makeatother

\renewcommand{\>}{\rangle}

\allowdisplaybreaks

\usepackage{etoolbox}
\AtBeginEnvironment{tabular}{\small}
\newcolumntype{R}[1]{>{\raggedleft\let\newline\\\arraybackslash\hspace{0pt}}m{#1}}

\begin{document}

\title[]{Multilevel Picard algorithm for general semilinear parabolic PDEs with gradient-dependent nonlinearities}
\author[A.Neufeld]{Ariel Neufeld}

\author[S. Wu]{Sizhou Wu}

\date{}
\thanks{
	\textit{Acknowledgment.} Financial support by the Nanyang Assistant Professorship Grant (NAP Grant) \textit{Machine Learning based Algorithms in Finance and Insurance} is gratefully acknowledged.}
	\thanks{\textit{Email addresses.} ariel.neufeld@ntu.edu.sg \ and \ wusizhou@sufe.edu.cn}

\subjclass[2010]{}

\begin{abstract}
In this paper we introduce a multilevel Picard approximation algorithm for general semilinear parabolic
PDEs with gradient-dependent nonlinearities whose coefficient functions do not need to be constant. We also provide a full convergence and complexity analysis of our algorithm.
\\
To obtain our main results, we consider a particular stochastic fixed-point equation (SFPE) motivated by the Feynman-Kac representation and the Bismut-Elworthy-Li formula.  
We show that the PDE under consideration has a unique viscosity solution which coincides with the first component  of the unique solution of the stochastic fixed-point equation. 
Moreover, the gradient of the unique viscosity solution of the PDE exists and coincides with the second component of the unique solution of the stochastic fixed-point equation.
Furthermore, we also provide a numerical example in up to $300$ dimensions to
demonstrate the practical applicability of our multilevel Picard algorithm.
\end{abstract}
\maketitle

\section{\textbf{Introduction}}

Nonlinear partial differential equations (PDEs) 
can be used to model numerous important phenomena in many fields,
e.g., finance, physics, biology, economics, and engineering.
In recent years, 
neural network based 
\cite{al2022extensions,beck2020deep,beck2021deep,beck2019machine,
beck2020overview,berner2020numerically,castro2022deep,cioica2022deep,
ew2017deep,ew2018deep,weinan2021algorithms,frey2022convergence,frey2022deep,GPW2022,
gnoatto2022deep,gonon2023random,gonon2021deep,gonon2023deep,grohs2023proof,
han2018solving,han2020convergence, 
han2019solving,hure2020deep,hutzenthaler2020proof,ito2021neural,
jacquier2023deep,jacquier2023random,jentzen2018proof,
lu2021deepxde,nguwi2022deep,nguwi2022numerical,nguwi2023deep,
raissi2019physics,reisinger2020rectified,sirignano2018dgm,zhang2020learning} 
or multilevel Monte-Carlo based 
\cite{BGJ2020,beck2020overcoming,becker2020numerical,
hutzenthaler2019multilevel,hutzenthaler2021multilevel,
giles2019generalised,HJK2022,HJKNW2020,
HK2020,HJKN2020,
hutzenthaler2020overcoming,hutzenthaler2022multilevel,
hutzenthaler2022multilevel1,NW2022} 
numerical methods to solve high-dimensional 
nonlinear PDEs have been widely developed. 
While efficient in practice, neural networks based algorithms lack a rigorous convergence analysis caused by the non-convexity of the corresponding optimization problems when training neural networks. On the other hand
one can provide a thorough convergence and complexity analysis for multilevel Monte-Carlo based methodologies. In particular, it has been proven in the literature that under some moderate assumptions (typically Lipschitz continuity)
on the coefficient functions, the source term function describing the nonlinearity, and the initial (or terminal) condition function of the PDE under consideration, 
the multilevel Picard approximation algorithms can overcome the curse of
dimensionality in the sense that the computational complexity of
the algorithms grows at most polynomially in both the PDE dimension~$d$ and the
reciprocal of the prescribed approximation accuracy~$\varepsilon$, see
\cite{BGJ2020,beck2020overcoming,hutzenthaler2019multilevel,
hutzenthaler2021multilevel,giles2019generalised,HJK2022,HJKNW2020,
HK2020,HJKN2020,
hutzenthaler2020overcoming,hutzenthaler2022multilevel,
hutzenthaler2022multilevel1,NW2022}.
 
We highlight that for semilinear PDEs with nonlinearities also in the gradient, the development of numerical schemes that can approximately solve such high-dimensional equations for which there is theoretical convergence guarantees and complexity analysis is at its infancy.
To the best of our knowledge, only in \cite{HJK2022,HK2020} multilevel Picard approximation algorithms together with its convergence and complexity analysis has been developed so far for semilinear PDEs with gradient-dependent nonlinearities. In particular, as the above mentioned papers \cite{HJK2022,HK2020} 
consider semilinear heat equations with
gradient-dependent nonlinearities, there is no literature on numerical schemes suitable for \textit{general} high-dimensional semilinear PDEs  with
gradient-dependent nonlinearities whose coefficient functions
are not constant for which there exists a thorough convergence and complexity analysis.

The goal of this paper hence is to develop a full-history recursive 
multilevel Picard (MLP) approximation algorithm, together with its convergence and complexity analysis, which can solve general semilinear PDEs
with gradient-dependent nonlinearities. In particular, we do not require the coefficient functions of the PDE to be constant.
We also provide a numerical example in up to $300$ dimensions to
demonstrate the practical applicability of our MLP algorithm.

The idea of our MLP approximation algorithm is the following. Consider the semilinear PDE with gradient-dependent nonlinearity defined on $[0,T]\times \R^d$  by 
\begin{align}
	\frac{\partial}{\partial t}u^d(t,x)+\langle\nabla_x u^d(t,x),\mu^d(x)\rangle
	&
	+\frac{1}{2}\operatorname{Trace}\big(\sigma^d(x)[\sigma^d(x)]^T
	\operatorname{Hess}_xu^d(t,x)\big)
	\nonumber\\
	&
	+f^d(t,x,u^d(t,x),\nabla_x u^d(t,x))=0
	\quad \text{on} \quad [0,T)\times\bR^d
	\label{PDE0}
\end{align}
with terminal condition $u^d(T,x)=g^d(x)$. 
Under some suitable conditions on the coefficients of the PDE,
the unique 
solution of the PDE \eqref{PDE0} satisfies the following \textit{Feynman-Kac representation}
\begin{equation}                                                   \label{FK}
	u^d(t,x)=\bE\big[g^d(X^{d,t,x}_T)\big]+\int_t^T\bE\big
	[f^d\big(s,X^{d,t,x}_s,u^d(s,X^{d,t,x}_s),
	(\nabla_x u^d)(s,X^{d,t,x}_s)\big)\big]\,ds, 
\end{equation}
where
\begin{equation}
\label{SDE 0}
X^{d,t,x}_s
=
x
+ 
\int_t^s\mu^d(X^{d,t,x}_s)\,ds
+
\int_t^s\sigma^d(X^{d,t,x}_{s})\,dW^d_s, \quad s\in[t,T].
\end{equation}
Note that the right-hand side of the Feynman-Kac representation \eqref{FK}
depends on both the solution of the PDE $u^d$ and its gradient $(\nabla_x u^d)$. To obtain a stochastic representation (called \textit{Bismut-Elworthy-Li formula}) also for the gradient $(\nabla_x u^d)$ of the PDE \eqref{PDE0}, observe that for each $k\in\{1,\dots d\}$
\begin{equation}                                                  \label{BEL1}
	\frac{\partial}{\partial x_k}\bE\big[g^d(X^{d,t,x}_T)\big]
	=\bE\left[g^d(X^{d,t,x}_T)\frac{1}{T-t}\int_t^T\left(
	\big[\sigma^d(X^{d,t,x}_{s})\big]^{-1}
	D_{x_k}X^{d,t,x}_{s}
	\right)^T dW^d_s\right]
\end{equation} 
and 
\begin{align}
	&
	\int_t^T\frac{\partial}{\partial x_k}
	\bE\left[f^d\big(s,X^{d,t,x}_s,u^d(s,X^{d,t,x}_s),
	(\nabla_x u^d)(s,X^{d,t,x}_s)\big)\right]ds
	\nonumber\\
	&
	=\int_t^T\bE\bigg[
	f^d\big(s,X^{d,t,x}_s,u^d(s,X^{d,t,x}_s),
	(\nabla_x u^d)(s,X^{d,t,x}_s)\big)
	\frac{1}{s-t}\int_t^s\left(
	\big[\sigma^d(X^{d,t,x}_{r})\big]^{-1}
	D_{x_k}X^{d,t,x}_{r}
	\right)^T\,dW^d_r\bigg]\,ds,
	\label{BEL2}
\end{align}
where $D_{x_k} X^{d,t,x}_s$ denotes the
$L_2(\bP)$-derivative of $X^{d,t,x}_s$ with respect to $x_k$, i.e., 
$D_{x_k}X^{d,t,x}_s$ is the $L_2(\bP)$-limit 
of the random variable
$(X^{d,t,x+\delta e_k}_s-X^{d,t,x}_s)/\delta$  as  $\delta\to 0$. 
Therefore, by \eqref{BEL1} and \eqref{BEL2} we obtain that
\begin{align}
	&
	\nabla_x u^d(t,x)
	\nonumber\\
	&
	=\bE\left[\frac{g^d(X^{d,t,x}_T)}{T-t}
	\int_t^T\left(\big[\sigma^d(X^{d,t,x}_{r})\big]^{-1}DX^{d,t,x}_{r}\right)^T\,dW^d_r\right]
	\nonumber\\
	& \quad                                        
	+\int_t^T \bE\bigg[
	f^d\big(s,X_s^{d,t,x},u^d(s,X_s^{d,t,x}),
	(\nabla_x u^d)(s,X^{d,t,x}_s)\big)
	\frac{1}{s-t}
	\int_t^s\left(\big[\sigma^d(X^{d,t,x}_{r})\big]^{-1}DX^{d,t,x}_{r}\right)^T
	\,dW^d_r\bigg]\,ds, 
	\label{BEL3}
\end{align}
where
\begin{equation}
\label{def D intro}
	DX^{d,t,x}_s:=\left(D_{x_1}X^{d,t,x}_s,
	D_{x_2}X^{d,t,x}_s,
	\dots,D_{x_d}X^{d,t,x}_s\right).
\end{equation}

%
Motivated by the Feynman-Kac representation \eqref{FK} and the Bismut-Elworthy-Li formula \eqref{BEL3}, we define the following map
\begin{align}
    &
	(\mathbf{\Phi}^d\circ\mathbf{v}^d)(t,x)
	\nonumber\\
	&
	=\bE\bigg[g^d(X^{d,t,x}_T)\left(1,\frac{1}{T-t}
	\int_t^T\left(\big[\sigma^d(X^{d,t,x}_{r})\big]^{-1}
	DX^{d,t,x}_{r}\right)^T\,dW^d_r\right)\bigg]  
	\nonumber\\
	& \quad    
	+\int_t^T\bE\bigg[
	f^d\big(s,X^{d,t,x}_s,\mathbf{v}^d(s,X^{d,t,x}_s)\big)
	\left(1,\frac{1}{s-t}
	\int_t^s\left(\big[\sigma^d(X^{d,t,x}_{r})\big]^{-1}
	DX^{d,t,x}_{r}\right)^T dW^d_r\right)
	\bigg]\,ds.
	\label{FP}
\end{align}
We show that $\mathbf{\Phi}^d$ defines a contraction mapping on a suitable Banach space, which then by the Banach fixed-point theorem ensures the existence of a unique fixed-point $\mathbf{v}^d=(v_1^d,v_2^d)$ for $\mathbf{\Phi}^d$. We show that the PDE $\eqref{PDE0}$ has a unique viscosity solution  $u^d$ which coincides with the first component $v_1^d$ of the fixed-point $\mathbf{v}^d$. Moreover, 
its gradient $(\nabla_x u^d)$ exists and coincides with the second component $v_2^d$ of the fixed-point $\mathbf{v}^d$. 

This then allows us to consider the sequence of Picard iteration defined recursively
by
\begin{equation*}
\mathbf{v}^d_k(t,x)=(\mathbf{\Phi}^d\circ\mathbf{v}^d_{k-1})(t,x)
\end{equation*}
with $\mathbf{v}^d_0\equiv 0$, for which the Banach fixed-point theorem ensures its convergence to the unique fixed-point. 
Note that by \eqref{FP}, we see that
\begin{align}
	\mathbf{v}^d_k(t,x)
	&
	=\mathbf{v}^d_1(t,x)
	+\sum_{l=1}^{k-1}[\mathbf{v}^d_{l+1}(t,x)-\mathbf{v}^d_{l}(t,x)]
	\nonumber\\
	&
	=(\mathbf{\Phi}^d\circ\mathbf{v}^d_0)(t,x)
	+\sum_{l=1}^{k-1}[(\mathbf{\Phi}^d\circ\mathbf{v}^d_l)(t,x)
	-(\mathbf{\Phi}^d\circ\mathbf{v}^d_{l-1})(t,x)]
	\nonumber\\
	&
	=\bE\left[g^d(X^{d,t,x}_T)\left(1,\frac{1}{T-t}
	\int_t^T\left(\big[\sigma^d(X^{d,t,x}_{r})\big]^{-1}DX^{d,t,x}_{r}\right)^T
	\,dW^d_r\right)\right]
	\nonumber\\
	&  \quad                                    
	+\sum_{l=0}^{k-1}\int_t^T\bE\Bigg[
	\left(
	f^d\big(s,X^{d,t,x}_s,\mathbf{v}^d_l(s,X^{d,t,x}_s)\big)
	-
	\mathbf{1}_{\{l\geq 1\}}f^d\big(s,X^{d,t,x}_s,\mathbf{v}^d_{l-1}(s,X^{d,t,x}_s)\big)\right)
	\nonumber\\
	& \quad
	\cdot
	\left(1,\frac{1}{s-t}
	\int_t^s\left(\big[\sigma^d(X^{d,t,x}_{r})\big]^{-1}DX^{d,t,x}_{r}\right)^T
	\,dW^d_r\right)\Bigg]\,ds.
	\label{iteration1}
\end{align}
By replacing in \eqref{iteration1} all expectations and integrals 
by corresponding Monte-Carlo approximations, 
as well as by replacing all paths of SDEs 
by corresponding Euler-Maruyama approximations 
in case the SDEs cannot be simulated directly, 
we derive our multilevel Picard approximation scheme.

Our approach hence can be seen as the extension of 
\cite{BGJ2020,beck2020overcoming,
hutzenthaler2019multilevel,hutzenthaler2021multilevel,
giles2019generalised,HJKNW2020,
HJKN2020,
hutzenthaler2020overcoming,
hutzenthaler2022multilevel1,NW2022}, 
where in these papers, the unique viscosity solution of a nonlinear PDE similar to
\eqref{PDE0} \textit{but without nonlinearity in the gradient} 
has been identified as the unique fixed-point of a stochastic fixed-point equation 
using only the Feynman-Kac representation. 

Moreover, \cite{HJK2022,HK2020} consider semilinear heat equations 
with gradient-dependent nonlinearities, 
which is a special case of our PDE \eqref{PDE0} under consideration 
where there the drift and diffusion coefficients are constant. 
This allows them to directly apply (a simple version of)
 the classical Bismut-Elworthy-Li formula 
(see, e.g. \cite[Proposition~3.2]{FLLLT1999}), 
without identifying the unique viscosity solution and its gradient 
as the unique solution of a particular stochastic fixed-point equation, 
see \cite[Lemma 4.2, (ii)]{HJK2022}, \cite[Lemma 4.2]{HK2020}. 

Furthermore, recently, \cite{gobet2024numerical} provides 
nonlinear Feynman-Kac formulas 
for a class of ergodic backward stochastic differential equations (BSDEs). 
These formulas provide a probabilistic representation for elliptic PDEs 
of ergodic type with gradient-dependent nonlinearities. 
This allows \cite{gobet2024numerical} 
to derive a fully implementable numerical scheme for ergodic BSDEs 
based on a Picard iteration methodology.

The structure of this paper is the following. 
The precise setting and assumptions are introduced in Section \ref{section setting}.
In Sections \ref{section euler}--\ref{section main thm},
we present our MLP approximation algorithm 
\eqref{def MLP no euler} and \eqref{def MLP euler} 
and formulate the main results of this paper 
(see Theorem \ref{thm MLP conv} and Theorem \ref{MLP complexity}).
The pseudocode of our MLP algorithm is presented in Section~\ref{section pseudocode},
whereas a numerical example is provided in Section \ref{section example}.
In Section~\ref{section FP}, we study a family of stochastic fixed-point equations,
which will be used to construct a viscosity solution of semilinear PDE~\eqref{PDE0}
(c.f.\ Proposition~\ref{Prop FP} and Corollary~\ref{corollary FP}).
In Section \ref{section PIDE}, we show the existence and uniqueness of the
viscosity solution of semilinear PDE~\eqref{PDE0} and establish 
a Feynman-Kac and Bismut-Elworthy-Li type formula 
(c.f.\ Proposition~\ref{theorem PDE existence}).
In Section~\ref{section general MLP}, we introduce a new class of full-history recursive multilevel
Picard approximation schemes (see \eqref{def MLP general}) 
in a general setting and provide an approximation error bound 
for \eqref{def MLP general}
(c.f.\ Proposition \ref{corollary MLP error}),
which we then employ to prove the convergence of our MLP approximation algorithm
\eqref{def MLP no euler} and \eqref{def MLP euler}.
At last, the proofs of the main theorems, 
Theorem~\ref{thm MLP conv} and Theorem~\ref{MLP complexity},
are presented in Section~\ref{section proof main}.
In addition, some technical lemmas and proofs are presented in 
the appendices.

\subsection*{Notations}
In conclusion, we introduce some notations used throughout this paper.
Denote by $\bN$ and $\bN_0$ the set of all positive integers 
and the set of all natural numbers, respectively,
and denote by $\bZ$ the set of all integers.

Let $d\in\bN$ and $T\in[0,\infty)$.
We use $\mathbf{I}_d$ to denote the $d\times d$ identity  matrix,
and use $\bS^d$ to denote the space of $d\times d$
symmetric matrices. 
For matrices $A,B\in\bS^d$ the notation $A\geq B$
means $A-B$ is positive semi-definite. 
For each $d\in\bN$ and any vectors $a,b\in\bR^d$, 
we denote by $\<a,b\>$ the Euclidean scalar
product of $a$ and $b$, and denote by $\|a\|$ the Euclidean norm of $a$.
For each $d\in\bN$ and every matrix $A\in\bR^{d\times d}$, we
denote by $\|A\|_F$ the Frobenius norm of $A$, 
and we use $A^{ij}$ to denote the element on the $i$-th row and $j$-th column
of $A$ for $i,j=1,...,d$.
Moreover, denote by $\{e_1,\dots,e_d\}$ the canonical basis of $\bR^d$.

For every matrix $A\in\bR^{d\times d}$, 
we denote by $A^T$ the transpose of $A$. 
For any metric spaces $(E,d_E)$ and $(F,d_F)$,
we denote by $C(E,F)$ the
set of continuous functions from $E$ to $F$. For every topological space $E$,
denote by $\cB(E)$ the Borel $\sigma$-algebra of $E$.  
For all measurable spaces $(X_1,\Sigma_1)$ and $(X_2,\Sigma_2)$, we denote by
$\cM(\Sigma_1,\Sigma_2)$ the set of $\Sigma_1/\Sigma_2$-measurable functions
from $X_1$ to $X_2$.
For all $a,b\in\bR$, we
use the notations $a\wedge b:=\min\{a,b\}$ and $a\vee b:=\max\{a,b\}$.
For any set $B$, 
we use $\mathbf{1}_B$ to denote the indicator function of $B$.

Let $\mathbb{T}$ denote any interval in the forms of $(0,T)$, $[0,T)$, and $[0,T]$,
and denote by $LSC(\mathbb{T}\times\bR^d)$ 
($USC(\mathbb{T}\times\bR^d)$)
the space of lower (upper) semicontinuous functions 
$u:\mathbb{T}\times\bR^d\to\bR$.
We denote by
$LSC_{lin}(\mathbb{T}\times\bR^d)$ 
($USC_{lin}(\mathbb{T}\times\bR^d)$)
the space of functions $u\in LSC(\mathbb{T}\times\bR^d)$ 
$(USC(\mathbb{T}\times\bR^d))$
satisfying the linear growth condition  
\begin{equation}                                               \label{p growth}
\sup_{(t,x)\in\mathbb{T}\times\bR^d}\frac{|u(t,x)|}{1+\|x\|}<\infty,
\end{equation}
and we denote by $C_{lin}(\mathbb{T}\times\bR^d)$ the space of continuous functions
$u:\mathbb{T}\times\bR^d\to\bR$ satisfying \eqref{p growth}.
Moreover, 
$SC(\mathbb{T}\times\bR^d):=LSC(\mathbb{T}\times\bR^d)
\cup USC(\mathbb{T}\times\bR^d)$,
and $SC_{lin}(\mathbb{T}\times\bR^d)
:=LSC_{lin}(\mathbb{T}\times\bR^d)\cup USC_{lin}(\mathbb{T}\times\bR^d)$.
We use the notation $C^{1,2}(\mathbb{T}\times\bR^d)$ to denote the space of once
in time $t\in\mathbb{T}$ and twice in space $x\in\bR^d$ continuously differentiable
functions $u:\mathbb{T}\times\bR^d\to\bR$.
The notation $C^\infty_c(\bR^d)$ means the space of 
infinitely differentiable real-valued
functions on $\bR^d$ with compact support.
Moreover, for every $d\in\bN$ and $q\in[1,\infty)$ we use the notation 
$L_q=L_q((\Omega,\cF,\bP),\bR^d)$ to denote the
space of random variables $X:\Omega\to\bR^d$ on a probability space 
$(\Omega,\cF,\bP)$ such that
$
\|X\|_{L_q}:=(\bE[\|X\|^q])^{1/q}<\infty.
$

In addition, for differentiable functions $h:\bR^d\to\bR$ and $f:\bR^d\to\bR^d$,
and twice differentiable functions $F:\bR^d\to\bR$,
we use $\nabla h$ to denote the gradient of $h$,
use $\nabla f$ to denote the Jacobian matrix of $f$,
and use $\operatorname{Hess}(F)$ to denote the Hessian matrix of $F$.
For differentiable functions $G=(G^1,G^2,\dots,G^d):\bR^d\to\bR^{d\times d}$,
we use the notation
$$
\nabla G(x):=\big(
\nabla G^1(x),\nabla G^2(x),\dots, \nabla G^d(x)
\big)\in\bR^{d^3},
\quad x\in\bR^d.
$$

\section{\textbf{Setting and assumptions}}           
\label{section setting}
Let $T>0$ be a fixed constant, and let $(\Omega,\cF,P)$ be a complete probability space
equipped with a filtration $\bF:=(\cF_t)_{t\in[0,T]}$ 
satisfying the usual conditions.
For each $d\in \bN$ we are given an $\bR^d$-valued standard $\bF$-Brownian motion
denoted by $(W^d_t)_{t\in[0,T]}$.
Moreover, for each $d\in\bN$ let
$\mu^d=(\mu^{d,1},...,\mu^{d,d})
\in C^3(\bR^d,\bR^d)$
and
$
\sigma^{d}=(\sigma^{d,ij})_{i,j\in\{1,2,...,d\}}=(\sigma^{d,1},...,\sigma^{d,d})
\in C^3(\bR^d,\bR^{d\times d})
$.
Then for each $d\in\bN$ and $(t,x)\in[0,T]\times\bR^d$ 
consider the following stochastic differential equation (SDE) on $[t,T]$
\begin{equation}                                               \label{SDE}                              
dX^{d,t,x}_{s}
=\mu^d(X^{d,t,x}_{s})\,ds
+\sigma^{d}(X^{d,t,x}_{s})\,dW^{d}_s
\end{equation}
with initial condition $X^{d,t,x}_{t}=x$.
For each $d\in\bN$, let $g^d\in C(\bR^d,\bR)$ 
and $f^d\in C([0,T]\times \bR^d \times\bR \times\bR^d,\bR)$.
Then we make the following assumptions for the coefficient functions.

\subsubsection*{Regularity and growth conditions}

\begin{assumption}[Lipschitz and linear growth conditions]
\label{assumption Lip and growth}
There exists constants $L,p\in(0,\infty)$ satisfying 
for all $d\in\bN$, $x,y\in\bR^d$, $v_1,v_2\in\bR$, $w_1,w_2\in\bR^d$
and $t\in[0,T]$ that
\begin{align} 
&                                   
|f^d(t,x,v_1,w_1)-f^d(t,y,v_2,w_2)|^2
\leq L\big(|v_1-v_2|^2+\|w_1-w_2\|^2+\|x-y\|^2\big),
\label{assumption Lip f}
\\
&
|g^d(x)-g^d(y)|^2
+                             
\|\mu^d(x)-\mu^d(y)\|^2
+\|\sigma^{d}(x)-\sigma^{d}(y)\|_F^2
\leq L\|x-y\|^2,
\label{assumption Lip mu sigma}
\\      
&                                  
\|\mu^d(x)\|^2+\|\sigma^{d}(x)\|_F^2
\leq Ld^p(1+\|x\|^2),
\label{assumption growth}
\end{align}
and
\begin{equation}                                  \label{assumption growth f g}
|f^d(t,x,0,\mathbf{0})|^2+|g^d(x)|^2\leq L(d^p+\|x\|^2).
\end{equation}
\end{assumption}

\begin{remark} 
\label{Remark bbd deri}
Note that \eqref{assumption Lip mu sigma} and the mean-value theorem
ensure for all $d\in\bN$, $x\in\bR^d$, 
and $k\in\{1,2,\dots,d\}$ that
\begin{equation}                                            \label{bbd partials}                             
\Big\|\frac{\partial}{\partial x_k}\mu^d(x)\Big\|^2\leq L, \quad
\Big\|\frac{\partial}{\partial x_k}\sigma^d(x)\Big\|_F^2\leq L.
\end{equation}
\end{remark}

\begin{assumption}
\label{assumption bbd partials global}
There exists a constant $K>0$ satisfying for all $d\in\bN$ and $x\in\bR^d$ that
\begin{equation}                                         \label{bbd partials global}                             
\big\|\nabla\mu^d(x)\big\|_F^2\leq K, \quad
\sum_{j=1}^d\big\|\nabla\sigma^{d,j}(x)\big\|_F^2\leq K.
\end{equation}
\end{assumption}

\begin{assumption}[Strong ellipticity]                       
\label{assumption ellip}
For every $d\in\bN$, assume that $\sigma^d(x)$
is invertible for all $x\in\bR^d$, 
and that there exists a constant $\varepsilon_d\in(0,1]$ such that
\begin{equation}                                    \label{sigma ellip}
y^T\sigma^d(x)\big[\sigma^d(x)\big]^Ty\geq \varepsilon_d\|y\|^2
\end{equation}
for all $x,y\in\bR^d$. 
\end{assumption}

\begin{remark}
For each $d\in\bN$ and $x\in\bR^d$, we denote by 
$\lambda_{d,i}(x)$, $i=1,2,\dots,d$,
the eigenvalues of $\sigma^d[\sigma^d]^T(x)$.
Then Assumption \ref{assumption ellip} ensures
for all $x\in\bR^d$ and $d\in\bN$ that
\begin{equation*}                                        
\min_{i\in\{1,2,\dots,d\}}\lambda_{d,i}(x)\geq \varepsilon_d,
\end{equation*}
where $\varepsilon_d$ is the positive constant defined in \eqref{sigma ellip}.
This implies for all $x\in\bR^d$ and $d\in\bN$ that
\begin{equation}                                          \label{max eigen}
\max_{i\in\{1,2,\dots,d\}}\lambda_{d,i}^{-1}(x)\leq \varepsilon_d^{-1}.
\end{equation}
Taking into account the eigendecomposition of $\sigma^d[\sigma^d]^T$ and
$[(\sigma^d)^{-1}]^T(\sigma^d)^{-1}$, we notice for all $x\in\bR^d$
that $\lambda^{-1}_{d,i}(x)$, $i=1,2,\dots,d$, are the eigenvectors of 
$[(\sigma^d)^{-1}]^T(\sigma^d)^{-1}(x)$. 
Hence, it holds for all $x\in\bR^d$ that
\begin{equation}
\label{bbd inverse sigma}
\big\|\big(\sigma^d(x)\big)^{-1}\big\|_F^2
=
\operatorname{Trace}\Big\{
\big[\big(\sigma^d(x)\big)^{-1}\big]^T\big(\sigma^d(x)\big)^{-1}
\Big\}
\leq d\varepsilon_d^{-1}.
\end{equation}
Moreover, \eqref{max eigen} and the eigendecomposition of 
$[(\sigma^d)^{-1}]^T(\sigma^d)^{-1}$ ensure for all $x,y\in\bR^d$ 
and $d\in\bN$ that
\begin{equation}                                           \label{sigma inverse est}
y^T\big[\big(\sigma^d(x)\big)^{-1}\big]^T\big(\sigma^d(x)\big)^{-1}y
\leq\varepsilon_d^{-1}\|y\|^2.
\end{equation}
\end{remark}
 
\begin{assumption}[Regularity of derivatives]
\label{assumption gradient}
For every $d\in\bR^d$ we assume that the derivatives
of $\mu^d$ and $\sigma^d$ up to order $3$ exist, and are continuous.
Moreover, there exists a constant $L_0\in(0,\infty)$ 
satisfying for all $d\in\bN$, $k\in\{1,2,\dots,d\}$, 
$x,y\in\bR^d$ that
\begin{equation}                                     \label{gradient mu sigma}                               
\Big\|\frac{\partial}{\partial x_k} \mu^d(x)
-\frac{\partial}{\partial y_k} \mu^d(y)\Big\|^2
+\Big\|\frac{\partial}{\partial x_k}\sigma^d(x)
-\frac{\partial}{\partial y_k} \sigma^d(y)\Big\|_F^2
\leq L_0\|x-y\|^2.
\end{equation}
\end{assumption}

\begin{remark}
Analogous to Remark \ref{Remark bbd deri}, Assumption \ref{assumption gradient}
and the mean-value theorem
ensure for all $d\in\bN$, $x\in\bR^d$,
and $k,l\in\{1,2,\dots,d\}$ that
\begin{equation}                                            \label{bbd 2nd partials}                             
\Big\|\frac{\partial^2}{\partial x_k \partial x_l}\mu^d(x)\Big\|^2\leq L_0, \quad
\Big\|\frac{\partial^2}{\partial x_k \partial x_l}\sigma^d(x)\Big\|_F^2\leq L_0.
\end{equation}
\end{remark}

Let $\Theta=\cup_{n\in \bN}\bZ^n$ be an index set which we will use for 
the families of independent random variables needed for the Monte Carlo
approximations. For each $d\in\bN$, let 
$W^{d,\theta}=(W^{d,\theta,1},...,W^{d,\theta,d})
:[0,T]\times \Omega\to \bR^d, \theta\in\Theta$, be independent 
$\bR^d$-valued standard $\bF$-Brownian motions. 
For each $d\in\bN$ and $(t,x)\in[0,T]\times\bR^d$, 
let $\big(X^{d,\theta,t,x}_s\big)_{s\in[t,T]}: 
[t,T]\times\Omega\to \bR^d$, $\theta\in\Theta$, be 
$\cB([0,T])\otimes\cF/\cB(\bR^d)$-measurable 
functions satisfying for all $\theta\in\Theta$ and $s\in[t,T]$ 
almost surely that $X^{d,\theta,t,x}_{t}=x$ and
\begin{equation}                                               \label{SDE theta}                              
X^{d,\theta,t,x}_{s}
=x+\int_t^s\mu^d(X^{d,\theta,t,x}_r)\,dr
+\int_t^s\sigma^{d}(X^{d,\theta,t,x}_r)\,dW^{d,\theta}_r.
\end{equation}
It is well-known that Assumption \ref{assumption Lip and growth} 
guarantees that for each $d\in\bN$, the SDE in \eqref{SDE theta} has an unique solution satisfying 
for all $q\in[2,\infty)$ that
\begin{equation} 
\label{SDE moment est}                                         
\bE\Big[\sup_{s\in[t,T]}\big\|X^{d,0,t,x}_s\big\|^q\Big]<C_{(d,q)}(1+\|x\|^q),
\end{equation}
where $C_{(d,q)}$ is a positive constant only 
depending on $L$, $d$, $p$, $T$, and $q$
(see, e.g., Theorem 2.9 in \cite{KS1991}, and Theorems 3.1 and 4.1 in \cite{Mao2007}).

\begin{remark}
If Assumptions \ref{assumption Lip and growth} and \ref{assumption gradient},
then it is well-known (see, e.g., Theorem 3.4 in \cite{Kunita}) that for all 
$d\in\bN$, $\theta\in\Theta$, $(t,x)\in[0,T]\times\bR^d$, $s\in[t,T]$, 
and $k\in\{1,2,\dots,d\}$,
$\frac{\partial}{\partial x_k}X^{d,\theta,t,x}_s$ exists,  and satisfies
$$
\frac{\partial}{\partial x_k} X^{d,\theta,t,x}_s
=
e_k+\int_t^s
(\nabla\mu^d)\big(X^{d,\theta,t,x}_{r}\big)
\frac{\partial}{\partial x_k}X^{d,\theta,t,x}_{r}\,dr
+
\sum_{j=1}^d
\int_t^s(\nabla\sigma^{d,j})\big(X^{d,\theta,t,x}_{r}\big)
\frac{\partial}{\partial x_k}X^{d,\theta,t,x}_{r}\,dW^{d,j}_r,
$$
where $e_k$ denotes the $k$-th unit vector on $\bR^d$.
In the sequel, we will use the notation 
\begin{equation}
\label{D proc def}
DX^{d,\theta,t,x}_s
:=\Big(\frac{\partial}{\partial x_1}X^{d,\theta,t,x}_s,
\frac{\partial}{\partial x_2}X^{d,\theta,t,x}_s,
\dots,\frac{\partial}{\partial x_d}X^{d,\theta,t,x}_s\Big),
\end{equation}
and
\begin{equation}
\label{D proc def 0}
DX^{d,t,x}_s
:=\Big(\frac{\partial}{\partial x_1}X^{d,t,x}_s,
\frac{\partial}{\partial x_2}X^{d,t,x}_s,
\dots,\frac{\partial}{\partial x_d}X^{d,t,x}_s\Big)
\end{equation}
for all 
$d\in\bN$, $\theta\in\Theta$, $(t,x)\in[0,T]\times\bR^d$, and $s\in[t,T]$.
Notice that in \eqref{def D intro} we use the same notation 
$DX^{d,t,x}_s$
to denote the $L_2(\bP)$-gradient of $X^{d,t,x}_s$ as well. 
Actually, we prove in Lemma \ref{lemma L2 derivative} that the
$L_2(\bP)$-derivative
$D_{x_k}X^{d,t,x}_s$ 
coincides with the classical derivative
$\frac{\partial}{\partial x_k}X^{d,t,x}_s$
for all $d\in\bN$, $k\in\{1,2,\dots,d\}$,  
$(t,x)\in[0,T]\times\bR^d$, and $s\in[t,T]$.
\end{remark}

\section{Multilevel Picard approximation scheme: the main results}
\subsection{Euler approximations}
\label{section euler}
For each $d\in\bN$, $N\in\bN$, and $(t,x)\in[0,T]\times\bR^d$, 
let $\big(\cX^{d,\theta,t,x,N}_s\big)_{s\in[t,T]}: 
[t,T]\times\Omega\to \bR^d$, $\theta\in\Theta$, be measurable 
functions satisfying for all 
$n\in \bN_0$, $s\in \Big[t+\frac{n(T-t)}{N},t+\frac{(n+1)(T-t)}{N}\Big]
\cap [t,T]$ that
$\cX^{d,\theta,t,x,N}_{t}=x$ and
\begin{align}
\cX^{d,\theta,t,x,N}_{s}= & \cX^{d,\theta,t,x,N}_{t+\frac{n(T-t)}{N}}
+\mu^d\Big(\cX^{d,\theta,t,x,N}_{t+\frac{n(T-t)}{N}}\Big)
\left[s-\Big(t+\frac{n(T-t)}{N}\Big)\right]
+\sigma^{d}\Big(\cX^{d,\theta,t,x,N}_{t+\frac{n(T-t)}{N}}\Big)
\left(W^{d,\theta}_s-W^{d,\theta}_{t+\frac{n(T-t)}{N}}\right).                  \label{discrete Euler}
\end{align}
To ease notations we define 
$$
\kappa_N(s):=t+\frac{\lfloor N(s-t)/(T-t)\rfloor \cdot(T-t)}{N},\quad s\in[t,T],
$$
where $\lfloor y\rfloor:=\max\{n\in\bN_0:n\leq y\}$ for $y\in[0,\infty)$.
Then for all $d,N\in\bN$, $\theta\in\Theta$, and $(t,x)\in[0,T]\times\bR^d$, 
\eqref{discrete Euler} can be written as
\begin{align}
d\cX^{d,\theta,t,x,N}_{s}=
&
\mu^d\Big(\cX^{d,\theta,t,x,N}_{\kappa_N(s)}\Big)\,ds
+\sigma^{d}\Big(\cX^{d,\theta,t,x,N}_{\kappa_N(s)}\Big)
\,dW^{d,\theta}_s.
\label{Euler 1}
\end{align}
For each $d\in\bN$ and $(t,x)\in[0,T)\times\bR^d$, 
let $\big(V^{d,\theta,t,x}_s\big)_{s\in[t,T)}: 
(t,T]\times\Omega\to \bR^d$, $\theta\in\Theta$, be continuous adapted processes
given by
\begin{equation}
V^{d,\theta,t,x}_s
:=
\frac{1}{s-t}
\int_t^s\left(\big[\sigma^d(X^{d,\theta,t,x}_{r})\big]^{-1}
DX^{d,\theta,t,x}_{r}\right)^T
\,dW^{d,\theta}_r, \quad s\in(t,T].
\label{def proc V}
\end{equation}
Moreover, For each $d\in\bN$, $N\in\bN$, $k\in\{1,2,\dots,d\}$, 
and $(t,x)\in[0,T]\times\bR^d$, 
let $\big(\cD_{x_k}\cX^{d,\theta,t,x,N}_s\big)_{s\in[t,T]}: 
[t,T]\times\Omega\to \bR^d$, $\theta\in\Theta$, be continuous adapted processes
given by
\begin{align}
\cD_{x_k}\cX^{d,\theta,t,x,N}_s
=
&
e_k
+
\int_t^s(\nabla \mu^d)\big(\cX^{d,\theta,t,x,N}_{\kappa_N(r)}\big)
\cD_{x_k}\cX^{d,\theta,t,x,N}_{\kappa_N(r)}\,dr
\nonumber\\
&
+
\sum_{j=1}^d(\nabla \sigma^{d,j})\big(\cX^{d,\theta,t,x,N}_{\kappa_N(r)}\big)
\cD_{x_k}\cX^{d,\theta,t,x,N}_{\kappa_N(r)}\,dW^j_r,
\quad s\in[t,T].
\label{def proc cD}
\end{align}
Then, for each $d\in\bN$, $N\in\bN$, and $(t,x)\in[0,T)\times\bR^d$, 
let $\big(\cV^{d,\theta,t,x,N}_s\big)_{s\in[t,T)}: 
(t,T]\times\Omega\to \bR^d$, $\theta\in\Theta$, be continuous adapted processes
given by
\begin{equation}
\cV^{d,\theta,t,x,N}_s
:=
\frac{1}{s-t}
\int_t^s\left(\big[\sigma^d(\cX^{d,\theta,t,x,N}_{\kappa_N(r)})\big]^{-1}
\cD\cX^{d,\theta,t,x,N}_{\kappa_N(r)}\right)^T
\,dW^{d,\theta}_r, \quad s\in(t,T],
\label{def euler proc V}
\end{equation}
where
$$
\cD\cX^{d,\theta,t,x,N}_s
:=(\cD_{x_1}\cX^{d,\theta,t,x,N}_s,\cD_{x_2}\cX^{d,\theta,t,x,N}_s,
\dots,\cD_{x_d}\cX^{d,\theta,t,x,N}_s), \quad s\in[t,T].
$$
For each $d\in\bN$, $N\in\bN$, $(t,x)\in[0,T)\times\bR^d$, $s\in(t,T]$,
and $\theta\in\Theta$
we also use the notation 
$$
\cV^{d,\theta,t,x,N}_s
=(\cV^{d,\theta,t,x,N,1}_s,\cV^{d,\theta,t,x,N,2}_s,\dots,\cV^{d,\theta,t,x,N,d}_s).
$$

\subsection{Multilevel Picard (MLP) approximation scheme}
\label{section MLP}
Let $\alpha\in[1/2,1)$, 
and define the function $\varrho:(0,1)\to(0,\infty)$
by 
\begin{equation}
\label{def pdf rho 0}
\varrho(z):=\frac{z^{-\alpha}(1-z)^{-\alpha}}{\cB(1-\alpha,1-\alpha)}, 
\quad z\in(0,1),
\end{equation}
where $\cB(\beta,\gamma):=\frac{\Gamma(\beta)\Gamma(\gamma)}{\Gamma(\beta+\gamma)}$
denotes the Beta function with parameters $\beta,\gamma\in(0,\infty)$, 
and $\Gamma$ denotes the Gamma function.
Let $\xi^\theta:\Omega\to[0,1]$, $\theta\in\Theta$, 
be i.i.d.\ random variables
such that $\bP(\xi^0\leq y)=\int_0^y\varrho(z)\,dz$ for all $y\in[0,1]$,
and assume that $(\xi^{\theta})_{\theta\in\Theta}$ and 
$(W^{d,\theta})_{(d,\theta)\in\bN\times\Theta}$ are independent.
For each $\theta\in\Theta$ and $t\in[0,T)$, define
$\cR^\theta_t:=t+(T-t)\xi^\theta$,
and let
$\Big(\cR^{(\theta,l,i)}_t\Big)_{(l,i)
\in \bN\times\bN_0}$
are independent copies of $\cR^\theta_t$.
Moreover, for each $d,N\in \bN$, $(t,x)\in[0,T]\times\bR^d$, $s\in[t,T]$, 
and $\theta\in\Theta$, let 
$\big(X^{(d,\theta,t,x,l,i)}_{s},V^{(d,\theta,t,x,l,i)}_{s}\big)_{(l,i)
\in\bN\times\bZ}$
be independent copies of 
$\big(X^{d,\theta,t,x}_{s},V^{d,\theta,t,x}_{s}\big)$,
and let 
$\big(\cX^{(d,\theta,t,x,N,l,i)}_{s},\cV^{(d,\theta,t,x,N,l,i)}_{s}\big)_{(l,i)
\in\bN\times\bZ}$
be independent copies of 
$\big(\cX^{d,\theta,t,x,N}_{s},\cV^{d,\theta,t,x,N}_{s}\big)$.
Furthermore, for each $d\in\bN$, let 
$$
F^d: C([0,T)\times\bR^d,\bR^{d+1})\to C([0,T)\times\bR^d,\bR)
$$
be the operator such that
$$
[0,T)\times\bR^d\ni(t,x)\mapsto (F^d(\mathbf{v}))(t,x)
:=f^d(t,x,\mathbf{v}(t,x))\in\bR, \quad \mathbf{v}\in C([0,T)\times\bR^d,\bR^{d+1}).
$$
Then our MLP approximation is defined as follows.

For each $d\in\bN$, $n\in\bN_0\cup\{-1\}$, $M\in\bN$, and $\theta\in\Theta$, 
let $U^{d,\theta}_{n,M}:[0,T)\times \bR^d \times \Omega \to \bR^{d+1}$ 
satisfy for all $(t,x)\in[0,T)\times\bR^d$ and $\omega\in\Omega$ that
$U^{d,\theta}_{0,M}(t,x)=U^{d,\theta}_{-1,M}(t,x)=\mathbf{0}$ and
\begin{align}
U^{d,\theta}_{n,M}(t,x)
=
&
(g^d(x),0)
+
\frac{1}{M^n}\sum_{i=1}^{M^n}
\Big[g^d\Big(X^{(d,\theta,t,x,0,-i)}_{T}\Big)-g^d(x)\Big]
\Big(1,V^{(d,\theta,t,x,0,-i)}_{T}\Big)
\nonumber\\
&
+\sum_{l=0}^{n-1}\frac{T-t}{M^{n-l}}
\bigg[\sum^{M^{n-l}}_{i=1}
\varrho^{-1}\Big(\frac{\cR^{(\theta,l,i)}_t-t}{T-t}\Big)
\Big[F\big(U^{(d,\theta,l,i)}_{l,M}\big)
-\mathbf{1}_{\{l\geq 1\}}F\big(U^{(d,\theta,-l,i)}_{l-1,M}\big)\Big]
\nonumber\\
&
\Big(\cR^{(\theta,l,i)}_t,X^{(d,\theta,t,x,l,i)}_{\cR^{(\theta,l,i)}_t}\Big)
\Big(1,V^{(d,\theta,t,x,l,i)}_{\cR^{(\theta,l,i)}_t}\Big)
\bigg],
\label{def MLP no euler}
\end{align}
where $\Big(U^{(d,\theta,l,i)}_{n,M}(t,x)\Big)_{(l,i)
\in \bZ\times\bN_0}$
are independent copies of $U^{d,\theta}_{n,M}(t,x)$ 
for each $(t,x)\in[0,T)\times\bR^d$.

In case the SDE in \eqref{SDE} or the process in \eqref{def proc V} 
cannot be simulated directly,
we define our MLP approximation as follows.

For each $d\in\bN$, $n\in\bN_0\cup\{-1\}$, $M,N\in\bN$, and $\theta\in\Theta$, 
let $\cU^{d,\theta}_{n,M,N}:[0,T)\times \bR^d \times \Omega \to \bR^{d+1}$ 
satisfy for all $(t,x)\in[0,T)\times\bR^d$ and $\omega\in\Omega$ that
$\cU^{d,\theta}_{0,M,N}(t,x)=\cU^{d,\theta}_{-1,M,N}(t,x)=\mathbf{0}$ and
\begin{align}
\cU^{d,\theta}_{n,M,N}(t,x)
=
&
(g^d(x),0)
+
\frac{1}{M^n}\sum_{i=1}^{M^n}
\Big[g^d\Big(\cX^{(d,\theta,t,x,N,0,-i)}_{T}\Big)-g^d(x)\Big]
\Big(1,\cV^{(d,\theta,t,x,N,0,-i)}_{T}\Big)
\nonumber\\
&
+\sum_{l=0}^{n-1}\frac{T-t}{M^{n-l}}
\bigg[\sum^{M^{n-l}}_{i=1}
\varrho^{-1}\Big(\frac{\cR^{(\theta,l,i)}_t-t}{T-t}\Big)
\Big[F\big(\cU^{(d,\theta,l,i)}_{l,M,N}\big)
-\mathbf{1}_{\{l\geq 1\}}F\big(\cU^{(d,\theta,-l,i)}_{l-1,M,N}\big)\Big]
\nonumber\\
&
\Big(\cR^{(\theta,l,i)}_t,\cX^{(d,\theta,t,x,N,l,i)}_{\cR^{(\theta,l,i)}_t}\Big)
\Big(1,\cV^{(d,\theta,t,x,N,l,i)}_{\cR^{(\theta,l,i)}_t}\Big)
\bigg],
\label{def MLP euler}
\end{align}
where $\Big(\cU^{(d,\theta,l,i)}_{n,M,N}(t,x)\Big)_{(l,i)
\in \bZ\times\bN_0}$
are independent copies of $\cU^{d,\theta}_{n,M,N}(t,x)$ 
for each $(t,x)\in[0,T)\times\bR^d$.

\subsection{Viscosity solutions of PDEs}
For every $d\in\bN$, let
$
G^d:(0,T)\times \bR^d\times\bR\times\bR^d\times\bS^d \to \bR
$
be a function defined for all
$(t,x,r,y,A)
\in(0,T)\times\bR^d\times\bR\times\bR^d\times\bS^d$ by 
\begin{align}
G^d(t,x,r,y,A):=
&
-\langle y,\mu^d(x)\rangle
-\frac{1}{2}\operatorname{Trace}
(\sigma^d(x)[\sigma^d(x)]^TA)
-f^d(t,x,r,y).
\label{def operator G}
\end{align}
Then for every $d\in\bN$
 we consider the semilinear PDE of parabolic type
\begin{align}
&
-\frac{\partial}{\partial t}u^d(t,x)
+G^d(t,x,u^d(t,x),\nabla_x u^d(t,x)
,\operatorname{Hess}_xu^d(t,x))=0
\quad \text{on $(0,T)\times \bR^d$},
\label{APIDE}
\\
&
u^d(T,x)=g^d(x) \quad \text{on $\bR^d$}.
\label{APIDE initial}
\end{align}
We use the following definition of viscosity solutions
of PDE \eqref{APIDE} (cf.\,Definition 2.1 in \cite{JK}).

\begin{definition}
Let $d\in\bN$.
A function $u^d\in USC_{lin}((0,T)\times\bR^d)$ 
($u^d\in LSC_{lin}((0,T)\times\bR^d)$)
is called a viscosity subsolution (supersolution) of PDE \eqref{APIDE}
if for every $(t,x)\in(0,T)\times\bR^d$ 
and $\varphi\in C^{1,2}((0,T)\times\bR^d)$ such that 
$\varphi(t,x)=u^d(t,x)$, 
and $u^d\leq \varphi$ ($u^d\geq \varphi$), we have that
$$
-\frac{\partial}{\partial t}\varphi(t,x)
+G^d(t,x,\varphi(t,x),\nabla_x \varphi(t,x),
\operatorname{Hess}_x\varphi(t,x))
\leq 0\; (\geq 0).
$$
A function $u^d:(0,T)\times\bR^d\to\bR$ 
is said to be a viscosity solution of PDE \eqref{APIDE} if $u$ is both a 
viscosity subsolution and a viscosity supersolution of \eqref{APIDE}.
\end{definition}

\begin{remark}                                       
\label{remark v solution}
Let $d\in\bN$.
If $u^d\in C_{lin}([0,T]\times\bR^d)$ is a viscosity solution of
\eqref{APIDE} with $u^d(T,x)=g^d(x)$ for $x\in\bR^d$, 
then the function $v^d(t,x):=u^d(T-t,x)$, $(t,x)\in[0,T]\times\bR^d$,
satisfies the following:
\begin{enumerate}[(i)]
\item
$v^d(0,x)=g^d(x)$ for $x\in\bR^d$;
\item
for every $(t,x)\in(0,T)\times\bR^d$ and
$\varphi\in C^{1,2}((0,T)\times\bR^d)$ such that $\varphi(t,x)=v^d(t,x)$ 
and $v^d\leq \varphi$ ($v^d\geq \varphi$), we have
$$
\frac{\partial}{\partial t}\varphi(t,x)
+G^d(t,x,\varphi(t,x),\nabla_x \varphi(t,x),\operatorname{Hess}_x\varphi(t,x))
\leq 0\; (\geq 0).
$$
\end{enumerate}
The converse holds as well.
\end{remark}


\subsection{The main results}
Our first main result, Theorem \ref{thm MLP conv},  proves the existence
and uniqueness of the viscosity solutions of PDE \eqref{APIDE}, and provides
a Feynman-Kac and Bismut-Elworthy-Li type formula for the stochastic representation
of the unique viscosity solution. Moreover, Theorem \ref{thm MLP conv} also
proves the convergence together with the corresponding convergence rate 
for our MLP approximations \eqref{def MLP no euler}
and \eqref{def MLP euler}.
\label{section main thm}
\begin{theorem}                                                 \label{thm MLP conv}
Let Assumptions \ref{assumption Lip and growth}, \ref{assumption ellip}, 
and \ref{assumption gradient} hold. Then the following holds.
\begin{enumerate}[(i)]
\item
There exists a unique pair of Borel functions $(u^d,w^d)$ with
$u^d\in C([0,T)\times\bR^d,\bR)$ and $w^d\in C([0,T)\times\bR^d,\bR^d)$ 
satisfying for all $(t,x)\in[0,T)\times\bR^d$ that
\begin{align}
&
\big\|g^d(X^{d,0,t,x}_T)(1,V^{d,0,t,x}_T)\big\|_{L_1}
+\int_t^T\big\|
f^d\big(s,X^{d,0,t,x}_s,u^d(s,X^{d,0,t,x}_s),w^d(s,X^{d,0,t,x}_s)\big)
(1,V^{d,0,t,x}_s)
\big\|_{L_1}\,ds
\nonumber\\
&
+\sup_{(s,y)\in[0,T)\times\bR^d}
\left(\frac{|u^d(s,y)|+(T-s)^{1/2}\|w^d(s,y)\|}{(d^p+\|y\|^2)^{1/2}}
\right)<\infty,
\label{finite main}
\end{align}
and
\begin{align}
(u^d(t,x),w^d(t,x))
&
=\bE\left[g(X^{d,0,t,x}_T)(1,V^{d,0,t,x}_T)\right]  
\nonumber\\
& \quad    
+\int_t^T\bE\left[
f^d\big(s,X^{d,0,t,x}_s,u^d(s,X^{d,0,t,x}_s),w^d(s,X^{d,0,t,x}_s)\big)
(1,V^{d,0,t,x}_s)
\right]ds.
\label{BEL main}
\end{align}
\item{
For each $d\in\bN$ there exists a unique viscosity solution
$\tilde{u}^d\in C_{lin}([0,T)\times\bR^d)$ of PDE \eqref{APIDE} with
$\tilde{u}^d(T,x)=g^d(x)$ for all $x\in\bR^d$.
}

\item{
It holds for all $d\in\bN$ and $(t,x)\in[0,T)\times\bR^d$ that
$\tilde{u}^d(t,x)=u^d(t,x)$.
}

\item
For all $d\in\bN$ and $(t,x)\in[0,T)\times\bR^d$ the gradient of $u^d$ exists
and satisfies $\nabla_x u^d(t,x)=w^d(t,x)$. 

\item{
There exists positive constants 
$\mathfrak{c}_{d,1}=\mathfrak{c}_{d,1}(d,\varepsilon_d,L,L_0,T)$ 
and $\mathfrak{c}_{d,2}=\mathfrak{c}_{d,2}(d,\varepsilon_d,\alpha,L,L_0,T)$
satisfying
for all $d\in\bN$, $(t,x)\in[0,T)\times\bR^d$, $n\in\bN_0$, and $M,N\in\bN$ that
\begin{align}
\big\|\cU^{d,0}_{n,M,N}(t,x)-(u^d,\nabla_x u^d)(t,x)\big\|_{L_2}
\leq
\left[
\mathfrak{c}_{d,1}N^{-1/2}
+
\mathfrak{c}_{d,2}^{n-1}\exp\big\{M^3/6\big\}M^{-n/2}
\right]
\frac{d^p+\|x\|^2}{(T-t)^{1/2}}.
\end{align}
}

\item{
In addition, assume that Assumption \ref{assumption bbd partials global} holds.
Then there exists a positive constant 
$\mathfrak{c}_3=\mathfrak{c}_3(\alpha,L,L_0,K,T)$ satisfying
for all $d\in\bN$, $(t,x)\in[0,T)\times\bR^d$, $n\in\bN_0$, and $M\in\bN$ that
\begin{align}
\big\|U^{d,0}_{n,M}(t,x)-(u^d,\nabla_x u^d)(t,x)\big\|_{L_2}
\leq
\mathfrak{c}_3^{n-1}(d\varepsilon_d^{-1})^n\exp\big\{M^3/6\big\}M^{-n/2}
\frac{(d^p+\|x\|^2)^{1/2}}{(T-t)^{1/2}}.
\label{MLP error no Euler}
\end{align}
}
\end{enumerate}
\end{theorem}

The next theorem provides convergence results for the MLP approximation
algorithms \eqref{def MLP no euler} and \eqref{def MLP euler},
and it shows that the MLP approximation algorithm 
defined by \eqref{def MLP no euler} 
can overcome the curse of dimensionality
in the sense that the computational complexity of
the algorithm grows at most polynomially in both the PDE dimension~$d$ and the
reciprocal of the prescribed approximation accuracy~$\varepsilon$.
To describe the computational complexity of each of the two  algorithms 
\eqref{def MLP no euler} and \eqref{def MLP euler}, 
for each $d\in\bN$,
$n\in\bN_0$, 
and $M\in\bN$ 
we introduce a natural number $\mathfrak{C}^{(d)}_{n,M}$ to denote the sum of: 
the number of function evaluations of $g^d$,
the number of function evaluations of $\mu^d$,
the number of function evaluations of $\sigma^{d}$,
and the number of realizations of scalar random variables used to obtain
one realization of the corresponding MLP approximation algorithms in \eqref{def MLP no euler} and \eqref{def MLP euler}. 
Moreover, for each $d\in\bN$ we use
$\mathfrak{g}^{(d)}$ to denote the number of function evaluations of $g^d$, 
$\mathfrak{f}^{(d)}$ to denote the number of function evaluations of $f^d$, 
$\mathfrak{e}^{(d)}$ to denote the sum of:
the number of realizations of scalar random variables generated,
the number of function evaluations of $\mu^d$,
and the number of function evaluations of $\sigma^{d}$. 
Note that by the construction of the MLP approximation algorithms 
\eqref{def MLP no euler} and \eqref{def MLP euler} it holds for all $d,n,M\in\bN$ that
\begin{equation}                                            \label{cc 1}
\mathfrak{C}^{(d)}_{n,M}\leq M^n(M^M\mathfrak{e}^{(d)}
+\mathfrak{g}^{(d)})
+\sum_{l=0}^{n-1}[M^{n-l}(M^M\mathfrak{e}^{(d)}+\mathfrak{f}^{(d)}
+\mathfrak{C}^{(d)}_{l,M}
+\mathfrak{C}^{(d)}_{l-1,M})].
\end{equation}

\begin{theorem}                                            \label{MLP complexity}
Let Assumptions \ref{assumption Lip and growth},  
\ref{assumption ellip}, and \ref{assumption gradient} hold. 
Then the following holds.
\begin{enumerate}[(i)]
\item
For each $d\in\bN$, $\varepsilon\in(0,1]$, and $x\in\bR^d$
there exists a positive integer $\mathfrak{n}^d(x,\varepsilon)\geq 2$ 
such that for all $t\in[0,T]$
\begin{equation}                                           \label{error epsilon b U}
\sup_{n\in[\mathfrak{n}^d(x,\varepsilon),\infty)\cap \bN}
\big\|\cU^{d,0}_{n^3,n,n}(t,x)-(u^d,\nabla_x u^d)(t,x)\big\|_{L_2}<\varepsilon.
\end{equation}
\item
For the computational complexity of \eqref{def MLP euler} it holds for all $d\in\bN$ and $n\in\bN$ that
\begin{equation}                                               \label{cc 2}
\sum_{k=1}^{n+1}\mathfrak{C}^{(d)}_{k^3,k}
\leq 
12\big[3\mathfrak{e}^{(d)}+\mathfrak{g}^{(d)}+\mathfrak{f}^{(d)}\big]
(12)^{5n^3}\cdot n^{8n^3}.
\end{equation}
\item
In addition, assume that Assumption \ref{assumption bbd partials global} holds. 
Then for each $d\in\bN$, $\varepsilon\in(0,1]$, and $x\in\bR^d$
there exists a positive integer $\mathbf{n}^d(x,\varepsilon)\geq 2$ 
such that for all 
$\gamma\in(0,1]$ and $t\in[0,T]$
\begin{equation}                                            \label{error epsilon}
\sup_{n\in[\mathbf{n}^d(x,\varepsilon),\infty)\cap \bN}
\big\|U^{d,0}_{n^3,n}(t,x)-(u^d,\nabla_x u^d)(t,x)\big\|_{L_2}<\varepsilon,
\end{equation}
as well as for the computational complexity of \eqref{def MLP no euler} that
\begin{align}                                             
\left(\sum_{n=1}^{\mathbf{n}^d(x,\varepsilon)}\mathfrak{C}^{(d)}_{n^3,n}\right)
\varepsilon^{\gamma+16}  
\leq
&
12\big[3\mathfrak{e}^{(d)}+\mathfrak{g}^{(d)}+2\mathfrak{f}^{(d)}\big]
\big[(T-t)^{-1}(d^p+\|x\|^2)\big]^{\frac{\gamma+16}{2}}
\nonumber\\
& 
\cdot
\sup_{n\in\bN}\Big\{
12^{5n^3}\cdot n^{-\gamma n^3/2} 
\big[\mathfrak{c}_3^{n-1}(d\varepsilon_d^{-1})^n
\exp\big\{n^3/6\big\}\big]^{\gamma+16}\Big\}<\infty,
\label{cc 3}
\end{align}
where $\mathfrak{c}_3$ is the positive constant introduced in 
\eqref{MLP error no Euler}.
\end{enumerate}
\end{theorem}

\begin{remark}
We highlight that in the accompanying paper \cite{NNW2023}, 
we prove under more restrictive assumptions 
on the function describing the nonlinearity of the PDE 
that our developed MLP algorithm \eqref{def MLP euler} which involves
the Euler approximations of the corresponding SDEs in case 
the SDE in \eqref{SDE} or the process in \eqref{def proc V} 
cannot be directly simulated 
also overcomes the curse of dimensionality. 
\end{remark}

\clearpage

\subsection{Pseudocode}
\label{section pseudocode}
In this subsection we provide a pseudocode to show how the multilevel Picard
approximations \eqref{def MLP no euler} and \eqref{def MLP euler} can be implemented.

\begin{algorithm}
\footnotesize{
\caption{Multilevel Picard Approximation}
\begin{algorithmic}[1]
\Function{MLP}{$t,x,n,M,N$}\\
\% $(t,x)$: 
the evaluating time-space point of PDE \eqref{PDE0} 
/ the starting point $x$ of SDE \eqref{SDE 0} at initial time $t$\\
\% $n$: the number of different levels in the MLP algorithms
\eqref{def MLP no euler} and \eqref{def MLP euler}\\ 
\% $M$: the parameter determining the number of Monte-Carlo simulation samples,
see \eqref{def MLP no euler} and \eqref{def MLP euler}\\
\% $N$: the number of steps in Euler approximations 
\eqref{Euler 1} and \eqref{def euler proc V}

\State $t_1(j)\leftarrow (T-t)/N$ for all $j\in\{0,...,N-1\}$;
\Comment the length of time intervals in the time discretization

\If{the processes defined in \eqref{SDE} and \eqref{def proc V}
can be directly simulated} 
    \State \multiline{ Generate $M^n$ realizations 
    $(X_1(j),V_1(j))\in\bR^d\times\bR^d$, 
    $j\in\{1,2,\dots,M^n\}$,
    of the pair of the processes (at time $T$) 
    defined in \eqref{SDE} and \eqref{def proc V};}
 \Else
 \For{$i\leftarrow 1$ to $M^n$}

    \State $X_1(i)\leftarrow x$, $DX_1(i)\leftarrow \mathbf{I}_d$, $V_1(i)\leftarrow 0$;

    \State \multiline{Generate $N$ realizations $W(j)\in\bR^d$,
    $j\in\{0,...,N-1\}$, of i.i.d.\,standard normal random vectors;}
    \For{$k\leftarrow 0$ to $N-1$}
    \State \multiline{
    $
    V_1(i)\leftarrow  V_1(i)+(T-t)^{-1}\big([\sigma^d(X_1(i))]^{-1}DX_1(i)\big)^T
    \sqrt{t_1(k)}\cdot W(k);
    $
    }
    \State \multiline{
    $
    DX_1(i)\leftarrow  DX_1(i)+(\nabla\mu^d)(X_1(i))DX_1(i)t_1(k)
    +(\nabla\sigma^d)(X_1(i))DX_1(i)\sqrt{t_1(k)}\cdot W(k);
    $
    }
    \State \multiline{
    $
    X_1(i)\leftarrow  X_1(i)+\mu^d(X_1(i))t_1(k)
    +\sigma^d(X_1(i))\sqrt{t_1(k)}\cdot W(k);
    $
    }
    \EndFor
    
 \EndFor
\EndIf

\If{$n>0$}
       \State $\bar{u}= (g(x),0,\dots,0)+\frac{1}{M^n}\sum_{i=1}^{M^n}
       \big[g(X_1(i))-g(x)\big]\big(1,V_1(i)\big);$
       \Else 
       \State $\bar{u}=0;$
    \EndIf

\For{$l \leftarrow 0$ to $n-1$}

    \State \multiline{
    Generate $M^{n-l}$ realizations $\xi(i)\in[0,1]$, $i\in{1,...,M^{n-l}}$,
    of i.i.d.\,random variables with density function $\varrho$
    (cf. \eqref{def pdf rho 0});
    }
    \State \multiline{
    $
    \cR(i)\leftarrow t+(T-t)\xi(i);
    $
    }
  \If{the processes defined in \eqref{SDE} and \eqref{def proc V}
    can be directly simulated} 
    \For{$i \leftarrow 1$ to $M^{n-l}$}
     \State \multiline{Generate a realization of the pair
     $(X_2(i),V_2(i))\in\bR^d\times\bR^d$ 
     of the precesses (at time $\cR(i)$) 
     defined in \eqref{SDE} and \eqref{def proc V},;}
    \EndFor 
    \Else
    
    \For{$i \leftarrow 1$ to $M^{n-l}$}
      \State $X_2(i)\leftarrow x$, $DX_2(i)\leftarrow \mathbf{I}_d$, 
      $V_2(i)\leftarrow 0;$
      \State $S(i)\leftarrow \lfloor (\cR(i)-t)N/(T-t)\rfloor+1$;
      \Comment total time points of the time discretization of SDE
      \State $t_2(j)\leftarrow (T-t)/N$ for all $j\in\{0,...,S(i)-2\}$;
      \Comment the length of the first $S(i)-1$ time intervals
      \State $t_2(S(i)-1)\leftarrow \cR(i)-[t+N^{-1}(T-t)(S_2(i)-1)]$;
      \Comment the length of the last time interval
      \State \multiline{Generate $S(i)$ realizations $W(j)\in\bR^d$,
      $j\in\{0,...,S(i)-1\}$, of independent standard normal random vectors;}
        \For{$k\leftarrow 0$ to $S(i)-1$}
          \State \multiline{
          $
          V_2(i)\leftarrow  V_2(i)
          +(\cR(i)-t)^{-1}\big([\sigma^d(X_2(i))]^{-1}DX_2(i)\big)^T
          \sqrt{t_2(k)}\cdot W(k);
          $
          }
          \State \multiline{
          $
          DX_2(i)\leftarrow  DX_2(i)+(\nabla\mu^d)(X_2(i))DX_2(i)t_2(k)
          +(\nabla\sigma^d)(X_2(i))DX_2(i)\sqrt{t_2(k)}\cdot W(k);
          $
          }
          \State \multiline{
          $
          X_2(i)\leftarrow  X_2(i)+\mu^d(X_2(i))t_2(k)
          +\sigma^d(X_2(i))\sqrt{t_2(k)}\cdot W(k);
          $
          }
        \EndFor
    \EndFor
   \EndIf
 \State $\bar{u}\leftarrow \bar{u}+\frac{T-t}{M^{n-l}}\sum_{i=1}^{M^{n-l}}
    \varrho^{-1}\big(\frac{\cR(i)-t}{T-t}\big)
    f\big(\cR(i),X_2(i),\operatorname{MLP}(\cR(i),X_2(i),l,M,N)\big)
    \big(1,V_2(i)\big)$;
    \If{$l>0$}
    \State $\bar{u}\leftarrow \bar{u}-\frac{T-t}{M^{n-l}}\sum_{i=1}^{M^{n-l}}
    \varrho^{-1}\big(\frac{\cR(i)-t}{T-t}\big)
    f\big(\cR(i),X_2(i),\operatorname{MLP}(\cR(i),X_2(i),l-1,M,N)\big)
    \big(1,V_2(i)\big)$;
    \EndIf
\EndFor\\
\quad\:   \Return{$\bar{u}$;}
\EndFunction
\end{algorithmic}
}
\end{algorithm}

\clearpage

\subsection{\textbf{A Numerical Example}}
\label{section example}
In this subsection, we present a numerical example\footnote{All numerical experiments have been implemented in \texttt{Python} on an average laptop (AMD Ryzen 7 5800H with Radeon Graphics, 3.20 GHz, 8 Cores, 16 Logical Processors). The code for the MLP algorithm can be found here: 
\url{https://github.com/SizhouWu/MLP_Gradient_Nonlinearity}.}
to illustrate the applicability of our  
MLP approximation algorithm \eqref{def MLP euler} in approximately solving 
high-dimensional semilinear PDEs 
of the form \eqref{APIDE}--\eqref{APIDE initial}. 
To this end, we consider the following semilinear PDE on $(0,T)\times\bR^d$
\begin{align}
&
\frac{\partial}{\partial t}u^d(t,x)
+\tilde{\mu}\<x,\nabla_x u^d(t,x)\>
+\frac{1}{2}\tilde{\sigma}^2\sum_{i=1}^d\frac{\partial^2}{\partial x_i^2}u^d(t,x)
+\tilde{f}^d\big(\nabla_xu^d(t,x)\big)
=0
\label{simple_PDE}
\end{align}
with terminal condition $u^d(T,x)=\tilde{g}^d(x)$
for all $x\in \bR^d$,  
where $\tilde{\mu}\in[0,\infty)$ and $\tilde{\sigma}\in(0,\infty)$
are constants,
the terminal condition $\tilde{g}^d:\bR^d\to\bR$ is given by
\begin{equation}
\label{def g fin}
\tilde{g}^d(y)
:=\max\Big\{0,\max_{i\in[1,d]\cap \bN}y_i-K_1\Big\}
-2\max\Big\{0,\max_{i\in[1,d]\cap \bN}y_i-K_2\Big\},
\quad
y=(y_1,y_2,\dots,y_d)\in\bR^d
\end{equation} 
with $K_1\in[0,\infty)$ and $K_2\in(K_1,\infty)$,
and where the function 
$\tilde{f}^d:\bR^d\to\bR$ 
describing the nonlinearity is defined by
\begin{equation}
\label{def f fin}
\tilde{f}^d(v):= 
Ld^{-1}\max\Big\{0,\max_{i\in\bN\cap[1,d]}|v_i|-K_0\Big\},
\quad
v=(v_1,v_2,\dots,v_d)\in\bR^d,
\end{equation}
with $K_0\in[0,\infty)$.
Then one can easily verify that \eqref{def g fin}
and \eqref{def f fin} satisfy
Assumptions \ref{assumption Lip and growth}, \ref{assumption bbd partials global}, 
\ref{assumption ellip}, and \ref{assumption gradient} with $\mu^d(x):=\tilde{\mu}x$,
$\sigma^d(t,x):=\tilde{\sigma}\mathbf{I}_d$, $g^d:=\tilde{g}^d$,
and $f^d:=\tilde{f}^d$ in the notation of Section \ref{section setting}.

For the numerical example, we choose 
$T=0.25$, $\tilde{\mu}=0.06$, $\tilde{\sigma}=0.2$, $K_0=25$,
$K_1=95$, $K_2=120$ for the parameters in PDE \eqref{simple_PDE}, 
and aim to approximate the solution $u^d(0,x)$ of PDE \eqref{simple_PDE}
for $x=(100,\dots,100)\in\bR^d$ and $d\in\{10,100,200,300\}$.
We take $N=12$ for each $d\in\{10,100,200,300\}$ 
and level $n=M\in\{1,2,\dots,5\}$, and run our MLP algorithm \eqref{def MLP euler}
10 times to approximate $u^d(0,100,\dots,100)$. 
The numerical results are collected 
in Table \ref{simple_eg} and Figure \ref{figure std devi}.

\begin{table}[h!]
		\begin{tabular}{rl|R{2cm}R{2cm}R{2cm}R{2cm}R{2cm}|}
			& & \multicolumn{5}{c|}{Level} \\
			$d$ & & $M = n = 1$ & $M = n = 2$ & $M = n = 3$ & 
			$M = n = 4$ & $M = n = 5$ \\
			\hline
			10 & Avg. Sol. & $6.6832$ & $6.7109$ & $6.6781$ & $6.6681$ & $6.6645$ \\ 
 & \textit{Std. Dev.} & \textit{0.0682} & \textit{0.0715} & \textit{0.0369} & \textit{0.0031} & \textit{0.0017} \\ 
 & Avg. Eval. & $4.37 \cdot 10^{2}$ & $5.54 \cdot 10^{3}$ & $1.20 \cdot 10^{5}$ & $3.46 \cdot 10^{6}$ & $1.28 \cdot 10^{8}$ \\ 
 & \textit{Avg. Time} & \textit{0.0016} & \textit{0.0157} & \textit{0.3743} & \textit{11.4826} & \textit{421.8852} \\ 
\hline 
100 & Avg. Sol. & $6.7766$ & $6.8193$ & $6.7645$ & $6.7632$ & $6.7628$ \\ 
 & \textit{Std. Dev.} & \textit{0.0327} & \textit{0.0595} & \textit{0.0549} & \textit{0.0044} & \textit{0.0010} \\ 
 & Avg. Eval. & $1.82 \cdot 10^{3}$ & $3.01 \cdot 10^{4}$ & $6.28 \cdot 10^{5}$ & $1.81 \cdot 10^{7}$ & $6.67 \cdot 10^{8}$ \\ 
 & \textit{Avg. Time} & \textit{0.0051} & \textit{0.0561} & \textit{1.2532} & \textit{36.8444} & \textit{1388.9666} \\ 
\hline 
200 & Avg. Sol. & $6.7897$ & $6.8063$ & $6.7966$ & $6.7933$ & $6.7871$ \\ 
 & \textit{Std. Dev.} & \textit{0.0201} & \textit{0.0341} & \textit{0.0122} & \textit{0.0145} & \textit{0.0012} \\ 
 & Avg. Eval. & $3.25 \cdot 10^{3}$ & $5.46 \cdot 10^{4}$ & $1.21 \cdot 10^{6}$ & $3.44 \cdot 10^{7}$ & $1.27 \cdot 10^{9}$ \\ 
 & \textit{Avg. Time} & \textit{0.0151} & \textit{0.1928} & \textit{4.1178} & \textit{116.2634} & \textit{4358.3818} \\ 
\hline 
300 & Avg. Sol. & $6.7958$ & $6.8322$ & $6.7933$ & $6.8009$ & $6.7998$ \\ 
 & \textit{Std. Dev.} & \textit{0.0341} & \textit{0.0317} & \textit{0.0179} & \textit{0.0020} & \textit{0.0010} \\ 
 & Avg. Eval. & $5.53 \cdot 10^{3}$ & $8.20 \cdot 10^{4}$ & $1.77 \cdot 10^{6}$ & $5.06 \cdot 10^{7}$ & $1.86 \cdot 10^{9}$ \\ 
 & \textit{Avg. Time} & \textit{0.0781} & \textit{0.7497} & \textit{12.8255} & \textit{324.7100} & \textit{11556.8760} \\ 
\hline 
		\end{tabular}
		\caption{MLP solutions of PDE \eqref{simple_PDE}, 
		for different $d \in \mathbb{N}$ and $M = n \in \lbrace 1,...,5 \rbrace$. 
		The average solution (Avg.~Sol.), 
		the sample standard deviation (Std.~Dev.), 
		the average time (Avg.~Time [in seconds]), 
		and the average number of function evaluations 
		over $10$ independent runs
		$\mathfrak{C}^{(d)}_{n,M}$ (Avg.~Eval., where each evaluation 
		of $\mu^d$, $\sigma^d$, $f^d$, and $g^d$, 
		and the generation of a one-dimensional random variable 
		is counted as one unit) are computed over the 10 runs of the algorithm.}
		\label{simple_eg}
	\end{table}
	
\clearpage
	
\begin{figure}[h] 
    \centering
    \includegraphics[width=0.75\textwidth, height=7.5cm]{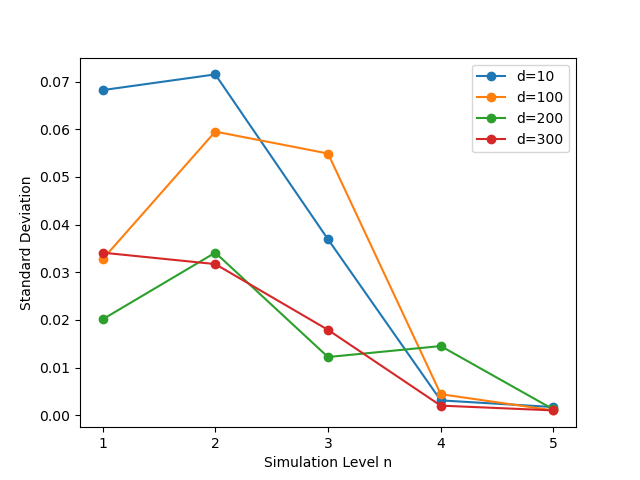}     
    \caption{Standard deviations of MLP solutions of PDE \eqref{simple_PDE}
    with different simulation levels}
    \label{fig:example}
    \label{figure std devi}
\end{figure}

Note that for each $d\in\{10,100,200,300\}$ 
the average solution $\bar{U}^d(0,x)$ is calculated by
$
\bar{U}^d(0,x):=\frac{\sum_{i=1}^{10} U^{d,i}(0,x)}{10},
$
where $U^{d,i}(0,x)$, $i\in\{1,2,\dots,10\}$, are $10$ independent realizations of 
the MLP approximation \eqref{def MLP euler}.
Moreover, for each $d\in\{10,100,200,300\}$ the standard deviation 
of the realizations is calculated by
$
\sqrt{\frac{\sum_{i=1}^{10}|U^{d,i}(0,x)-\bar{U}^d(0,x)|^2}{10}}
$.

\section{\textbf{Stochastic fixed-point equations}}
\label{section FP}
In this section, we show in Proposition \ref{Prop FP} below 
the existence and uniqueness of the solution 
of some stochastic fixed-point equation (see \eqref{FP gradient} below).
This fixed-point will be used to construct a viscosity solution
of PDE \eqref{APIDE} in Section \ref{section PIDE}.
Furthermore, we also present in Lemma \ref{lemma perturbation} below 
a perturbation result for the stochastic fixed-point equation mentioned above.
We first present a simple lemma which will be used in the calculation later on.
\begin{lemma}
\label{lemma time integral est}
It holds for all $t\in[0,T)$ and $\beta\in(0,\infty)$ that
\begin{equation}
\label{time integral est}
\int_t^T
(s-t)^{\frac{-(2+\beta)}{2(1+\beta)}}(T-s)^{\frac{-(2+\beta)}{2(1+\beta)}}\,ds
\leq (2(1+\beta)/\beta)2^{\frac{2+\beta}{1+\beta}}(T-t)^{\frac{-1}{1+\beta}}
\end{equation}
and
\begin{equation}
\label{time integral est 1/2}
\int_t^T
(s-t)^{-1/2}(T-s)^{-1/2}\,ds
\leq 4.
\end{equation}
\end{lemma}

\begin{proof}
We observe for all $t\in[0,T)$ and $\beta\in(0,\infty)$ that
\begin{align*}
&
\int_t^T
(s-t)^{\frac{-(2+\beta)}{2(1+\beta)}}(T-s)^{\frac{-(2+\beta)}{2(1+\beta)}}\,ds
\\
&
=
\int_t^{\frac{T+t}{2}}
(s-t)^{\frac{-(2+\beta)}{2(1+\beta)}}(T-s)^{\frac{-(2+\beta)}{2(1+\beta)}}\,ds
+\int_{\frac{T+t}{2}}^T
(s-t)^{\frac{-(2+\beta)}{2(1+\beta)}}(T-s)^{\frac{-(2+\beta)}{2(1+\beta)}}\,ds
\\
&
\leq
\Big(\frac{T-t}{2}\Big)^{\frac{-(2+\beta)}{2(1+\beta)}}
\int_t^{\frac{T+t}{2}}(s-t)^{\frac{-(2+\beta)}{2(1+\beta)}}\,ds
+
\Big(\frac{T-t}{2}\Big)^{\frac{-(2+\beta)}{2(1+\beta)}}
\int_{\frac{T+t}{2}}^T(T-s)^{\frac{-(2+\beta)}{2(1+\beta)}}\,ds
\\
&
=
(2(1+\beta)/\beta)2^{\frac{2+\beta}{1+\beta}}(T-t)^{\frac{-1}{1+\beta}}.
\end{align*}
Similarly, we obtain \eqref{time integral est 1/2},
which completes the proof of this lemma.
\end{proof}

\begin{proposition}[Existence and uniqueness of stochastic fixed-point]
\label{Prop FP}
Let $d\in \bN$, $a,b,b_1,c,T,p\in(0,\infty)$, 
and for each $t\in[0,T]$, $s\in[t,T]$, and $x\in\bR^d$ 
let $\bX^{t,x}_s:\Omega\to \bR^d$ be a random variable.  
For every nonnegative Borel function
$\varphi:[0,T]\times\bR^d\to [0,\infty)$, we assume that the mapping
$
\{(t,s)\in[0,T]^2:t\leq s\}\times \bR^d \ni (t,s,x)
\mapsto \bE[\varphi(s,\bX^{t,x}_s)]\in[0,\infty] 
$
is measurable. Let $F:[0,T]\times \bR^d \times \bR \times \bR^d \to \bR$
and $G:\bR^d\to\bR$ be Borel functions, and assume for all
$t\in[0,T]$, $s\in[t,T]$, $x,x'\in\bR^d$, $y,y'\in\bR$, and $v,v'\in\bR^d$ that
\begin{equation}                                             \label{est f g Z}
|F(t,x,0,\mathbf{0})|\leq a(d^p+\|x\|^2)^{1/2}, 
\quad |G(x)|\leq a(d^p+\|x\|^2)^{1/2},
\quad \big\|\big(d^p+\|\bX^{t,x}_s\|^2\big)\big\|_{L_1}
\leq b^2(d^p+\|x\|^2),
\end{equation}
\begin{equation}
\label{est Z q}
\big\|\big(d^p+\|\bX^{t,x}_s\|^2\big)\big\|_{L_q}
\leq
b^2_0(d^p+\|x\|^2),
\end{equation}
and
\begin{equation}                                          \label{Lip f FP}
|G(x)-G(x')|\leq c\|x-x'\|, \quad
|F(t,x,y,v)-F(t,x',y',v')|\leq c(\|x-x'\|+|y-y'|+\|v-v'\|).
\end{equation}
Moreover, for every $(t,x)\in[0,T)\times\bR^d$, 
let $\bV^{t,x}:(t,T]\times\Omega\to\bR^d$ be a stochastic process such that
\begin{equation}                                          
\label{est V t x s}
\big\|\bV^{t,x}_s\big\|_{L_2}^2\leq C_{d,T}(s-t)^{-1},
\quad s\in(t,T]
\end{equation}
with a positive constant $C_{d,T}$ only depending on $d$ and $T$.
Also assume that there exists a positive constant $\alpha_{d,T}$  
only depending on $d$ and $T$ such that it holds for all
$t\in[0,T)$, $t'\in[t,T)$ $s\in(t',T]$, and $x,x'\in\bR^d$ that
\begin{equation}
\label{cont est V}
\big\|\bV^{t,x}_s-\bV^{t',x'}_s\big\|_{L_2}^2
\leq 
\frac{\alpha_{d,T}(t'-t)}{(s-t)(s-t')}
+\alpha_{d,T}(s-t')^{-1}\big[(t'-t)(d^p+\|x\|^2)+\|x-x'\|^2\big],
\end{equation}
and
\begin{equation}
\label{cont est Z}
\big\|\bX^{t,x}_s-\bX^{t',x'}_s\big\|_{L_2}^2
\leq 
\alpha_{d,T}(d^p+\|x\|^2)\big[(t'-t)+\|x-x'\|^2\big].
\end{equation}
Then the following holds.
\begin{enumerate}[(i)]
\item
There exists a unique pair of Borel functions $(u_1,u_2)$ with
$u_1\in C([0,T)\times\bR^d,\bR)$ and $u_2\in C([0,T)\times\bR^d,\bR^d)$ 
satisfying for all $(t,x)\in[0,T)\times\bR^d$ that
\begin{align*}
&
\big\|G(\bX^{t,x}_T)(1,\bV^{t,x}_T)\big\|_{L_1}
+
\int_t^T\big\|F\big(s,\bX^{t,x}_s,u_1(s,\bX^{t,x}_s),u_2(s,\bX^{t,x}_s)\big)
(1,\bV^{t,x}_s)\big\|_{L_1}\,ds
\\
&
+\sup_{(s,y)\in[0,T)\times\bR^d}
\left(\frac{|u_1(s,y)|+(T-s)^{1/2}\|u_2(s,y)\|}{(d^p+\|y\|^2)^{1/2}}
\right)<\infty,
\end{align*}
and
\begin{align}
&
\big(u_1(t,x),u_2(t,x)\big)
\nonumber\\
&
=\bE\left[G(\bX^{t,x}_T)\left(1,\bV^{t,x}_T\right)\right]  
+\int_t^T\bE\left[
F\big(s,\bX^{t,x}_s,u_1(s,\bX^{t,x}_s),u_2(s,\bX^{t,x}_s)\big)
\left(1,\bV^{t,x}_s\right)
\right]ds.
\label{FP gradient}
\end{align}
\item
It holds for all $t\in[0,T)$ that
\begin{align}
&
\sup_{r\in[t,T)}\sup_{x\in\bR^d}
\left[
\frac{|u_1(r,x)|+(T-r)^{1/2}\|u_2(r,x)\|}{(d^p+\|x\|^2)^{1/2}}
\right]
\nonumber\\
&
\leq
ab\Big[1+C_{d,T}^{1/2}+T+2C_{d,T}^{1/2}T\Big]
+
\exp\left\{
4bc(4+T)\big(1+C_{d,T}^{1/2}T^{1/2}\big)
\right\}
<\infty.
\label{LG FP u}
\end{align}
\end{enumerate}
\end{proposition}

\begin{proof}
Throughout this proof we denote $V$ by the space
\begin{align}
V:= \Big\{
&
\mathbf{v}=(v^1,v^2)\in C([0,T)\times\bR^d,\bR)\times C([0,T)\times\bR^d,\bR^d):
\nonumber\\
&
\sup_{(t,x)\in[0,T)\times\bR^d}
\frac{|v^1(t,x)|+(T-t)^{1/2}\|v^2(t,x)\|}{(d^p+\|x\|^2)^{1/2}}
<\infty
\Big\}.
\label{def V space}
\end{align}
For every $\lambda\in\bR$, let $\|\cdot\|_\lambda$ be a norm on $V$ such that
\begin{equation}                                          \label{def norm}
\|\mathbf{v}\|_\lambda:=\sup_{(t,x)\in[0,T)\times\bR^d}
\frac{e^{\lambda t}\big(|v^1(t,x)|+(T-t)^{1/2}\|v^2(t,x)\|\big)}
{(d^p+\|x\|^2)^{1/2}},
\quad \mathbf{v}\in V.
\end{equation}
It is easy to check that $(V,\|\cdot\|_\lambda)$ 
is a normed real vector space for each $\lambda\in\bR$. 
Then we aim to show that $(V,\|\cdot\|_0)$ is a real Banach space.
Let $\{\mathbf{v}_n\}_{n=1}^\infty=\{(v^1_n,v^2_n)\}_{n=1}^\infty\subset V$ 
be a Cauchy sequence in
$(V,\|\cdot\|_0)$. Then we have
\begin{equation}                                       \label{cauchy v}
\lim_{N\to\infty}\sup_{n,m\geq N}\|\mathbf{v}_n-\mathbf{v}_m\|_0=0.
\end{equation}
This together with \eqref{def norm} imply that there exists functions
$
\varphi:[0,T)\times\bR^d\to\bR
$
and
$
\psi:[0,T)\times\bR^d\to\bR^d
$
satisfying for all $(t,x)\in[0,T)\times\bR^d$ that
\begin{equation}                                     \label{conv cauchy v}
\lim_{n\to\infty}\big[
|v^1_n(t,x)-\varphi(t,x)|
+\|v^2_n(t,x)-\psi(t,x)\|
\big]=0.
\end{equation}
Furthermore, by \eqref{conv cauchy v} it holds for all $N\in\bN$ that
\begin{align}
&
\sup_{(t,x)\in[0,T)\times\bR^d}
\frac{|\varphi(t,x)|+(T-t)^{1/2}\|\psi(t,x)\|}{(d^p+\|x\|^2)^{1/2}}
\nonumber\\
&
=\sup_{(t,x)\in[0,T)\times\bR^d}
\frac{|\lim_{n\to\infty}v^1_n(t,x)|+(T-t)^{1/2}\|\lim_{n\to\infty} v^2_n(t,x)\|}
{(d^p+\|x\|^2)^{1/2}}
\nonumber\\
&
\leq \sup_{(t,x)\in[0,T)\times\bR^d}
\frac{\sup_{n\in\bN}\left[|v^1_n(t,x)|+(T-t)^{1/2}\|v^2_n(t,x)\|\right]}
{(d^p+\|x\|^2)^{1/2}}
\nonumber\\
&
=\sup_{n\in\bN}\|\mathbf{v}_n\|_0
\nonumber\\
&
\leq \sup_{n,m\geq N}\|\mathbf{v}_n-\mathbf{v}_m\|_0+\sup_{n\leq N}\|\mathbf{v}_n\|_0.
\label{norm varphi}
\end{align}
By \eqref{cauchy v} taking limit to \eqref{norm varphi} as $N\to\infty$ yields that
\begin{equation}
\label{finite norm phi psi}
\sup_{(t,x)\in[0,T)\times\bR^d}
\frac{|\varphi(t,x)|+(T-t)^{1/2}\|\psi(t,x)\|}{(d^p+\|x\|^2)^{1/2}}
\leq \sup_{n\in\bN}\|\mathbf{v}_n\|_0<\infty.
\end{equation}
Moreover, by \eqref{cauchy v} we also notice that
\begin{align}
&
\limsup_{n\to\infty}\|(\varphi,\psi)-\mathbf{v}_n\|_0
\nonumber\\
&
=\limsup_{n\to\infty}\sup_{(t,x)\in[0,T)\times\bR^d}
\frac{\lim_{m\to\infty}\big[|v^1_m(t,x)-v^1_n(t,x)|
+(T-t)^{1/2}\|v^2_m(t,x)-v^2_n(t,x)\|\big]}{(d^p+\|x\|^2)^{1/2}}
\nonumber\\
&
\leq \limsup_{n\to\infty}\sup_{(t,x)\in[0,T)\times\bR^d}
\frac{\sup_{m\geq n}\big[|v^1_m(t,x)-v^1_n(t,x)|
+(T-t)^{1/2}\|v^2_m(t,x)-v^2_n(t,x)\|\big]}
{(d^p+\|x\|^2)^{1/2}}
\nonumber\\
&
= \limsup_{n\to\infty}\sup_{m\geq n}\|\mathbf{v}_m-\mathbf{v}_n\|_0
=0.
\label{conv v n}
\end{align}
This implies for all compact set $\cK\subseteq [0,T)\times\bR^d$ that
$$
\lim_{n\to\infty}\sup_{(t,x)\in\cK}
\left[|\varphi(t,x)-v^1_n|+\|\psi(t,x)-v^2_n(t,x)\|\right]
=0.
$$
Thus, the fact that $v^1_n$ and $v^2_n$ are continuous for all $n\in\bN$
ensures that 
$
\varphi:[0,T)\times\bR^d\to\bR
$
and
$
\psi:[0,T)\times\bR^d\to\bR^d
$ 
are also continuous.
This together with \eqref{finite norm phi psi} and \eqref{conv v n} ensure that 
$(V,\|\cdot\|_0)$ is a real Banach space.
We also notice that it holds
for all $\lambda_1\in\bR$, $\lambda_2\in[\lambda_1,\infty)$, and $\mathbf{v}\in V$ that
$\|\mathbf{v}\|_{\lambda_1}\leq \|\mathbf{v}\|_{\lambda_2}
\leq e^{(\lambda_2-\lambda_1)T}\|\mathbf{v}\|_{\lambda_1}$.
Considering this and the fact that $(V,\|\cdot\|_0)$ is a real Banach space, we have
that $(V,\|\cdot\|_\lambda)$ is a real Banach space for every $\lambda\in\bR$.

Next, by H\"older's inequality, \eqref{est f g Z}, and \eqref{est V t x s}
we observe for all $(t,x)\in[0,T)\times\bR^d$ that
\begin{align}
\big\|G\big(\bX^{t,x}_T\big)\big\|_{L_1}
&
\leq \bE\left[a\big(d^p+\big\|\bX^{t,x}_T\big\|^2\big)^{1/2}\right]
\leq ab(d^p+\|x\|^2)^{1/2}
<\infty,
\label{milk tea 1}
\end{align}
and
\begin{align}
(T-t)^{1/2}\big\|G\big(\bX^{t,x}_T\big)\bV^{t,x}_T\big\|_{L_1}
&
\leq
(T-t)^{1/2}
\big\|G\big(\bX^{t,x}_T\big)\big\|_{L_2}
\big\|\bV^{t,x}_T\big\|_{L_2}
\leq aC^{1/2}_{d,T}
\left(\bE\left[d^p+\big\|\bX^{t,x}_T\big\|^2\right]\right)^{1/2}
\nonumber\\
&
\leq abC^{1/2}_{d,T}(d^p+\|x\|^2)^{1/2}
<\infty.
\label{milk tea 2}
\end{align}
Furthermore, by H\"older's inequality, 
\eqref{est f g Z}, \eqref{Lip f FP}, and \eqref{est V t x s} we have for all 
$(t,x)\in[0,T)\times\bR^d$, and $\mathbf{v}=(v^1,v^2)\in V$ that
\begin{align}
&
(T-t)^{1/2}
\int_t^T\bE\left[
\big\|
F\big(
s,\bX^{t,x}_s,v^1(s,\bX^{t,x}_s),v^2(s,\bX^{t,x}_s)
\big)
\bV^{t,x}_s
\big\|
\right]ds
\nonumber\\
&
\leq
(T-t)^{1/2}
\int_t^T
\big\|
F\big(
s,\bX^{t,x}_s,v^1(s,\bX^{t,x}_s),v^2(s,\bX^{t,x}_s)
\big)
\big\|_{L_2}
\big\|\bV^{t,x}_s\big\|_{L_2}\,
ds
\nonumber\\
&
\leq 
(T-t)^{1/2}
C^{1/2}_{d,T}
\int_t^T(s-t)^{-1/2}\left(
\bE\left[
\big(
\big|F(s,\bX^{t,x}_s,0,\mathbf{0})\big|
+c\big|v^1(s,\bX^{t,x}_s)\big|
+c\big\|v^2(s,\bX^{t,x}_s)\big\|
\big)^2
\right]
\right)^{1/2}ds
\nonumber\\
&
\leq 
(T-t)^{1/2}C^{1/2}_{d,T}\int_t^T(s-t)^{1/2}
\Bigg(
\bE\Bigg[
\frac
{\big|F(s,\bX^{t,x}_s,0,\mathbf{0})\big|^2}
{d^p+\|\bX^{t,x}_s\|^2}
\big(d^p+\|\bX^{t,x}_s\|^2\big)
\Bigg]
\Bigg)^{1/2}ds
+(T-t)^{1/2}C^{1/2}_{d,T}c(1+T^{1/2})
\nonumber\\
& \quad
\cdot
\int_t^T(s-t)^{-1/2}(T-s)^{-1/2}
\left(
\bE\left[
\frac
{\big(\big|v^1(s,\bX^{t,x}_s)\big|+(T-s)\big\|v^2(s,\bX^{t,x}_s)\big\|\big)^2}
{d^p+\|\bX^{t,x}_s\|^2}
\big(d^p+\|\bX^{t,x}_s\|^2\big)
\right]
\right)^{1/2}ds
\nonumber\\
&
\leq
(T-t)^{1/2}C^{1/2}_{d,T}
\Bigg[
\sup_{t\in[t,T]}\sup_{y\in\bR^d}
\frac{\big|F(r,y,0,\mathbf{0})\big|}{(d^p+\|y\|^2)^{1/2}}
\Bigg]
\int_t^T(s-t)^{-1/2}\Big(
\bE\Big[
d^p+\|\bX^{t,x}_s\|^2
\Big]
\Big)^{1/2}\,ds
\nonumber\\
& \quad
+(T-t)^{1/2}C^{1/2}_{d,T}c(1+T^{1/2})
\int_t^T(s-t)^{-1/2}(T-s)^{-1/2}
\nonumber\\
& \quad
\cdot
\Bigg[
\sup_{r\in[s,T)}\sup_{y\in\bR^d}
\frac{|v^1(r,y)|+(T-r)^{1/2}\|v^2(r,y)\|}{(d^p+\|y\|^2)^{1/2}}
\Bigg]
\Big(
\bE\Big[d^p+\big\|\bX^{t,x}_s\big\|^2\Big]
\Big)^{1/2}
\,ds
\nonumber\\
&
\leq 
2abC_{d,T}^{1/2}(T-t)(d^p+\|x\|^2)^{1/2}
+(T-t)^{1/2}C^{1/2}_{d,T}bc(1+T^{1/2})(d^p+\|x\|^2)^{1/2}
\nonumber\\
& \quad
\cdot
\int_t^T(s-t)^{-1/2}(T-s)^{-1/2}
\Bigg[
\sup_{r\in[s,T)}\sup_{y\in\bR^d}
\frac{|v^1(r,y)|+(T-r)^{1/2}\|v^2(r,y)\|}{(d^p+\|y\|^2)^{1/2}}
\Bigg]
\,ds
<\infty.
\label{milk tea 3}
\end{align}
Analogously, by H\"older's inequality, 
\eqref{est f g Z}, \eqref{Lip f FP}, and \eqref{est V t x s}
it holds for all 
$(t,x)\in[0,T)\times\bR^d$, and $\mathbf{v}=(v^1,v^2)\in V$ that
\begin{align}
&
\int_t^T
\big\|
F\big(s,\bX^{t,x}_s,v^1(s,\bX^{t,x}_s), v^2(s,\bX^{t,x}_s)\big)
\big\|_{L_1}
ds
\nonumber\\
&
\leq
ab(T-t)(d^p+\|x\|^2)^{1/2}
+bc(1+T^{1/2})(d^p+\|x\|^2)^{1/2}
\nonumber\\
& \quad
\cdot
\int_t^T
(T-s)^{-1/2}
\left[
\sup_{r\in[s,T)}\sup_{y\in\bR^d}
\frac{|v^1(r,y)|+(T-r)^{1/2}\|v^2(r,y)\|}
{(d^p+\|y\|^2)^{1/2}}
\right]
ds
<\infty,
\label{milk tea 4}
\end{align}
Then \eqref{milk tea 1}--\eqref{milk tea 4} ensure that 
for every $\mathbf{v}=(v^1,v^2)\in V$ we are allowed to define the mapping 
$\psi_{\mathbf{v}}:[0,T)\times\bR^d\to\bR^{d+1}$ by
\begin{align}
\psi_{\mathbf{v}}(t,x)
:=\bE\left[G\big(\bX^{t,x}_T\big)\Big(1,\bV^{t,x}_T\Big)\right]
+\int_t^T\bE\left[
F\big(s,\bX^{t,x}_s,v^1(s,\bX^{t,x}_s),v^2(s,\bX^{t,x}_s)\big)
\left(1,\bV^{t,x}_s\right)
\right]ds.
\label{def map psi}
\end{align}
By \eqref{est f g Z}, \eqref{Lip f FP}, \eqref{est V t x s}, \eqref{cont est V},
\eqref{cont est Z}, and Cauchy-Schwarz inequality, we obtain for all
$t\in(0,T]$, $t'\in[t,T)$, and $x,x'\in\bR^d$ that
\begin{align}
\big\|G\big(\bX^{t,x}_T\big)-G\big(\bX^{t',x'}_T\big)\big\|_{L_1}
\leq c\big\|\bX^{t,x}_T-\bX^{t',x'}_T\big\|_{L_2}
\leq
c\alpha^{1/2}(d^p+\|x\|^2)^{1/2}\big[(t'-t)+\|x-x'\|^2\big]^{1/2},
\label{black tea 1}
\end{align}
and
\begin{align}
&
\big\|G\big(\bX^{t,x}_T\big)\bV^{t,x}_T
-G\big(\bX^{t',x'}_T\big)\bV^{t',x'}_T\big\|_{L_1}
\nonumber\\
&
\leq
\big\|G\big(\bX^{t,x}_T\big)\big\|_{L_2}
\big\|\bV^{t,x}_T-\bV^{t',x'}_T\big\|_{L_2}
+
c\big\|\bX^{t,x}_T-\bX^{t',x'}_T\big\|_{L_2}
\big\|\bV^{t',x'}_T\big\|_{L_2}
\nonumber\\
& 
\leq
a\left(\bE\left[d^p+\|\bX^{t,x}_s\|^2\right]\right)^{1/2}
\big\|\bV^{t,x}_T-\bV^{t',x'}_T\big\|_{L_2}
+
cC_{d,T}^{1/2}(T-t')^{-1/2}
\big\|\bX^{t,x}_T-\bX^{t',x'}_T\big\|_{L_2}
\nonumber\\
&
\leq
ab(d^p+\|x\|^2)^{1/2}
\left(\frac{\alpha(t'-t)}{(T-t)(T-t')}+\alpha(T-t')^{-1}
\big[(t'-t)(d^p+\|x\|^2)+\|x-x'\|^2\big]\right)^{1/2}
\nonumber\\
& \quad
+
c(\alpha C_{d,T})^{1/2}(d^p+\|x\|^2)^{1/2}\big[(t'-t)+\|x-x'\|^2\big]^{1/2}.
\label{black tea 2}
\end{align}
Furthermore, by the triangle inequality and Minkowski's integral inequality
we have for all $t\in[0,T)$, $t'\in[t,T)$, $x,x'\in\bR^d$,
and $\fv=(v^1,v^2)\in V$ that
\begin{align}
&
\bigg\|
\int_t^T
\bE\left[
F\big(s,\bX^{t,x}_s,\fv(s,\bX^{t,x}_s)\big)\bV^{t,x}_s
\right]
ds
-
\int_{t'}^T
\bE\left[
F\big(s,\bX^{t',x'}_s,\fv(s,\bX^{t',x'}_s)\big)\bV^{t',x'}_s
\right]
ds
\bigg\|
\leq
\sum_{i=1}^3A_{\fv,i}(t,t',x,x'),
\label{ineq A v}
\end{align}
where 
\begin{align*}
&
A_{\fv,1}(t,t',x,x')
:=\int_t^{t'}
\big\|F\big(s,\bX^{t,x}_s,\fv(s,\bX^{t,x}_s)\big)\bV^{t,x}_s\big\|_{L_1}ds
\\
&
A_{\fv,2}(t,t',x,x')
:=\int_{t'}^T
\big\|
\big(\bV^{t,x}_s-\bV^{t',x'}_s\big)F\big(s,\bX^{t',x'}_s,\fv(s,\bX^{t',x'}_s)\big)
\big\|_{L_1}ds,
\end{align*}
and
$$
A_{\fv,3}(t,t',x,x'):=\int_{t'}^T
\big\|
\big[F\big(s,\bX^{t,x}_s,\fv(s,\bX^{t,x}_s)-
F\big(s,\bX^{t',x'}_s,\fv(s,\bX^{t',x'}_s)\big)\big]
\bV^{t,x}_s
\big\|_{L_1}ds.
$$
By \eqref{time integral est 1/2}, \eqref{est f g Z}, \eqref{Lip f FP},
\eqref{est V t x s}, and Cauchy-Schwarz inequality we have for all
$t\in[0,T)$, $t'\in[t,T)$, $x,x'\in\bR^d$, and $\fv=(v^1,v^2)\in V$ that
\begin{align}
&
A_{\fv,1}(t,t',x,x')
\nonumber\\
&
\leq
\int_t^{t'}
\big\|F\big(s,\bX^{t,x}_s,\fv(s,\bX^{t,x}_s)\big)\big\|_{L_2}
\|\bV^{t,x}_s\|_{L_2}
ds
\nonumber\\
&
\leq
C_{d,T}^{1/2}\int_t^{t'}(s-t)^{-1/2}
\bigg[
c\left(\bE\left[
\big(\big|v^1(s,\bX^{t,x}_s)\big|+\big\|v^2(s,\bX^{t,x}_s)\big\|\big)^2
\right]\right)^{1/2}
+
a\left(\bE\left[d^p+\|\bX^{t,x}_s\|^2\right]\right)^{1/2}
\bigg]ds
\nonumber\\
&
\leq
C_{d,T}^{1/2}\int_t^{t'}(s-t)^{-1/2}\Bigg[
\frac{c(1+T^{1/2})}{(T-s)^{1/2}}
\bigg(
\bE\bigg[
\frac
{\big(\big|v^1(s,\bX^{t,x}_s)\big|+(T-s)^{1/2}\big\|v^2(s,\bX^{t,x}_s)\big\|\big)^2}
{d^p+\|\bX^{t,x}_s\|^2}
\nonumber\\
& \quad
\cdot
\big(d^p+\|\bX^{t,x}_s\|^2\big)
\bigg]
\bigg)^{1/2}
+
\frac{T^{1/2}ab(d^p+\|x\|^2)^{1/2}}{(T-s)^{1/2}}
\Bigg]\,ds
\nonumber\\
&
\leq
C_{d,T}^{1/2}b[c+(a+c)T^{1/2}](d^p+\|x\|^2)^{1/2}\|\fv\|_0(T-t')^{-1/2}
\int_t^{t'}(s-t)^{-1/2}\,ds
\nonumber\\
&
=
2C_{d,T}^{1/2}b[c+(a+c)T^{1/2}](d^p+\|x\|^2)^{1/2}\|\fv\|_0(T-t')^{-1/2}(t'-t)^{1/2}.
\label{A v 1}
\end{align}
Furthermore, by \eqref{time integral est 1/2}, \eqref{est f g Z}, \eqref{Lip f FP},
\eqref{est V t x s}, \eqref{cont est V}, and Cauchy-Schwarz inequality
it holds for all $t\in[0,T)$, $t'\in[t,T)$, $x,x'\in\bR^d$, and
$\fv=(v^1,v^2)\in V$ that
\begin{align}
A_{\fv,2}(t,t',x,x')
&
\leq
\int_{t'}^T\big\|\bV^{t,x}_s-\bV^{t',x'}_s\big\|_{L_2}
\big\|F\big(s,\bX^{t',x'}_s,\fv(s,\bX^{t',x'}_s)\big)\big\|_{L_2}
\,ds
\nonumber\\
&
\leq
\int_{t'}^T
\left(
\bigg[\frac{\alpha(t'-t)}{(s-t)(s-t')}\bigg]^{1/2}
+\frac{\alpha^{1/2}[(t'-t)(d^p+\|x\|^2)+\|x-x'\|^2]^{1/2}}{(s-t')^{1/2}}
\right)
\nonumber\\
& \quad
\cdot
\left[
c\left(\bE\left[\big(\big
|v^1(s,\bX^{t',x'}_s)\big|+\|v^2(s,\bX^{t',x'}_s)\big\|
\big)^2\right]\right)^{1/2}
+
a\left(
\bE\left[d^p+\|\bX^{t',x'}_s\|^2\right]
\right)^{1/2}
\right]ds
\nonumber\\
&
\leq
\int_{t'}^T
\left(
\bigg[\frac{\alpha(t'-t)}{(s-t)(s-t')}\bigg]^{1/2}
+\frac{\alpha^{1/2}\big[(t'-t)(d^p+\|x\|^2)+\|x-x'\|^2\big]^{1/2}}{(s-t')^{1/2}}
\right)
\nonumber\\
& \quad
\cdot
\Bigg[
\frac{c(1+T^{1/2})}{(T-s)^{1/2}}
\left(
\bE\left[
\frac
{\big(\big|v^1(s,\bX^{t',x'}_s)\big|+(T-s)^{1/2}\big\|v^2(s,\bX^{t',x'}_s)\big\|\big)^2}
{d^p+\|\bX^{t',x'}_s\|^2}
\big(d^p+\|\bX^{t',x'}_s\|^2\big)
\right]
\right)^{1/2}
\nonumber\\
& \quad
+
\frac{T^{1/2}ab(d^p+\|x\|^2)^{1/2}}{(T-s)^{1/2}}
\Bigg]\,ds
\nonumber\\
&
\leq
b[c+(a+c)T^{1/2}](d^p+\|x\|^2)^{1/2}\|\fv\|_0
\Bigg(
\int_t^T\left[
\frac{\alpha(t'-t)\mathbf{1}_{(t',T)}(s)}{(s-t)(s-t')(T-s)}
\right]^{1/2}ds
\nonumber\\
& \quad
+
\alpha^{1/2}\big[(t'-t)(d^p+\|x\|^2)+\|x-x'\|^2\big]^{1/2}
\int_{t'}^T(s-t')^{-1/2}(T-s)^{-1/2}\,ds
\Bigg)
\nonumber\\
&
\leq
b[c+(a+c)T^{1/2}](d^p+\|x\|^2)^{1/2}\|\fv\|_0
\Bigg(
\int_t^T\left[
\frac{\alpha(t'-t)\mathbf{1}_{(t',T)}(s)}{(s-t)(s-t')(T-s)}
\right]^{1/2}ds
\nonumber\\
& \quad
+
4\big[(t'-t)(d^p+\|x\|^2)+\|x-x'\|^2\big]^{1/2}
\Bigg).
\label{ineq A v 2}
\end{align}
Here and later on, we use the convention $0/0:=0$ whenever such terms appear.
For each $t\in[0,T)$, $s\in(t,T)$, and $\{t_k\}_{k=1}^\infty\subseteq[t,T)$
with $|t_1-t|\leq (T-t)/2$ and $t_k\downarrow t$ as $k\to\infty$, we define
$$
R_{t,t_k}(s):=
\left(\frac{\alpha(t_k-t)\mathbf{1}_{(t_k,T]}(s)}{(s-t)(s-t_k)(T-s)}\right)^{1/2}.
$$
Then it holds for all $\beta\in(0,1)$, $t\in[0,T)$, $s\in(t,T)$, and $k\in\bN$ that
\begin{align*}
\int_t^TR_{t,t_k}^{1+\beta}(s)\,ds
&
\leq
\int_{t_k}^T\alpha^{\frac{1+\beta}{2}}
(s-t_k)^{-\frac{1+\beta}{2}}(T-s)^{-\frac{1+\beta}{2}}\,ds
\\
&
=
\alpha^{\frac{1+\beta}{2}}
\left(
\int_{t_k}^{\frac{T+t_k}{2}}
(s-t_k)^{-\frac{1+\beta}{2}}(T-s)^{-\frac{1+\beta}{2}}\,ds
+
\int_{\frac{T+t_k}{2}}^T
(s-t_k)^{-\frac{1+\beta}{2}}(T-s)^{-\frac{1+\beta}{2}}\,ds
\right)
\\
&
\leq
\alpha^{\frac{1+\beta}{2}}\Big(\frac{T-t_k}{2}\Big)^{-\frac{1+\beta}{2}}
\left(
\int_{t_k}^{\frac{T+t_k}{2}}(s-t_k)^{-\frac{1+\beta}{2}}\,ds
+
\int_{\frac{T+t_k}{2}}^T(T-s)^{-\frac{1+\beta}{2}}\,ds
\right)
\\
&
=
\frac{4\alpha^{\frac{1+\beta}{2}}}{1-\beta}\Big(\frac{T-t_k}{2}\Big)^{-\beta}
\leq 
\frac{4\alpha^{\frac{1+\beta}{2}}}{1-\beta}\Big(\frac{T-t}{4}\Big)^{-\beta}.
\end{align*}
Hence, for every $t\in[0,T)$, 
the family of $\cB\big((t,T)\big)/\cB(\bR)$-measurable functions
$\big(R_{t,t_k}(s)\big)_{s\in(t,T)}$, $k\in\bN$, are uniformly integrable
on $\big((t,T),\cB\big((t,T)\big),\cL(ds)\big)$,
where $\cL(ds)$ denotes the Lebesgue measure 
on $\big((t,T),\cB\big((t,T)\big)\big)$. 
Moreover, notice for every $t\in[0,T)$ and $s\in(t,T)$ that
$\lim_{k\to\infty}R_{t,t_k}(s)=0$.
Thus, for all $t\in[0,T)$ and $s\in(t,T]$ we have that
$$
\lim_{k\to\infty}\int_t^TR_{t,t_k}(s)\,ds
=
\int_t^T\lim_{k\to\infty}R_{t,t_k}(s)\,ds=0.
$$
This together with \eqref{ineq A v 2} imply for all $\fv\in V$, 
$(t,x)\in[0,T)\times\bR^d$, and $(t_k,x_k)_{k=1}^\infty\subseteq [t,T)\times\bR^d$
with $|t_1-t|\leq (T-t)/2$ and $t_k\downarrow t$, $x_k\to x$ as $k\to\infty$ that
\begin{equation}
\label{A v 2}
\lim_{k\to\infty}A_{\fv,2}(t,t_k,x,x_k)=0.
\end{equation}
In addition, by \eqref{Lip f FP}, \eqref{est V t x s}, and Cauchy-Schwarz inequality
we obtain for all $t\in[0,T)$, $t'\in[t,T)$, $x,x'\in\bR^d$, 
and $\fv=(v^1,v^2)\in V$ that
\begin{align}
A_{\fv,3}(t,t',x,x')
&
=
\int_t^T
\big\|\mathbf{1}_{(t',T)}(s)
\big[F\big(s,\bX^{t,x}_s,\fv(s,\bX^{t,x}_s)\big)-
F\big(s,\bX^{t',x'}_s,\fv(s,\bX^{t',x'}_s)\big)\big]
\bV^{t,x}_s\big\|_{L_1}ds
\nonumber\\
&
\leq
\int_t^T\big\|
\mathbf{1}_{(t',T)}(s)\big|F\big(s,\bX^{t,x}_s,\fv(s,\bX^{t,x}_s)\big)-
F\big(s,\bX^{t',x'}_s,\fv(s,\bX^{t',x'}_s)\big\|_{L_2}
\big\|\bV^{t,x}_s\big\|_{L_2}ds
\nonumber\\
&
\leq
C_{d,T}^{1/2}c^{1/2}\int_t^T\frac{\mathbf{1}_{(t',T)}(s)}{(s-t)^{1/2}}
\Big(
\bE\Big[
\big(
\big\|\bX^{t,x}_s-\bX^{t',x'}_s\big\|
+\big|v^1(s,\bX^{t,x}_s)-v^1(s,\bX^{t',x'}_s)\big|
\nonumber\\
& \quad
+\big\|v^1(s,\bX^{t,x}_s)-v^1(s,\bX^{t',x'}_s)\big\|
\big)^2
\Big]
\Big)^{1/2}\,ds
\nonumber\\
& 
\leq
C_{d,T}^{1/2}c^{1/2}(1+T^{1/2})\int_t^T
\frac{\mathbf{1}_{(t',T)}(s)}{(s-t)^{1/2}(T-s)^{1/2}}
\big\|h_{\fv}(t,t',x,x',s)\big\|_{L_2}\,ds,
\label{ineq A v 3}
\end{align}
where
\begin{align*}
&
h_{\fv}(t,t',x,x',s)
\\
&
:=
\big\|\bX^{t,x}_s-\bX^{t',x'}_s\big\|
+\big|v^1(s,\bX^{t,x}_s)-v^1(s,\bX^{t',x'}_s)\big|
+(T-s)^{1/2}\big\|v^1(s,\bX^{t,x}_s)-v^1(s,\bX^{t',x'}_s)\big\|.
\end{align*}
By \eqref{est Z q} it holds for all $k\in\bN$, $\beta\in(0,1)$, 
$\fv=(v^1,v^2)\in V$, $(t,x)\in[0,T)\times\bR^d$, $s\in(t,T)$,
and $(t_k,x_k)_{k=1}^\infty\subseteq[t,T)\times\bR^d$ 
with $\|x_k-x\|\leq 1$ for all $k\in\bN$ that
\begin{align}
&
\big\|h_{\fv}(t,t_k,x,x_k,s)\big\|_{L_{2+\beta}}^{2+\beta}
\nonumber\\
&
\leq
\bE\Bigg[
\bigg(
\big\|\bX^{t,x}_s-\bX^{t_k,x_k}_s\big\|
+
\frac{\big|v^1(s,\bX^{t,x}_s)\big|+(T-s)^{1/2}\big\|v^2(s,\bX^{t,x}_s)\big\|}
{\big(d^p+\|\bX^{t,x}_s\|^2\big)^{1/2}}
\big(d^p+\|\bX^{t,x}_s\|^2\big)^{1/2}
\nonumber\\
& \quad
+
\frac{\big|v^1(s,\bX^{t_k,x_k}_s)\big|+(T-s)^{1/2}\big\|v^2(s,\bX^{t_k,x_k}_s)\big\|}
{\big(d^p+\|\bX^{t_k,x_k}_s\|^2\big)^{1/2}}
\big(d^p+\|\bX^{t_k,x_k}_s\|^2\big)^{1/2}
\bigg)^{2+\beta}
\Bigg]
\nonumber\\
&
\leq
2^{1+\beta}(1+\|\fv\|_0)^{2+\beta}
\bE\left[
\big(d^p+\|\bX^{t,x}_s\|^2\big)^{1+\beta/2}
+\big(d^p+\|\bX^{t_k,x_k}_s\|^2\big)^{1+\beta/2}
\right]
\nonumber\\
&
\leq
2^{1+\beta}(1+\|\fv\|_0)^{2+\beta}b_0^{2+\beta}
\left[(d^p+\|x\|^2)^{1+\beta/2}+(d^p+(1+\|x\|)^2)^{1+\beta/2}\right]
<\infty,
\label{h 1+beta est}
\end{align}
which demonstrates that $h_{\fv}(t,t_k,x,x_k,s)$, $k\in\bN$, are uniformly
integrable random variables.
Furthermore, by \eqref{cont est Z} we have for all $(t,x)\in[0,T)\times\bR^d$,
$s\in(t,T)$, and $(t_k,x_k)_{k=1}^\infty\subseteq[t,T)\times\bR^d$ with
$t_k\downarrow t$ and $x_k\to x$ as $k\to\infty$ that
$$
\lim_{k\to\infty}\big\|\bX^{t,x}_s-\bX^{t_k,x_k}_s\big\|_{L_2}=0.
$$
This together withy the continuity of $\fv=(v^1,v^2)\in V$ imply for all 
$\fv\in V$, $(t,x)\in[0,T)\times\bR^d$,
$s\in(t,T)$, and $(t_k,x_k)_{k=1}^\infty\subseteq[t,T)\times\bR^d$ with
$t_k\downarrow t$ and $x_k\to x$ as $k\to\infty$ that
$$
h_{\fv}(t,t_k,x,x_k,s)\to 0 \quad \text{as $k\to\infty$ in probability}.
$$
Combining this and the uniform integrability of
$\{h_{\fv}(t,t_k,x,x_k,s)\}_{k=1}^\infty$ obtained before yields for all
$(t,x)\in[0,T)\times\bR^d$, $s\in(t,T)$,
and $(t_k,x_k)_{k=1}^\infty\subseteq[t,T)\times\bR^d$ with $t_k\downarrow t$,
$x_k\to x$ as $k\to\infty$ for all $k\in\bN$ that
\begin{equation}
\label{h v conv 0}
\lim_{k\to\infty}\big\|h_{\fv}(t,t_k,x,x_k,s)\big\|_{L_2}=0.
\end{equation}
Moreover, \eqref{time integral est 1/2}, \eqref{ineq A v 3}, 
\eqref{h 1+beta est}, \eqref{h v conv 0}, and the dominated convergence theorem
ensure for all $\fv\in V$, $(t,x)\in[0,T)\times\bR^d$,
$s\in(t,T)$, and $(t_k,x_k)_{k=1}^\infty\subseteq[t,T)\times\bR^d$ with
$t_k\downarrow t$ and $x_k\to x$ as $k\to\infty$ that
\begin{align}
&
\lim_{k\to\infty}A_{\fv,3}(t,t_k,x,x_k)
\nonumber\\
&
\leq
C^{1/2}_{d,T}c^{1/2}(1+T^{1/2})\lim_{k\to\infty}
\int_t^T\frac{\mathbf{1}_{(t_k,T)}(s)}{(s-t)^{1/2}(T-s)^{1/2}}
\big\|h_{\fv}(t,t_k,x,x_k,s)\big\|_{L_2}\,ds
\nonumber\\
&
\leq
C_{d,T}^{1/2}c^{1/2}(1+T^{1/2})
\int_t^T(s-t)^{-1/2}(T-s)^{-1/2}
\lim_{k\to\infty}\big\|h_{\fv}(t,t_k,x,x_k,s)\big\|_{L_2}\,ds
=0.
\label{A v 3}
\end{align}
Combining \eqref{ineq A v}, \eqref{A v 1}, \eqref{A v 2}, and \eqref{A v 3} yields
for all $\fv=(v^1,v^2)\in V$, $(t,x)\in[0,T)\times\bR^d$,  
and $(t_k,x_k)_{k=1}^\infty\subseteq[t,T)\times\bR^d$ with
$t_k\downarrow t$ and $x_k\to x$ as $k\to\infty$ that
\begin{equation}
\label{RC int F}
\lim_{k\to\infty}
\bigg\|
\int_t^T
\bE\left[
F\big(s,\bX^{t,x}_s,\fv(s,\bX^{t,x}_s)\big)\bV^{t,x}_s
\right]
ds
-
\int_{t'}^T
\bE\left[
F\big(s,\bX^{t_k,x_k}_s,\fv(s,\bX^{t_k,x_k}_s)\big)\bV^{t_k,x_k}_s
\right]
ds
\bigg\|=0.
\end{equation}
Similarly, we have for all $\fv=(v^1,v^2)\in V$, $(t,x)\in[0,T)\times\bR^d$,  
and $(t_k,x_k)_{k=1}^\infty\subseteq[t,T)\times\bR^d$ with
$t_k\uparrow t$ and $x_k\to x$ as $k\to\infty$ that
\begin{equation*}
\lim_{k\to\infty}
\bigg\|
\int_t^T
\bE\left[
F\big(s,\bX^{t,x}_s,\fv(s,\bX^{t,x}_s)\big)\bV^{t,x}_s
\right]
ds
-
\int_{t'}^T
\bE\left[
F\big(s,\bX^{t_k,x_k}_s,\fv(s,\bX^{t_k,x_k}_s)\big)\bV^{t_k,x_k}_s
\right]
ds
\bigg\|=0.
\end{equation*}
This together with \eqref{RC int F} ensure for all $\fv\in V$ that the mapping
\begin{equation}
\label{cont int F}
[0,T)\times\bR^d\ni(t,x)
\longmapsto
\int_t^T
\bE\left[
F\big(s,\bX^{t,x}_s,\fv(s,\bX^{t,x}_s)\big)\bV^{t,x}_s
\right]
ds
\in \bR^d
\end{equation}
is continuous. 
Analogously, it holds for all $\fv\in V$ that the mapping
\begin{equation}
\label{cont int F 1}
[0,T)\times\bR^d\ni(t,x)
\longmapsto
\int_t^T
\bE\left[
F\big(s,\bX^{t,x}_s,\fv(s,\bX^{t,x}_s)\big)
\right]
ds
\in \bR
\end{equation}
is continuous.
Then combining \eqref{def map psi}, \eqref{black tea 1}, \eqref{black tea 2},
\eqref{cont int F}, and \eqref{cont int F 1} yields for all $\fv\in V$ 
that the mapping $\psi_{\fv}:[0,T)\times\bR^d\to\bR^{d+1}$ is continuous.
Therefore, there exists $\Phi: V\to V$ such that for all 
$\mathbf{v}=(v^1,v^2)\in V$ and $(t,x)\in[0,T)\times\bR^d$ that
\begin{align}
\big(\Phi(\mathbf{v})\big)(t,x)
=\bE\left[G(\bX^{t,x}_T)\left(1,\bV^{t,x}_T\right)\right]  
+\int_t^T\bE\left[
F\big(s,\bX^{t,x}_s,v^1(s,\bX^{t,x}_s),v^2(s,\bX^{t,x}_s)\big)
\left(1,\bV^{t,x}_s\right)
\right]ds.
\label{def mapping Phi}
\end{align}
Next, by Jensen's inequality, H\"older's inequality, 
\eqref{est f g Z}, and \eqref{Lip f FP}
we notice for all $(t,x)\in[0,T)\times\bR^d$, 
$\lambda\in(0,\infty)$, $\beta\in(0,1)$,
$\mathbf{v}:=(v^1,v^2)\in V$, and $\mathbf{w}:=(w^1,w^2)\in V$ that
\begin{align}
&
\int_t^T
\left\|
F\big(s,\bX^{t,x}_s,v^1(s,\bX^{t,x}_s),v^2(s,\bX^{t,x}_s)\big)
-
F\big(s,\bX^{t,x}_s,w^1(s,\bX^{t,x}_s),w^2(s,\bX^{t,x}_s)\big)
\right\|_{L_1}
ds
\nonumber\\
&
\leq c\int_t^T
\left(\big\|v^1(s,\bX^{t,x}_s)-w^1(s,\bX^{t,x}_s)\big\|_{L_1}
+
\big\|
v^2(s,\bX^{t,x}_s)-w^2(s,\bX^{t,x}_s)
\big\|_{L_1}
\right)ds
\nonumber\\
&
\leq
c(1+T^{1/2})\int_t^T\bE\Bigg[
\frac
{e^{\lambda s}\big(\big|v^1(s,\bX^{t,x}_s)-w^1(s,\bX^{t,x}_s)\big|
+
(T-s)^{1/2}\big\|
v^2(s,\bX^{t,x}_s)-w^2(s,\bX^{t,x}_s)
\big\|\big)}
{\big(d^p+\big\|\bX^{t,x}_s\big\|^2\big)^{1/2}}
\nonumber\\
& \quad \cdot
\big(d^p+\big\|\bX^{t,x}_s\big\|^2\big)^{1/2}
\Bigg]e^{-\lambda s}(T-s)^{-1/2}\,ds
\nonumber\\
&
\leq c(1+T^{1/2})\int_t^T
\|\mathbf{v}-\mathbf{w}\|_\lambda\cdot
\bE
\left[\big(d^p+\big\|\bX^{t,x}_s\big\|^2\big)^{1/2}\right]
e^{-\lambda s}(T-s)^{-1/2}\,ds
\nonumber\\
&
\leq c(1+T^{1/2})\|\mathbf{v}-\mathbf{w}\|_\lambda\cdot b(d^p+\|x\|^2)^{1/2}
\left(\int_t^Te^{-(2+\beta)\lambda s}\,ds\right)^{\frac{1}{2+\beta}}
\left(
\int_t^T(T-s)^{\frac{-(2+\beta)}{2(1+\beta)}}\,ds
\right)^{\frac{1+\beta}{2+\beta}}
\nonumber\\
&
\leq 
\frac{bc(1+T^{1/2})(2(1+\beta)/\beta)^{\frac{1+\beta}{2+\beta}}(T-t)^{\frac{\beta}{2(2+\beta)}}}
{\lambda^{\frac{1}{2+\beta}}}
\|\mathbf{v}-\mathbf{w}\|_\lambda\cdot
(d^p+\|x\|^2)^{1/2}e^{-\lambda t}.
\label{FP est 1}
\end{align}
Analogously, by Minkowski's integral inequality, H\"older inequality,
\eqref{time integral est}, \eqref{est f g Z}, \eqref{Lip f FP}, 
and \eqref{est V t x s} it holds for all
$(t,x)\in[0,T)\times\bR^d$, $\lambda\in(0,\infty)$, $\beta\in(0,1)$,
$\mathbf{v}:=(v^1,v^2)\in V$, and $\mathbf{w}:=(w^1,w^2)\in V$ that
\begin{align}
&
(T-t)^{1/2}
\left\|
\int_t^T\bE\left[
\left(
F\big(s,\bX^{t,x}_s,v^1(s,\bX^{t,x}_s),v^2(s,\bX^{t,x}_s)\big)
-
F\big(s,\bX^{t,x}_s,w^1(s,\bX^{t,x}_s),w^2(s,\bX^{t,x}_s)\big)
\right)\bV^{t,x}_s
\right]ds
\right\|
\nonumber\\
&
\leq 
(T-t)^{1/2}
\int_t^T
\big\|
F\big(s,\bX^{t,x}_s,v^1(s,\bX^{t,x}_s),v^2(s,\bX^{t,x}_s)\big)
-
F\big(s,\bX^{t,x}_s,w^1(s,\bX^{t,x}_s),w^2(s,\bX^{t,x}_s)\big)
\big\|_{L_2}
\big\|\bV^{t,x}_s\big\|_{L_2}\,
ds
\nonumber\\
& 
\leq
(T-t)^{1/2}c(1+T^{1/2})C^{1/2}_{d,T}
\int_t^T
(s-t)^{-1/2}(T-s)^{-1/2}
\nonumber\\
& \quad \cdot
\left(
\bE\left[
\left(
\left|
v^1(s,\bX^{t,x}_s)-w^1(s,\bX^{t,x}_s)
\right|
+
(T-s)^{1/2}
\left\|
v^2(s,\bX^{t,x}_s)-w^2(s,\bX^{t,x}_s)
\right\|
\right)^2
\right]
\right)^{1/2}\,ds
\nonumber\\
&
\leq
(T-t)^{1/2}c(1+T^{1/2})C_{d,T}^{1/2}
\left(
\int_t^T(s-t)^{-\frac{2+\beta}{2(1+\beta)}}(T-s)^{-\frac{2+\beta}{2(1+\beta)}}\,ds
\right)^{\frac{1+\beta}{2+\beta}}
\nonumber\\
& \quad
\cdot
\left(
\int_t^T\left(
\bE\left[
\left(
\left|
v^1(s,\bX^{t,x}_s)-w^1(s,\bX^{t,x}_s)
\right|
+
(T-s)^{1/2}
\left\|
v^2(s,\bX^{t,x}_s)-w^2(s,\bX^{t,x}_s)
\right\|
\right)^2
\right]
\right)^{\frac{2+\beta}{2}}ds
\right)^{\frac{1}{2+\beta}}
\nonumber\\
&
\leq
(T-t)^{1/2}
cC_{d,T}^{1/2}(T-t)^{\frac{\beta}{2(2+\beta)}}
\Bigg(\int_t^T e^{-(2+\beta)\lambda s}
\nonumber\\
& \quad \cdot
\Bigg(
\bE\Bigg[
\frac
{
e^{2\lambda s}
\big(
\big|
v^1(s,\bX^{t,x}_s)-w^1(s,\bX^{t,x}_s)
\big|
+(T-s)^{1/2}
\big\|
v^2(s,\bX^{t,x}_s)-w^2(s,\bX^{t,x}_s)
\big\|
\big)^2
}
{d^p+\big\|\bX^{t,x}_s\big\|^2}
\big(d^p+\big\|\bX^{t,x}_s\big\|^2\big)
\Bigg]
\Bigg)^{\frac{2+\beta}{2}}
ds\Bigg)^{\frac{1}{2+\beta}}
\nonumber\\
&
\leq 
c(1+T^{1/2})C_{d,T}^{1/2}
2(2(1+\beta)/\beta)^{\frac{1+\beta}{2+\beta}}(T-t)^{\frac{\beta}{2(2+\beta)}}
\left(
\int_t^T
e^{-(2+\beta)\lambda s}
\|\mathbf{v}-\mathbf{w}\|_\lambda^{2+\beta}
\left(
\bE\left[
d^p+\big\|\bX^{t,x}_s\big\|^2
\right]
\right)^{\frac{2+\beta}{2}}
ds
\right)^{\frac{1}{2+\beta}}
\nonumber\\
&
\leq
\frac
{bc(1+T^{1/2})C_{d,T}^{1/2}
2(2(1+\beta)/\beta)^{\frac{1+\beta}{2+\beta}}(T-t)^{\frac{\beta}{2(2+\beta)}}}
{\lambda^{\frac{1}{2+\beta}}}
\|\mathbf{v}-\mathbf{w}\|_\lambda
\cdot
(d^p+\|x\|^2)^{1/2}e^{-\lambda t}.
\label{FP est 2}
\end{align}
Then combining \eqref{def mapping Phi}, \eqref{FP est 1}, \eqref{FP est 2},
and Minkowski's integral inequality
yields for all $\lambda\in(0,\infty)$ and $\mathbf{v},\mathbf{w}\in V$ that
\begin{equation}
\label{contraction mapping}
\|\Phi(\mathbf{v})-\Phi(\mathbf{w})\|_{\lambda}
\leq
\frac
{bc\big(1+T^{1/2}\big)\big(1+2C_{d,T}^{1/2}\big)
(2(1+\beta)/\beta)^{\frac{1+\beta}{2+\beta}}T^{\frac{\beta}{2(2+\beta)}}}
{\lambda^{\frac{1}{2+\beta}}}
\|\mathbf{v}-\mathbf{w}\|_{\lambda}.
\end{equation}
If we fix a $\beta\in(0,1)$, \eqref{contraction mapping} implies for all
$\lambda\geq \big[bc\big(1+T^{1/2}\big)\big(1+2C_{d,T}^{1/2}\big)\big]^{2+\beta}
(2(1+\beta)/\beta)^{1+\beta}T^{\beta/2}$ that
$\Phi$ forms a contraction mapping on $(V,\|\cdot\|_{\lambda})$.
Hence, by Banach's fixed-point theorem there exists a unique 
$\mathbf{u}=(u_1,u_2)\in V$ such that $\Phi(\mathbf{u})=\mathbf{u}$. 
This together with \eqref{def V space}, 
\eqref{milk tea 1}--\eqref{milk tea 4},
and \eqref{def mapping Phi} establish (i).
Furthermore, by \eqref{def mapping Phi} and Minkowski's integral inequality
we have for all $t\in[0,T)$ that
\begin{align*}
&
\sup_{r\in[t,T)}\sup_{x\in\bR^d}
\frac{|u_1(r,x)|+(T-r)^{1/2}\|u_2(r,x)\|}{(d^p+\|x\|^2)^{1/2}}
\\
&
\leq
\sup_{r\in[t,T)}\sup_{x\in\bR^d}
\frac
{
\bE\left[\big|G\big(\bX^{r,x}_T\big)\big|\right]
+(T-r)^{1/2}\int_r^T\bE\left[
\big|
F\big(
s,\bX^{r,x}_s,u_1(s,\bX^{r,x}_s),u_2(s,\bX^{r,x}_s)
\big)
\big|
\right]ds
}
{(d^p+\|x\|^2)^{1/2}}
\\ 
& \quad
+\sup_{r\in[t,T)}\sup_{x\in\bR^d}
\frac
{
\bE\left[\big\|G\big(\bX^{r,x}_T\big)\bV^{r,x}_T\big\|\right]
+(T-r)^{1/2}\int_r^T\bE\left[
\big\|
F\big(
s,\bX^{r,x}_s,u_1(s,\bX^{r,x}_s),u_2(s,\bX^{r,x}_s)
\big)
\bV^{r,x}_s
\big\|
\right]ds
}
{(d^p+\|x\|^2)^{1/2}}.
\end{align*}
This together with \eqref{milk tea 1}--\eqref{milk tea 4} imply for all
$t\in[0,T)$ that
\begin{align*}
&
\sup_{r\in[t,T)}\sup_{x\in\bR^d}
\left[
\frac{|u_1(r,x)|+(T-r)^{1/2}\|u_2(r,x)\|}{(d^p+\|x\|^2)^{1/2}}
\right]
\\
&
\leq 
ab\Big[1+C_{d,T}^{1/2}+T+2C_{d,T}^{1/2}T\Big]
+
bc(1+T^{1/2})\big(1+C_{d,T}^{1/2}T^{1/2}\big)
\\
& \quad
\cdot 
\int_t^T
\left[  
(T-s)^{-1/2}+(s-t)^{-1/2}(T-s)^{-1/2}
\right]
\left[
\sup_{r\in[s,T)}\sup_{x\in\bR^d}
\frac{|u_1(r,x)|+(T-r)^{1/2}\|u_2(r,x)\|}{(d^p+\|x\|^2)^{1/2}}
\right]ds.
\end{align*}
Therefore, by \eqref{time integral est 1/2}, the fact that $\fv\in V$, 
and Gr\"onwall's lemma we obtain 
for all $t\in[0,T)$ that
\begin{align}
&
\sup_{r\in[t,T)}\sup_{x\in\bR^d}
\left[
\frac{|u_1(r,x)|+(T-r)^{1/2}\|u_2(r,x)\|}{(d^p+\|x\|^2)^{1/2}}
\right]
\leq 
ab\Big[1+C_{d,T}^{1/2}+T+2C_{d,T}^{1/2}T\Big]
\nonumber\\
& \quad
+
\exp\left\{
bc(1+T^{1/2})\big(1+C_{d,T}^{1/2}T^{1/2}\big)
\int_t^T
\left[  
(T-s)^{-1/2}+(s-t)^{-1/2}(T-s)^{-1/2}
\right]
ds
\right\}
\nonumber\\
&
\leq
ab\Big[1+C_{d,T}^{1/2}+T+2C_{d,T}^{1/2}T\Big]
+
\exp\left\{
4bc(4+T)\big(1+C_{d,T}^{1/2}T^{1/2}\big)
\right\}
<\infty.
\label{est gronwall FP}
\end{align}
This proves (ii). Hence the proof of this proposition is completed.
\end{proof}

\begin{remark}
Note that for each $T>0$ every bounded monotone continuous function 
$f:(0,T)\to\bR$ is uniformly continuous.
Hence there is a unique continuous extension $\bar{f}:[0,T]\to\bR$ of $f$ such that
$\bar{f}(s)=f(s)$ for all $s\in(0,T)$.
Then there is no obstacle for us to apply Gr\"onwall's lemma 
to obtain \eqref{est gronwall FP}.
\end{remark}

Then we apply Proposition \ref{Prop FP} to SDE \eqref{SDE} to obtain
the following corollary.

\begin{corollary}                                         
\label{corollary FP}
Let Assumptions \ref{assumption Lip and growth}, \ref{assumption ellip},
and \ref{assumption gradient} hold,
and let $d,N\in\bN$. 
For each $(t,x)\in[0,T]\times\bR^d$ 
let $\big(X^{d,0,t,x}_s\big)_{s\in[t,T]}$ and $\big(\cX^{d,0,t,x,N}_s\big)_{s\in[t,T]}$
be the stochastic processes defined in \eqref{SDE} and \eqref{Euler 1}, 
respectively, with $\theta=0$. 
Moreover, for each $(t,x)\in[0,T)\times\bR^d$ 
let $\big(V^{d,0,t,x}_s\big)_{s\in(t,T]}$ 
and $\big(\cV^{d,0,t,x,N}_s\big)_{s\in(t,T]}$ 
be the stochastic processes defined in
\eqref{def proc V} and \eqref{def euler proc V}, respectively, with $\theta=0$.
Then the following holds.
\begin{enumerate}[(i)]
\item
There exists a unique pair of Borel functions $(u^d,w^d)$ with
$u^d\in C([0,T)\times\bR^d,\bR)$ and $w^d\in C([0,T)\times\bR^d,\bR^d)$ 
satisfying for all $(t,x)\in[0,T)\times\bR^d$ that
\begin{align*}
&
\bE\left[\big\|g^d(X^{d,0,t,x}_T)(1,V^{d,0,t,x}_T)\big\|\right]
+\int_t^T\bE
\Big[\big\|
f^d\big(s,X^{d,0,t,x}_s,u^d(s,X^{d,0,t,x}_s),w^d(s,X^{d,0,t,x}_s)\big)
(1,V^{d,0,t,x}_s)
\big\|\Big]\,ds
\\
&
+\sup_{(s,y)\in[0,T)\times\bR^d}
\left(\frac{|u^d(s,y)|+(T-s)^{1/2}\|w^d(s,y)\|}{(d^p+\|y\|^2)^{1/2}}
\right)<\infty,
\end{align*}
and
\begin{align}
(u^d(t,x),w^d(t,x))
&
=\bE\left[g^d(X^{d,0,t,x}_T)(1,V^{d,0,t,x}_T)\right]  
\nonumber\\
& \quad    
+\int_t^T\bE\left[
f^d\big(s,X^{d,0,t,x}_s,u^d(s,X^{d,0,t,x}_s),w^d(s,X^{d,0,t,x}_s)\big)
(1,V^{d,0,t,x}_s)
\right]ds.
\label{FP gradient SDE}
\end{align}
\item
There exists a constant $C_{d,1}$ 
only depending on $d$, $\varepsilon_d$, $L$, and $T$
satisfying for all $t\in[0,T)$ that
\begin{align}                                              
\sup_{s\in[t,T)}\sup_{x\in\bR^d}
\frac{|u^d(s,x)|+(T-s)^{1/2}\|w^d(s,x)\|}{(d^p+\|x\|^2)^{1/2}}
\leq 
C_{d,1}.
\label{LG FP SDE}
\end{align}
\item
There exists a unique pair of Borel functions $(u^d_N,w^d_N)$ with
$u^d_N\in C([0,T)\times\bR^d,\bR)$ and $w^d_N\in C([0,T)\times\bR^d,\bR^d)$ 
satisfying for all $(t,x)\in[0,T)\times\bR^d$ that
\begin{align*}
&
\bE\left[\big\|g^d(\cX^{d,0,t,x,N}_T)(1,\cV^{d,0,t,x,N}_T)\big\|\right]
\\
&
+\int_t^T\bE
\Big[\big\|
f^d\big(s,\cX^{d,0,t,x,N}_s,u^d_N(s,\cX^{d,0,t,x,N}_s),w^d_N(s,\cX^{d,0,t,x,N}_s)\big)
(1,\cV^{d,0,t,x,N}_s)
\big\|\Big]\,ds
\\
&
+\sup_{(s,y)\in[0,T)\times\bR^d}
\left(\frac{|u^d(s,y)|+(T-s)^{1/2}\|w^d(s,y)\|}{(d^p+\|y\|^2)^{1/2}}
\right)<\infty,
\end{align*}
and
\begin{align}
(u^d_N(t,x),w^d_N(t,x))
&
=\bE\left[g^d(\cX^{d,0,t,x,N}_T)(1,V^{d,0,t,x,N}_T)\right]  
\nonumber\\
& \quad    
+\int_t^T\bE\left[
f^d\big(s,\cX^{d,0,t,x,N}_s,u^d_N(s,\cX^{d,0,t,x,N}_s),w^d_N(s,\cX^{d,0,t,x,N}_s)\big)
(1,\cV^{d,0,t,x,N}_s)
\right]ds.
\label{FP gradient Euler}
\end{align}
\item
There exists a constant $C_{d,2}$ 
only depending on $d$, $\varepsilon_d$, $L$, and $T$ 
satisfying for all $t\in[0,T)$ that
\begin{align}                                              
\sup_{s\in[t,T)}\sup_{x\in\bR^d}
\frac{|u^d_N(s,x)|+(T-s)^{1/2}\|w^d_N(s,x)\|}{(d^p+\|x\|^2)^{1/2}}
\leq 
C_{d,2}.
\label{LG FP Euler}
\end{align}
\end{enumerate}
\end{corollary}

\begin{proof}
We first notice that
by Corollary \ref{corollary prob 1}, for $d\in \bN$ the mapping
$\Lambda\times\bR^d\ni(t,s,x)
\mapsto\big(s,X^{d,0,t,x}_{s}\big)\in\cL_0(\Omega,[0,T]\times\bR^d)$
is continuous, where $\cL_0(\Omega,[0,T]\times\bR^d)$ denotes the metric space
of all measurable functions from $\Omega$ to $[0,T]\times\bR^d$ equipped with
the metric deduced by convergence in probability.
Moreover, notice that for all $d\in\bN$ and nonnegative Borel functions 
$\varphi:[0,T]\times\bR^d\to[0,\infty)$, the mapping
$
\cL_0(\Omega,[0,T]\times\bR^d)\ni Z \mapsto \bE\big[\varphi(Z)\big]\in [0,\infty]
$
is measurable.
Hence for all $d\in\bN$ and all nonnegative Borel functions
$\varphi:[0,T]\times\bR^d\to[0,\infty)$ it holds that the mapping
\begin{equation}                                  \label{measurability 1}
\Lambda\times\bR^d\ni(t,s,x)\mapsto
\bE\Big[\varphi\big(s,X^{d,0,t,x}_{s}\big)\Big]
\in[0,\infty]
\end{equation}
is measurable. Analogously, Corollary \ref{corollary prob 1}
ensures for all $d\in\bN$, $N\in\bN$ and all nonnegative Borel functions
$\varphi:[0,T]\times\bR^d\to[0,\infty)$ that the mapping
\begin{equation}                                  \label{measurability 2}
\Lambda\times\bR^d\ni(t,s,x)\mapsto
\bE\Big[\varphi\big(s,\cX^{d,0,t,x,N}_{s}\big)\Big]
\in[0,\infty]
\end{equation}
is measurable.
Then combining \eqref{assumption Lip f}, \eqref{assumption growth f g},
\eqref{measurability 1}, 
\eqref{q moment est SDE}, \eqref{L q continuity SDE}, 
\eqref{L2 est proc V}, and \eqref{L2 cont est proc V},
and applying Proposition \ref{Prop FP}
(with
$F\cal f^d$, $G\cal g^d$, $\bX^{t,x}_s\cal X^{d,0,t,x}_s$,
and $\bV^{t,x}_s\cal V^{d,0,t,x}_s$
in the notation of Proposition \ref{Prop FP}),
we obtain (i) and (ii).
Similarly, by \eqref{assumption Lip f}, \eqref{assumption growth f g},
\eqref{measurability 2},
\eqref{L q est Euler SDE}, \eqref{cV L q est 0}, \eqref{L2 cont est proc cV},
Proposition \ref{Prop FP}, and e.g., Lemma 3.10 in \cite{NW2022},
we get (iii) and (iv).
The proof of this lemma is completed.
\end{proof}

Next, we present a perturbation lemma for stochastic fixed-point equations
associated with different random variables, 
which will be applied to prove the main results, Theorems \ref{thm MLP conv}
and \ref{MLP complexity}, in Section \ref{section proof main}.

\begin{lemma}[Perturbation of stochastic fixed-point equations]                                           
\label{lemma perturbation}
Let $d\in\bN$, $a,a_1,b,c,L\in[0,\infty)$, 
$\alpha,\alpha_1,\beta,\delta,\gamma,\gamma_1,\kappa,\rho\in[0,\infty)$,
$T\in(0,\infty)$, $q\in(2,\infty)$, 
let $(\bX^{t,x,k}_{s})_{s\in[t,T]}:[t,T]\times\Omega\to\bR^d$,
$t\in[0,T]$, $x\in\bR^d$, $k\in\{1,2\}$, be 
$\cB([0,T])\otimes\cF/\cB(\bR^d)$-measurable functions,
let $F:[0,T]\times\bR^d\times\bR\times\bR^d\to\bR$, 
$G:\bR^d\to\bR$,
and 
$u_k:[0,T)\times\bR^d\to\bR$, 
$w_k:[0,T)\times\bR^d\to\bR^d$,
$k\in\{1,2\}$, be Borel functions.
For all $t\in[0,T]$, $s\in[t,T]$, and all Borel functions 
$h:\bR^d\times\bR^d\to[0,\infty)$ assume that
$
\bR^d\times\bR^d\ni(y_1,y_2)\mapsto\bE[h(\bX^{t,y_1,1}_{s},\bX^{t,y_2,1}_{s})]
\in[0,\infty]
$ 
is measurable. 
For every $(t,x)\in[0,T)\times\bR^d$ and $k\in\{1,2\}$, 
let $\bV^{t,x,k}:[t,T)\to\bR^d$ be a stochastic process such that
\begin{equation}                                          \label{est V t x k s}
\big\|\bV^{t,x,k}_s\big\|_{L_2}^2\leq C_{d,T}(s-t)^{-1}
\quad
\text{for all $s\in[t,T)$},
\end{equation}
where $C_{d,T}$ is a positive constant only depending on $d$ and $T$.
Moreover, for all $x,x',y,y'\in\bR^d$, $v,v'\in\bR$, $t\in[0,T)$, $s\in[t,T)$,
$r\in[s,T)$, $k\in\{1,2\}$, and all Borel functions 
$h:\bR^d\times\bR^d\to[0,\infty)$ assume that 
\begin{align}
&
\bX^{t,x,k}_{t}=x, 
\quad
\big\|\big(d^p+\big\|\bX^{t,x,k}_{s}\big\|^2\big)\big\|_{L_1}
\leq a(d^p+\|x\|^2),
\label{a b est}
\\
&
\big\|\big(d^p+\big\|\bX^{t,x,k}_{s}\big\|^2\big)\big\|_{L_{q/2}}
\leq a_1(d^p+\|x\|^2),
\quad
\big\|\big(d^p+\big\|\bX^{t,x,k}_{s}\big\|^2\big)\big\|_{L_2}
\leq a_2(d^p+\|x\|^2), 
\label{a b est 2 q}
\\
&
|G(x)|^2\leq b(d^p+\|x\|^2),
\quad
|F(t,x,v,y)|^2\leq b(d^p+\|x\|^2+|v|^2+\|y\|^2),
\label{growth F G}
\\
&
\big\|\bX^{t,x,1}_s-\bX^{t,x',1}_s\big\|_{L_2}^2
\leq
\alpha\|x-x'\|^2,
\quad
\big\|\bX^{t,x,1}_s-\bX^{t,x',1}_s\big\|_{L_{2q/(q-2)}}^{(q-2)/q}
\leq
\alpha_1\|x-x'\|^2,
\label{L2 cont Z}
\\
&
\big\|\bX^{t,x,1}_s-\bX^{t,x,2}_s\big\|_{L_2}^2
\leq
\gamma\delta^2(d^p+\|x\|^2),
\quad
\big\|\bX^{t,x,1}_s-\bX^{t,x,2}_s\big\|_{L_{2q/(q-2)}}^{(q-2)/q}
\leq
\gamma_1\delta^2(d^p+\|x\|^2),
\label{error Z 1 2}
\\
&
\big\|\bV^{t,x,1}_s-\bV^{t,x',1}_s\big\|_{L_2}^2
\leq
\beta(s-t)^{-1}\|x-x'\|^2,
\quad
\big\|\bV^{t,x,1}_s-\bV^{t,x,2}_s\big\|_{L_2}^2
\leq
\kappa\delta^2(s-t)^{-1}(d^p+\|x\|^2),
\label{L2 cont error V}
\\
&
\bE\left[\bE\Big[h(\bX^{s,x',1}_{r},\bX^{s,y',1}_{r})\Big]
\Big|_{(x',y')=(\bX^{t,x,1}_{s},\bX^{t,y,1}_{s})}\right]
=\bE\Big[h(\bX^{t,x,1}_{r},\bX^{t,y,1}_{r})\Big],
\label{flow property Z}
\\
&
|F(t,x,v,y)-F(t,x',v',y')|^2
\leq L(\|x-x'\|^2+|v-v'|^2+\|y-y'\|^2),
\quad
|G(x)-G(x')|^2\leq L\|x-x'\|^2,
\label{F G Lip}
\\
&
\big\|G(\bX^{t,x,k}_{T})\big(1,\bV^{t,x,k}_T\big)\big\|_{L_1}
+
\int_t^T
\big\|F\big(s,\bX^{t,x,k}_{s},u_k(s,\bX^{t,x,k}_{s}),w_k(s,\bX^{t,x,k}_{s})\big)
\big(1,\bV^{t,x,k}_s\big)\big\|_{L_1}
\,ds<\infty,
\nonumber\\
&
(u_k(t,x),w_k(t,x))
=\bE\bigg[G(\bX^{t,x,k}_{T})\big(1,\bV^{t,x,k}_T\big)
+
\int_t^T
F\big(s,\bX^{t,x,k}_{s},u_k(s,\bX^{t,x,k}_{s}),w_k(s,\bX^{t,x,k}_{s})\big)
\big(1,\bV^{t,x,k}_s\big)\,ds\bigg],
\label{BEL u w k}
\\
&                                     
|u_k(t,x)|^2+(T-t)\|w_k(t,x)\|^2
\leq c(d^p+\|x\|^2).
\label{cond growth u w k}
\end{align}
Then the following holds.
\begin{enumerate}[(i)]
\item
For all $t\in[0,T)$ and $x,y\in\bR^d$ we have that
\begin{equation}
\label{local Lip u 1}
|u_1(t,x)-u_1(t,y)|^2+(T-t)\|w_1(t,x)-w_1(t,y)\|^2
\leq 
c_{d,1}e^{c_{d,2}T}(d^p+\|x\|^2)\|x-y\|^2,
\end{equation}
where
\begin{equation}
\label{def c 1}
c_{d,1}:=4\Big[
\big(2(2q/(q-1))^{\frac{2(q-1)}{q}}[2(q-1)/q]+1\big)
a_1b\alpha_1\beta
+
\alpha L\big(1+3C_{d,T}\big)
\Big]
(1+T)^2(1+c),
\end{equation}
and
\begin{equation}
\label{def c 2}
c_{d,2}:=64L\big(C_{d,T}+1\big)(1+T)T.
\end{equation}
\item
For all $(t,x)\in[0,T)\times\bR^d$ it holds that
\begin{equation}                                     \label{u w difference}
|u_1(t,x)-u_2(t,x)|^2+(T-t)\|w_1(t,x)-w_2(t,x)\|^2
\leq 
c_{d,3}e^{c_{d,4}T}\delta^2(d^p+\|x\|^2)^2,
\end{equation}
where
\begin{align}
c_{d,3}
:=
&
2\Big[
L\gamma+2(\gamma LC_{d,T}+\kappa ab)
+
4LT(1+T)\big(\gamma+8a_1\gamma_1c_{d,1}e^{c_{d,2}T}\big)
\nonumber\\
&
+
4T(1+T)
\left(
C_{d,T}L\big(2\gamma+c_{d,1}e^{c_{d,2}T}a_1\gamma_1\big)
+4ab\kappa(1+4c)
\right)
\Big],
\label{def c 3}
\end{align}
and
\begin{equation*}
c_{d,4}:=8L(1+T)T\big(C_{d,T}a_2^2T+4a^2\big).
\end{equation*}
\end{enumerate}
\end{lemma}

\begin{proof}
By \eqref{BEL u w k} and the triangle inequality, 
we first notice for all $t\in[0,T)$, $s\in[t,T)$, and $x,y\in\bR^d$ that
\begin{align}
E^{t,x,y}_1(s)
&
:=
\|u_1(s,\bX^{t,x,1}_s)-u_1(s,\bX^{t,y,1}_s)\|_{L_2}^2
+
(T-s)\|w_1(s,\bX^{t,x,1}_s)-w_1(s,\bX^{t,y,1}_s)\|_{L_2}^2
\nonumber\\
&
\leq 
\sum_{i=1}^3A_i^{t,x,y}(s),
\label{ineq u 1 t x y}
\end{align}
where
\begin{align*}
&
A_1^{t,x,y}(s)
:=
\bE\Bigg[
\Big(\bE\Big[
\big(G(\bX^{s,x',1}_{T})-G(\bX^{s,y',1}_{T})\big)
+\int_s^T
\big[
F\big(r,\bX^{s,x',1}_{r},u_1(r,\bX^{s,x',1}_{r}),w_1(r,\bX^{s,x',1}_{r})\big)
\\
& \qquad \qquad \qquad
-
F\big(r,\bX^{s,y',1}_{s},u_1(r,\bX^{s,y',1}_{r}),w_1(r,\bX^{t,y',1}_{r})\big)
\big]
\,dr
\Big]\Big)^2\bigg|_{(x',y')=(\bX^{t,x,1}_s,\bX^{t,y,1}_s)}
\Bigg],
\\
&
A_2^{t,x,y}(s)
:=
2(T-s)
\bE\left[
\big\|
G(\bX^{s,x',1}_{T})\bV^{s,x',1}_T
-
G(\bX^{s,y',1}_{T})\bV^{s,y',1}_T
\big\|_{L_1}^2\Big|_{(x',y')=(\bX^{t,x,1}_s,\bX^{t,y,1}_s)}
\right],
\end{align*}
and
\begin{align*}
A_3^{t,x,y}(s)
:=
2(T-s)
\bE
&
\Bigg[
\bigg(
\int_s^T
\big\|F\big(r,\bX^{s,x',1}_{r},u_1(r,\bX^{s,x',1}_{r}),w_1(r,\bX^{s,x',1}_{r})\big)
\bV^{s,x',1}_r
\nonumber\\
&
-
F\big(r,\bX^{s,y',1}_{r},u_1(r,\bX^{s,y',1}_{r}),w_1(r,\bX^{s,y',1}_{r})\big)
\bV^{s,y',1}_r\big\|_{L_1}
\,dr
\bigg)^2
\bigg|_{(x',y')=(\bX^{t,x,1}_s,\bX^{t,y,1}_s)}
\Bigg].
\end{align*}
By \eqref{L2 cont Z}, \eqref{flow property Z}, \eqref{F G Lip},
the triangle inequality, Fubini's theorem, Cauchy-Schwarz inequality, 
and Jensen's inequality
we have for all $t\in[0,T)$, $s\in[t,T)$, and $x,y\in\bR^d$ that
\begin{align}
&
A_1^{t,x,y}(s)
\nonumber\\
&
\leq 
2\big\|G(\bX^{t,x,1}_T)-G(\bX^{t,y,1}_T)\big\|_{L_2}^2
+
2\bE\Bigg[
\bigg(\int_s^T\bE\Big[
F\big(r,\bX^{s,x',1}_r,u_1(r,\bX^{s,x',1}_r),w_1(r,\bX^{s,x',1}_r)\big)
\nonumber\\
& \quad
-
F\big(r,\bX^{s,y',1}_r,u_1(r,\bX^{s,y',1}_r),w_1(r,\bX^{s,y',1}_r)\big)\Big]\,dr\bigg)^2
\bigg|_{(x',y')=(\bX^{t,x,1}_s,\bX^{t,y,1}_s)}
\Bigg]
\nonumber\\
&
\leq 2L\big\|\bX^{t,x,1}_T-\bX^{t,y,1}_T\big\|_{L_2}^2
+2L
\bE\left[
\left(
\int_s^T
\big\|\bX^{s,x',1}_r-\bX^{s,y',1}_r\big\|_{L_2}
dr
\right)^2
\bigg|_{(x',y')=(\bX^{t,x,1}_s,\bX^{t,y,1}_s)}
\right]
\nonumber\\
& \quad
+2L\bE\Bigg[
\bigg(
\int_s^T
\Big(
\big\|u_1(r,\bX^{s,x',1}_r)-u_1(r,\bX^{s,y',1}_r)\big\|_{L_2}^2
\nonumber\\
& \quad
+
\big\|w_1(r,\bX^{s,x',1}_r)-w_1(r,\bX^{s,y',1}_r)\big\|_{L_2}^2
\Big)^{1/2}
dr
\bigg)^2
\bigg|_{(x',y')=(\bX^{t,x,1}_s,\bX^{t,y,1}_s)}
\Bigg]
\nonumber\\
&
\leq 
2\alpha L(1+T)\|x-y\|^2
+
2L(1+T)\bE\Bigg[
\bigg(
\int_s^T
\Big(
\big\|u_1(r,\bX^{s,x',1}_r)-u_1(r,\bX^{s,y',1}_r)\big\|_{L_2}^2
\nonumber\\
& \quad
+
(T-r)\big\|w_1(r,\bX^{s,x',1}_r)-w_1(r,\bX^{s,y',1}_r)\big\|_{L_2}^2
\Big)^{1/2}
(T-r)^{-1/2}
dr
\bigg)^2
\bigg|_{(x',y')=(\bX^{t,x,1}_s,\bX^{t,y,1}_s)}
\Bigg]
\nonumber\\
&
\leq 
2\alpha L(1+T)\|x-y\|^2
+4L(1+T)T^{1/2}
\int_s^T
E^{t,x,y}_1(r)
(T-r)^{-1/2}
\,dr.
\label{A 1 t x y est}
\end{align}
Furthermore, by \eqref{est V t x k s}, \eqref{a b est 2 q}, \eqref{growth F G}, 
\eqref{L2 cont Z}, \eqref{L2 cont error V}, \eqref{flow property Z}, 
and H\"older's inequality
it holds for all $t\in[0,T)$, $s\in[t,T)$, and $x,y\in\bR^d$ that
\begin{align}
A_2^{t,x,y}(s)
&
\leq
4(T-s)\bE\Bigg[
\big\|
G(\bX^{s,x',1}_{T})
(
\bV^{s,x',1}_T
-
\bV^{s,y',1}_T
)
\big\|_{L_2}^2
\bigg|_{(x',y')=(\bX^{t,x,1}_s,\bX^{t,y,1}_s)}
\Bigg]
\nonumber\\
& \quad
+
4(T-s)\bE\Bigg[
\big\|
\big[
G(\bX^{s,x',1}_{T})
-
G(\bX^{s,y',1}_{T})
\big]
\bV^{s,y',1}_T
\big\|_{L_2}^2
\bigg|_{(x',y')=(\bX^{t,x,1}_s,\bX^{t,y,1}_s)}
\Bigg]
\nonumber\\
&
\leq
4(T-s)
\bE\left[\big\|G(\bX^{s,x',1}_T)\big\|_{L_2}^2
\cdot
\big\|
\bV^{s,x',1}_T-\bV^{s,y',1}_T
\big\|_{L_2}^2
\Big|_{(x',y')=(\bX^{t,x,1}_s,\bX^{t,y,1}_s)}
\right]
\nonumber\\
& \quad
+
4(T-s)
\bE\left[
\big\|G(\bX^{s,x',1}_T)-G(\bX^{s,y',1}_T)\big\|_{L_2}^2
\cdot
\big\|
\bV^{s,y',1}_T
\big\|_{L_2}^2
\Big|_{(x',y')=(\bX^{t,x,1}_s,\bX^{t,y,1}_s)}
\right]
\nonumber\\
&
\leq
4(T-s)
\big\|G(\bX^{t,x,1}_T)\big\|_{L_q}^2
\left(
\bE\left[
\big\|
\bV^{s,x',1}_T-\bV^{s,y',1}_T
\big\|_{L_2}^{2q/(q-2)}
\Big|_{(x',y')=(\bX^{t,x,1}_s,\bX^{t,y,1}_s)}
\right]
\right)^{\frac{q-2}{q}}
\nonumber\\
& \quad
+
4C_{d,T}\big\|G(\bX^{t,x,1}_T)-G(\bX^{t,y,1}_T)\big\|_{L_2}^2
\nonumber\\
&
\leq
4b\beta
\big\|(d^p+\big\|\bX^{t,x,1}_T\big\|^2\big)\big\|_{L_{q/2}}
\cdot
\big\|
\bX^{t,x,1}_s-\bX^{t,x,1}_s
\big\|_{L_{2q/(q-2)}}^2
+4C_{d,T}L
\big\|
\bX^{t,x,1}_s-\bX^{t,x,1}_s
\big\|_{L_2}^2
\nonumber\\
&
\leq
4\big(a_1b\alpha_1\beta+\alpha LC_{d,T}\big)
(d^p+\|x\|^2)\|x-y\|^2.
\label{A 2 t x y est}
\end{align}
In addition, by the triangle inequality 
we notice for all $t\in[0,T)$, $s\in[t,T)$, and $x,y\in\bR^d$ that
\begin{equation}
\label{ineq A t x y 3}
A^{t,x,y}_3(s)\leq A^{t,x,y}_{3,1}(s)+A^{t,x,y}_{3,2}(s),
\end{equation}
where 
\begin{align*}
A^{t,x,y}_{3,1}(s)
:=
2(T-s)
\bE
\Bigg[
&
\bigg(
\int_s^T
\big\|
F\big(r,\bX^{s,x',1}_{r},u_1(r,\bX^{s,x',1}_{r}),w_1(r,\bX^{s,x',1}_{r})\big)
\nonumber\\
&
\big(\bV^{s,x',1}_r-\bV^{s,y',1}_r\big)
\big\|_{L_1}
\,dr
\bigg)^2
\bigg|_{(x',y')=(\bX^{t,x,1}_s,\bX^{t,y,1}_s)}
\Bigg],
\end{align*}
and
\begin{align*}
A^{t,x,y}_{3,2}(s)
:=
2(T-s)
\bE
\Bigg[
&
\bigg(
\int_s^T
\big\|
\big(
F\big(r,\bX^{s,x',1}_{r},u_1(r,\bX^{s,x',1}_{r}),w_1(r,\bX^{s,x',1}_{r})\big)
\nonumber\\
&
-
F\big(r,\bX^{s,y',1}_{r},u_1(r,\bX^{s,y',1}_{r}),w_1(r,\bX^{s,y',1}_{r})\big)
\big)
\bV^{s,y',1}_r
\big\|_{L_1}
\,dr
\bigg)^2
\bigg|_{(x',y')=(\bX^{t,x,1}_s,\bX^{t,y,1}_s)}
\Bigg].
\end{align*}
By \eqref{time integral est}, \eqref{a b est 2 q}, \eqref{growth F G}, 
\eqref{flow property Z}, \eqref{cond growth u w k},
H\"older's inequality, and Fubini theorem it holds for all $t\in[0,T)$,
$s\in[t,T)$, and $x,y\in\bR^d$ that
\begin{align}
&
A^{t,x,y}_{3,1}(s)
\nonumber\\
&
\leq 
2(T-s)\bE\Bigg[
\bigg(
\int_s^T
\big\|
F\big(
r,\bX^{s,x',1}_r,u_1(r,\bX^{s,x',1}_r),w_1(r,\bX^{s,x',1}_r)
\big)
\big\|_{L_2}
\nonumber\\
& \quad
\cdot
\big\|\bV^{s,x',1}_r-\bV^{s,y',1}_r\big\|_{L_2}\,
dr
\bigg)^2\bigg|_{(x',y')=(\bX^{t,x,1}_s,\bX^{t,y,1}_s)}
\Bigg]
\nonumber\\
&
\leq
2b\beta(1+T)(T-s)
\bE\Bigg[
\bigg(
\int_s^T
\frac
{\left(\bE\left[d^p+\big\|\bX^{s,x',1}_r\big\|^2+\big|u_1(r,\bX^{s,x',1}_r)\big|^2
+(T-r)\big\|w_1(r,\bX^{s,x',1}_r)\big\|^2
\right]\right)^{1/2}}
{(T-r)^{1/2}}
\nonumber\\
& \quad
\cdot
\frac{\|x'-y'\|}{(r-s)^{1/2}}
\,dr
\bigg)^2\bigg|_{(x',y')=(\bX^{t,x,1}_s,\bX^{t,y,1}_s)}
\Bigg]
\nonumber\\
&
\leq
2b\beta(1+T)(1+c)(T-s)
\bE\Bigg[
\bigg(
\int_s^T
\frac
{\big\|\big(d^p+\big\|\bX^{s,x',1}_r\big\|^2\big)\big\|_{L_1}^{1/2}\cdot\|x'-y'\|}
{(T-r)^{1/2}(r-s)^{1/2}}
\,dr
\bigg)^2\bigg|_{(x',y')=(\bX^{t,x,1}_s,\bX^{t,y,1}_s)}
\Bigg]
\nonumber\\
&
\leq 
2b\beta(1+T)(1+c)(T-s)
\bE\Bigg[
\left(
\int_s^T
\big\|\big(d^p+\big\|\bX^{s,x',1}_r\big\|^2\big)\big\|_{L_1}^{q/2}\,dr
\right)^{2/q}
\nonumber\\
& \quad
\cdot
\left(
\int_s^T(T-r)^{\frac{-q}{2(q-1)}}(r-s)^{\frac{-q}{2(q-1)}}\,dr
\right)^{\frac{2(q-1)}{q}}
\|x'-y'\|^2
\bigg|_{(x',y')=(\bX^{t,x,1}_s,\bX^{t,y,1}_s)}
\Bigg]
\nonumber\\
&
\leq
8\Big(\frac{2q}{q-1}\Big)^{\frac{2(q-1)}{q}}
b\beta(1+T)(1+C)(T-s)^{\frac{q-2}{q}}
\nonumber\\
& \quad
\cdot
\bE\Bigg[
\left(
\int_s^T
\big\|
\big(
d^p+\big\|\bX^{s,x',1}_r\big\|^2
\big)
\big\|_{L_{q/2}}^{q/2}\,dr
\right)^{2/q}
\|x'-y'\|^2
\bigg|_{(x',y')=(\bX^{t,x,1}_s,\bX^{t,y,1}_s)}
\Bigg]
\nonumber\\
&
\leq
8(2q/(q-1))^{\frac{2(q-1)}{q}}
b\beta(1+T)(1+C)(T-s)^{\frac{q-2}{q}}
\nonumber\\
& \quad
\cdot
\left(
\int_s^T
\big\|
\big(
d^p+\big\|\bX^{t,x,1}_s\big\|^2
\big)
\big\|_{L_{q/2}}^{q/2}\,dr
\right)^{2/q}
\big\|
\bX^{t,x,1}_s-\bX^{t,y,1}_s
\big\|_{L_{2q/(q-2)}}^2
\nonumber\\
&
\leq
8(2q/(q-1))^{\frac{2(q-1)}{q}}[2(q-1)/q]
a_1b\alpha_1\beta(1+T)T(1+c)(d^p+\|x\|^2)\|x-y\|^2.
\label{A t x y 3 1 est}
\end{align}
Moreover, by \eqref{est V t x k s}, \eqref{L2 cont Z}, \eqref{F G Lip} 
Cauchy-Schwarz inequality, and Jensen's iequality we have
for all $t\in[0,T)$, $s\in[t,T)$, and $x,y\in\bR^d$ that
\begin{align}
&
A^{t,x,y}_{3,2}(s)
\nonumber\\
&
\leq 
2(T-s)
\bE\Bigg[
\bigg(
\int_s^T
\big\|
F\big(r,\bX^{s,x',1}_r,u_1(r,\bX^{s,x',1}_r),w_1(r,\bX^{s,x',1}_r)\big)
\nonumber\\
& \quad
-
F\big(r,\bX^{s,y',1}_r,u_1(r,\bX^{s,y',1}_r),w_1(r,\bX^{s,y',1}_r)\big)
\big\|_{L_2}^2
\cdot
\big\|\bV^{s,y',1}_r\big\|_{L_2}^2
\,dr
\bigg)^2\bigg|_{(x',y')=(\bX^{t,x,1}_s,\bX^{t,y,1}_s)}
\Bigg]
\nonumber\\
& 
\leq
2LC_{d,T}(T-s)\bE\Bigg[
\bigg(\int_s^T
\Big(
\big\|\bX^{s,x',1}_r-\bX^{s,y',1}_r\big\|_{L_2}^2
+\big\|u_1(r,\bX^{s,x',1}_r)-u_1(r,\bX^{s,y',1}_r)\big\|_{L_2}^2
\nonumber\\
& \quad
+\big\|w_1(r,\bX^{s,x',1}_r)-w_1(r,\bX^{s,y',1}_r)\big\|^2
\Big)^{1/2}
(r-s)^{-1/2}
\,dr
\bigg)^2\bigg|_{(x',y')=(\bX^{t,x,1}_s,\bX^{t,y,1}_s)}
\Bigg]
\nonumber\\
&
\leq
4LC_{d,T}(T-s)
\big\|\bX^{t,x,1}_s-\bX^{t,y,1}_s\big\|_{L_2}^2
\left(
\int_s^T(r-s)^{-1/2}\,dr
\right)^2
+4LC_{d,T}(1+T)(T-s)
\nonumber\\
& \quad
\cdot
\bE\Bigg[
\bigg(
\int_s^T
\left(
\big\|u_1(r,\bX^{s,x',1}_r)-u_1(r,\bX^{s,y',1}_r)\big\|_{L_2}^2
+(T-r)\big\|w_1(r,\bX^{s,x',1}_r)-w_1(r,\bX^{s,y',1}_r)\big\|_{L_2}^2
\right)^{1/2}
\nonumber\\
& \quad
\cdot
(T-r)^{-1/2}(r-s)^{-1/2}
\,dr
\bigg)^2
\bigg|_{(x',y')=(\bX^{t,x,1}_s,\bX^{t,y,1}_s)}
\Bigg]
\nonumber\\
&
\leq
8\alpha LC_{d,T}(T-s)^2\|x-y\|^2
+
16LC_{d,T}(1+T)(T-s)
\int_s^T
(T-r)^{-1/2}(r-s)^{-1/2}
E^{t,x,y}_1(r)
\,dr.
\nonumber\\
\end{align}
This together with \eqref{ineq A t x y 3} and \eqref{A t x y 3 1 est}
imply for all $t\in[0,T)$, $s\in[t,T)$, and $x,y\in\bR^d$ that
\begin{align}
A^{t,x,y}_3(s)
\leq
&
8\Big[
(2q/(q-1))^{\frac{2(q-1)}{q}}[2(q-1)/q]
a_1b\alpha_1\beta
+
\alpha LC_{d,T}
\Big]
(1+T)T(1+c)(d^p+\|x\|^2)\|x-y\|^2
\nonumber\\
&
+
16LC_{d,T}(1+T)T
\int_s^T
(T-r)^{-1/2}(r-s)^{-1/2}
E^{t,x,y}_1(r)
\,dr.
\label{A 3 t x y est}
\end{align}
Then combining \eqref{ineq u 1 t x y}, 
\eqref{A 1 t x y est}, \eqref{A 2 t x y est}, and \eqref{A 3 t x y est} yields 
for all $t\in[0,T)$, $s\in[t,T)$, and $x,y\in\bR^d$ that  
\begin{align}
E^{t,x,y}_1(s)
\leq
c_{d,1}(d^p+\|x\|^2)\|x-y\|^2
+\frac{c_{d,2}}{4}\int_s^T
(T-r)^{-1/2}(r-s)^{-1/2}
E^{t,x,y}_1(r)
\,dr,
\label{E t x y 1 est}
\end{align}
where $c_{d,1}$ and $c_{d,2}$ are the positive constants 
defined by \eqref{def c 1} and \eqref{def c 2}, respectively.
By \eqref{time integral est 1/2}, \eqref{a b est}, \eqref{cond growth u w k}, 
\eqref{E t x y 1 est}, and Gr\"onwall's lemma, 
it holds for all $t\in[0,T)$, $s\in[t,T)$, and $x,y\in\bR^d$ that
\begin{align*}
E^{t,x,y}_1(s)
&
\leq 
c_{d,1}(d^p+\|x\|^2)\|x-y\|^2
\exp\bigg\{
\frac{c_{d,2}}{4}\int_s^T
(T-r)^{-1/2}(r-s)^{-1/2}
\,dr
\bigg\}
\\
&
\leq
c_{d,1}e^{c_{d,2}T}(d^p+\|x\|^2)\|x-y\|^2.
\end{align*}
This ensures that \eqref{local Lip u 1} holds.
Next, by \eqref{BEL u w k}
we notice for all $(t,x)\in[0,T)\times\bR^d$ that
\begin{equation}
\label{ineq u w 1 2}
|u_1(t,x)-u_2(t,x)|^2+(T-t)\|w_1(t,x)-w_2(t,x)\|^2
\leq
2\sum_{i=1}^4B^{t,x}_i,
\end{equation}
where
\begin{align*}
&
B^{t,x}_1
:=
\big\|G(\bX^{t,x,1}_{T})-G(\bX^{t,x,2}_{T})\big\|_{L_1}^2
, \quad
B^{t,x}_2
:=
(T-t)
\big\|G(\bX^{t,x,1}_{T})\bV^{t,x,1}_T
-G(\bX^{t,x,2}_{T})\bV^{t,x,2}_T\big\|_{L_1}^2.
\\
&
B^{t,x}_3
:=
\Bigg(
\bE\bigg[
\int_t^T
\Big[
F\big(s,\bX^{t,x,1}_{s},u_1(s,\bX^{t,x,1}_{s}),w_1(s,\bX^{t,x,1}_{s})\big)
-
F\big(s,\bX^{t,x,2}_{s},u_2(s,\bX^{t,x,2}_{s}),w_2(s,\bX^{t,x,2}_{s})\big)
\Big]
\,ds
\bigg]
\Bigg)^2,
\end{align*}
and
\begin{align*}
B^{t,x}_4
:=
(T-t)
\bigg(
&
\bE\bigg[
\int_t^T
\big\|
F\big(s,\bX^{t,x,1}_{s},u_1(s,\bX^{t,x,1}_{s}),w_1(s,\bX^{t,x,1}_{s})\big)\bV^{t,x,1}_s
\\
&
-
F\big(s,\bX^{t,x,2}_{s},u_2(s,\bX^{t,x,2}_{s}),w_2(s,\bX^{t,x,2}_{s})\big)\bV^{t,x,2}_s
\big\|
\,ds
\bigg]
\bigg)^2.
\end{align*}
By \eqref{est V t x k s}, \eqref{a b est}, \eqref{growth F G},
\eqref{error Z 1 2}, \eqref{L2 cont error V}, \eqref{F G Lip}, 
and Cauchy-Schwarz inequality we have for all $(t,x)\in[0,T)\times\bR^d$ that
\begin{equation}
\label{B t x 1 est}
B^{t,x}_1
\leq
L\big\|\bX^{t,x,1}_T-\bX^{t,x,2}_T\big\|_{L_2}^2
\leq
L\gamma\delta^2(d^p+\|x\|^2),
\end{equation}
and
\begin{align}
B^{t,x}_2
&
\leq
2(T-t)\left(
\Big(
\bE\Big[
\big|G(\bX^{t,x,1}_T)-G(\bX^{t,x,2}_T)\big|
\cdot
\big\|\bV^{t,x,1}_T\big\|
\Big]
\Big)^2
+
\Big(
\bE\Big[
\big|G(\bX^{t,x,2}_T)\big|
\cdot
\big\|\bV^{t,x,1}_T-\bV^{t,x,2}_T\big\|
\Big]
\Big)^2
\right)
\nonumber\\
&
\leq
2L(T-t)\big\|\bX^{t,x,1}_T-\bX^{t,x,2}_T\big\|_{L_2}^2
\cdot
\big\|\bV^{t,x,1}_T\big\|_{L_2}^2
+
2b(T-t)\big\|\big(d^p+\big\|\bX^{t,x,2}_T\big\|^2\big)\big\|_{L_1}
\cdot
\big\|\bV^{t,x,1}_T-\bV^{t,x,2}_T\big\|_{L_2}^2
\nonumber\\
&
\leq
2(\gamma LC_{d,T}+\kappa ab)\delta^2(d^p+\|x\|^2).
\label{B t x 2 est}
\end{align}
Furthermore, we notice for all $(t,x)\in[0,T)\times\bR^d$ that
\begin{equation}
\label{ineq B t x 3}
B^{t,x}_3\leq 2B^{t,x}_{3,1}+2B^{t,x}_{3,2},
\end{equation}
where
$$
B^{t,x}_{3,1}
:=
\Bigg(
\bE
\bigg[
\int_t^T
\Big[
F\big(s,\bX^{t,x,1}_{s},u_1(s,\bX^{t,x,1}_{s}),w_1(s,\bX^{t,x,1}_{s})\big)
-
F\big(s,\bX^{t,x,2}_{s},u_1(s,\bX^{t,x,2}_{s}),w_1(s,\bX^{t,x,2}_{s})\big)
\Big]
\,ds
\bigg]
\Bigg)^2,
$$
and
$$
B^{t,x}_{3,2}
:=
\Bigg(
\bE
\bigg[
\int_t^T
\Big[
F\big(s,\bX^{t,x,2}_{s},u_1(s,\bX^{t,x,2}_{s}),w_1(s,\bX^{t,x,2}_{s})\big)
-
F\big(s,\bX^{t,x,2}_{s},u_2(s,\bX^{t,x,2}_{s}),w_2(s,\bX^{t,x,2}_{s})\big)
\Big]
\,ds
\bigg]
\Bigg)^2.
$$
By \eqref{F G Lip}, \eqref{error Z 1 2}, \eqref{local Lip u 1},
Jensen's inequality, and H\"older's inequality, we have
for all $(t,x)\in[0,T)\times\bR^d$ that
\begin{align}
&
B^{t,x}_{3,1}
\nonumber\\
&
\leq
L\bE\Bigg[
\bigg(
\int_t^T
\Big(
\|\bX^{t,x,1}_s-\bX^{t,x,2}_s\|
+\big|u_1(s,\bX^{t,x,1}_{s})-u_1(s,\bX^{t,x,2}_{s})\big|
+\big\|w_1(s,\bX^{t,x,1}_{s})-w_1(s,\bX^{t,x,2}_{s})\big\|
\Big)
\,ds
\bigg)^2
\Bigg]
\nonumber\\
&
\leq
2L(T-t)\int_t^T\big\|\bX^{t,x,1}_s-\bX^{t,x,2}_s\big\|_{L_2}^2
+4L(1+T)\bE\Bigg[
\bigg(
\int_t^T
(T-s)^{1/2}
\nonumber\\
& \quad
\cdot
\Big(
\big|u_1(s,\bX^{t,x,1}_{s})-u_1(s,\bX^{t,x,2}_{s})\big|^2
+(T-s)\big\|w_1(s,\bX^{t,x,1}_{s})-w_1(s,\bX^{t,x,2}_{s})\big\|^2
\Big)^{1/2}
\,ds
\bigg)^2
\Bigg]
\nonumber\\
&
\leq
2LT^2\gamma\delta^2(d^p+\|x\|^2)
+4L(1+T)c_{d,1}e^{4c_{d,2}T}
\bE
\left[
\bigg(
\int_t^T
\frac
{\big(d^p+\big\|\bX^{t,x,1}_s\big\|^2\big)^{1/2}\big\|\bX^{t,x,1}_s-\bX^{t,x,2}_s\big\|}
{(T-s)^{1/2}}
\,ds
\bigg)^2
\right]
\nonumber\\
&
\leq
2LT^2\gamma\delta^2(d^p+\|x\|^2)+8L(1+T)(T-t)^{1/2}c_{d,1}e^{4c_{d,2}T}
\int_t^T
\frac
{
\bE
\left[
\big(d^p+\big\|\bX^{t,x,1}_s\big\|^2\big)\big\|\bX^{t,x,1}_s-\bX^{t,x,2}_s\big\|^2
\right]
}
{(T-s)^{1/2}}
\,ds
\nonumber\\
&
\leq
2LT^2\gamma\delta^2(d^p+\|x\|^2)
+16L(1+T)(T-t)c_{d,1}e^{4c_{d,2}T}
\sup_{s\in[t,T]}
\big\|\big(d^p+\big\|\bX^{t,x,1}_s\big\|^2\big)\big\|_{L_{q/2}}
\nonumber\\
& \quad
\cdot
\sup_{s\in[t,T]}
\big\|\bX^{t,x,1}_s-\bX^{t,x,2}_s\big\|_{L_{2q/(q-2)}}^2
\nonumber\\
&
\leq
2LT^2\gamma\delta^2(d^p+\|x\|^2)
+16LT(1+T)c_{d,1}e^{4c_{d,2}T}a_1\gamma_1\delta^2(d^p+\|x\|^2)^2.
\label{B t x 3 1 est}
\end{align}
For every $s\in[0,T)$, we define
\begin{equation}
\label{def E s}
E(s):=\sup_{y\in\bR^d}
\frac{|u_1(s,y)-u_2(s,y)|^2+(T-s)\|w_1(s,y)-w_2(s,y)\|^2}{(d^p+\|y\|^2)^2}.
\end{equation}
Then by \eqref{a b est}, \eqref{F G Lip}, and Jensen's inequality,
we obtain for all $(t,x)\in[0,T)\times\bR^d$ that
\begin{align*}
B^{t,x}_{3,2}
&
\leq
L\left(
\bE\left[
\int_t^T
\left(
\big|u_1(s,\bX^{t,x,2}_s)-u_2(s,\bX^{t,x,2}_s)\big|
+
\big\|w_1(s,\bX^{t,x,2}_s)-w_2(s,\bX^{t,x,2}_s)\big\|
\right)
ds
\right]
\right)^2 
\\
&
\leq 2L(1+T)
\Bigg(
\bE\bigg[
\int_t^T
(T-s)^{-1/2}\left(d^p+\big\|\bX^{t,x,2}_s\big\|^2\right)
\\
& \quad
\cdot
\frac
{
\Big(
\big|u_1(s,\bX^{t,x,2}_s)-u_2(s,\bX^{t,x,2}_s)\big|^2
+
(T-s)\big\|w_1(s,\bX^{t,x,2}_s)-w_2(s,\bX^{t,x,2}_s)\big\|^2
\Big)^{1/2}
}
{\left(d^p+\big\|\bX^{t,x,2}_s\big\|^2\right)}
\,ds
\bigg]
\Bigg)^2
\\
&
\leq
2L(1+T)
\left(
\int_t^T
\big(T-s)^{-1/2}\bE\Big[d^p+\big\|\bX^{t,x,2}_s\big\|^2\Big][E(s)\big]^{1/2}
\,ds
\right)^2
\\
&
\leq
4a^2L(1+T)(d^p+\|x\|^2)^2(T-t)^{1/2}\int_t^T(T-s)^{-1/2}E(s)\,ds.
\end{align*}
This together with \eqref{ineq B t x 3} and \eqref{B t x 3 1 est} imply 
for all $(t,x)\in[0,T)\times\bR^d$ that
\begin{align}
B^{t,x}_3
\leq
&
4LT(1+T)\left(\gamma+8a_1\gamma_1c_{d,1}e^{c_{d,2}T}\right)\delta^2(d^p+\|x\|^2)^2
\nonumber\\
&
+
8a^2L(1+T)(d^p+\|x\|^2)^2(T-t)^{1/2}\int_t^T(T-s)^{-1/2}E(s)\,ds.
\label{B t x 3 est}
\end{align}
Next, we notice for all $(t,x)\in[0,T)\times\bR^d$ that
\begin{equation}
\label{ineq B t x 4}
B^{t,x}_4\leq \sum_{i=1}^3B^{t,x}_{4,i},
\end{equation}
where
\begin{align*}
&
B^{t,x}_{4,1}
:=
2(T-t)
\bigg(
\int_t^T
\big\|
\big[
F\big(s,\bX^{t,x,1}_{s},u_1(s,\bX^{t,x,1}_{s}),w_1(s,\bX^{t,x,1}_{s})\big)
\\
& \qquad\qquad\qquad\qquad\quad
-
F\big(s,\bX^{t,x,2}_{s},u_1(s,\bX^{t,x,2}_{s}),w_1(s,\bX^{t,x,2}_{s})\big)\big]
\bV^{t,x,1}_s
\big\|_{L_1}
\,ds
\bigg)^2,
\\
&
B^{t,x}_{4,2}
:=
4(T-t)
\bigg(
\int_t^T
\big\|
F\big(s,\bX^{t,x,2}_{s},u_1(s,\bX^{t,x,2}_{s}),w_1(s,\bX^{t,x,2}_{s})\big)
\big(\bV^{t,x,1}_s-\bV^{t,x,2}_s\big)
\big\|_{L_1}
\,ds
\bigg)^2,
\end{align*}
and
\begin{align*}
B^{t,x}_{4,3}
:=
2(T-t)
\bigg(
&
\int_t^T
\big\|
\big[
F\big(s,\bX^{t,x,2}_{s},u_1(s,\bX^{t,x,2}_{s}),w_1(s,\bX^{t,x,2}_{s})\big)
\\
&
-
F\big(s,\bX^{t,x,2}_{s},u_2(s,\bX^{t,x,2}_{s}),w_2(s,\bX^{t,x,2}_{s})\big)
\big]
\bV^{t,x,2}_s
\big\|_{L_1}
\,ds
\bigg)^2.
\end{align*}
By \eqref{time integral est 1/2}, \eqref{est V t x k s}, \eqref{a b est 2 q}, 
\eqref{error Z 1 2}, \eqref{F G Lip}, \eqref{local Lip u 1}, 
and H\"older's inequality we have for all $(t,x)\in[0,T)\times\bR^d$ that
\begin{align}
B^{t,x}_{4,1}
&
\leq
2L(T-t)\bigg(
\int_t^T
\bE
\Big[
\big(
\big\|\bX^{t,x,1}_s-\bX^{t,x,2}_s\big\|+\big|u_1(s,\bX^{t,x,1}_s)-u_1(s,\bX^{t,x,2}_s)\big|
\nonumber\\
& \quad
+\big\|w_1(s,\bX^{t,x,1}_s)-w_1(s,\bX^{t,x,2}_s\big\|
\big)
\big\|\bV^{t,x,1}_s\big\|
\Big]
ds
\bigg)^2
\nonumber\\
&
\leq 
4LT\left(\int_t^T
\big\|\bX^{t,x,1}_s-\bX^{t,x,2}_s\big\|_{L_2}
\cdot
\big\|\bV^{t,x,1}_s\big\|_{L_2}
\,ds
\right)^2
+8L(1+T)T
\bigg(
\int_t^T(T-s)^{-1/2}
\nonumber\\
& \quad
\cdot
\bE\left[
\left(
\big|u_1(s,\bX^{t,x,1}_s)-u_1(s,\bX^{t,x,2}_s)\big|^2
+(T-s)\big\|w_1(s,\bX^{t,x,1}_s)-w_1(s,\bX^{t,x,2}_s)\big\|^2
\right)^{1/2}
\big\|\bV^{t,x,1}_s\big\|
\right]
ds
\bigg)^2
\nonumber\\
&
\leq
4C_{d,T}LT\gamma\delta^2(d^p+\|x\|^2)
\left(\int_t^T(s-t)^{-1/2}\,ds\right)^2
+8L(1+T)Tc_{d,1}e^{c_{d,2}T}
\nonumber\\
& \quad
\cdot
\left(
\int_t^T
(T-s)^{-1/2}
\bE\left[
\left(d^p+\big\|\bX^{t,x,1}_s\big\|^2\right)^{1/2}
\big\|\bX^{t,x,1}_s-\bX^{t,x,2}_s\big\|
\cdot
\big\|\bV^{t,x,1}_s\big\|
\right]
ds
\right)^2
\nonumber\\
&
\leq
16C_{d,T}LT^2\gamma\delta^2(d^p+\|x\|^2)
+8L(1+T)Tc_{d,1}e^{c_{d,2}T}
\bigg(
\int_t^T(T-s)^{-1/2}
\nonumber\\
& \quad
\cdot
\left(
\bE\left[
\left(d^p+\big\|\bX^{t,x,1}_s\big\|^2\right)
\big\|\bX^{t,x,1}_s-\bX^{t,x,2}_s\big\|^2
\right]
\right)^{1/2}
\big\|\bV^{t,x,1}_s\big\|_{L_2}
\,ds
\bigg)^2
\nonumber\\
&
\leq
16C_{d,T}LT^2\gamma\delta^2(d^p+\|x\|^2)
+8C_{d,T}L(1+T)Tc_{d,1}e^{c_{d,2}T}
\bigg(
\int_t^T
(T-s)^{-1/2}(s-t)^{-1/2}
\nonumber\\
& \quad
\cdot
\big\|\big(
d^p+\big\|\bX^{t,x,1}_s\big\|^2
\big)\big\|_{L_{q/2}}^{1/2}
\cdot
\big\|
\bX^{t,x,1}_s-\bX^{t,x,2}_s
\big\|_{L_{2q/(q-2)}}
\,ds
\bigg)^2
\nonumber\\
&
\leq
16C_{d,T}LT^2\gamma\delta^2(d^p+\|x\|^2)
+128C_{d,T}L(1+T)Tc_{d,1}e^{c_{d,2}T}a_1\gamma_1\delta^2(d^p+\|x\|^2)^2.
\label{B t x 4 1 est}
\end{align}
Moreover, by \eqref{time integral est 1/2}, \eqref{a b est}, \eqref{growth F G},
\eqref{L2 cont error V}, \eqref{cond growth u w k}, and Cauchy-Schwarz inequality
it holds for all $(t,x)\in[0,T)\times\bR^d$ that
\begin{align}
B^{t,x}_{4,2}
&
\leq
4bT\left(
\int_t^T
\bE\left[
\left(
d^p+\big\|\bX^{t,x,2}_s\big\|^2+\big|u_1(s,\bX^{t,x,2}_s)\big|^2
+\big\|w_1(s,\bX^{t,x,2}_s)\big\|^2
\right)^{1/2}
\big\|\bV^{t,x,1}_s-\bV^{t,x,2}_s\big\|
\right]
ds
\right)^2
\nonumber\\
& 
\leq
8bT\left(
\int_t^T
\bE\left[
\left(d^p+\big\|\bX^{t,x,2}_s\big\|^2\right)^{1/2}
\big\|\bV^{t,x,1}_s-\bV^{t,x,2}_s\big\|
\right]
ds
\right)^2
\nonumber\\
& \quad
+8bT(1+T)
\Bigg(
\int_t^T
\bE\bigg[
\frac
{
\big(
\big|u_1(s,\bX^{t,x,2}_s)\big|^2
+(T-s)\big\|w_1(s,\bX^{t,x,2}_s)\big\|^2
\big)^{1/2}
}
{(T-s)^{1/2}}
\big\|\bV^{t,x,1}_s-\bV^{t,x,2}_s\big\|
\bigg]
\,ds
\Bigg)^2
\nonumber\\
&
\leq
8bT\left(
\int_t^T
\big\|\big(d^p+\big\|\bX^{t,x,2}_s\big\|^2\big)\big\|_{L_1}^{1/2}
\cdot
\big\|\bV^{t,x,1}_s-\bV^{t,x,2}_s\big\|_{L_2}
\,ds
\right)^2
\nonumber\\
& \quad
+8bcT(1+T)
\left(
\int_t^T
(T-s)^{-1/2}
\big\|\big(d^p+\big\|\bX^{t,x,2}_s\big\|^2\big)\big\|_{L_1}^{1/2}
\cdot
\big\|\bV^{t,x,1}_s-\bV^{t,x,2}_s\big\|_{L_2}
\,ds
\right)^2
\nonumber\\
&
\leq
8abT\kappa\delta^2(d^p+\|x\|^2)^2
\left(\int_t^T(s-t)^{-1/2}\,ds\right)^2
\nonumber\\
& \quad
+
8abcT(1+T)\kappa\delta^2(d^p+\|x\|^2)^2
\left(\int_t^T(T-s)^{-1/2}(s-t)^{-1/2}\,ds\right)^2
\nonumber\\
&
\leq
32abT^2\kappa\delta^2(d^p+\|x\|^2)^2
+
128abcT(1+T)\kappa\delta^2(d^p+\|x\|^2)^2.
\label{B t x 4 2 est}
\end{align}
In addition, by \eqref{est V t x k s}, \eqref{a b est}, \eqref{a b est 2 q},
\eqref{F G Lip}, H\"older's inequality, and Jensen's inequality
we obtain for all $(t,x)\in[0,T)\times\bR^d$ that
\begin{align}
B^{t,x}_{4,3}
&
\leq
2L(T-t)
\bigg(
\int_t^T
\bE\Big[
\Big(
\big|u_1(s,\bX^{t,x,2}_s)-u_2(s,\bX^{t,x,2}_s)\big|^2
\nonumber\\
& \quad
+
\big\|w_1(s,\bX^{t,x,2}_s)-w_2(s,\bX^{t,x,2}_s)\big\|^2
\Big)^{1/2}
\big\|\bV^{t,x,2}_s\big\|
\Big]
\,ds
\bigg)^2
\nonumber\\
&
\leq 
2L(1+T)(T-t)
\Bigg(
\int_t^T
\bE\bigg[
\big\|\bV^{t,x,2}_s\big\|
\big(d^p+\big\|\bX^{t,x,2}_s\big\|^2\big)^{1/2}
\nonumber\\
& \quad
\cdot
\frac
{
\Big(
\big|u_1(s,\bX^{t,x,2}_s)-u_2(s,\bX^{t,x,2}_s)\big|^2
+
(T-s)\big\|w_1(s,\bX^{t,x,2}_s)-w_2(s,\bX^{t,x,2}_s)\big\|^2
\Big)^{1/2}
}
{(T-s)^{1/2}\big(d^p+\big\|\bX^{t,x,2}_s\big\|^2\big)^{1/2}}
\bigg]
\,ds
\Bigg)^2
\nonumber\\
&
\leq
2L(1+T)(T-t)
\bigg(
\int_t^T
(T-s)^{-1/2}\big[E(s)\big]^{1/2}
\big\|\big(d^p+\big\|\bX^{t,x,2}_s\big\|^2\big)\big\|_{L_1}^{1/2}
\cdot
\big\|\bV^{t,x,2}_s\big\|_{L_2}
\,ds
\bigg)^2
\nonumber\\
& 
\leq
2C_{d,T}La_2^2(1+T)(T-t)(d^p+\|x\|^2)^2
\left(
\int_t^T(T-s)^{-1/2}\big[E(s)\big]^{1/2}\,ds
\right)^2
\nonumber\\
&
\leq
4C_{d,T}La_2^2(1+T)(T-t)^{3/2}(d^p+\|x\|^2)^2
\int_t^T(T-s)^{-1/2}E(s)\,ds.
\label{B t x 4 3 est}
\end{align}
Combining \eqref{ineq B t x 4}, \eqref{B t x 4 1 est}, \eqref{B t x 4 2 est},
and \eqref{B t x 4 3 est} implies for all $(t,x)\in[0,T)\times\bR^d$ that
\begin{align}
B^{t,x}_4
\leq
&
8T(1+T)
\left[
C_{d,T}L\big(2\gamma+c_{d,1}e^{c_{d,2}T}a_1\gamma_1\big)
+4ab\kappa(1+4c)
\right]
\delta^2(d^p+\|x\|^2)^2
\nonumber\\
&
+4C_{d,T}La_2^2(1+T)(T-t)^{3/2}(d^p+\|x\|^2)^2
\int_t^T(T-s)^{-1/2}E(s)\,ds.
\label{B t x 4 est}
\end{align}
Then by \eqref{ineq u w 1 2}, \eqref{B t x 1 est}, \eqref{B t x 2 est},
\eqref{B t x 3 est}, and \eqref{B t x 4 est}, 
we obtain for all $(t,x)\in[0,T)\times\bR^d$ that
\begin{align}
&
|u_1(t,x)-u_2(t,x)|^2+(T-t)\|w_1(t,x)-w_2(t,x)\|^2
\nonumber\\
&
\leq
c_{d,3}\delta^2(d^p+\|x\|^2)^2
+
4L(1+T)(C_{d,T}a_2^2T+4a^2)(d^p+\|x\|^2)^2(T-t)^{1/2}
\int_t^T(T-s)^{-1/2}E(s)\,ds,
\label{est u w 1 2 t x}
\end{align}
where $c_{d,3}$ is defined by \eqref{def c 3}.
By \eqref{cond growth u w k} and \eqref{est u w 1 2 t x},
the application of Gr\"onwall's lemma ensures \eqref{u w difference}.
The proof of this proposition is therefore completed.
\end{proof}

\section{\textbf{Semilinear Parabolic Partial Differential Equations}}
\label{section PIDE}
In this section, we assume the settings in Section \ref{section setting},
and show the following proposition and theorem 
which demonstrate the uniqueness and existence of viscosity solutions to
PDE \eqref{APIDE}, and establish a probabilistic representation 
for the unique viscosity solution and its gradient. The details for the
construction of viscosity solutions of PDE \eqref{APIDE} can be found
in Appendix \ref{appendix existence PDE}.
\begin{proposition}[Uniqueness]                   \label{proposition uniqueness PDE}
Let Assumption \ref{assumption Lip and growth} hold, 
and let $d\in\bN$.
Assume that $u_1,u_2\in C_{lin}([0,T)\times\bR^d)$ 
are two viscosity solutions of PDE \eqref{APIDE} such that 
$u_1(T,x)=u_2(T,x)=g(x)$ for all $x\in\bR^d$.
Then we have for all
$(t,x)\in[0,T]\times\bR^d$ that 
$u_1(t,x)=u_2(t,x)$.
\end{proposition}

\begin{proof}
We define the function $f^*:[0,T]\times\bR^d\times\bR\times\bR^d\to\bR$ by
$$
f^*(t,x,v,w):=f(t,x,v,w\sigma^{-1}(x)), 
\quad (t,x,v,w)\in[0,T]\times\bR^d\times\bR\times\bR^d.
$$
Then by \eqref{assumption Lip f}, \eqref{bbd inverse sigma}, 
and Cauchy-Schwarz inequality
it holds for all $(t,x)\in[0,T]\times\bR^d$, 
$v_1,v_2\in\bR$, and $w_1,w_2\in\bR^d$ that
\begin{align}
|f^*(t,x,v_1,w_1)-f^*(t,x,v_2,w_2)|^2
&
\leq
L\big(|v_1-v_2|^2+\|(w_1-w_2)\sigma^{-1}(x)\|^2\big)
\nonumber\\
&
\leq
L\big(1+d\varepsilon_d^{-1}\big)\big(|v_1-v_2|^2+\|(w_1-w_2)\|^2\big).
\label{Lip f *}
\end{align}
By \eqref{assumption Lip mu sigma}, \eqref{assumption growth f g}, 
and \eqref{Lip f *}, the application of Theorem 3.5 in \cite{BBP1997}
(with $b\cal\mu$, $\sigma\cal\sigma$, $\beta\cal 0$, $g_i\cal g$, $f_i\cal f^*$,
and $k=0$ in the notation of Theorem 3.5 in \cite{BBP1997})
proves Proposition \ref{proposition uniqueness PDE}.
\end{proof}

\begin{proposition}                            \label{theorem PDE existence}
Let Assumptions \ref{assumption Lip and growth}, \ref{assumption ellip},
and \ref{assumption gradient} hold,
and let $d\in\bN$.
Then the following holds:
\begin{enumerate}[(i)]
\item
There exists a unique pair of Borel functions $(u,w)$ such that
$u\in C_{lin}([0,T)\times\bR^d,\bR)$, $w\in C([0,T)\times\bR^d,\bR^d)$, 
and
\begin{equation}
\label{growth FP PIDE}
\sup_{(s,y)\in[0,T)\times\bR^d}
\left(\frac{|u(s,y)|+(T-s)^{1/2}\|w(s,y)\|}{(d^p+\|y\|^2)^{1/2}}
\right)<\infty,
\end{equation}
and it holds for all $(t,x)\in[0,T)\times\bR^d$ that
\begin{align}
&
(u(t,x),w(t,x))
\nonumber\\
&
=\bE\left[g(X^{d,t,x}_T)\left(1,\frac{1}{T-t}
\int_t^T\left[\sigma^{-1}(X^{d,t,x}_{r})
DX^{d,t,x}_{r}\right]^T\,dW^d_r\right)\right]  
\nonumber\\
& \quad    
+\int_t^T\bE\left[
f\big(s,X^{d,t,x}_s,u(s,X^{d,t,x}_s),w(s,X^{d,t,x}_s)\big)
\left(1,\frac{1}{s-t}
\int_t^s\left[\sigma^{-1}(X^{t,x}_{r})
DX^{d,t,x}_{r}\right]^T\,dW^d_r\right)
\right]ds.
\label{BEL PDE}
\end{align}
\item
The function $u:[0,T)\times\bR^d\to\bR$ defined in \eqref{BEL PDE} 
with $u(T,\cdot)=g(\cdot)$
is a viscosity solution of PDE~\eqref{APIDE}.
\item
For all $(t,x)\in[0,T)\times\bR^d$ the gradient of $u$ exists and satisfies
$\nabla_x u(t,x)=w(t,x)$.
\end{enumerate}
\end{proposition}

\begin{proof}
This proposition is proved in Appendix \ref{section proof existence PDE}.
\end{proof}

\section{\textbf{Multilevel Picard approximations}}
\label{section general MLP}

In this section, we introduce and investigate a new class of full-history recursive
multilevel Picard approximation algorithms applicable to semilinear PDEs with
gradient-dependent nonlinearity (c.f. \eqref{APIDE}).
In the main result of this section (see Proposition \ref{corollary MLP error}),
we show an error analysis for these multilevel Picard approximation algorithms,
which will be applied to prove the main results of this paper, 
namely Theorems \ref{thm MLP conv} and \ref{MLP complexity}, 
in Section \ref{section proof main}.

\subsection{Setting}                                 \label{MLP setting}

Let $d\in\bN$, $T\in(0,\infty)$, and $\Theta=\cup_{n=1}^\infty\bZ^n$,
and define
$$
\Delta:=\{(t,s)\in[0,T)\times[0,T]:t\leq s\}.
$$ 
Let $\bX^\theta=(\bX^{\theta,t,x}_{s})_{(t,s,x)\in\Delta\times\bR^d}$:
$\Delta\times\bR^d\times\Omega\to \bR^d$, $\theta\in\Theta$,
be $\cB(\Delta)\otimes\cB(\bR^d)\otimes\cF/\cB(\bR^d)$-measurable functions.
For each $d\in\bN$,
let $\bV^\theta=(\bV^{\theta,t,x}_{s})_{(t,s,x)\in\Delta\times\bR^d}$:
$\Delta\times\bR^d\times \Omega\to \bR^d$, $\theta\in\Theta$, 
be $\cB(\Delta)\otimes\cB(\bR^d)\otimes\cF/\cB(\bR^d)$-measurable functions such that
$\bE\big[\bV^{\theta,t,x}_s\big]=\mathbf{0}$
for all $\theta\in\Theta$ and $(t,s,x)\in\Delta\times\bR^d$.
Assume for all $(t,x)\in[0,T)\times\bR^d$ and $s\in[t,T]$ that 
$\big(\bX^{\theta,t,x}_{s},\bV^{\theta,t,x}_{s}\big)$,
$\theta\in\Theta$, are independent and identically distributed.
Let $\alpha\in[1/2,1)$, and define the function $\varrho:(0,1)\to(0,\infty)$ by 
\begin{equation}
\label{def pdf rho}
\varrho(z):=\frac{z^{-\alpha}(1-z)^{-\alpha}}{\cB(1-\alpha,1-\alpha)}, 
\quad z\in(0,1),
\end{equation}
where $\cB(\beta,\gamma):=\frac{\Gamma(\beta)\Gamma(\gamma)}{\Gamma(\beta+\gamma)}$
denotes the Beta function for all $\beta,\gamma\in(0,\infty)$, 
and $\Gamma$ denotes the Gamma function.
Let $\xi^\theta:\Omega\to[0,1]$, $\theta\in\Theta$, 
be i.i.d. random variables
such that $\bP(\xi^0\leq y)=\int_0^y\varrho(z)\,dz$ for all $y\in[0,1]$.
For each $\theta\in\Theta$ and $t\in[0,T)$, define
$\cR^\theta_t:=t+(T-t)\xi^\theta$.
Moreover, we assume
that 
$$
\big(
\bX^{\theta,t,x}_{s},\bV^{\theta,t,x}_{s}
\big)
_{(\theta,t,s,x)\in\Theta\times\Delta\times\bR^d}
\quad
\text{and} 
\quad
\big(\xi^\theta\big)_{\theta\in\Theta} 
$$
are independent.
For each $d\in \bN$, $(t,x)\in[0,T]\times\bR^d$, $s\in[t,T]$, 
and $\theta\in\Theta$, let 
$\big(\bX^{(\theta,t,x,l,i)}_{s},\bV^{(\theta,t,x,l,i)}_{s}\big)_{(l,i)\in\bN\times\bZ}$
be independent copies of 
$\big(\bX^{\theta,t,x}_{s},\bV^{\theta,t,x}_{s}\big)$.
Let $g\in C([0,T]\times\bR^d\to\bR)$ 
and $f\in C([0,T]\times\bR^d\times\bR\times\bR^d)$.
Moreover, let 
$$
F: C([0,T]\times\bR^d,\bR^{d+1})\to C([0,T]\times\bR^d,\bR)
$$
be the operator such that
\begin{equation}
\label{def F MLP}
[0,T]\times\bR^d\ni(t,x)\mapsto (F(\mathbf{v}))(t,x)
:=f(t,x,\mathbf{v}(t,x))\in\bR,\quad \mathbf{v}\in C([0,T]\times\bR^d,\bR^{d+1}).
\end{equation}
Then for each $n\in\bN_0$, $M\in\bN$, and $\theta\in\Theta$, 
let $U^{\theta}_{n,M}:[0,T)\times \bR^d \times \Omega \to \bR^{d+1}$ 
satisfy for all $(t,x)\in[0,T)\times\bR^d$ and $\omega\in\Omega$ that
$U^{\theta}_{0,M}(t,x)=\mathbf{0}$ and
$$
U^{\theta}_{n,M}(t,x)
=
(g(x),0)
+
\frac{1}{M^n}\sum_{i=1}^{M^n}
\Big[g\Big(\bX^{(\theta,t,x,0,-i)}_{T}\Big)-g(x)\Big]
\Big(1,\bV^{(\theta,t,x,0,-i)}_{T}\Big)
$$
\begin{equation}
\label{def MLP general}
+\sum_{l=0}^{n-1}\frac{T-t}{M^{n-l}}
\left[\sum^{M^{n-l}}_{i=1}
\varrho^{-1}\Big(\frac{\cR^{(\theta,l,i)}_t-t}{T-t}\Big)
\Big[F\big(U^{(\theta,l,i)}_{l,M}\big)
-\mathbf{1}_{\{l\geq 1\}}F\big(U^{(\theta,-l,i)}_{l-1,M}\big)\Big]
\Big(\cR^{(\theta,l,i)}_t,\bX^{(\theta,t,x,l,i)}_{\cR^{(\theta,l,i)}_t}\Big)
\Big(1,\bV^{(\theta,t,x,l,i)}_{\cR^{(\theta,l,i)}_t}\Big)
\right],
\end{equation}
where $\Big(U^{(\theta,l,i)}_{n,M}(t,x)\Big)_{(l,i)
\in \bZ\times\bN_0}$
are independent copies of $U^{\theta}_{n,M}(t,x)$ 
for each $(t,x)\in[0,T)\times\bR^d$,
and $\Big(\cR^{(\theta,l,i)}_t\Big)_{(l,i)
\in \bN\times\bN_0}$
are independent copies of $\cR^\theta_t$ for each $t\in[0,T)$.

Furthermore, let $a,a_1,a_2,a_3,b,b_1,c,L,p,\rho\in(0,\infty)$, 
and let $u:[0,T)\times\bR^d\to\bR$, $w:[0,T)\times\bR^d\to\bR^d$ be
measurable functions. We assume for all $(t,x)\in[0,T)\times\bR^d$, 
$s\in(t,T]$, and $v_1,v_2\in\bR^{d+1}$ that
\begin{align}
&
\big\|g(\bX^{0,t,x}_T)\big\|_{L_1}
+
\big\|g(\bX^{0,t,x}_T)\bV^{0,t,x}_T\big\|_{L_1}
+
\int_t^T
\big\|f\big(s,\bX^{0,t,x}_s,u(s,\bX^{0,t,x}_s),w(s,\bX^{0,t,x}_s)\big)\big\|_{L_1}
ds
\nonumber\\
&
+
\int_t^T
\big\|f\big(s,\bX^{0,t,x}_s,u(s,\bX^{0,t,x}_s),w(s,\bX^{0,t,x}_s)\big)\bV^{0,t,x}_s\big\|_{L_1}
ds
<\infty,
\label{MLP integrability}
\\
&
\big\|\big(d^p+\|\bX^{0,t,x}_s\|^2\big)\big\|_{L_1}
\leq ae^{\rho(s-t)}(d^p+\|x\|^2),
\quad
\big\|\big(d^p+\|\bX^{0,t,x}_s\|^2\big)\big\|_{L_2}
\leq a_1e^{\rho(s-t)}(d^p+\|x\|^2),
\label{MLP Z est}
\\
&
\big\|\bX^{0,t,x}_s-x\big\|_{L_2}^2
\leq
a_2(s-t)e^{\rho(s-t)}(d^p+\|x\|^2),
\quad
\big\|\bX^{0,t,x}_s-x\big\|_{L_4}^2
\leq
a_3(s-t)e^{\rho(s-t)}(d^p+\|x\|^2),
\label{MLP Z x est}
\\
&
\|\bV^{0,t,x}_s\|_{L_2}^2\leq b(s-t)^{-1},
\quad
\|\bV^{0,t,x}_s\|_{L_4}^2
\leq b_1(s-t)^{-1},
\label{MLP V est}
\\
&
\big(u(t,x),w(t,x)\big)
=
\bE\big[g\big(\bX^{0,t,x}_T\big)\big(1,\bV^{0,t,x}_T\big)\big]
+
\int_t^T\bE\big[
f\big(s,\bX^{0,t,x}_s,u(s,\bX^{0,t,x}_s),w(s,\bX^{0,t,x}_s)\big)
\big(1,\bV^{0,t,x}_s\big)
\big]\,
ds,
\label{MLP FP}
\\
&
|f(t,x,v_1)-f(t,x,v_2)|^2\leq L\|w_1-w_2\|^2,
\label{MLP Lip f}
\\
&
|f(t,x,0,\mathbf{0})|^2+|g(x)|^2\leq c(d^p+\|x\|^2),
\label{MLP LG}
\\
&
(|u(t,x)|+(T-t)^{1/2}\|w(t,x)\|)^2\leq c'(d^p+\|x\|^2).
\label{MLP growth u v}
\end{align}

\subsection{Error Analysis of Multilevel Picard Approximations}
In the subsection, we provide the error analysis for the MLP approximations
introduced in \eqref{def MLP general}. Some details of the proofs are collected 
in Appendix \ref{appendix MLP}.

\begin{lemma}
\label{lemma MLP property}
Assume Setting \ref{MLP setting}.
Then the following holds:
\begin{enumerate}[(i)]
\item
for all $n\in\bN_0$, $M\in\bN$, and $\theta\in \Theta$,
$U^\theta_{n,M}:[0,T)\times\bR^d\times\Omega\to\bR^{d+1}$
is measurable;
\item
for all $n\in\bN_0$, $M\in\bN$, and $\theta\in \Theta$,
$$
\sigma\left(\big(U^\theta_{n,M}(t,x)\big)_{(t,x)\in[0,T)\times\bR^d}\right)
\subseteq
\sigma\left(
\big(\xi^{(\theta,\upsilon)}\big)_{\upsilon\in\Theta},
\big(\bX^{\theta,\upsilon,t,x}_s,\bV^{\theta,\upsilon,t,x}_s\big)
_{(t,s,x,\upsilon)\in\Delta\times\bR^d\times\Theta}
\right);
$$ 
\item
for all $n\in\bN_0$, $M\in\bN$, and $\theta\in \Theta$,
$\big(U^\theta_{n,M}(t,x)\big)_{(t,x)\in[0,T)\times\bR^d}$,
$\big(\bX^{\theta,t,x}_s,\bV^{\theta,t,x}_s\big)_{(t,s,x)\in\Delta\times\bR^d}$,
and $\xi^\theta$ are independent;
\item
for all $n,m\in\bN_0$, $M\in\bN$, $\theta\in \Theta$,
and $i,j,k,l\in \bZ$ with $(i,j)\neq(k,l)$,
\\
$\big(U^{(\theta,i,j)}_{n,M}(t,x)\big)_{(t,x)\in[0,T)\times\bR^d}$,
$\big(U^{(\theta,k,l)}_{m,M}(t,x)\big)_{(t,x)\in[0,T)\times\bR^d}$,
$\big(\bX^{\theta,i,j,t,x}_s,\bV^{\theta,i,j,t,x}_s\big)_{(t,s,x)\in\Delta\times\bR^d}$,
and $\xi^{(\theta,i,j)}$ are independent;
\item
for all $n\in\bN_0$, $M\in\bN$, and $(t,x)\in[0,T)\times\bR^d$,
$U^{\theta}_{n,M}(t,x)$, $\theta\in\Theta$, are i.i.d..
\end{enumerate}
\end{lemma}

\begin{proof}
See Appendix \ref{sec proof C1}.
\end{proof}

\begin{lemma}
\label{Lemma MLP property}
Assume Setting \ref{MLP setting}.
Let $M\in\bN$, and let $\operatorname{dim}:\Theta\to\bN$ be the mapping
satisfying for all $n\in\bN$ and $\theta\in\bZ^n$ that
$\dim(\theta)=n$. Then the following holds:
\begin{enumerate}[(i)]
\item
for all $t\in[0,T)$, $l\in N_0$, 
and $\eta,\zeta,\upsilon\in\Theta$ 
with $\min\{\dim(\eta),\dim(\zeta)\}\geq\dim(\upsilon)$,
\begin{align}
&
\sup_{x\in\bR^d}
\left(
\frac{e^{\rho t}(T-t)^2}{d^p+\|x\|^2}
\bE\left[
\Big\|
\varrho^{-1}\Big(\frac{\cR^{\upsilon}_t-t}{T-t}\Big)
\big(
F\big(U^\eta_{l,M}\big)
-
\mathbf{1}_{\{l\geq 1\}}F\big(U^\zeta_{l-1,M}\big)
\big)
\big(\cR^{\upsilon}_t,\bX^{\upsilon,t,x}_{\cR^\upsilon_t}\big)
\big(1,\bV^{\upsilon,t,x}_{\cR^\upsilon_t}\big)
\Big\|^2
\right]
\right)^{1/2}
\nonumber\\
&
\leq
(a+a_1b_1)^{1/2}(T-t)^{1/2}
\Bigg[
\bigg(\int_t^T
\frac{[(s-t)^{-1}+1]}{(T-s)\varrho(\frac{s-t}{T-t})}
\,ds\bigg)^{1/2}
\sup_{(r,x)\in[t,T]\times\bR^d}
\bigg(
\frac{\mathbf{1}_{\{l=0\}}[e^{\rho r}(T-r)]^{1/2}}{(d^p+\|x\|^2)^{1/2}}
\big|(F(\mathbf{0}))(r,x)\big|
\bigg)
\nonumber\\
& \quad
+
\left(
\int_t^T
\frac{[(s-t)^{-1}+1]}{(T-s)\varrho(\frac{s-t}{T-t})}
\sup_{(r,x)\in[s,T)\times\bR^d}
\left(
\frac{\mathbf{1}_{\{l\geq 1\}}Le^{\rho r}(T-r)}{d^p+\|x\|^2}
\bE\left[\big\|(U^\eta_{l,M}-U^\zeta_{l-1,M})(r,x)\big\|^2\right]
\right)
ds
\right)^{1/2}
\Bigg];
\label{MLP est F l l-1}
\end{align}
\item
for all $\theta\in\Theta$,
\begin{align}
&
\sup_{(t,x)\in[0,T)\times\bR^d}
\left(
(d^p+\|x\|^2)^{-1}e^{\rho t}
\big\|
\big[g\big(\bX^{\theta,t,x}_T\big)-g(x)\big]\big(1,\bV^{\theta,t,x}_T\big)
\big\|_{L_2}^2
\right)
\leq
e^{\rho T}L(a_2T+a_3b_1c);
\label{MLP est g}
\end{align}
\item
for all $l\in N_0$ and $\eta,\zeta,\upsilon\in\Theta$ 
with $\min\{\dim(\eta),\dim(\zeta)\}\geq\dim(\upsilon)$,
\begin{align}
&
\sup_{(t,x)\in[0,T)\times\bR^d}
\left(
\frac{e^{\rho t}(T-t)^2}{d^p+\|x\|^2}
\bE\left[
\Big\|
\varrho^{-1}\Big(\frac{\cR^{\upsilon}_t-t}{T-t}\Big)
\big(
F\big(U^\eta_{l,M}\big)
-
\mathbf{1}_{\{l\geq 1\}}F\big(U^\zeta_{l-1,M}\big)
\big)
\big(\cR^{\upsilon}_t,\bX^{\upsilon,t,x}_{\cR^\upsilon_t}\big)
\big(1,\bV^{\upsilon,t,x}_{\cR^\upsilon_t}\big)
\Big\|^2
\right]
\right)
\nonumber\\
&
<\infty;
\label{MLP F l l-1 finite}
\end{align}
\item
\label{MLP U finite}
for all $\theta\in\Theta$,
\begin{equation}
\sup_{(t,x)\in[0,T)\times\bR^d}
\left(
(d^p+\|x\|^2)^{-1}e^{\rho t}
\big\|U^\theta_{n,M}(t,x)\big\|_{L_2}^2
\right)
<\infty;
\end{equation}
\item
for all $t\in[0,T)$, $n\in\bN_0$, and $\eta,\upsilon\in\Theta$
with $\dim(\eta)\geq\dim(\upsilon)$,
\begin{align}
&
\sup_{(t,x)\in[0,T)\times\bR^d}
\left(
\frac{e^{\rho t}}{d^p+\|x\|^2}
\Big\|
(T-t)\varrho^{-1}\Big(\frac{\cR^{\upsilon}_t-t}{T-t}\Big)
\big(F\big(U^\eta_{l,M}\big)\big)
\big(\cR^{\upsilon}_t,\bX^{\upsilon,t,x}_{\cR^\upsilon_t}\big)
\big(1,\bV^{\upsilon,t,x}_{\cR^\upsilon_t}\big)
\Big\|_{L_2}^2
\right)
\nonumber\\
&
=
\sup_{x\in\bR^d}
\left(
\frac{e^{\rho t}(T-t)}{d^p+\|x\|^2}
\int_t^T
\varrho^{-1}\Big(\frac{s-t}{T-t}\Big)
\cdot
\big\|
\big(F\big(U^\eta_{l,M}\big)\big)
\big(s,\bX^{\upsilon,t,x}_s\big)
\big(1,\bV^{\upsilon,t,x}_s\big)
\big\|_{L_2}^2\,
ds
\right)
<\infty.
\label{MLP est F}
\end{align}
\end{enumerate}
\end{lemma}

\begin{proof}
See Appendix \ref{sec proof C2}.
\end{proof}

\begin{lemma}
\label{lemma iid expectation}
Assume Setting \ref{MLP setting}, and let $\theta\in\Theta$.
Then the following holds:
\begin{enumerate}[(i)]
\item
for all $l\in\bN_0$ and $(t,x)\in[0,T)\times\bR^d$,
$$
\Big(
F\big(U^{(\theta,l,i)}_{l,M}\big)
-
\mathbf{1}_{\{l\geq 1\}}F\big(U^{(\theta,-l,i)}_{l-1,M}\big)
\Big)
\Big(
\cR^{(\theta,l,i)}_t,\bX^{(\theta,l,i,t,x)}_{\cR^{(\theta,l,i)}_t}
\Big)
\Big(
1,\bV^{(\theta,l,i,t,x)}_{\cR^{(\theta,l,i)}_t}
\Big),
\quad
i\in\bN,
$$
are independently and identically distributed;
\item
for all $n\in\bN$ and $(t,x)\in[0,T)\times\bR^d$,
\begin{align}
&
\bE\left[U^\theta_{n,M}(t,x)\right]
\nonumber\\
&
=
\bE\left[g(\bX^{\theta,t,x}_T)(1,\bV^{\theta,t,x}_T)\right]
+
(T-t)\bE\left[
\varrho^{-1}\Big(\frac{\cR^\theta_t-t}{T-t}\Big)
\Big(F\big(U^\theta_{n-1,M}\big)\Big)
\big(
\cR^\theta_t,\bX^{\theta,t,x}_T
\big)
\big(
1,\bV^{\theta,t,x}_{\cR^\theta_t}
\big)
\right]
\nonumber\\
&
=
\bE\left[g(\bX^{\theta,t,x}_T)(1,\bV^{\theta,t,x}_T)\right]
+
\int_t^T
\bE\left[
\Big(F\big(U^\theta_{n-1,M}\big)\Big)
\big(s,\bX^{\theta,t,x}_s\big)
\big(1,\bV^{\theta,t,x}_s\big)
\right]
ds.
\label{E U t x}
\end{align}
\end{enumerate}
\end{lemma}

\begin{proof}
See Appendix \ref{sec proof C3}.
\end{proof}

\begin{lemma}[Recursive Error]
\label{lemma recursive error}
Assume Setting \ref{MLP setting}. Let $n,M\in\bN$, $t\in[0,T)$.
Then we have for all $\beta\in(0,1/2)$ that
\begin{align}
&
\sup_{x\in\bR^d}
\left(
e^{\rho t}(d^p+\|x\|^2)^{-1}(T-t)
\big\|
U^0_{n,M}(t,x)-(u,w)(t,x)
\big\|_{L_2}^2
\right)^{1/2}
\nonumber\\
&
\leq
\frac{e^{\rho T/2}}{\sqrt{M^n}}
\left(
(a+a_1b_1)^{1/2}[cT\alpha^{-1}(4^{-\alpha}+T)]^{1/2}
\sqrt{\cB(1-\alpha,1-\alpha)}
+
[LT(a_2T+a_3b_1c)]^{1/2}
\right)
\nonumber\\
& \quad
+\big[L(a+a_1b_1)(b+T^{3/2})\big]^{1/2}(1+\sqrt{2})
5(1+T)\sqrt{\cB(1-\alpha,1-\alpha)}
\nonumber\\
& \quad
\cdot
\sum_{l=0}^{n-1}
\Bigg[
\frac
{1}
{\sqrt{M^{n-l-1}}}
\left(
\int_t^T
\sup_{(r,y)\in[s,T)\times\bR^d}
\left(
\frac{e^{\rho r}(T-r)}{d^p+\|y\|^2}
\big\|
U^0_{l,M}(r,y)-(u,w)(r,y)
\big\|_{L_2}^2
\right)^{\frac{1+\beta}{\beta}}
ds
\right)^{\frac{\beta}{2(1+\beta)}}
\Bigg]
\nonumber\\
&
<\infty.
\label{MLP error recursive}
\end{align}
\end{lemma}

\begin{proof}
Throughout the proof of this lemma, for every 
$x\in\bR^d$, $s\in[t,T]$, and $i\in\bZ$ we use the notations 
\begin{align*}
&
U^0_{n,M}(t,x)
=
\left(U^{0,(1)}_{n,M}(t,x),
U^{0,(2)}_{n,M}(t,x),
\dots,
U^{0,(d+1)}_{n,M}(t,x)\right),
\\
&
(u,w)(t,x)=
\left(
(u,w)^{(1)}(t,x),(u,w)^{(2)}(t,x),\dots,(u,w)^{(d+1)}(t,x)
\right),
\\
&
\bV^{(0,t,x,0,i)}_s=
\left(
\bV^{(0,t,x,0,i),(1)}_s,\bV^{(0,t,x,0,i),(2)}_s,\dots,\bV^{(0,t,x,0,i),(d)}_s
\right),
\end{align*}
and set
$$
\bV^{(0,t,x,0,i),(0)}_s=1.
$$
By \eqref{def MLP general}, the triangle inequality, 
(i) in Lemma \ref{lemma iid expectation}, 
and the fact that it holds for all $x\in\bR^d$ that
$\big(\bX^{t,x,\theta}_T,\bV^{t,x,\theta}_T\big)$, $\theta\in\Theta$, are i.i.d.,
we first notice for all $x\in\bR^d$ that
\begin{align}
&
\left(
\sum_{k=1}^{d+1}
\bVar
\left[U^{0,(k)}_{n,M}(t,x)
\right]
\right)^{1/2}
\nonumber\\
&
\leq
\left(
\bVar\left[
\frac{1}{M^n}\sum_{i=1}^{M^n}\sum_{k=0}^d
\left(g\big(\bX^{(0,t,x,0,-i)}_T\big)-g(x)\right)
\bV^{(0,t,x,0,i),(k)}_T
\right]
\right)^{1/2}
+\sum_{l=0}^{n-1}
\Bigg(
\bVar\bigg[
\frac{T-t}{M^{n-l}}
\nonumber\\
& \quad
\sum_{i=1}^{M^{n-l}}\sum_{k=0}^d
\varrho^{-1}\Big(\frac{\cR^{(0,l,i)}_t-t}{T-t}\Big)
\left(
F\big(U^{0,l,i}_{l,M}\big)-\mathbf{1}_{\{l\geq 1\}}F\big(U^{0,l-1,i}_{l-1,M}
\right)
\Big(\cR^{(0,l,i)}_t,\bX^{(0,t,x,l,i)}_{\cR^{(0,l,i)}_t}\Big)
\bV^{(0,t,x,l,i),(k)}_{\cR^{(0,l,i)}_t}
\bigg]
\Bigg)^{1/2}
\nonumber\\
&
\leq
\frac{1}{\sqrt{M^n}}
\big\|
\big(g(\bX^{0,t,x}_T)-g(x)\big)
\big(1,\bV^{0,t,x}_T\big)
\big\|_{L_2}
+
\sum_{l=0}^{n-1}\frac{T-t}{\sqrt{M^{n-l}}}
\Big\|
\varrho^{-1}\Big(\frac{\cR^{(0,l,1)}_t-t}{T-t}\Big)
\nonumber\\
& \quad
\cdot
\left(
F\big(U^{0,l,i}_{l,M}\big)-\mathbf{1}_{\{l\geq 1\}}F\big(U^{0,l-1,i}_{l-1,M}\big)
\right)
\Big(\cR^{(0,l,i)}_t,\bX^{(0,t,x,l,i)}_{\cR^{(0,l,i)}_t}\Big)
\Big(1,\bV^{(0,t,x,l,1)}_{\cR^{(0,l,1)}_t}\Big)
\Big\|_{L_2}.
\nonumber
\end{align}
Hence, the application of (i) and (ii) in Lemma \ref{Lemma MLP property}
(applied for every $l\in[0,n-1]\cap \bN$ with $\eta\cal(0,l,1)$, 
$\zeta\cal(0,-l,1)$, and $\upsilon\cal(0,l,1)$) in the notation of 
Lemma \ref{Lemma MLP property} implies for all $x\in\bR^d$ that
\begin{align}
&
\left(
\sum_{k=1}^{d+1}
\bVar
\left[U^{0,(k)}_{n,M}(t,x)
\right]
\right)^{1/2}
\nonumber\\
&
\leq
\frac{e^{\rho(T-t)/2}(d^p+\|x\|^2)^{1/2}[L(a_2T+a_3b_1c)]^{1/2}}{\sqrt{M^n}}
+
\frac{e^{\rho(T-t)/2}(d^p+\|x\|^2)^{1/2}(a+a_1b_1)^{1/2}(T-t)}{\sqrt{M^n}}
\nonumber\\
& \quad
\cdot
\left(
\int_t^T[(s-t)^{-1}+1]\varrho^{-1}\Big(\frac{s-t}{T-t}\Big)(T-s)^{-1}\,ds
\right)^{1/2}
\sup_{(r,z)\in[0,T]\times\bR^d}\left(\frac{|(F(0))(r,z)|}{(d^p+\|z\|^2)^{1/2}}\right)
\nonumber\\
& \quad
+\sum_{l=1}^{n-1}
\Bigg[
\frac
{e^{-\rho t/2}(d^p+\|x\|^2)^{1/2}(a+a_1b_1)^{1/2}(T-t)^{1/2}L^{1/2}}
{\sqrt{M^{n-l}}}
\nonumber\\
& \quad
\cdot
\left(
\int_t^T\frac{(s-t)^{-1}+1}{(T-s)\varrho(\frac{s-t}{T-t})}
\sup_{(r,y)\in[s,T)\times\bR^d}
\left(
\frac{e^{\rho r}(T-r)}{d^p+\|y\|^2}
\big\|
U^{(0,l,1)}_{l,M}(r,y)-U^{(0,-l,1)}_{l-1,M}(r,y)
\big\|_{L_2}^2
\right)
ds
\right)^{1/2}
\Bigg].
\label{Var est 1}
\end{align}
Furthermore, we notice that (v) in Lemma \ref{lemma MLP property} ensures for all
$l\in\bN$, $\eta,\zeta\in\Theta$, and $(s,x)\in[0,T)\times\bR^d$ that
\begin{align*}
\big\|
U^{\eta}_{l,M}(s,x)-U^{\zeta}_{l-1,M}(s,x)
\big\|_{L_2}
&
\leq
\big\|
U^{\eta}_{l,M}(s,x)-(u,w)(s,x)
\big\|_{L_2}
+
\big\|
U^{\zeta}_{l-1,M}(s,x)-(u,w)(s,x)
\big\|_{L_2}
\nonumber\\
&
=
\big\|
U^0_{l,M}(s,x)-(u,w)(s,x)
\big\|_{L_2}
+
\big\|
U^0_{l-1,M}(s,x)-(u,w)(s,x)
\big\|_{L_2}.
\end{align*}
This together with \eqref{Var est 1} and the fact that it holds for all
$\{a_i\}_{i=1}^n\subseteq[0,\infty]$ that
$\sum_{l=1}^{n-1}(a_l+a_{l-1})
\leq \sum_{l=0}^{n-1}\big(2-\mathbf{1}_{\{l=n-1\}}\big)a_l$
imply for all $x\in\bR^d$ that
\begin{align}
&
\left(
\sum_{k=1}^{d+1}
\bVar
\left[U^{0,(k)}_{n,M}(t,x)
\right]
\right)^{1/2}
\nonumber\\
&
\leq
\frac{e^{\rho(T-t)/2}(d^p+\|x\|^2)^{1/2}[L(a_2T+a_3b_1c)]^{1/2}}{\sqrt{M^n}}
+
\frac{e^{\rho(T-t)/2}(d^p+\|x\|^2)^{1/2}(a+a_1b_1)^{1/2}(T-t)}{\sqrt{M^n}}
\nonumber\\
& \quad
\cdot
\left(\int_t^T[(s-t)^{-1}+1]\varrho\Big(\frac{s-t}{T-t}\Big)(T-s)^{-1}
\,ds\right)^{1/2}
\sup_{(r,z)\in[s,T)\times\bR^d}\left(\frac{|(F(0))(r,z)|}{(d^p+\|z\|^2)^{1/2}}\right)
\nonumber\\
& \quad
+\sum_{l=0}^{n-1}
\Bigg[
\frac
{\big(2-\mathbf{1}_{\{l=n-1\}}\big)
e^{-\rho t/2}(d^p+\|x\|^2)^{1/2}(a+a_1b_1)^{1/2}(T-t)^{1/2}L^{1/2}}
{\sqrt{M^{n-l-1}}}
\nonumber\\
& \quad
\cdot
\left(
\int_t^T\frac{(s-t)^{-1}+1}{(T-s)\varrho(\frac{s-t}{T-t})}
\sup_{(r,y)\in[s,T)\times\bR^d}
\left(
\frac{e^{\rho r}(T-r)}{d^p+\|y\|^2}
\big\|
U^0_{l,M}(r,y)-(u,w)(r,y)
\big\|_{L_2}^2
\right)
ds
\right)^{1/2}
\Bigg].
\nonumber
\end{align}
Therefore, we have that
\begin{align}
&
\sup_{x\in\bR^d}
\left(
(d^p+\|x\|^2)^{-1}e^{\rho t}\sum_{k=1}^{d+1}\bVar\left[U^{0,(k)}_{n,M}(t,x)\right]
\right)^{1/2}
\nonumber\\
&
\leq
\frac{e^{\rho T/2}[L(a_2T+a_3b_1c)]^{1/2}}{\sqrt{M^n}}
\nonumber\\
& \quad
+
\frac{e^{\rho T/2}(a+a_1b_1)^{1/2}(T-t)}{\sqrt{M^n}}
\left(\int_t^T\frac{(s-t)^{-1}+1}{(T-s)\varrho(\frac{s-t}{T-t})}\,ds\right)^{1/2}
\sup_{(r,z)\in[0,T]\times\bR^d}\left(\frac{|(F(0))(r,z)|}{(d^p+\|z\|^2)^{1/2}}\right)
\nonumber\\
& \quad
+\sum_{l=0}^{n-1}
\Bigg[
\frac
{\big(2-\mathbf{1}_{\{l=n-1\}}\big)(a+a_1b_1)^{1/2}(T-t)^{1/2}L^{1/2}}
{\sqrt{M^{n-l-1}}}
\nonumber\\
& \quad
\cdot
\left(
\int_t^T\frac{(s-t)^{-1}+1}{(T-s)\varrho(\frac{s-t}{T-t})}
\sup_{(r,y)\in[s,T)\times\bR^d}
\left(
\frac{e^{\rho r}(T-r)}{d^p+\|y\|^2}
\big\|
U^0_{l,M}(r,y)-(u,w)(r,y)
\big\|_{L_2}^2
\right)
ds
\right)^{1/2}
\Bigg]
\nonumber\\
&
<\infty.
\label{Var est 2}
\end{align}
Next, by \eqref{MLP FP} and (ii) in Lemma \ref{lemma iid expectation}
we observe that it holds for all $x\in\bR^d$ that
$$
\bE\left[U^0_{n,M}(t,x)-(u,w)(t,x)\right]
=
\int_t^T
\bE\left[
\big(
(F(U^0_{n-1,M}))(s,\bX^{0,t,x}_s)
-
(F(u,w))(s,\bX^{0,t,x}_s)
\big)
\big(1,\bV^{0,t,x}_s\big)
\right]
ds.
$$
Thus, by (iii) in Lemma \ref{lemma MLP property}, \eqref{MLP Z est}, 
\eqref{MLP V est}, \eqref{MLP Lip f}, Minkowski's integral inequality,
Cauchy-Schwarz inequality, and e.g., Lemma 2.2 in \cite{HJKNW2020}
we obtain for all $x\in\bR^d$ that
\begin{align*}
&
\big\|
\bE\left[U^0_{n,M}(t,x)-(u,w)(t,x)\right]
\big\|
\\
&
\leq
\int_t^T
\big\|
\big[
(F(U^0_{n-1,M}))(s,\bX^{0,t,x}_s)
-
(F(u,w))(s,\bX^{0,t,x}_s)
\big]
\big(1,\bV^{0,t,x}_s\big)
\big\|_{L_1}
\,ds
\\
&
=
\int_t^T
\bE\left[
\big\|
[(F(U^0_{n-1,M}))(s,z)-(F(u,w))(s,z)]
(1,y)
\big\|_{L_1}
\Big|_{(z,y)=(\bX^{0,t,x}_s,\bV^{0,t,x}_s)}
\right]
ds
\\
&
\leq
\int_t^T
\left[
\sup_{(r,z)\in[s,T)\times\bR^d}
\frac
{(T-r)\big\|(F(U^0_{n-1,M}))(r,z)-(F(u,w))(r,z)\big\|_{L_1}}
{(d^p+\|z\|^2)^{1/2}}
\right]
\\
& \quad
\cdot
\frac
{\bE\left[(d^p+\|\bX^{0,t,x}_s\|^2)^{1/2}(1+\|\bV^{0,t,x}_s\|)\right]}
{(T-s)^{1/2}}
\,ds
\\
&
\leq
L^{1/2}\int_t^T
\Bigg[
\sup_{(r,z)\in[s,T)\times\bR^d}
\frac
{(T-r)\big\|U^0_{n-1,M}(r,z)-(u,w)(r,z)\big\|_{L_1}}
{(d^p+\|z\|^2)^{1/2}}
\Bigg]
(T-s)^{-1/2}
\\
& \quad
\cdot
\left(\bE\big[d^p+\|\bX^{0,t,x}_s\|^2\big]\right)^{1/2}
\left(1+\|\bV^{0,t,x}_s\|_{L_2}\right)
ds
\\
&
\leq
(aL)^{1/2}
\int_t^T
\Bigg[
\sup_{(r,z)\in[s,T)\times\bR^d}
\frac
{(T-r)\big\|U^0_{n-1,M}(r,z)-(u,w)(r,z)\big\|_{L_1}}
{(d^p+\|z\|^2)^{1/2}}
\Bigg]
\frac{e^{\rho(s-t)/2}(d^p+\|x\|^2)^{1/2}}{(T-s)^{1/2}}
\,ds
\\
& \quad
+
(abL)^{1/2}
\int_t^T
\Bigg[
\sup_{(r,z)\in[s,T)\times\bR^d}
\frac
{(T-r)\big\|U^0_{n-1,M}(r,z)-(u,w)(r,z)\big\|_{L_1}}
{(d^p+\|z\|^2)^{1/2}}
\Bigg]
\frac{e^{\rho(s-t)/2}(d^p+\|x\|^2)^{1/2}}{(s-t)^{1/2}(T-s)^{1/2}}
\,ds.
\end{align*}
Hence, by Jensen's inequality and \eqref{time integral est 1/2} we have that
\begin{align*}
&
\sup_{x\in\bR^d}
\left[
\frac{e^{\rho t}\big\|\bE\big[U^0_{n,M}(t,x)-(u,w)(t,x)\big]\big\|^2}
{d^p+\|z\|^2}
\right]^{1/2}
\\
&
\leq
\left(
\int_t^T
\left[
\sup_{(r,z)\in[s,T)\times\bR^d}
\frac
{e^{\rho r}(T-r)\big\|U^0_{n-1,M}(r,z)-(u,w)(r,z)\big\|_{L_2}^2}
{d^p+\|z\|^2}
\right]
\frac{2aL(T-t)}{(T-s)^{1/2}}
\,ds
\right)^{1/2}
\\
& \quad
+
\left(
\int_t^T
\left[
\sup_{(r,z)\in[s,T)\times\bR^d}
\frac
{e^{\rho r}(T-r)\big\|U^0_{n-1,M}(r,z)-(u,w)(r,z)\big\|_{L_2}^2}
{d^p+\|z\|^2}
\right]
\frac
{2abL}{(s-t)^{1/2}(T-s)^{1/2}}
\,ds
\right)^{1/2}.
\end{align*}
This together with \eqref{Var est 2} ensure that
\begin{align}
&
\sup_{x\in\bR^d}
\left(
e^{\rho t}(d^p+\|x\|^2)^{-1}
\big\|
U^0_{n,M}(t,x)-(u,w)(t,x)
\big\|_{L_2}^2
\right)^{1/2}
\nonumber\\
&
\leq
\sup_{x\in\bR^d}
\left[
e^{\rho t}(d^p+\|x\|^2)^{-1}
\sum_{k=1}^{d+1}
\left(
\left(\bE\left[U^{0,(k)}_{n,M}(t,x)\right]-(u,w)^{(k)}(t,x)\right)^2
+
\bVar\left[U^{0,(k)}_{n,M}(t,x)\right]
\right)
\right]^{1/2}
\nonumber\\
&
\leq
\sup_{x\in\bR^d}
\left(
e^{\rho t}(d^p+\|x\|^2)^{-1}
\big\|
\bE\left[U^0_{n,M}(t,x)\right]-(u,w)(t,x)
\big\|^2
\right)^{1/2}
\nonumber\\
& \quad
+\sup_{x\in\bR^d}
\left(
e^{\rho t}(d^p+\|x\|^2)^{-1}
\sum_{k=1}^{d+1}\bVar\left[U^{0,(k)}_{n,M}(t,x)\right]
\right)^{1/2}
\nonumber\\
&
\leq
\big[2aL(b+T^{3/2})\big]^{1/2}
\Bigg(
\int_t^T
\left[
\sup_{(r,z)\in[s,T)\times\bR^d}
\frac
{e^{\rho r}(T-r)\big\|U^0_{n-1,M}(r,z)-(u,w)(r,z)\big\|_{L_2}^2}
{d^p+\|z\|^2}
\right]
\nonumber\\
& \quad
\cdot
(s-t)^{-1/2}(T-s)^{-1/2}
\,ds
\Bigg)^{1/2}
+
\frac{e^{\rho T/2}[L(a_2T+a_3b_1c)]^{1/2}}{\sqrt{M^n}}
\nonumber\\
& \quad
+
\frac{e^{\rho T/2}(a+a_1b_1)^{1/2}(T-t)}{\sqrt{M^n}}
\left(\int_t^T\frac{(s-t)^{-1}+1}{(T-s)\varrho(\frac{s-t}{T-t})}\,ds\right)^{1/2}
\sup_{(r,z)\in[0,T]\times\bR^d}\left(\frac{|(F(0))(r,z)|}{(d^p+\|z\|^2)^{1/2}}\right)
\nonumber\\
& \quad
+\sum_{l=0}^{n-1}
\Bigg[
\frac
{\big(2-\mathbf{1}_{\{l=n-1\}}\big)(a+a_1b_1)^{1/2}(T-t)^{1/2}L^{1/2}}
{\sqrt{M^{n-l-1}}}
\nonumber\\
& \quad
\cdot
\left(
\int_t^T\frac{(s-t)^{-1}+1}{(T-s)\varrho(\frac{s-t}{T-t})}
\sup_{(r,y)\in[s,T)\times\bR^d}
\left(
\frac{e^{\rho r}(T-r)}{d^p+\|y\|^2}
\big\|
U^0_{l,M}(r,y)-(u,w)(r,y)
\big\|_{L_2}^2
\right)
ds
\right)^{1/2}
\Bigg]
\nonumber\\
&
\leq
\frac{e^{\rho T/2}(a+a_1b_1)^{1/2}(T-t)}{\sqrt{M^n}}
\left(\int_t^T\frac{(s-t)^{-1}+1}{(T-s)\varrho(\frac{s-t}{T-t})}\,ds\right)^{1/2}
\sup_{(r,z)\in[0,T]\times\bR^d}\left(\frac{|(F(0))(r,z)|}{(d^p+\|z\|^2)^{1/2}}\right)
\nonumber\\
& \quad
+
\frac{e^{\rho T/2}[L(a_2T+a_3b_1c)]^{1/2}}{\sqrt{M^n}}
+
\big[L(a+a_1b_1)(b+T^{3/2})\big]^{1/2}(1+\sqrt{2})
\nonumber\\
& \quad
\cdot
\sum_{l=0}^{n-1}
\Bigg[
\frac
{1}
{\sqrt{M^{n-l-1}}}
\Bigg(
\int_t^T
\left[
\frac{(s-t)^{-1}+1}{(T-s)\varrho(\frac{s-t}{T-t})}
+(s-t)^{-1/2}(T-s)^{-1/2}
\right]
\nonumber\\
& \quad
\cdot
\sup_{(r,y)\in[s,T)\times\bR^d}
\left(
\frac{e^{\rho r}(T-r)}{d^p+\|y\|^2}
\big\|
U^0_{l,M}(r,y)-(u,w)(r,y)
\big\|_{L_2}^2
\right)
ds
\Bigg)^{1/2}
\Bigg].
\label{error recursive 1}
\end{align}
Next, by \eqref{def pdf rho} we observe for all $\beta\in[0,1/2)$ that
\begin{align}
&
(T-t)^{1/2}
\left(
\int_t^T\left[
\frac{(s-t)^{-1}+1}{(T-s)\varrho(\frac{s-t}{T-t})}
+
(s-t)^{-1/2}(T-s)^{-1/2}
\right]^{1+\beta}
\,ds
\right)^{\frac{1}{2(1+\beta)}}
\nonumber\\
&
\leq
(T-t)^{1/2}
\Bigg[
\Bigg(
\int_t^T
\Big[
(s-t)^{-(1-\alpha)}(T-s)^{-(1-\alpha)}\cB(1-\alpha,1-\alpha)
\Big]^{1+\beta}
\,ds
\Bigg)^{\frac{1}{1+\beta}}
\nonumber\\
& \quad
+
\Bigg(
\int_t^T
\Big[
(s-t)^\alpha(T-s)^{-(1-\alpha)}\cB(1-\alpha,1-\alpha)
\Big]^{1+\beta}
\,ds
\Bigg)^{\frac{1}{1+\beta}}
\nonumber\\
& \quad
+
\Bigg(
\int_t^T
\bigg[
(s-t)^{-1/2}(T-s)^{-1/2}
\bigg]^{1+\beta}
\,ds
\Bigg)^{\frac{1}{1+\beta}}
\Bigg]^{1/2}
\label{beta integral est 2}
\end{align}
By the assumption that $\alpha\in[1/2,1)$, we have for all $\beta\in[0,1/2)$ that
\begin{align}
&
\Bigg(
\int_t^T
\Big[
(s-t)^{-(1-\alpha)}(T-s)^{-(1-\alpha)}\cB(1-\alpha,1-\alpha)
\Big]^{1+\beta}
\,ds
\Bigg)^{\frac{1}{1+\beta}}
\nonumber\\
&
\leq
\cB(1-\alpha,1-\alpha)
\Bigg[
\left(
\int_t^{\frac{T+t}{2}}
(s-t)^{(\alpha-1)(1+\beta)}\Big(\frac{T-t}{2}\Big)^{(\alpha-1)(1+\beta)}
\,ds
\right)^{\frac{1}{\beta+1}}
\nonumber\\
& \quad
+
\left(
\int_{\frac{T+t}{2}}^T
(T-s)^{(\alpha-1)(1+\beta)}\Big(\frac{T-t}{2}\Big)^{(\alpha-1)(1+\beta)}
\,ds
\right)^{\frac{1}{\beta+1}}
\Bigg]
\nonumber\\
&
=
\frac
{\cB(1-\alpha,1-\alpha)
2^{2(1-\alpha)+1+\frac{1}{1+\beta}}(T-t)^{2(\alpha-1)+\frac{1}{\beta+1}}}
{[(\alpha-1)(1+\beta)+1]^{\frac{1}{1+\beta}}}
\nonumber\\
&
\leq
16(T-t)^{2(\alpha-1)+\frac{1}{\beta+1}}\cB(1-\alpha,1-\alpha).
\label{beta integral est 3}
\end{align}
Similarly, it holds for all $\beta\in[0,1/2)$ that
\begin{align}
\Bigg(
\int_t^T
\Big[
(s-t)^\alpha(T-s)^{-(1-\alpha)}\cB(1-\alpha,1-\alpha)
\Big]^{1+\beta}
\,ds
\Bigg)^{\frac{1}{1+\beta}}
\leq
2(T-t)^{2(\alpha-1)+\frac{1}{1+\beta}+1}\cB(1-\alpha,1-\alpha),
\label{beta integral est 4}
\end{align}
and
\begin{equation}
\label{beta integral est 5}
\Bigg(
\int_t^T
\bigg[
(s-t)^{-1/2}(T-s)^{-1/2}
\bigg]^{1+\beta}
\,ds
\Bigg)^{\frac{1}{1+\beta}}
\leq 
4(T-t)^{-\frac{\beta}{1+\beta}}.
\end{equation}
Combining \eqref{beta integral est 2}--\eqref{beta integral est 5} yields for all
$\beta\in[0,1/2)$ that
\begin{align*}
&
(T-t)^{1/2}
\left(
\int_t^T\left[
\frac{(s-t)^{-1}+1}{(T-s)\varrho(\frac{s-t}{T-t})}
+
(s-t)^{-1/2}(T-s)^{-1/2}
\right]^{1+\beta}
\,ds
\right)^{\frac{1}{2(1+\beta)}}
\\
&
\leq
\left[
\big(16(1+T)+2(1+T)^2\big)\cB(1-\alpha,1-\alpha)+(1+T)
\right]^{1/2}
\leq 
5(1+T)\sqrt{\cB(1-\alpha,1-\alpha)}.
\end{align*}
This together with \eqref{def F MLP}, \eqref{MLP LG}, \eqref{error recursive 1},
\eqref{beta integral est 1},
and H\"older's inequality imply for all $\beta\in(0,1/2)$ that
\begin{align}
&
\sup_{x\in\bR^d}
\left(
e^{\rho t}(d^p+\|x\|^2)^{-1}(T-t)
\big\|
U^0_{n,M}(t,x)-(u,w)(t,x)
\big\|_{L_2}^2
\right)^{1/2}
\nonumber\\
&
\leq
\frac{e^{\rho T/2}(a+a_1b_1)^{1/2}[cT\alpha^{-1}(4^{-\alpha}+T)]^{1/2}
\sqrt{\cB(1-\alpha,1-\alpha)}}
{\sqrt{M^n}}
+
\frac{e^{\rho T/2}[LT(a_2T+a_3b_1c)]^{1/2}}{\sqrt{M^n}}
\nonumber\\
& \quad
+
\big[L(a+a_1b_1)(b+T^{3/2})\big]^{1/2}(1+\sqrt{2})
\nonumber\\
& \quad
\cdot
\sum_{l=0}^{n-1}
\Bigg[
\frac
{(T-t)^{1/2}}
{\sqrt{M^{n-l-1}}}
\left(
\int_t^T
\left[
\frac{(s-t)^{-1}+1}{(T-s)\varrho(\frac{s-t}{T-t})}
+(s-t)^{-1/2}(T-s)^{-1/2}
\right]^{1+\beta}
ds
\right)^{\frac{1}{2(1+\beta)}}
\nonumber\\
& \quad
\cdot
\left(
\int_t^T
\sup_{(r,y)\in[s,T)\times\bR^d}
\left(
\frac{e^{\rho r}(T-r)}{d^p+\|y\|^2}
\big\|
U^0_{l,M}(r,y)-(u,w)(r,y)
\big\|_{L_2}^2
\right)^{\frac{1+\beta}{\beta}}
ds
\right)^{\frac{\beta}{2(1+\beta)}}
\Bigg]
\nonumber\\
&
\leq
\frac{e^{\rho T/2}}{\sqrt{M^n}}
\left(
(a+a_1b_1)^{1/2}[cT\alpha^{-1}(4^{-\alpha}+T)]^{1/2}
\sqrt{\cB(1-\alpha,1-\alpha)}
+
[LT(a_2T+a_3b_1c)]^{1/2}
\right)
\nonumber\\
& \quad
+\big[L(a+a_1b_1)(b+T^{3/2})\big]^{1/2}(1+\sqrt{2})
5(1+T)\sqrt{\cB(1-\alpha,1-\alpha)}
\nonumber\\
& \quad
\cdot
\sum_{l=0}^{n-1}
\Bigg[
\frac
{1}
{\sqrt{M^{n-l-1}}}
\left(
\int_t^T
\sup_{(r,y)\in[s,T)\times\bR^d}
\left(
\frac{e^{\rho r}(T-r)}{d^p+\|y\|^2}
\big\|
U^0_{l,M}(r,y)-(u,w)(r,y)
\big\|_{L_2}^2
\right)^{\frac{1+\beta}{\beta}}
ds
\right)^{\frac{\beta}{2(1+\beta)}}
\Bigg].
\nonumber\\
\end{align}
This establishes \eqref{MLP error recursive}.
Hence, we have completed the proof of this lemma.
\end{proof}

The following proposition provides a global error analysis for the MLP approximation
algorithm \eqref{def MLP general}, 
which will be used to prove Theorems \ref{thm MLP conv} and \ref{MLP complexity} 
(see Section \ref{section proof main}).

\begin{proposition}[Global approximation error]
\label{corollary MLP error}
Assume Setting \ref{MLP setting}, and let $n,M\in\bN$, $t\in[0,T)$, $\beta\in(0,1/2]$.
Then it holds that
\begin{align}
&
\sup_{x\in\bR^d}
\left(
(d^p+\|x\|^2)^{-1}(T-t)
\big\|
U^0_{n,M}(t,x)-(u,w)(t,x)
\big\|_{L_2}^2
\right)^{1/2}
\nonumber\\
&
\leq
\left[
A+B(T-t)^{\frac{\beta}{2(1+\beta)}}e^{\rho T/2}(1+T^{1/2})c^{1/2}
\right]
\exp\Big\{\frac{\beta}{2(1+\beta)}M^{\frac{1+\beta}{\beta}}\Big\}
M^{-n/2}
\left[1+B(T-t)^{\frac{\beta}{2(1+\beta)}}\right]^{n-1},
\label{MLP error}
\end{align}
where
$$
A:=e^{\rho T/2}
\left(
(a+a_1b_1)^{1/2}[cT\alpha^{-1}(4^{-\alpha}+T)]^{1/2}
\sqrt{\cB(1-\alpha,1-\alpha)}
+
[LT(a_2T+a_3b_1c)]^{1/2}
\right),
$$
and
$$
B:=\big[L(a+a_1b_1)(b+T^{3/2})\big]^{1/2}(1+\sqrt{2})
5(1+T)\sqrt{\cB(1-\alpha,1-\alpha)}.
$$
\end{proposition}

\begin{proof}
Throughout the proof of this lemma, we define the Borel functions 
$\mathfrak{f}_k:[t,T)\to[0,\infty]$, $k\in\{0,1,\dots,n\}$ by
$$
\mathfrak{f}_k(s):=
\sup_{s\in[t,T)}
\sup_{x\in\bR^d}
\left(
e^{\rho s}
(d^p+\|x\|^2)^{-1}
(T-s)
\big\|
U^0_{k,M}(s,x)-(u,w)(s,x)
\big\|_{L_2}^2
\right)^{1/2},
\quad s\in [t,T).
$$
Then \eqref{MLP growth u v} and the fact that $U^0_{0,M}\equiv 0$ imply that
\begin{equation}
\label{sup f 0}
\sup_{s\in[t,T)}|\mathfrak{f}_0(s)|
\leq
e^{\rho T/2}(1+T^{1/2})\sqrt{c'}<\infty.
\end{equation}
Moreover, notice that Lemma \ref{lemma recursive error} ensures for all 
$k\in\{1,2,\dots,n\}$ and $s\in[t,T)$ that
$$
\mathfrak{f}_k(s)
\leq
\frac{A}{\sqrt{M^k}}
+
\sum_{l=0}^{k-1}
\frac{B}{\sqrt{M^{k-l-1}}}
\left(
\int_t^T
|\mathfrak{f}_l(r)|^{\frac{2(1+\beta)}{\beta}}
\,dr
\right)^{\frac{\beta}{2(1+\beta)}}.
$$
Hence, by \eqref{sup f 0} the application of Lemma 3.11 in \cite{HJKN2020}
(with $a\cal A$, $b\cal B$, $M\cal M$, $N\cal N$, $T\cal T$, $\tau\cal t$, 
$p\cal 2(1+\beta)/\beta$, and $f_n\cal\mathfrak{f}_n$ in the notations
of Lemma 3.11 in \cite{HJKN2020})
yields that
\begin{align*}
\mathfrak{f}_N(t)
\leq
\left[
A+B(T-t)^{\frac{\beta}{2(1+\beta)}}e^{\rho T/2}(1+T^{1/2})c^{1/2}
\right]
\exp\bigg\{\frac{\beta M^{\frac{1+\beta}{\beta}}}{2(1+\beta)}\bigg\}
M^{-N/2}
\left[1+B(T-t)^{\frac{\beta}{2(1+\beta)}}\right]^{N-1}.
\end{align*}
This shows that \eqref{MLP error} holds for $\beta\in(0,1/2)$.
Thus, a straightforward limiting argument as $\beta\to 1/2$ ensures
\eqref{MLP error} for $\beta=1/2$.
We have therefore completed the proof of this corollary.
\end{proof}

\section{\textbf{Proof of the main results}}
\label{section proof main}
In this section, 
we present the proof of Theorems \ref{thm MLP conv} and \ref{MLP complexity}.
\begin{proof}[Proof of Theorem \ref{thm MLP conv}]
First notice that Corollary \ref{corollary FP} ensures that (i),
and it holds for all $(t,x)\in[0,T)\times\bR^d$ and $d\in\bN$ that
\begin{align}                                              
\sup_{s\in[t,T)}\sup_{x\in\bR^d}
\frac{|u^d(s,x)|+(T-s)^{1/2}\|w^d(s,x)\|}{(d^p+\|x\|^2)^{1/2}}
\leq 
C_{d,1}.
\label{growth u v}
\end{align}
By Proposition \ref{proposition uniqueness PDE} and
Proposition~\ref{theorem PDE existence} we obtain (ii), (iii), and (iv).
Furthermore, Corollary \ref{corollary FP} also implies for each $N,d\in\bN$ and that
there exists a unique pair of Borel functions $(u^d_N,w^d_N)$ with
$u^d_N\in C([0,T)\times\bR^d,\bR)$ and $w^d_N\in C([0,T)\times\bR^d,\bR^d)$ 
satisfying for all $(t,x)\in[0,T)\times\bR^d$ that
\begin{align*}
&
\big\|g^d(\cX^{d,0,t,x,N}_T)(1,\cV^{d,0,t,x,N}_T)\big\|_{L_1}
\\
&
+\int_t^T[\big\|
f^d\big(s,\cX^{d,0,t,x,N}_s,u^d_N(s,\cX^{d,0,t,x,N}_s),w^d_N(s,\cX^{d,0,t,x,N}_s)\big)
(1,\cV^{d,0,t,x,N}_s)
\big\|_{L_1}\,ds
\\
&
+\sup_{(s,y)\in[0,T)\times\bR^d}
\left(\frac{|u^d(s,y)|+(T-s)^{1/2}\|w^d(s,y)\|}{(d^p+\|y\|^2)^{1/2}}
\right)<\infty,
\end{align*}
and
\begin{align}
(u^d_N(t,x),w^d_N(t,x))
&
=\bE\left[g^d(\cX^{d,0,t,x,N}_T)(1,V^{d,0,t,x,N}_T)\right]  
\nonumber\\
& \quad    
+\int_t^T\bE\left[
f^d\big(s,\cX^{d,0,t,x,N}_s,u^d_N(s,\cX^{d,0,t,x,N}_s),w^d_N(s,\cX^{d,0,t,x,N}_s)\big)
(1,\cV^{d,0,t,x,N}_s)
\right]ds,
\label{BEL u v N}
\end{align}
as well as for all $t\in[0,T)$ that
\begin{align}                                              
\sup_{s\in[t,T)}\sup_{x\in\bR^d}
\frac{|u^d_N(s,x)|+(T-s)^{1/2}\|w^d_N(s,x)\|}{(d^p+\|x\|^2)^{1/2}}
\leq C_{d,2}.
\label{growth u v N}
\end{align}
Next, to prove (v) and (vi) we observe for all
$d\in\bN$, $(t,x)\in[0,T)\times\bR^d$, $n\in\bN_0$, and $M,N\in\bN$ that
\begin{align}
&
\big\|\cU^{d,0}_{n,M,N}(t,x)-(u^d,\nabla_x u^d)(t,x)\big\|_{L_2}
\nonumber\\
&
\leq
\big\|\cU^{d,0}_{n,M,N}(t,x)-(u^d_N,w^d_N)(t,x)\big\|_{L_2}
+
\big\|(u^d_N,w^d_N)(t,x)-(u^d,\nabla_x u^d)(t,x)\big\|_{L_2}.
\label{ineq bU u v}
\end{align}
Moreover, by \eqref{L q continuity SDE}
we notice that for each $d\in\bN$, $s\in[0,T]$, and $r\in[s,T]$ the mapping 
$
\bR^d\times\bR^d\ni(x,y)\mapsto \big(X^{d,0,s,x}_r,X^{d,0,s,y}_r\big)
\in \cL_0(\Omega,\bR^d\times\bR^d)
$
is continuous and hence measurable, 
and we have for all nonnegative Borel functions 
$h:\bR^d\times\bR^d\to[0,\infty)$ that the mapping
$
\cL_0(\Omega,\bR^d\times\bR^d)\ni Z \mapsto \bE\big[h(Z)\big]\in [0,\infty]
$
is measurable. 
Hence, it holds for all $d\in\bN$, $s\in[0,T]$,
$r\in[s,T]$ and all nonnegative Borel functions 
$h:\bR^d\times\bR^d\to[0,\infty)$ that the mapping
\begin{equation}                                           \label{measurability 3}
\bR^d\times\bR^d\ni(x,y)\mapsto
\bE\Big[h\big(X^{d,0,s,x}_{r},X^{d,0,s,y}_{r}\big)\Big]\in[0,\infty]
\end{equation}
is measurable.
Furthermore, Lemma 2.2 in \cite{HJKNW2020} ensures that for all 
$d\in\bN$, $t\in[0,T]$, $s\in[t,T]$, $r\in[s,T]$, $x,y\in\bR^d$ and
all nonnegative Borel functions $h:\bR^d\times\bR^d\to[0,\infty)$
it holds that
\begin{equation}                                     \label{expectation equality}
\bE\left[\bE \left[h\Big(X^{d,0,s,x'}_{r},X^{d,0,s,y'}_{r}\Big)\right]
\Big|_{(x',y')=(X^{d,0,t,x}_{s},X^{d,0,t,y}_{s})}\right]
=\bE\left[h\Big(X^{d,0,t,x}_{r},X^{d,0,t,y}_{r}\Big)\right].
\end{equation}
Then by \eqref{assumption Lip f}, \eqref{assumption Lip mu sigma},
\eqref{assumption growth f g}, \eqref{BEL main},  
\eqref{BEL u v N}, \eqref{measurability 3}, \eqref{expectation equality},
\eqref{q moment est SDE}, \eqref{L q continuity SDE}, 
\eqref{L2 est proc V}, \eqref{L2 cont est proc V}, 
\eqref{L q est Euler SDE}, \eqref{L q error Euler SDE}, \eqref{cV L q est 0},  
and \eqref{V cV error 0},
the application of Lemma \ref{lemma perturbation}
(with $F\cal f^d$, $G\cal g^d$, 
$\bX^{t,x,1}\cal X^{d,0,t,x}$, $\bX^{t,x,2}\cal \cX^{d,0,t,x,N}$,
$\bV^{t,x,1}\cal V^{d,0,t,x}$, and $\bV^{t,x,2}\cal \cV^{d,0,t,x,N}$
in the notation of Lemma \ref{lemma perturbation}) yields that there
exists a positive constant 
$\mathfrak{c}_{d,1}=\mathfrak{c}_{d,1}(d,\varepsilon_d,L,L_0,T)$
satisfying for all
$d\in\bN$, $(t,x)\in[0,T)\times\bR^d$, $n\in\bN_0$, and $M,N\in\bN$ that
\begin{align}
&\big\|(u^d_N,w^d_N)(t,x)-(u^d,\nabla_x u^d)(t,x)\big\|_{L_2}
\nonumber\\
&
\leq
\left(
\frac{1+T}{T-t}
\bE\left[\big|u^d_N(t,x)-u^d(t,x)\big|^2+(T-t)\big\|w^d_N(t,x)-\nabla_x u^d(t,x)\big\|^2\right]
\right)^{1/2}
\nonumber\\
&
\leq
\mathfrak{c}_{d,1}(T-t)^{-1/2}N^{-1/2}(d^p+\|x\|^2).
\label{ineq MLP conv 1}
\end{align}
Moreover, by
\eqref{assumption Lip f}, \eqref{assumption growth f g},
\eqref{BEL u v N}, \eqref{growth u v N}, 
\eqref{L q est Euler SDE}, \eqref{L q est Euler SDE -x}, 
and \eqref{cV L q est 0},
we apply Proposition \ref{corollary MLP error}
(with $\beta=1/2$, $\varrho\cal\varrho$, $\cR^\theta\cal\cR^\theta$, 
$g\cal g^d$, $f\cal f^d$, $F\cal F^d$,
$\bX^{\theta,t,x}\cal \cX^{d,\theta,t,x,N}$,
$\bV^{\theta,t,x}\cal \cV^{d,\theta,t,x,N}$,
$(u,w)\cal (u^d_N,w^d_N)$,
and $\cU^{\theta}_{n,M,N}(t,x)\cal\cU^{d,\theta}_{n,M,N}(t,x)$
in the notation of Proposition \ref{corollary MLP error})
to show that there exists a positive constant 
$\mathfrak{c}_{d,2}=\mathfrak{c}_{d,2}(d,\varepsilon_d,\alpha,L,L_0,T)$
satisfying for all
$d\in\bN$, $(t,x)\in[0,T)\times\bR^d$, $n\in\bN_0$, and $M,N\in\bN$ that
\begin{align}
\big\|\cU^{d,0}_{n,M,N}(t,x)-(u^d_N,w^d_N)(t,x)\big\|_{L_2}
\leq
\mathfrak{c}_{d,2}^{n-1}\exp\big\{M^3/6\big\}M^{-n/2}
(T-t)^{-1/2}(d^p+\|x\|^2)^{1/2}.
\label{ineq MLP conv 2}
\end{align}
Then combining \eqref{ineq bU u v}, \eqref{ineq MLP conv 1},
and \eqref{ineq MLP conv 2} yields for all
$d\in\bN$, $(t,x)\in[0,T)\times\bR^d$, $n\in\bN_0$, and $M,N\in\bN$ that
\begin{align}
&
\big\|\cU^{d,0}_{n,M,N}(t,x)-(u^d,\nabla_x u^d)(t,x)\big\|_{L_2}
\nonumber\\
&
\leq
\left[
\mathfrak{c}_{d,1}N^{-1/2}
+
\mathfrak{c}_{d,2}^{n-1}\exp\big\{M^3/6\big\}M^{-n/2}
\right]
(T-t)^{-1/2}(d^p+\|x\|^2).
\end{align}
This proves (v).
Next, taking Assumption \ref{assumption bbd partials global} into account,
by
\eqref{assumption Lip f}, \eqref{assumption growth f g}, \eqref{finite main},
\eqref{BEL main}, \eqref{growth u v}, 
\eqref{q moment est SDE}, \eqref{q moment est SDE -x}, and \eqref{L q est proc V}
we apply Proposition \ref{corollary MLP error}
(with $\beta=1/2$, $\varrho\cal\varrho$, $\cR^\theta\cal\cR^\theta$, 
$g\cal g^d$, $f\cal f^d$, $F\cal F^d$,
$\bX^{\theta,t,x}\cal X^{d,\theta,t,x}$,
$V^{\theta,t,x}\cal V^{d,\theta,t,x}$,
$(u,w)\cal (u^d,w^d)$,
and $\cU^{\theta}_{n,M}(t,x)\cal U^{d,\theta}_{n,M}(t,x)$
in the notation of Proposition \ref{corollary MLP error})
to show that there exists a positive constant 
$\mathfrak{c}_3=\mathfrak{c}_3(\alpha,L,L_0,K,T)$
satisfying for all
$d\in\bN$, $(t,x)\in[0,T)\times\bR^d$, $n\in\bN_0$, and $M,N\in\bN$ that
\begin{align}
\big\|U^{d,0}_{n,M}(t,x)-(u^d,\nabla_x u^d)(t,x)\big\|_{L_2}
\leq
\mathfrak{c}_3^{n-1}(d\varepsilon_d^{-1})^n\exp\big\{M^3/6\big\}M^{-n/2}
(T-t)^{-1/2}(d^p+\|x\|^2)^{1/2}.
\end{align}
This ensures (vi). 
We have therefore completed the proof of Theorem \ref{thm MLP conv}.
\end{proof}

\begin{proof}[Proof of Theorem \ref{MLP complexity}]
For each $d\in\bN$, $\varepsilon\in(0,1]$, and $x\in\bR^d$ we define
$\mathfrak{n}^d(x,\varepsilon)$ by
\begin{equation}                                                \label{def n frak}
\mathfrak{n}^d(x,\varepsilon):=\inf\Big\{n\in\bN\cap [2,\infty):
\sup_{k\in[n,\infty)\cap\bN}\sup_{t\in[0,T)}
\big\|\cU^{d,0}_{(k)}(t,x)-(u^d,\nabla_x u^d)(t,x)\big\|_{L_2}<\varepsilon\Big\},
\end{equation}
where we use the shorter notation
$$
\cU^{(d)}_{(k)}(t,x):=\cU^{d,0}_{k^3,k,k}(t,x)
\quad
\text{for all $d\in\bN$, $k\in\bN$, and $(t,x)\in[0,T)\times\bR^d$,}
$$ 
and use the convention $\inf(\emptyset)=\infty$
($\emptyset$ denotes the empty set). 
Applying (v) in Theorem \ref{thm MLP conv} (with
$n \cal n^3$, $M\cal n$, and $N\cal n$ 
in the notation of Theorem \ref{thm MLP conv}),
we have for all $d\in\bN$, 
$n\in\bN$, and $(t,x)\in[0,T)\times\bR^d$ that
\begin{align}
&
\big\|\cU^{d,0}_{(n)}(t,x)-(u^d,\nabla_x u^d)(t,x)\big\|_{L_2}
\leq
\Big[\mathfrak{c}_{d,1}n^{-1/2}
+\mathfrak{c}_{d,2}^{n-1}\exp\big\{n^3/6\big\}n^{-n^3/2}\Big]
(T-t)^{-1/2}(d^p+\|x\|^2)^{1/2}              
\label{MLP n bU}
\end{align}
with $\mathfrak{c}_{d,1}=\mathfrak{c}_{d,1}(d,\varepsilon_d,L,L_0,T)$ 
and $\mathfrak{c}_{d,2}=\mathfrak{c}_{d,2}(d,\varepsilon_d,\alpha,L,L_0,T)$ 
being the positive constants
introduced in (v).
Moreover, for each $d\in\bN$ we observe for all integers 
$n\geq \max\{(1+\mathfrak{c}_{d,2})/2,e\}$ that
$$
\mathfrak{c}_{d,2}^{n-1}\exp\big\{n^3/6\big\}n^{-n^3/2}
\leq
\frac{(1+\mathfrak{c}_{d,2})^n}{n^{n^3/3}}
\cdot\frac{e^{n^3/6}}{n^{n^3/6}}
\leq 2^{-n},
$$
which implies that
$$
\lim_{n\to\infty}
\Big[
\mathfrak{c}_{d,1}n^{-1/2}+
\mathfrak{c}_{d,2}^{n-1}\exp\big\{n^3/6\big\}n^{-n^3/2}
\Big]=0.
$$
Therefore, by \eqref{MLP n bU} we have for all 
$d\in\bN$, $n\in\bN$, $\varepsilon\in(0,1]$, and $(t,x)\in[0,T)\times\bR^d$ that
$$
\mathfrak{n}^d(x,\varepsilon)<\infty \quad \text{and} \quad
\sup_{n\in[\mathfrak{n}^d(x,\varepsilon),\infty)\cap \bN}
\big\|\cU^{d,0}_{(n)}(t,x)-(u^d,\nabla_x u^d)(t,x)\big\|_{L_2}
<\varepsilon,
$$
which proves (i).
Next, note that \eqref{cc 1} and e.g., Lemma 3.14 in \cite{BGJ2020}
(applied with $M\cal M$, $n\cal n$, 
$\alpha\cal (2M^M\mathfrak{e}^{(d)}+\mathfrak{g}^{(d)}+\mathfrak{f}^{(d)})$,
$\beta\cal (M^M\mathfrak{e}^{(d)}+\mathfrak{f}^{(d)})$,
and $(C_n)_{n\in\bN_0}\cal(\mathfrak{C}^{(d)}_{n,M})_{n\in\bN_0}$
in the notation of Lemma 3.14 in \cite{BGJ2020})
ensure for all $d\in\bN$ and $n,M\in\bN$ that
$$
\mathfrak{C}^{(d)}_{n,M}
\leq
\left[\frac{3M^M\mathfrak{e}^{(d)}+\mathfrak{g}^{(d)}+\mathfrak{f}^{(d)}}{2}\right]
(3M)^n.
$$
Hence, it holds for all $d\in\bN$, $n\in\bN$, and $k\in\{1,2,\dots,n+1\}$ that
\begin{align*}
\mathfrak{C}^{(d)}_{k^3,k}
&
\leq
\frac{\big[3\mathfrak{e}^{(d)}+\mathfrak{g}^{(d)}+\mathfrak{f}^{(d)}\big]
(3(n+1)^2)^{(n+1)^3}}{2}
\leq
\frac{\big[3\mathfrak{e}^{(d)}+\mathfrak{g}^{(d)}+\mathfrak{f}^{(d)}\big]
(3(2n)^2)^{(n+1)^3}}{2}
\\
&
\leq
\frac{\big[3\mathfrak{e}^{(d)}+\mathfrak{g}^{(d)}+\mathfrak{f}^{(d)}\big]
(12n^2)^{4(n^3+1)}}{2}.
\end{align*}
This together with the fact that $n^9\leq 9^n$ imply for all $d\in\bN$ and $n\in\bN$
that
\begin{align*}
\sum_{k=1}^{n+1}\mathfrak{C}^{(d)}_{k^3,k}
&
\leq
\frac{\big[3\mathfrak{e}^{(d)}+\mathfrak{g}^{(d)}+\mathfrak{f}^{(d)}\big]
(n+1)(12n^2)^{4(n^3+1)}}{2}
\leq
\big[3\mathfrak{e}^{(d)}+\mathfrak{g}^{(d)}+\mathfrak{f}^{(d)}\big]
n(12n^2)^{4(n^3+1)}
\\
&
\leq
12\big[3\mathfrak{e}^{(d)}+\mathfrak{g}^{(d)}+\mathfrak{f}^{(d)}\big]
9^n\cdot 12^{4n^3}\cdot n^{8n^3}
\leq
12\big[3\mathfrak{e}^{(d)}+\mathfrak{g}^{(d)}+\mathfrak{f}^{(d)}\big]
(12)^{5n^3}\cdot n^{8n^3}.
\end{align*}
This establishes (ii).
Next, for each $d\in\bN$, $\varepsilon\in(0,1]$, and $x\in\bR^d$ define
$\mathbf{n}^d(x,\varepsilon)$ by
\begin{equation}                                                \label{def n}
\mathbf{n}^d(x,\varepsilon):=\inf\Big\{n\in\bN\cap [2,\infty):
\sup_{k\in[n,\infty)\cap\bN}\sup_{t\in[0,T)}
\big\|U^{d,0}_{(k)}(t,x)-(u^d,\nabla_x u^d)(t,x)\big\|_{L_2}<\varepsilon^2\Big\},
\end{equation}
where we use the shorter notation
$$
U^{(d)}_{(k)}(t,x):=U^{d,0}_{k^3,k}(t,x)
\quad
\text{for all $d\in\bN$, $k\in\bN$, and $(t,x)\in[0,T)\times\bR^d$}.
$$ 
Applying (vi) in Theorem \ref{thm MLP conv} (with
$n \cal n$, and $M\cal n$ in the notation of Theorem \ref{thm MLP conv}),
we have for all $d\in\bN$, 
$n\in\bN$, and $(t,x)\in[0,T)\times\bR^d$ that
\begin{align}
&
\big\|U^{d,0}_{(n)}(t,x)-(u^d,\nabla_x u^d)(t,x)\big\|_{L_2}
\leq
\mathfrak{c}_3^{n-1}(d\varepsilon_d^{-1})^n\exp\big\{n^3/6\big\}n^{-n^3/2}
(T-t)^{-1/2}(d^p+\|x\|^2)^{1/2},               
\label{MLP n}
\end{align}
where $\mathfrak{c}_3=\mathfrak{c}_3(\alpha,L,L_0,K,T)$
is the positive constant introduced in (vi).
Moreover, for each $d\in\bN$ we observe for all integers 
$n\geq \max\{(1+\mathfrak{c}_3)d\varepsilon_d^{-1}/2,e\}$ that
$$
\mathfrak{c}_3^{n-1}(d\varepsilon_d^{-1})^n\exp\big\{n^3/6\big\}n^{-n^3/2}
\leq
\frac{\big[(1+\mathfrak{c}_3)d\varepsilon_d^{-1}\big]^n}{n^{n^3/3}}
\cdot\frac{e^{n^3/6}}{n^{n^3/6}}
\leq 2^{-n},
$$
which implies that
$$
\lim_{n\to\infty}
\mathfrak{c}_3^{n-1}(d\varepsilon_d^{-1})^n\exp\big\{n^3/6\big\}n^{-n^3/2}=0.
$$
Therefore, by \eqref{MLP n} we have for all $d\in\bN$, $n\in\bN$, $\varepsilon\in(0,1]$, and $(t,x)\in[0,T)\times\bR^d$ that
$$
\mathbf{n}^d(x,\varepsilon)<\infty \quad \text{and} \quad
\sup_{n\in[\mathbf{n}^d(x,\varepsilon),\infty)\cap \bN}
\big\|U^{d,0}_{(n)}(t,x)-(u^d,\nabla_x u^d)(t,x)\big\|_{L_2}<\varepsilon,
$$
which proves \eqref{error epsilon}.
Furthermore, by \eqref{cc 2} and \eqref{MLP n}
we obtain for all $d,n\in\bN$, $\gamma\in(0,1]$, 
and $(t,x)\in[0,T]\times\bR^d$ that
\begin{align}
&
\left(\sum_{k=1}^{n+1}\mathfrak{C}^{(d)}_{k^3,k}\right)
\big\|U^{d,0}_{(n)}(t,x)-(u^d,\nabla_x u^d)(t,x)\big\|_{L_2}^{\gamma+16}
\nonumber\\
&
\leq 
12\big[3\mathfrak{e}^{(d)}+\mathfrak{g}^{(d)}+2\mathfrak{f}^{(d)}\big]
(12)^{5n^3}\cdot n^{8n^3} 
\big[(T-t)^{-1}(d^p+\|x\|^2)\big]^{\frac{\gamma+16}{2}}
\big[\mathfrak{c}_3^{n-1}(d\varepsilon_d^{-1})^n
\exp\big\{n^3/6\big\}n^{-n^3/2}\big]^{\gamma+16}
\nonumber\\
& 
=         
12\big[3\mathfrak{e}^{(d)}+\mathfrak{g}^{(d)}+2\mathfrak{f}^{(d)}\big]
(12)^{5n^3}\cdot n^{-\gamma n^3/2} 
\big[(T-t)^{-1}(d^p+\|x\|^2)\big]^{\frac{\gamma+16}{2}}
\big[\mathfrak{c}_3^{n-1}(d\varepsilon_d^{-1})^n
\exp\big\{n^3/6\big\}\big]^{\gamma+16} .      
\label{cc 4}
\end{align} 
Then \eqref{def n} and \eqref{cc 4} show for all $d\in\bN$,
$\varepsilon,\gamma\in(0,1]$ and $(t,x)\in[0,T]\times\bR^d$ that
\begin{align}
\left(\sum_{k=1}^{\mathbf{n}^d(x,\varepsilon)}\mathfrak{C}^{(d)}_{k^3,k}\right)
\varepsilon^{\gamma+16}
&
\leq \left(\sum_{k=1}^{\mathbf{n}^d(x,\varepsilon)}\mathfrak{C}^{(d)}_{k^3,k}\right)
\big\|U^{d,0}_{(\mathbf{n}^d(x,\varepsilon)-1)}(t,x)
-(u^d,\nabla_x u^d)(t,x)\big\|_{L_2}^{\gamma+16}
\nonumber\\
&
\leq
12\big[3\mathfrak{e}^{(d)}+\mathfrak{g}^{(d)}+2\mathfrak{f}^{(d)}\big]
\big[(T-t)^{-1}(d^p+\|x\|^2)\big]^{\frac{\gamma+16}{2}}
\nonumber\\
& \quad
\cdot
\sup_{n\in\bN}\Big\{
12^{5n^3}\cdot n^{-\gamma n^3/2} 
\big[\mathfrak{c}_3^{n-1}(d\varepsilon_d^{-1})^n
\exp\big\{n^3/6\big\}\big]^{\gamma+16}\Big\}.
\label{complexity}  
\end{align}
Moreover, it holds for all $\gamma\in(0,1]$ and $n\in\bN$ satisfying 
$$
n\geq 
\max\left\{
\big(2\cdot 12^5\big)^{\frac{6}{\gamma}},(1+\mathfrak{c}_3)d\varepsilon_d^{-1},
6(\gamma+16)\gamma^{-1},\exp\big\{\gamma+16\gamma^{-1}\big\}
\right\}
$$ that
\begin{align*}
&
12^{5n^3}\cdot n^{-\gamma n^3/2} 
\big[\mathfrak{c}_3^{n-1}(d\varepsilon_d^{-1})^n
\exp\big\{n^3/6\big\}\big]^{\gamma+16}
\\
&
\leq
\frac{12^{5n^3}}{n^{\gamma n^3/6}}
\cdot
\frac{\big[(1+\mathfrak{c}_3)d\varepsilon_d^{-1}\big]^{n(\gamma+16)}}
{n^{\gamma n^3/6}}
\cdot
\frac{\exp\big\{(\gamma+16)n^3/6\big\}}{n^{\gamma n^3/6}}
\\
&
\leq 
2^{-n^3}.
\end{align*}
which implies for all $\gamma\in(0,1]$ that
$$
\sup_{n\in\bN}\Big\{
12^{5n^3}\cdot n^{-\gamma n^3/2} 
\big[\mathfrak{c}_3^{n-1}(d\varepsilon_d^{-1})^n
\exp\big\{n^3/6\big\}\big]^{\gamma+16}\Big\}
<\infty.
$$
This together with \eqref{complexity} establish \eqref{cc 3},
which proves (iii). We have therefore completed the proof of 
Theorem \ref{MLP complexity}.
\end{proof}

\clearpage

\appendix

\section{\textbf{Preliminaries}}
\label{section pre}

In this section, we show some important lemmas and tools used in
Sections \ref{section FP}, \ref{section PIDE}, and \ref{section proof main}. 
Sections \ref{section bounds SDE}--\ref{section bounds euler} provide some results 
on moment estimates, stability, continuity, and discretization errors for SDEs.
In Section \ref{section property coe}, we collect some lemmas for the coefficient
functions $\mu^d$ and $\sigma^d$ of PDE \eqref{APIDE}. 
Throughout this section, we assume the settings in Section \ref{section setting},
fix $d\in\bN$ and $\theta\in\Theta$, 
and omit the superscripts $d$ and $\theta$ for
the notations introduced in Section \ref{section setting} 
(e.g., for each $(t,x)\in[0,T]\times\bR^d$, 
$(X^{d,\theta,t,x}_{s})_{s\in[0,T]}$ 
will be denoted by $(X^{t,x}_{s})_{s\in[0,T]}$).

\subsection{Dimension-depending bounds for SDEs}
\label{section bounds SDE}

\begin{lemma}
\label{lemma q moment est SDE}
Let Assumptions \ref{assumption Lip and growth} hold.
For every $(t,x)\in[0,T]\times\bR^d$, let $(X^{t,x}_s)_{s\in[t,T]}$ be the
stochastic process defined in \eqref{SDE}.
Then it holds for all $(t,x)\in[0,T]\times\bR^d$, $s\in[t,T]$, 
and $q\in[2,\infty)$ that
\begin{equation}
\label{q moment est SDE}
\bE\left[\sup_{r\in[t,s]}\big(d^p+\|X^{t,x}_r\|^2\big)^{q/2}\right]
\leq
\big[C_{q,1}e^{\rho_{q,1}(s-t)}(d^p+\|x\|^2)\big]^{q/2},
\end{equation}
and
\begin{equation}
\label{q moment est SDE -x}
\bE\bigg[\sup_{r\in[t,s]}\big\|X^{t,x}_r-x\big\|^q\bigg]
\leq
\left[K_{q,0}(s-t)e^{\rho_{q,1}(s-t)}(d^p+\|x\|^2)\right]^{q/2},
\end{equation}
where
\begin{equation}
\label{def C q 1}
C_{q,1}
:=
\left(
2^{q/2}6^{q-1}\big[1+2^{q-1}L^{q/2}+T^{q/2}\big((4q)^q+T^{q/2}\big)\big]
\right)^{2/q},
\end{equation}
\begin{equation}
\label{def rho q 1}
\rho_{q,1}
:=
2q^{-1}6^{q-1}L^{q/2}T^{\frac{q-2}{2}}\big((4q)^q+T^{q/2}\big),
\end{equation}
and
\begin{equation}
\label{def K q 0}
K_{q,0}:=LC_{q,1}
\big[4^{q-1}\big(T^{q/2}+(4q)^q\big)\big]^{2/q}.
\end{equation}
\end{lemma}

\begin{proof}
By \eqref{SDE}, \eqref{assumption Lip mu sigma}, \eqref{assumption growth},
Jensen's inequality, and Burkholder-Davis-Gundy inequality
we have for all $(t,x)\in[0,T]\times\bR^d$ and $q\in[2,\infty)$ that
\begin{align}
&
\bE\left[\sup_{r\in[t,s]}\|X^{t,x}_r\|^q\right]
\nonumber\\
&
\leq
3^{q-1}\|x\|^q
+
3^{q-1}\bE\left[\bigg\|\int_t^s\mu(X^{t,x}_r)\,dr\bigg\|^q\right]
+
3^{q-1}\bE\left[\sup_{u\in[t,s]}\bigg\|\int_t^u\sigma(X^{t,x}_r)\,dW_r\bigg\|^q\right]
\nonumber\\
&
\leq
3^{q-1}\|x\|^q
+
3^{q-1}(s-t)^{q-1}\int_t^s\bE\left[\|\mu(X^{t,x}_r)\|^q\right]dr
+
3^{q-1}(4q)^q
\left(\bE\left[\int_t^s\|\sigma(X^{t,x}_r)\|_F^2\,dr\right]\right)^{q/2}
\nonumber\\
&
\leq
3^{q-1}\|x\|^q
+
6^{q-1}T^{q-1}\int_t^s\bE\left[\|\mu(X^{t,x}_r)-\mu(0)\|^q\right]dr
+
6^{q-1}T^{q-1}\int_t^s\|\mu(0)\|^q\,dr
\nonumber\\
& \quad
+
6^{q-1}T^{\frac{q-2}{2}}(4q)^q
\int_t^s\bE\left[\|\sigma(X^{t,x}_r)-\sigma(0)\|_F^q\right]dr
+
6^{q-1}T^{\frac{q-2}{2}}(4q)^q
\int_t^s\|\sigma(0)\|_F^q\,dr
\nonumber\\
&
\leq
3^{q-1}\|x\|^q+6^{q-1}T^{q-1}L^{q/2}\int_t^s
\bE\bigg[\sup_{u\in[t,r]}\|X^{t,x}_u\|^q\bigg]
\,dr
+
6^{q-1}T^q(Ld^p)^{q/2}
\nonumber\\
& \quad
+
6^{q-1}T^{\frac{q-2}{2}}(4q)^qL^{q/2}\int_t^s
\bE\bigg[\sup_{u\in[t,r]}\|X^{t,x}_u\|^q\bigg]
\,dr
+
6^{q-1}T^{q/2}(4q)^q(Ld^p)^{q/2}
\nonumber\\
&
\leq
3^{q-1}\big[1+2^{q-1}L^{q/2}T^{q/2}\big((4q)^q+T^{q/2}\big)\big]
(d^p+\|x\|^2)^{q/2}
\nonumber\\
& \quad
+
6^{q-1}L^{q/2}T^{\frac{q-2}{2}}\big((4q)^q+T^{q/2}\big)
\int_t^s\bE
\bigg[\sup_{u\in[t,r]}\|X^{t,x}_u\|^q\bigg]
\,dr.
\label{integral est q moment}
\end{align}
Moreover, under Assumption \ref{assumption Lip and growth} 
it is well-known (see, e.g., Theorem 4.1 in \cite{Mao2007}) that
$$
\bE\bigg[\sup_{s\in[t,T]}\|X^{t,x}_s\|^q\bigg]<\infty
$$
for all $(t,x)\in[0,T]\times\bR^d$ and $q\in[2,\infty)$.
Hence, \eqref{integral est q moment} and Gr\"onwall's lemma ensure
for all $(t,x)\in[0,T]\times\bR^d$ and $q\in[2,\infty)$ that
\begin{align*}
\bE\bigg[\sup_{r\in[t,s]}\|X^{t,x}_r\|^q\bigg]
\leq
&
3^{q-1}\big[1+2^{q-1}L^{q/2}+T^{q/2}\big((4q)^q+T^{q/2}\big)\big]
\\
&
\cdot
\exp\big\{6^{q-1}L^{q/2}T^{\frac{q-2}{2}}\big((4q)^q+T^{q/2}\big)(s-t)\big\}
(d^p+\|x\|^2)^{q/2}.
\end{align*}
This together with the fact that $(a+b)^m\leq 2^{q-1}(a^m+b^m)$ for all
$a,b\in[0,\infty)$ and $m\in[1,\infty)$ 
imply \eqref{q moment est SDE}.
Next, by \eqref{SDE}, \eqref{assumption Lip mu sigma}, \eqref{assumption growth}, 
\eqref{q moment est SDE}, Jensen's inequality, and Burkholder-Davis-Gundy inequality
we notice for all $(t,x)\in[0,T]\times\bR^d$, $s\in[t,T]$, $N\in\bN$,
and $q\in[2,\infty)$ that
\begin{align*}
&
\bE\bigg[
\sup_{r\in[t,s]}
\big\|X^{t,x}_r-x\big\|^q
\bigg]
\\
&
\leq
2^{q-1}
\bE\left[
\sup_{u\in[t,s]}
\bigg\|
\int_t^u
\mu\big(X^{t,x}_r\big)
\,dr
\bigg\|^q
\right]
+
2^{q-1}
\bE\left[
\sup_{u\in[t,s]}
\bigg\|
\int_t^u
\sigma\big(X^{t,x}_r\big)
\,dW_r
\bigg\|^q
\right]
\\
&
\leq
2^{q-1}(s-t)^{q-1}\int_t^s
\bE\left[\big\|\mu\big(X^{t,x}_r\big)\big\|^q\right]
dr
+
2^{q-1}(4q)^q
\left(
\bE\left[
\int_t^s
\big\|
\sigma\big(X^{t,x}_r\big)
\big\|_F^2
\,dr
\right]
\right)^{q/2}
\\
&
\leq
4^{q-1}(s-t)^{q-1}
\left(
\int_t^s\bE\left[
\big\|
\mu\big(X^{t,x}_r\big)-\mu(0)
\big\|^q
\right]dr
+
\int_t^s
\big\|
\mu(0)
\big\|^q
\,dr
\right)
\\
& \quad
+
4^{q-1}(s-t)^{\frac{q-2}{2}}(4q)^q
\left(
\int_t^s\bE\left[
\big\|
\sigma\big(X^{t,x}_r\big)-\sigma(0)
\big\|_F^q
\right]dr
+
\int_t^s
\big\|
\sigma(0)
\big\|_F^q
\,dr
\right)
\\
&
\leq
4^{q-1}(s-t)^qL^{q/2}
\bE\bigg[
\sup_{r\in[t,s]}\big\|X^{t,x}_r\big\|^q
\bigg]
+
4^{q-1}(s-t)^q(Ld^p)^{q/2}
\\
& \quad
+
4^{q-1}(s-t)^{q/2}(4q)^qL^{q/2}
\bE\bigg[
\sup_{r\in[t,s]}\big\|X^{t,x}_r\big\|^q
\bigg]
+
4^{q-1}(s-t)^{q/2}(4q)^q(Ld^p)^{q/2}
\\
&
\leq
4^{q-1}(s-t)^{q/2}(Ld^p)^{q/2}[T^{q/2}+(4q)^q]
+
4^{q-1}(s-t)^{q/2}L^{q/2}[T^{q/2}+(4q)^q]
\bE\bigg[
\sup_{r\in[t,s]}\big\|X^{t,x}_r\big\|^q
\bigg]
\\
&
\leq
4^{q-1}(s-t)^{q/2}L^{q/2}[T^{q/2}+(4q)^q]
\bE\bigg[
\sup_{r\in[t,s]}
\big(
d^p+\big\|X^{t,x}_r\big\|^2
\big)^{q/2}
\bigg]
\\
&
\leq
4^{q-1}(s-t)^{q/2}L^{q/2}[T^{q/2}+(4q)^q]
\left[C_{q,1}e^{\rho_{q,1}(s-t)}(d^p+\|x\|^2)\right]^{q/2}.
\end{align*}
This shows \eqref{q moment est SDE -x}.
Therefore, we have completed the proof of this lemma.
\end{proof}

\begin{lemma}
Let Assumption \ref{assumption Lip and growth} hold.
For every $(t,x)\in[0,T]\times\bR^d$, 
let $\big(X^{t,x}_s\big)_{s\in[t,T]}$ be the
stochastic process defined in \eqref{SDE}.
Then it holds for all $x,y\in\bR^d$, $t\in[0,T]$,
$t'\in[t,T]$, $s\in[t',T]$ and $q\in[2,\infty)$ that
\begin{equation}                                    \label{L q continuity SDE}
\bE\left[\big\|
X^{t,x}_s-X^{t',y}_s\big\|^q\right]
\leq 
C_{q,2}
\left[(1+T^{q/2})\big(C_{q,1}Ld^p\big)^{q/2}(t'-t)^{q/2}(d^p+\|x\|^2)^{q/2}
+\|x-y\|^q\right],
\end{equation}
where $C_{q,1}$ is defined in \eqref{def C q 1}, and
\begin{equation}
\label{def C q 2}
C_{q,2}:=5^{q-1}\exp\left\{5^{q-1}L^{q/2}T^q\big((4q)^q+T^{q/2}\big)\right\}.
\end{equation}
\end{lemma}

\begin{proof}
We fix $x,y\in\bR^d$, $t\in[0,T]$,
$t'\in[t,T]$, $s\in[t',T]$, and $q\in[2,\infty)$  
throughout the proof of this lemma.
By \eqref{SDE}, we first observe that
\begin{equation}
\bE\left[\big\|
X^{t,x}_s-X^{t',y}_s\big\|^q\right]
\leq 
5^{q-1}\|x-y\|^q
+
5^{q-1}\sum_{i=1}^3A_i,
\label{ineq X t t' x y}
\end{equation}
where
\begin{align*}
&
A_1
:=
\bE\left[\bigg|\int_t^{t'}\mu(X^{t,x}_r)\,dr\bigg|^q\right]
+
\bE\left[\bigg\|\sum_{j=1}^d\int_t^{t'}\sigma^j(X^{t,x}_r)\,dW^j_r\bigg\|^q\right],
\\
&
A_2
:=
\bE\left[\bigg|
\int_{t'}^s\big(\mu(X^{t,x}_r)-\mu(X^{t',y}_r)\big)\,dr
\bigg|^q\right],
\end{align*}
and
$$
A_3
:=
\bE\left[\bigg\|
\sum_{j=1}^d
\int_{t'}^s\big(\sigma^j(X^{t,x}_r)-\sigma^j(X^{t',y}_r)\big)\,dW^j_r
\bigg\|^q\right].
$$
By \eqref{assumption growth}, \eqref{q moment est SDE}, Jensen's inequality,
and Burkholder-Davis-Gundy inequality it holds that
\begin{align}
A_1
&
\leq
(t'-t)^{q-1}\int_t^{t'}\bE\left[|\mu(X^{t,x}_r)|^q\right]dr
+
(4q)^q\left(
\bE\left[\int_t^{t'}\|\sigma(X^{t,x}_r)\|_F^2\,dr\right]
\right)^{q/2}
\nonumber\\
&
\leq
(t'-t)^{q-1}(Ld^p)^{q/2}\int_t^{t'}\bE\left[(d^p+\|X^{t,x}_r\|^2)^{q/2}\right]dr
+
(4q)^q(Ld^p)^{q/2}
\left(\int_t^{t'}\bE\left[d^p+\|X^{t,x}_r\|^2\right]dr\right)^{q/2}
\nonumber\\
&
\leq
(t'-t)^q\big(C_{q,1}Ld^p\big)^{q/2}(d^p+\|x\|^2)^{q/2}
+
(t'-t)^{q/2}\big(C_{q,1}Ld^p\big)^{q/2}(d^p+\|x\|^2)^{q/2},
\label{est A 1 q}
\end{align}
where $C_{q,1}$ is the positive constant defined in \eqref{def C q 1}.
Moreover, by \eqref{assumption Lip mu sigma}, Jensen's inequality,
and Burkholder-Davis-Gundy inequality we have that
\begin{align}
A_2
&
\leq
(s-t')^{q-1}\int_{t'}^s
\bE\left[
\big|\mu(X^{t,x}_r)-\mu(X^{t',y}_r)\big|^q
\right]dr
\leq
T^{q-1}L^{q/2}\int_{t'}^s
\bE\left[
\big\|X^{t,x}_r-X^{t',y}_r\big\|^q
\right]
dr,
\label{est A 2 q}
\end{align}
and
\begin{align}
A_3
&
\leq 
(4q)^q
\left(
\bE\left[
\int_{t'}^s
\big\|\sigma(X^{t,x}_r)-\sigma(X^{t',y}_r)\big\|_F^2
\,dr
\right]
\right)^{q/2}
\leq
(4q)^qT^{\frac{q-2}{2}}L^{q/2}
\int_{t'}^s\bE\left[\big\|X^{t,x}_r-X^{t',y}_r\big\|^q\right]dr.
\label{est A 3 q}
\end{align}
Then combining \eqref{ineq X t t' x y}, \eqref{est A 1 q}, \eqref{est A 2 q},
and \eqref{est A 3 q} yields that
\begin{align*}
\bE\left[\big\|
X^{t,x}_s-X^{t',y}_s\big\|^q\right]
\leq 
&
5^{q-1}\|x-y\|^q
+
5^{q-1}(1+T^{q/2})\big(C_{q,1}Ld^p\big)^{q/2}(t'-t)^{q/2}(d^p+\|x\|^2)^{q/2}
\\
&
+
5^{q-1}L^{q/2}T^{\frac{q-2}{2}}\big((4q)^q+T^{q/2}\big)
\int_{t'}^s\bE\left[\big\|X^{t,x}_r-X^{t',y}_r\big\|^q\right]dr.
\end{align*}
This together with \eqref{SDE moment est} and Gr\"onwall's lemma imply that
\eqref{L q continuity SDE}. The proof of this lemma is therefore completed.
\end{proof}

\subsection{Derivatives of the solutions of SDEs}
\label{section deri SDE}

\quad\\
The following well-known lemma describes the differentiability of the solution of
\eqref{SDE} with respect to the initial value, and we refer to
Theorem 3.4 in \cite{Kunita} for its proof. 

\begin{lemma}[\cite{Kunita}, Theorem 3.4]                                 
\label{lemma classical derivative}
Let Assumptions \ref{assumption Lip and growth} and \ref{assumption gradient} hold.
For every $(t,x)\in[0,T]\times\bR^d$, let $(X^{t,x}_s)_{s\in[t,T]}$ be the
stochastic process defined in \eqref{SDE}.
Then for every $(t,x)\in[0,T]\times\bR^d$, $k\in\{1,2,\dots,d\}$,
and $s\in[t,T]$, the classical derivative of 
$X^{t,x}_s$ with respect to $x_k$, denoted by 
$\frac{\partial}{\partial x_k} X^{t,x}_s$, 
exists and is given by the following SDE 
\begin{align}
\frac{\partial}{\partial x_k} X^{t,x}_s
=
&
e_k+\int_t^s
(\nabla\mu)\big(X^{t,x}_{r}\big)
\frac{\partial}{\partial x_k}X^{t,x}_{r}\,dr
+
\sum_{j=1}^d
\int_t^s(\nabla\sigma^j)\big(X^{t,x}_{r}\big)
\frac{\partial}{\partial x_k}X^{t,x}_{r}\,dW^j_r.
\label{SDE Classical Derivative}
\end{align}
\end{lemma}

\begin{lemma}                                       \label{lemma L2 derivative}
Let Assumptions \ref{assumption Lip and growth},
and \ref{assumption gradient} hold.
For every $(t,x)\in[0,T]\times\bR^d$, let $(X^{t,x}_s)_{s\in[t,T]}$ be the
stochastic process defined in \eqref{SDE}.
Then for all $(t,x)\in[0,T]\times\bR^d$, $k\in\{1,2,\dots,d\}$
and $s\in[t,T]$, the $L_2(\bP)$-derivative of 
$X^{t,x}_s$ with respect to $x_k$, denoted by 
$D_{x_k} X^{t,x}_s$, 
exists and coincides with the classical derivative 
$\frac{\partial}{\partial x_k}X^{t,x}_s$ given by 
\eqref{SDE Classical Derivative}, i.e.,
\begin{equation}                                          \label{L2derivative}
\lim_{\delta\to 0}
\bE\left[\left\|\frac{X^{t,x+\delta e_k}-X^{t,x}_s}{\delta}-
\frac{\partial}{\partial x_k}X^{t,x}_s\right\|^2\right]=0.
\end{equation}
\end{lemma}

\begin{proof}
For every $k\in\{1,2,\dots,d\}$, $(t,x)\in[0,T]\times\bR^d$, 
$s\in[t,T]$, and $\delta\in(0,1)$, 
we define $N_t(s,x,k,\delta)$ by
$$
N_t(s,x,k,\delta):=\frac{X^{t,x+\delta e_k}_s-X^{t,x}_s}{\delta}.
$$
Assumptions \ref{assumption Lip and growth} and \ref{assumption gradient}, 
and e.g., Theorem 3.3 in \cite{Kunita}
ensure that there exists a
positive constant $C_{T,d}$ only depending on $T$ and $d$ satisfying for all 
$k\in\{1,2,\dots,d\}$, 
$t\in[0,T]$, $x,x'\in\bR^d$, $s\in[t,T]$, and $\delta,\delta'\in(0,1)$
\begin{equation}                                             \label{est1}
\bE \Big[\sup_{s\in[t,T]}\big\|N_t(s,x,k,\delta)\big\|^2\Big]\leq C_{T,d},
\end{equation}
and
\begin{equation}                                            \label{est2}
\bE\Big[\sup_{s\in[t,T]}
\big\|N_t(s,x,k,\delta)-N_t(s,x',k,\delta')\big\|^2\Big]
\leq C_{T,d}\big(\|x-x'\|^{2}+|\delta-\delta'|^{2}\big).
\end{equation}
Hence, by Fatou's lemma we have for all $k\in\{1,2,\dots,d\}$, 
$(t,x)\in[0,T]\times\bR^d$, and $s\in[t,T]$ that
\begin{align*}
\lim_{\delta\to 0}
\bE\left[
\left\|
\frac{X^{t,x+\delta e_k}-X^{t,x}_s}{\delta}-
\frac{\partial}{\partial x_k}X^{t,x}_s
\right\|^2
\right]
&
=\lim_{\delta\to 0} \bE\left[
\big\|N_t(s,x,k,\delta)-\lim_{\delta'\to 0}N_t(s,x,k,\delta')\big\|^2
\right]
\\
&
\leq \lim_{\delta\to 0}
\left(
\liminf_{\delta'\to 0}
\bE\Big[\big\|N_t(s,x,k,\delta)-N_t(s,x,k,\delta')\big\|^2\Big]
\right)
\\
&
\leq \lim_{\delta\to 0}
\left(\liminf_{\delta'\to 0} C_{T,d}|\delta-\delta'|^{2}\right)=0.
\end{align*}
The proof of this lemma is therefore completed.
\end{proof}

\begin{lemma}                                      \label{lemma SDE partial est}
Let Assumptions \ref{assumption Lip and growth} and \ref{assumption gradient} hold.
For every $(t,x)\in[0,T]\times\bR^d$ and $k\in\{1,2,\dots,d\}$, 
let $\big(\frac{\partial}{\partial x_k} X^{t,x}_s\big)_{s\in[t,T]}$ be the
stochastic process defined in \eqref{SDE Classical Derivative}.
Then it holds for all $(t,x)\in[0,T]\times\bR^d$, 
$k\in\{1,2,\dots,d\}$, and $q\in[2,\infty)$ that
\begin{equation}                                      \label{moment est partial q}
\bE\left[\sup_{s\in[t,T]}\Big\|
\frac{\partial}{\partial x_k}X^{t,x}_s\Big\|^q\right]
\leq C_{d,q,0},
\end{equation}
where 
\begin{equation}
\label{def C q 0}
C_{d,q,0}
:=3^{q-1}
\exp\big\{
3^{q-1}(Ld)^{\frac{q}{2}}T^{\frac{q-2}{2}}
\big[T^{\frac{q}{2}}+(4q)^q\big]
\big\}.
\end{equation}
In particular, it holds for all $d\in\bN$, $(t,x)\in[0,T]\times\bR^d$, 
and $k\in\{1,2,\dots,d\}$ that
\begin{equation}                                      \label{moment est partial}
\bE\left[\sup_{s\in[t,T]}\Big\|
\frac{\partial}{\partial x_k}X^{t,x}_s\Big\|^2\right]
\leq C_{d,0},
\end{equation}
where $C_{d,0}:=3\exp\{3Ld(T+64)T\}$.

If we further let Assumption \ref{assumption bbd partials global} hold, 
then it holds for all $(t,x)\in[0,T]\times\bR^d$, 
$k\in\{1,2,\dots,d\}$, and $q\in[2,\infty)$ that
\begin{equation}                                      \label{moment est partial q 1}
\bE\left[\sup_{s\in[t,T]}\Big\|
\frac{\partial}{\partial x_k}X^{t,x}_s\Big\|^q\right]
\leq C'_{q ,0},
\end{equation}
where 
\begin{equation}
\label{def C q 0 '}
C'_{q,0}
:=3^{q-1}
\exp\big\{
3^{q-1}K^{\frac{q}{2}}T^{\frac{q-2}{2}}
\big[T^{\frac{q}{2}}+(4q)^q\big]
\big\}.
\end{equation}
\end{lemma}

\begin{proof}
We fix $(t,x)\in[0,T]\times\bR^d$ and $q\in[2,\infty)$ throughout this proof.
For each $k\in\{1,2,\dots,d\}$, we define a sequence of stopping times by
\begin{equation}
\tau^k_n:=\inf\left\{s\geq t: 
\Big\|\frac{\partial}{\partial x_k}X^{t,x}_s\Big\|\geq n\right\}
\wedge T, \quad n\in\bN,
\end{equation} 
By Lemma \ref{lemma classical derivative}, 
H\"older’s inequality, Burkholder-Davis-Gundy inequality 
(see, e.g., Theorem VII.92 in \cite{DM1982}), 
it holds for all $k\in\{1,2,\dots,d\}$, $n\in\bN$, and $s\in [t,T]$ that
\begin{align}
a^{k,n}(s):= \bE\left[\sup_{r\in[t,s\wedge\tau^k_n]}\Big\|
\frac{\partial}{\partial x_k}X^{t,x}_r\Big\|^q\right]
\leq 
3^{q-1}+3^{q-1}\big(A^{k,n}_1(s)+A^{k,n}_2(s)\big),
\label{ineq a k n}
\end{align}
where 
$$
A^{k,n}_1(s)
:=(s-t)^{q-1}\bE\left[\int_t^{s\wedge \tau^k_n}\Big\|
\big(\nabla\mu\big)\big(X^{t,x}_{r}\big)
\frac{\partial}{\partial x_k}X^{t,x}_{r}\Big\|^q\,dr\right],
$$
and
$$
A^{k,n}_2(s)
:=
(4q)^q
\left(
\bE\left[
\sum_{j=1}^d
\int_t^{s\wedge\tau^k_n}\Big\|\big(\nabla\sigma^j\big)\big(X^{t,x}_{r}\big)
\frac{\partial}{\partial x_k}X^{t,x}_{r}\Big\|^2\,dr
\right]
\right)^{\frac{q}{2}}.
$$
Then by \eqref{bbd partials} and Cauchy-Schwarz inequality, 
we notice for $k\in\{1,2,\dots,d\}$, $n\in\bN$, and $s\in [t,T]$ that
\begin{align}
A^{k,n}_1(s)
&
=
(s-t)^{q-1}\bE\left[\int_t^{s\wedge \tau^k_n}
\left[
\sum_{j=1}^d\left(
\sum_{i=1}^d
\frac{\partial}{\partial x_i}\mu^j(X^{t,x}_r)
\frac{\partial}{\partial x_k}X^{t,x,i}_r
\right)^2
\right]^{\frac{q}{2}}
\,dr\right]
\nonumber\\
&
\leq 
(s-t)^{q-1}\bE\left[\int_t^{s\wedge \tau^k_n}
\big\|(\nabla\mu)(X^{t,x}_r)\big\|_F^q
\cdot\Big\|\frac{\partial}{\partial x_k}X^{t,x}_r\Big\|^q
\,dr\right]
\nonumber\\
&
\leq
(s-t)^{q-1}(Ld)^{\frac{q}{2}}
\bE\left[\int_t^{s\wedge \tau^k_n}
\Big\|\frac{\partial}{\partial x_k}X^{t,x}_r\Big\|^q
\,dr\right]
\label{turtle 1}
\end{align} 
and
\begin{align}
A^{k,n}_2(s)
\leq 
&
(4q)^q
\left(
\bE\left[\sum_{j=1}^d\int_t^{s\wedge \tau^k_n}
\big\|(\nabla\sigma^j)(X^{t,x})\big\|_F^2
\cdot\Big\|\frac{\partial}{\partial x_k}X^{t,x}_r\Big\|^2
\,dr\right]
\right)^{\frac{q}{2}}
\nonumber\\
&
\leq
(4q)^q(Ld)^{\frac{q}{2}}(s-t)^{\frac{q-2}{2}}
\bE\left[\int_t^{s\wedge \tau^k_n}
\Big\|\frac{\partial}{\partial x_k}X^{t,x}_r\Big\|^q
\,dr\right].
\label{turtle 2}
\end{align}
Combing \eqref{ineq a k n}, \eqref{turtle 1}, and \eqref{turtle 2} yields for all 
$k\in\{1,2,\dots,d\}$, $n\in\bN$, and $s\in[t,T]$ that
\begin{align}
a^{k,n}(s)
&
\leq 
3^{q-1}+3^{q-1}(Ld)^{\frac{q}{2}}T^{\frac{q-2}{2}}
\big[T^{\frac{q}{2}}+(4q)^q\big]
\bE\left[\int_t^{s\wedge \tau^k_n}
\Big\|\frac{\partial}{\partial x_k}X^{t,x}_r\Big\|^q\,dr\right]
\nonumber\\
&
\leq 
3^{q-1}+3^{q-1}(Ld)^{\frac{q}{2}}T^{\frac{q-2}{2}}
\big[T^{\frac{q}{2}}+(4q)^q\big]
\int_t^{s}
a^{k,n}(r)\,dr.
\label{a k n est 1}
\end{align}
This together with Gr\"onwall's lemma imply 
for all $k\in\{1,2,\dots,d\}$ and $n\in\bN$ that
\begin{equation}                                  \label{a k n est 2}
a^{k,n}(T)
\leq 
3^{q-1}
\exp\left\{
3^{q-1}(Ld)^{\frac{q}{2}}T^{\frac{q-2}{2}}
\big[T^{\frac{q}{2}}+(4q)^q\big]T
\right\}.
\end{equation}
Furthermore, applying Fatou's lemma and taking limit 
as $n\to\infty$ in \eqref{a k n est 2}
yields \eqref{moment est partial q}.
Taking $q=2$ in \eqref{moment est partial q}, we get \eqref{moment est partial}.
Moreover, Assumption \ref{assumption bbd partials global} and
the application of analogous arguments as used to obtain \eqref{moment est partial}
ensure \eqref{moment est partial q 1}.
Therefore we have completed the proof of this lemma. 
\end{proof}

\begin{lemma}
Let Assumptions \ref{assumption Lip and growth} and \ref{assumption gradient} hold.
For each $(t,x)\in[0,T]\times\bR^d$ and $k\in\{1,2,\dots,d\}$, 
let $(X^{t,x}_s)_{s\in[t,T]}:[t,T]\times\bR^d\to \bR^d$,
and 
$\big(\frac{\partial}{\partial x_k}X^{t,x}_s\big)_{s\in[t,T]}
:[t,T]\times\bR^d\to \bR^d$ 
be the stochastic process defined in \eqref{SDE} 
and \eqref{SDE Classical Derivative}.
Then it holds for all $k\in\{1,2,\dots,d\}$, $x,y\in\bR^d$, $t\in[0,T]$,
$t'\in[t,T]$, and $s\in[t',T]$ that
\begin{equation}
\label{L 2 continuity partial X}
\bE\left[\Big\|
\frac{\partial}{\partial x_k}X^{t,x}_s
-
\frac{\partial}{\partial y_k}X^{t',y}_s
\Big\|^2\right]
\leq
C_{d,1}\big[(t'-t)(d^p+\|x\|^2)+\|x-y\|^2\big]
e^{8Ld(1+T)T},
\end{equation}
where $C_{d,1}$ is a positive constant defined by
\begin{equation}
\label{def C 1}
C_{d,1}:=4(1+T)\big[2L_0dT((1+T)C_{4,1}Ld^p+1)(C_{d,4,0}C_{4,2})^{1/2}+LdC_{d,0}\big],
\end{equation}
with $C_{d,0}$ being the constant defined in \eqref{moment est partial},
and $C_{d,4,0}$, $C_{4,1}$, and $C_{4,2}$ being the constants defined by 
\eqref{def C q 0}, \eqref{def C q 1}, and \eqref{def C q 2}, respectively,
with $q=4$.

Moreover, if we further let Assumption \ref{assumption bbd partials global} hold,
then it holds for all $k\in\{1,2,\dots,d\}$, $x,y\in\bR^d$, $t\in[0,T]$,
$t'\in[t,T]$, and $s\in[t',T]$ that
\begin{equation}
\label{L 2 continuity partial X 2}
\bE\left[\Big\|
\frac{\partial}{\partial x_k}X^{t,x}_s
-
\frac{\partial}{\partial y_k}X^{t',y}_s
\Big\|^2\right]
\leq
C_{d,2}\big[(t'-t)(d^p+\|x\|^2)+\|x-y\|^2\big]
e^{8K(1+T)T},
\end{equation}
where $C_{d,2}$ is a positive constant defined by
\begin{equation}
\label{def C 2}
C_{d,2}:=4(1+T)\big[2L_0dT((1+T)C_{4,1}Ld^p+1)(C'_{4,0}C_{4,2})^{1/2}+KC'_{2,0}\big],
\end{equation}
with
$C'_{2,0}$ and $C'_{4,0}$ being the positive constant 
defined by \eqref{def C q 0 '} with $q=2$ and $q=4$, respectively.
\end{lemma}

\begin{proof}
We fix $k\in\{1,2,\dots,d\}$, $x,y\in\bR^d$, $t\in[0,T]$,
$t'\in[t,T]$, $s\in[t',T]$, and $q\in(2,\infty)$  
throughout the proof of this lemma.
By \eqref{SDE Classical Derivative}, we first notice that
\begin{equation}
\label{est t t' x y s k}
\bE\left[\Big\|
\frac{\partial}{\partial x_k}X^{t,x}_s
-
\frac{\partial}{\partial y_k}X^{t',y}_s
\Big\|^2\right]
\leq 
4\sum_{i=1}^3A_i,
\end{equation}
where
\begin{align*}
&
A_1
:=
\bE\left[\bigg\|
\int_t^{t'}
(\nabla\mu)\big(X^{t,x}_{r}\big)
\frac{\partial}{\partial x_k}X^{t,x}_{r}
\,dr
\bigg\|^2\right]
+
\bE\left[\bigg\|
\sum_{j=1}^d
\int_t^{t'}
(\nabla\sigma^j)\big(X^{t,x}_{r}\big)
\frac{\partial}{\partial x_k}X^{t,x}_{r}
\,dW^j_r
\bigg\|^2\right],
\\
&
A_2
:=
\bE\left[\bigg\|
\int_{t'}^{s}
\Big[
(\nabla\mu)\big(X^{t,x}_{r}\big)
\frac{\partial}{\partial x_k}X^{t,x}_{r}
-
(\nabla\mu)\big(X^{t',y}_{r}\big)
\frac{\partial}{\partial y_k}X^{t',y}_{r}
\Big]
\,dr
\bigg\|^2\right],
\end{align*}
and
\begin{equation}
A_3
:=
\bE\left[\bigg\|
\sum_{j=1}^d
\int_{t'}^{s}
\Big[
(\nabla\sigma^j)\big(X^{t,x}_{r}\big)
\frac{\partial}{\partial x_k}X^{t,x}_{r}
-
(\nabla\sigma^j)\big(X^{t',y}_{r}\big)
\frac{\partial}{\partial y_k}X^{t',y}_{r}
\Big]
\,dW^j_r
\bigg\|^2\right].
\end{equation}
By \eqref{bbd partials}, \eqref{moment est partial}, 
It\^o's isometry, Jensen's inequality, and Cauchy-Schwarz inequality, 
we obtain that
\begin{align}
A_1
&
\leq
(t'-t)
\bE\left[
\int_t^{t'}
\Big\|
(\nabla\mu)\big(X^{t,x}_{r}\big)
\frac{\partial}{\partial x_k}X^{t,x}_{r}
\Big\|^2
\,dr
\right]
+
\bE\left[
\int_t^{t'}
\sum_{j=1}^d
\Big\|
(\nabla\sigma^j)\big(X^{t,x}_{r}\big)
\frac{\partial}{\partial x_k}X^{t,x}_{r}
\Big\|^2
\,dr
\right]
\nonumber\\
&
\leq
2(t'-t)\bE\left[
\int_t^{t'}
\big\|(\nabla \mu)(X^{t,x}_r)\big\|_F^2
\Big\|\frac{\partial}{\partial x_k}X^{t,x}_r\Big\|^2
\,dr
\right]
+
2\bE\left[
\int_t^{t'}
\sum_{j=1}^d
\big\|(\nabla \sigma^j)(X^{t,x}_r)\big\|_F^2
\Big\|\frac{\partial}{\partial x_k}X^{t,x}_r\Big\|^2
\,dr
\right]
\nonumber\\
&
\leq
2LdC_{d,0}(t'-t)^2+2LdC_{d,0}(t'-t),
\label{est A 1}
\end{align}
where $C_{d,0}$ is the positive constant defined below \eqref{moment est partial}.
Moreover, by \eqref{bbd partials}, \eqref{gradient mu sigma}, 
\eqref{moment est partial q}, \eqref{moment est partial}, 
Jensen's inequality, and H\"older's inequality it holds that
\begin{align}
A_2
&
\leq
2(s-t')
\bE\left[
\int_{t'}^s
\big\|(\nabla \mu)(X^{t,x}_r)\big\|_F^2
\cdot
\Big\|
\frac{\partial}{\partial x_k}X^{t,x}_r
-
\frac{\partial}{\partial y_k}X^{t',y}_r
\Big\|^2
\,dr
\right]
\nonumber\\
& \quad
+
2(s-t')
\bE\left[
\int_{t'}^s
\big\|(\nabla\mu)(X^{t,x}_r)-(\nabla\mu)(X^{t',y}_r)\big\|_F^2
\cdot
\Big\|
\frac{\partial}{\partial y_k}X^{t',y}_r
\Big\|^2
\,dr
\right]
\nonumber\\
&
\leq
2Ld(s-t')
\int_{t'}^s
\bE\left[
\Big\|
\frac{\partial}{\partial x_k}X^{t,x}_r
-
\frac{\partial}{\partial y_k}X^{t',y}_r
\Big\|^2
\right]
dr
\nonumber\\
& \quad
+
2L_0d(s-t')
\bE\left[
\int_{t'}^s
\big\|X^{t,x}_r-X^{t',y}_r\big\|^2
\cdot
\Big\|
\frac{\partial}{\partial y_k}X^{t',y}_r
\Big\|^2
\,dr
\right]
\nonumber\\
&
\leq
2Ld(s-t')
\int_{t'}^s
\bE\left[
\Big\|
\frac{\partial}{\partial x_k}X^{t,x}_r
-
\frac{\partial}{\partial y_k}X^{t',y}_r
\Big\|^2
\right]
dr
\nonumber\\
& \quad
+2L_0d(s-t')^2C_{d,4,0}^{1/2}
\bigg(
\sup_{r\in[t,T]}
\bE\left[
\big\|
X^{t,x}_r-X^{t',y}_r
\big\|^4
\right]
\bigg)^{1/2},
\label{est A 2}
\end{align}
where $C_{d,4,0}$ is defined by \eqref{def C q 0} with $q=4$.
Analogously, by \eqref{bbd partials}, \eqref{gradient mu sigma}, 
\eqref{moment est partial q}, \eqref{moment est partial}, 
It\^o's isometry, and H\"older's inequality
we obtain that
\begin{align}
A_3
&
\leq
2
\bE\left[
\int_{t'}^s
\sum_{j=1}^d
\big\|(\nabla \sigma^j)(X^{t,x}_r)\big\|_F^2
\cdot
\Big\|
\frac{\partial}{\partial x_k}X^{t,x}_r
-
\frac{\partial}{\partial y_k}X^{t',y}_r
\Big\|^2
\,dr
\right]
\nonumber\\
& \quad
+
2
\bE\left[
\int_{t'}^s
\sum_{j=1}^d
\big\|(\nabla\sigma^j)(X^{t,x}_r)-(\nabla\sigma^j)(X^{t',y}_r)\big\|_F^2
\cdot
\Big\|
\frac{\partial}{\partial y_k}X^{t',y}_r
\Big\|^2
\,dr
\right]
\nonumber\\
&
\leq
2Ld
\int_{t'}^s
\bE\left[
\Big\|
\frac{\partial}{\partial x_k}X^{t,x}_r
-
\frac{\partial}{\partial y_k}X^{t',y}_r
\Big\|^2
\right]
dr
+
2L_0d
\bE\left[
\int_{t'}^s
\big\|X^{t,x}_r-X^{t',y}_r\big\|_F^2
\cdot
\Big\|
\frac{\partial}{\partial y_k}X^{t',y}_r
\Big\|^2
\,dr
\right]
\nonumber\\
&
\leq
2Ld
\int_{t'}^s
\bE\left[
\Big\|
\frac{\partial}{\partial x_k}X^{t,x}_r
-
\frac{\partial}{\partial y_k}X^{t',y}_r
\Big\|^2
\right]
dr
+2L_0d(s-t')C_{d,4,0}^{1/2}
\bigg(
\sup_{r\in[t,T]}
\bE\left[
\big\|
X^{t,x}_r-X^{t',y}_r
\big\|^4
\right]
\bigg)^{1/2}.
\label{est A 3}
\end{align}
Combining \eqref{L q continuity SDE}, \eqref{est t t' x y s k},
\eqref{est A 1}, \eqref{est A 2}, and \eqref{est A 3} yields that
\begin{align*}
&
\bE\left[\Big\|
\frac{\partial}{\partial x_k}X^{t,x}_s
-
\frac{\partial}{\partial y_k}X^{t',y}_s
\Big\|^2\right]
\\
&
\leq
4LdC_{d,0}(1+T)(t'-t)
+
8L_0dT(1+T)(C_{d,4,0}C_{4,2})^{1/2}
\big[(1+T)(C_{4,1}Ld^p)(t'-t)(d^p+\|x\|^2)+\|x-y\|^2\big]
\\
& \quad
+
8Ld(1+T)\int_{t'}^s
\bE\left[\Big\|
\frac{\partial}{\partial x_k}X^{t,x}_r
-
\frac{\partial}{\partial y_k}X^{t',y}_r
\Big\|^2\right]
dr
\nonumber\\
&
\leq
C_{d,1}\big[(t'-t)(d^p+\|x\|^2)+\|x-y\|^2\big]
+
8Ld(1+T)\int_{t'}^s
\bE\left[\Big\|
\frac{\partial}{\partial x_k}X^{t,x}_r
-
\frac{\partial}{\partial y_k}X^{t',y}_r
\Big\|^2\right]
dr,
\nonumber
\end{align*}
where $C_{d,1}$ is the positive constant defined by \eqref{def C 1}.
This together with \eqref{moment est partial q} and Gr\"onwall's lemma imply
\eqref{L 2 continuity partial X}.
Moreover, using Assumption \ref{assumption bbd partials global} and an
analogous argument we obtain \eqref{L 2 continuity partial X 2}.
Thus, the proof of this lemma is completed.
\end{proof}

\begin{lemma}
Let Assumptions \ref{assumption Lip and growth}, \ref{assumption ellip}, 
and \ref{assumption gradient} hold.
For each $(t,x)\in[0,T)\times\bR^d$, 
let $V^{t,x}=(V^{t,x,k})_{k\in\{1,2,\dots,d\}}:[t,T)\times\bR^d\to \bR^d$
be the stochastic process defined in \eqref{def proc V}.
Then it holds for all $t\in[0,T)$, $t'\in[t,T)$, $s\in(t',T]$, $x,x'\in\bR^d$ 
and $q\in[2,\infty)$ that
\begin{equation}
\label{L2 est proc V}
\bE\Big[\big\|V^{t,x}_s\big\|^2\Big]
\leq
d\varepsilon_d^{-1}C_{d,0}(s-t)^{-1},
\end{equation}
and
\begin{align}
\bE\Big[
\big\|V^{t,x}_s-V^{t',x'}_s\big\|^2
\Big]
\leq
&
4d\frac{(t'-t)\varepsilon_d^{-1}C_{d,0}}{(s-t)(s-t')}
+
4d^2\varepsilon_d^{-1}(s-t')^{-1}
\nonumber\\
&
\cdot
\big[
d^2\varepsilon_d^{-1}(C_{d,4,0}C_{4,2})^{1/2}((1+T)C_{4,1}Ld^p+1)
+C_{d,2}e^{8Ld(1+T)T}
\big]
\nonumber\\
&
\cdot
\big[(t'-t)(d^p+\|x\|^2)+\|x-x'\|^2\big],
\label{L2 cont est proc V}
\end{align}
where $\varepsilon_d$, $C_{d,0}$, $C_{d,1}$, $C_{d,4,0}$, $C_{4,1}$, and $C_{4,2}$
are the positive constants introduced in
\eqref{sigma ellip}, \eqref{moment est partial}, \eqref{def C 1},
\eqref{def C q 0}, \eqref{def C q 1}, and \eqref{def C q 2}, respectively,
with $q=4$.

Moreover, if we further let Assumption \ref{assumption bbd partials global} hold,
then it holds for all $t\in[0,T)$, $t'\in[t,T)$, $s\in(t',T]$, 
$x,x'\in\bR^d$, and $q\in[2,\infty)$ that
\begin{equation}
\label{L q est proc V}
\bE\Big[\big\|V^{t,x}_s\big\|^q\Big]
\leq
(4q)^q(d\varepsilon_d)^{-q/2}(C'_{2,0})^{q/2}(s-t)^{-q/2},
\end{equation}
and
\begin{align}
&
\bE\Big[
\big\|V^{t,x}_s-V^{t',x'}_s\big\|^2
\Big]
\nonumber\\
&
\leq
4d\frac{(t'-t)\varepsilon_d^{-1}C'_{2,0}}{(s-t)(s-t')}
+
4d^2\varepsilon_d^{-1}(s-t')^{-1}
\left[
d^2\varepsilon_d^{-1}(C_{d,4,0}C_{4,2})^{1/2}((1+T)C_{4,1}Ld^p+1)
+C_{d,2}e^{8K(1+T)T}
\right]
\nonumber\\
& \quad
\cdot
\big[(t'-t)(d^p+\|x\|^2)+\|x-x'\|^2\big],
\label{L2 cont est proc V *}
\end{align}
where $K$ is the positive constant introduced in \eqref{bbd partials global},
and $C'_{2,0}$ is the positive constant defined by \eqref{def C q 0 '}.
\end{lemma}

\begin{proof}
Throughout the proof of this lemma, 
we fix $t\in[0,T)$, $t'\in[t,T)$, $s\in[t',T)$, and $x,x'\in\bR^d$.
By \eqref{sigma inverse est}, \eqref{moment est partial}, and It\^o's isometry,
we have for all $k\in\{1,2,\dots,d\}$ that
\begin{align*}
\bE\Big[\big\|V^{t,x,k}_s\big\|^2\Big]
&
=
\bE\left[
\left(
\frac{1}{s-t}
\int_t^s
\left[
\sigma^{-1}(X^{t,x}_r)\frac{\partial}{\partial x_k}X^{t,x}_r
\right]^T
dW_r
\right)^2
\right]
\\
&
=
(s-t)^{-2}\bE\left[
\int_t^s\Big\|
\sigma^{-1}(X^{t,x}_r)\frac{\partial}{\partial x_k}X^{t,x}_r
\Big\|^2\,dr
\right]
\\
&
\leq 
\varepsilon_d^{-1}(s-t)^{-2}
\int_t^s\bE\left[\Big\|\frac{\partial}{\partial x_k}X^{t,x}_r\Big\|^2\right]dr
\\
&
\leq
\varepsilon_d^{-1}C_{d,0}(s-t)^{-1}.
\end{align*}
This proves \eqref{L2 est proc V}.
Furthermore, by \eqref{sigma inverse est}, \eqref{moment est partial q 1}, 
and Burkholder-Davis-Gundy inequality
we notice for all $k\in\{1,2,\dots,d\}$ that
\begin{align*}
\bE\left[\big\|V^{t,x,k}_s\big\|^q\right]
&
\leq
\frac{(4q)^q}{(s-t)^q}
\left(
\bE\left[
\int_t^s
\Big\|
\sigma^{-1}(X^{t,x}_r)\frac{\partial}{\partial x_k}X^{t,x}_r
\Big\|^2
\,dr
\right]
\right)^{q/2}
\\
&
\leq
\frac{(4q)^q\varepsilon_d^{-q/2}}{(s-t)^q}
\left(
\int_t^s\bE\left[\Big\|\frac{\partial}{\partial x_k}X^{t,x}_r\Big\|^2\right]dr
\right)^{q/2}
\\
&
\leq
(4q)^q\varepsilon_d^{-q/2}(C'_{2,0})^{q/2}(s-t)^{-q/2}.
\end{align*}
This proves \eqref{L q est proc V}.
Next, to show \eqref{L2 cont est proc V} we notice for all $k\in\{1,2,\dots,d\}$ that
\begin{equation}
\label{ineq V t t' x x' k}
\bE\Big[
\big\|V^{t,x,k}_s-V^{t',x',k}_s\big\|^2
\Big]
\leq
4\sum_{i=1}^4A_{k,i},
\end{equation}
where
\begin{align*}
&
A_{k,1}
:=
\bE\left[
\left(
\frac{1}{s-t}
\int_t^{t'}
\left[
\sigma^{-1}(X^{t,x}_r)\frac{\partial}{\partial x_k}X^{t,x}_r
\right]^T
dW_r
\right)^2
\right],
\\
&
A_{k,2}
:=
\bE\left[
\left(
\Big(\frac{1}{s-t'}-\frac{1}{s-t}\Big)
\int_{t'}^s
\left[
\sigma^{-1}(X^{t,x}_r)\frac{\partial}{\partial x_k}X^{t,x}_r
\right]^T
dW_r
\right)^2
\right],
\\
&
A_{k,3}
:=
\bE\left[
\left(
\frac{1}{s-t'}
\int_{t'}^s
\left[
\big(\sigma^{-1}(X^{t,x}_r)-\sigma^{-1}(X^{t',x'}_r)\big)
\frac{\partial}{\partial x_k}X^{t,x}_r
\right]^T
dW_r
\right)^2
\right],
\end{align*}
and
$$
A_{k,4}
:=
\bE\left[
\left(
\frac{1}{s-t'}
\int_{t'}^s
\left[
\sigma^{-1}(X^{t',x'}_r)
\Big(
\frac{\partial}{\partial x_k}X^{t,x}_r
-
\frac{\partial}{\partial x'_k}X^{t',x'}_r
\Big)
\right]^T
dW_r
\right)^2
\right].
$$
By \eqref{sigma inverse est}, \eqref{moment est partial}, and It\^o's isometry,
it holds for all $k\in\{1,2,\dots,d\}$ that
\begin{align}
A_{k,1}
&
=
\frac{1}{(s-t)^2}
\bE\left[
\int_t^{t'}
\Big\|\sigma^{-1}(X^{t,x}_r)\frac{\partial}{\partial x_k}X^{t,x}_r\Big\|^2\,dr
\right]
\leq
\frac{\varepsilon_d^{-1}}{(s-t)^2}\int_t^{t'}
\bE\left[\Big\|\frac{\partial}{\partial x_k}X^{t,x}_r\Big\|^2\right]dr
\nonumber
\\
&
\leq
\frac{(t'-t)\varepsilon_d^{-1}C_{d,0}}{(s-t)^2},
\label{est A k 1}
\end{align}
and
\begin{align}
A_{k,2}
&
=
\frac{(t'-t)^2}{(s-t')^2(s-t)^2}
\bE\left[
\int_{t'}^s
\Big\|\sigma^{-1}(X^{t,x}_r)\frac{\partial}{\partial x_k}X^{t,x}_r\Big\|^2
\,dr
\right]
\nonumber\\
&
\leq
\frac{(t'-t)^2\varepsilon_d^{-1}}{(s-t')^2(s-t)^2}
\int_{t'}^s
\bE\left[\Big\|\frac{\partial}{\partial x_k}X^{t,x}_r\Big\|^2\right]
dr
\nonumber\\
&
\leq
\frac{(t'-t)^2\varepsilon_d^{-1}C_{d,0}}{(s-t')(s-t)^2}.
\label{est A k 2}
\end{align}
Furthermore, by \eqref{moment est partial q}, \eqref{bbd partial sigma inv},
\eqref{L q continuity SDE}, It\^o's isometry, the mean-value theorem,
and Cauchy-Schwarz inequality we have for all $k\in\{1,2,\dots,d\}$ that
\begin{align}
A_{k,3}
&
=
\frac{1}{(s-t')^2}
\bE\left[
\int_{t'}^s
\Big\|
\big[\sigma^{-1}(X^{t,x}_r)-\sigma^{-1}(X^{t',x'}_r)\big]
\frac{\partial}{\partial x_k}X^{t,x}_r
\Big\|^2
\,dr
\right]
\nonumber\\
&
\leq
\frac{1}{(s-t')^2}
\int_{t'}^s
\bE\left[
\big\|\sigma^{-1}(X^{t,x}_r)-\sigma^{-1}(X^{t',x'}_r)\big\|_F^2
\cdot 
\Big\|\frac{\partial}{\partial x_k}X^{t,x}_r\Big\|^2
\right]
dr
\nonumber\\
&
\leq
\frac{Ld^3\varepsilon_d^{-2}}{(s-t')^2}
\int_{t'}^s
\bE\left[
\big\|X^{t,x}_r-X^{t',x'}_r\big\|^2
\cdot
\Big\|\frac{\partial}{\partial x_k}X^{t,x}_r\Big\|^2
\right]
dr
\nonumber\\
&
\leq
\frac{Ld^3\varepsilon_d^{-2}}{(s-t')^2}
\int_{t'}^s
\left(\bE\left[\big\|X^{t,x}_r-X^{t',x'}_r\big\|^4\right]\right)^{1/2}
\left(\bE\left[\Big\|\frac{\partial}{\partial x_k}X^{t,x}_r\Big\|^4\right]\right)^{1/2}
dr
\nonumber\\
&
\leq
\frac{Ld^3\varepsilon_d^{-2}(C_{d,4,0}C_{4,2})^{1/2}}{s-t'}
\big[
(1+T)C_{4,1}Ld^p(t'-t)(d^p+\|x\|^2)+\|x-x'\|^2
\big],
\label{est A k 3}
\end{align}
where $C_{d,4,0}$, $C_{4,1}$, and $C_{4,2}$ are the positive constants defined in
\eqref{def C q 0}, \eqref{def C q 1}, and \eqref{def C q 2}, respectively,
with $q=4$.
In addition, by \eqref{sigma inverse est}, \eqref{L 2 continuity partial X},
It\^o's isometry, and Cauchy-Schwarz inequality we obtain for all
$k\in\{1,2,\dots,d\}$ that
\begin{align}
A_{k,4}
&
=
\frac{1}{(s-t')^2}
\bE\left[
\int_{t'}^s
\Big\|
\sigma^{-1}(X^{t',x'}_r)
\Big(
\frac{\partial}{\partial x_k}X^{t,x}_r
-
\frac{\partial}{\partial x'_k}X^{t',x'}_r
\Big)
\Big\|_F^2
\,dr
\right]
\nonumber\\
&
\leq
\frac{d\varepsilon_d^{-1}}{(s-t')^2}
\int_{t'}^s
\bE\left[
\Big\|
\frac{\partial}{\partial x_k}X^{t,x}_r
-
\frac{\partial}{\partial x'_k}X^{t',x'}_r
\Big\|^2
\right]
dr
\nonumber\\
&
\leq
(s-t')^{-1}d\varepsilon_d^{-1}C_{d,2}e^{8Ld(1+T)T}
\big[(t'-t)(d^p+\|x\|^2)+\|x-x'\|^2\big],
\label{est A k 4}
\end{align}
where $C_{d,2}$ and $\varepsilon_d$ are the positive constants defined in
\eqref{def C 2} and \eqref{sigma ellip}, respectively.
Combining \eqref{ineq V t t' x x' k}, \eqref{est A k 1}, \eqref{est A k 2},
\eqref{est A k 3}, and \eqref{est A k 4} yields \eqref{L2 cont est proc V}.
By Assumption \ref{assumption bbd partials global} 
and analogous arguments to obtain \eqref{L2 cont est proc V},
we can obtain \eqref{L2 cont est proc V *}.
Hence, the proof of this lemma is completed.
\end{proof}

\begin{lemma}                                      \label{lemma two SDEs}
Let Assumptions \ref{assumption Lip and growth}  
and \ref{assumption gradient}
hold, and let
$\big(X^{t,x}_s\big)_{s\in[t,T]}:[t,T]\times\Omega\to\bR^d$ be the unique 
$\bF$-adapted continuous process satisfying \eqref{SDE}.
Let $B\subseteq \bR^d$ be a closed set,
and let $\bar{\mu}\in C^3(\bR^d,\bR^d)$ and 
$\bar{\sigma}\in C^3(\bR^d,\bR^{d\times d})$ such that
\begin{equation}                                        \label{cond B}
\bar{\mu}(x)=\mu(x),\quad \bar{\sigma}(x)=\sigma(x) 
\quad \text{for all $x\in\ B$}.
\end{equation}
Assume that $\bar{\mu}$ and $\bar{\sigma}$ satisfy
Assumptions \ref{assumption Lip and growth}
and \ref{assumption gradient}.
Moreover, for each $(t,x)\in[0,T]\times\bR^d$ 
let $\big(\bar{X}^{t,x}_s\big)_{s\in[t,T]}:[t,T]\times\Omega\to\bR^d$ be an 
$\bF$-adapted continuous process satisfying that $\bar{X}^{t,x}_t=x$,
and almost surely for all $s\in[t,T]$ 
$$                                             
d\bar{X}^{t,x}_{s}
=\bar{\mu}\big(\bar{X}^{t,x}_s\big)\,ds
+\bar{\sigma}\big(\bar{X}^{t,x}_s\big)\,dW_s.
$$
For each $(t,x)\in[0,T]\times\bR^d$ and $k\in\{1,2,\dots,d\}$,
let $\tau^{t,x}:\Omega\to[t,T]$ and $\tau^{t,x,k}:\Omega\to[t,T]$ 
be stopping times defined by
\begin{equation}                                     \label{def tau B}
\tau^{t,x}:=\inf
\big\{
s\geq t: X^{t,x}_s\notin B \; \text{or} \; \bar{X}^{t,x}_s\notin B
\big\}
\wedge T,
\end{equation}
and
\begin{equation}                                     \label{def tau k B}
\tau^{t,x,k}:=\inf
\Big\{
s\geq t: 
X^{t,x}_s\notin B \; \text{or} \; \bar{X}^{t,x}_s\notin B
\; \text{or} \; 
\frac{\partial}{\partial x_k} X^{t,x}_s\notin B 
\; \text{or} \; 
\frac{\partial}{\partial x_k}\bar{X}^{t,x}_s\notin B
\Big\}
\wedge T.
\end{equation}
Then it holds for all $(t,x)\in[0,T]\times\bR^d$ and $k\in\{1,2,\dots,d\}$ that
\begin{equation}                                               \label{prob 1 tau}
\bP\left(
\mathbf{1}_{\{s\leq\tau^{t,x}\}}\big\|X^{t,x}_s-\bar{X}^{t,x}_s\big\|=0
\; \text{for} \; \text{all} \; s\in[t,T]
\right)=1,
\end{equation}
and
\begin{equation}                                             \label{prob 1 tau k}
\bP\left(
\mathbf{1}_{\{s\leq\tau^{t,x,k}\}}
\Big\|\frac{\partial}{\partial x_k} X^{t,x}_s
-\frac{\partial}{\partial x_k} \bar{X}^{t,x}_s\Big\|=0
\; \text{for} \; \text{all} \; s\in[t,T]
\right)=1.
\end{equation}
\end{lemma}

\begin{proof}
\eqref{prob 1 tau} has been proved in \cite[Lemma~5.14]{NW2022}.
Throughout the proof of \eqref{prob 1 tau k}, 
we fix $(t,x)\in[0,T]\times\bR^d$ and $k\in\{1,2,\dots,d\}$, 
and use the shorter notations $\tau=\tau^{t,x}$ and $\tau^k=\tau^{t,x,k}$.
By Assumptions \ref{assumption Lip and growth} and \ref{assumption gradient}, 
we apply Lemma \ref{lemma classical derivative} to obtain for every $s\in[t,T]$ that
\begin{align}
\frac{\partial}{\partial x_k} X^{t,x}_s
=
&
e_k+\int_t^s
(\nabla\mu)\big(X^{t,x}_{r}\big)
\frac{\partial}{\partial x_k}X^{t,x}_{r}\,dr
+
\sum_{j=1}^d
\int_t^s(\nabla\sigma^j)\big(X^{t,x}_{r}\big)
\frac{\partial}{\partial x_k}X^{t,x}_{r}\,dW^j_r,
\label{partial X}
\end{align}
and
\begin{align}
\frac{\partial}{\partial x_k} \bar{X}^{t,x}_s
=
&
e_k+\int_t^s
(\nabla\bar{\mu})\big(\bar{X}^{t,x}_{r}\big)
\frac{\partial}{\partial x_k}\bar{X}^{t,x}_{r}\,dr
+
\sum_{j=1}^d
\int_t^s(\nabla\bar{\sigma}^j)\big(\bar{X}^{t,x}_{r}\big)
\frac{\partial}{\partial x_k}\bar{X}^{t,x}_{r}\,dW^j_r.
\label{partial X bar}
\end{align}
By \eqref{def tau B} and \eqref{def tau k B}, we also notice that
$$
\bP\left(
\mathbf{1}_{\{s\leq\tau^{k}\}}\big\|X^{t,x}_s-\bar{X}^{t,x}_s\big\|=0
\; \text{for} \; \text{all} \; s\in[t,T]
\right)
\geq
\bP\left(
\mathbf{1}_{\{s\leq\tau\}}\big\|X^{t,x}_s-\bar{X}^{t,x}_s\big\|=0
\; \text{for} \; \text{all} \; s\in[t,T]
\right).
$$
This together with \eqref{prob 1 tau} imply that
$$
\bP\left(
\mathbf{1}_{\{s\leq\tau^{k}\}}\big\|X^{t,x}_s-\bar{X}^{t,x}_s\big\|=0
\; \text{for} \; \text{all} \; s\in[t,T]
\right)=1.
$$
Thus, by \eqref{cond B}, \eqref{partial X}, \eqref{partial X bar}, 
Jensen's inequality, and It\^o's isometry 
we obtain for all $s\in[t,T]$ that
\begin{align}
&
\bE\left[
\Big\|
\frac{\partial}{\partial x_k}X^{t,x}_{s\wedge\tau^k}
-\frac{\partial}{\partial x_k}\bar{X}^{t,x}_{s\wedge\tau^k}
\Big\|^2
\right]
\nonumber\\
&
\leq 
2\int_t^s(s-t)\bE\left[
\mathbf{1}_{\{r\leq \tau^k\}}\Big\|
(\nabla\mu)\big(X^{t,x}_r\big)\frac{\partial}{\partial x_k}X^{t,x}_r
-(\nabla\mu)\big(X^{t,x}_r\big)\frac{\partial}{\partial x_k}\bar{X}^{t,x}_r
\Big\|^2
\right]dr
\nonumber\\
& \quad
+2\int_t^s\bE\left[
\mathbf{1}_{\{r\leq \tau^k\}}\Big\|
(\nabla\sigma)\big(X^{t,x}_r\big)\frac{\partial}{\partial x_k}X^{t,x}_r
-(\nabla\sigma)\big(X^{t,x}_r\big)\frac{\partial}{\partial x_k}\bar{X}^{t,x}_r
\Big\|_F^2
\right]dr
\nonumber\\
&
\leq A_1(s)+A_2(s),
\label{wangba 0}
\end{align}
where
\begin{align*}
&
A_1(s):=2\int_t^s(s-t)\bE\left[
\Big\|
(\nabla\mu)\big(X^{t,x}_{r\wedge\tau_k}\big)
\Big(
\frac{\partial}{\partial x_k}X^{t,x}_{r\wedge\tau^k}
-\frac{\partial}{\partial x_k}\bar{X}^{t,x}_{r\wedge\tau^k}
\Big)
\Big\|^2
\right]dr,
\end{align*}
and
$$
A_2(s):=2\int_t^s\bE\left[
\Big\|
(\nabla\sigma)\big(X^{t,x}_{r\wedge\tau_k}\big)
\Big(
\frac{\partial}{\partial x_k}X^{t,x}_{r\wedge\tau^k}
-\frac{\partial}{\partial x_k}\bar{X}^{t,x}_{r\wedge\tau^k}
\Big)
\Big\|_F^2
\right]dr.
$$
Then, by Cauchy-Schwarz inequality and \eqref{bbd partials}, 
we observe for all $s\in[t,T]$ that
\begin{align}
A_1(s) 
&
\leq 2T\int_t^s\bE\left[
\big\|(\nabla\mu)\big(X^{t,x}_{r\wedge\tau_k}\big)\big\|^2\cdot
\Big\|
\frac{\partial}{\partial x_k}X^{t,x}_{r\wedge\tau^k}
-\frac{\partial}{\partial x_k}\bar{X}^{t,x}_{r\wedge\tau^k}
\Big\|^2
\right]dr 
\nonumber\\
&
\leq 2LdT\int_t^s\bE\left[
\Big\|
\frac{\partial}{\partial x_k}X^{t,x}_{r\wedge\tau^k}
-\frac{\partial}{\partial x_k}\bar{X}^{t,x}_{r\wedge\tau^k}
\Big\|^2
\right]dr,
\label{wangba 1}
\end{align}
and
\begin{equation}                                                \label{wangba 2}
A_2(s)\leq 2Ld\int_t^s\bE\left[
\Big\|
\frac{\partial}{\partial x_k}X^{t,x}_{r\wedge\tau^k}
-\frac{\partial}{\partial x_k}\bar{X}^{t,x}_{r\wedge\tau^k}
\Big\|^2
\right]dr.
\end{equation}
Combining \eqref{wangba 0}--\eqref{wangba 2} shows for all $s\in[t,T]$ that
$$
\bE\left[\Big\|
\frac{\partial}{\partial x_k}X^{t,x}_{s\wedge\tau^k}
-\frac{\partial}{\partial x_k}\bar{X}^{t,x}_{s\wedge\tau^k}
\Big\|^2\right]
\leq 2Ld(T+1)\int_t^s\bE\left[\Big\|
\frac{\partial}{\partial x_k}X^{t,x}_{r\wedge\tau^k}
-\frac{\partial}{\partial x_k}\bar{X}^{t,x}_{r\wedge\tau^k}
\Big\|^2\right]dr.
$$
Hence, Lemma \ref{lemma SDE partial est} allows us to apply Gr\"onwall's lemma
to obtain for all $s\in[t,T]$ that
\begin{equation}                                             \label{diff partial bar}
\bE\left[\Big\|
\frac{\partial}{\partial x_k}X^{t,x}_{s\wedge\tau^k}
-\frac{\partial}{\partial x_k}\bar{X}^{t,x}_{s\wedge\tau^k}
\Big\|^2\right]=0.
\end{equation}
Moreover, we notice that 
$\big(\frac{\partial}{\partial x_k}X^{t,x}_s\big)_{s\in[t,T]}$
and $\big(\frac{\partial}{\partial x_k}X^{t,x}_s\big)_{s\in[t,T]}$
have continuous sample paths.
Thus, by \eqref{diff partial bar} we obtain \eqref{prob 1 tau k},
which completes the proof of this lemma.
\end{proof}

\subsection{Dimension-dependent bounds for Euler approximations}
\label{section bounds euler}

\begin{lemma}
Let Assumption \ref{assumption Lip and growth} hold.
For each $(t,x)\in[0,T]\times\bR^d$ and $N\in\bN$, 
let $(X^{t,x}_s)_{s\in[t,T]}$ and $(\cX^{t,x,N}_s)_{s\in[t,T]}$ be the 
stochastic processes defined by \eqref{SDE} and \eqref{Euler 1}, respectively.
Then it holds for all $(t,x)\in[0,T]\times\bR^d$, $s\in[t,T]$, $N\in\bN$, 
and $q\in[2,\infty)$ that
\begin{equation}
\label{L q est Euler SDE}
\bE\bigg[\sup_{r\in[t,s]}\big(d^p+\big\|\cX^{t,x,N}_r\big\|^2\big)^{q/2}\bigg]
\leq
\big[C_{q,1}e^{\rho_{q,1}(s-t)}(d^p+\|x\|^2)\big]^{q/2},
\end{equation}
\begin{equation}
\label{L q est Euler SDE -x}
\bE\bigg[\sup_{r\in[t,s]}\big\|\cX^{t,x,N}_r-x\big\|^q\bigg]
\leq
\left[K_{q,0}(s-t)e^{\rho_{q,1}(s-t)}(d^p+\|x\|^2)\right]^{q/2},
\end{equation}
and
\begin{equation}
\label{L q error Euler SDE}
\bE\bigg[\sup_{r\in[t,T]}\big\|\cX^{t,x,N}_r-X^{t,x}_r\big\|^q\bigg]
\leq
\left[K_{q,1}N^{-1}(d^p+\|x\|^2)\right]^{q/2},
\end{equation}
where $C_{q,1}$, $\rho_{q,1}$, and $K_{q,0}$ are the constants defined in 
\eqref{def C q 1}, \eqref{def rho q 1}, and \eqref{def K q 0}, respectively, and
\begin{equation}
\label{def K q 1}
K_{q,1}:=
C_{q,1}e^{\rho_{q,1}T}
\left(
16^{q-1}T^{3q/2}L^q[T^{q/2}+(4q)^q]^2
\exp\big\{4^{q-1}T^{\frac{q-2}{2}}L^{q/2}[T^{q/2}+(4q)^q]T\big\}
\right)^{2/q}
.
\end{equation}
\end{lemma}

\begin{proof}
By analogous calculation as in the proof of Lemma \ref{lemma q moment est SDE},
we obtain \eqref{L q est Euler SDE} and \eqref{L q est Euler SDE -x}.
Next, by \eqref{SDE}, \eqref{assumption Lip mu sigma}, \eqref{Euler 1},
and \eqref{est X t x s s'} we have for all $(t,x)\in[0,T]\times\bR^d$,
$s\in[t,T]$, $N\in\bN$, and $q\in[2,\infty)$ that
\begin{align}
&
\bE\bigg[
\sup_{u\in[t,s]}
\big\|
\cX^{t,x,N}_u-X^{t,x}_u
\big\|^q
\bigg]
\nonumber\\
&
\leq
2^{q-1}\bE\left[
\sup_{u\in[t,s]}
\bigg\|
\int_t^u
\big(
\mu\big(\cX^{t,x,N}_{\kappa_N(r)}\big)
-
\mu\big(X^{t,x}_r\big)
\big)
\,dr
\bigg\|^q
\right]
+
2^{q-1}\bE\left[
\sup_{u\in[t,s]}
\bigg\|
\int_t^u
\big(
\sigma\big(\cX^{t,x,N}_{\kappa_N(r)}\big)
-
\sigma\big(X^{t,x}_r\big)
\big)
\,dW_r
\bigg\|^q
\right]
\nonumber\\
&
\leq
4^{q-1}(s-t)^{q-1}
\left(
\int_t^s
\bE\left[
\big\|
\mu\big(\cX^{t,x,N}_{\kappa_N(r)}\big)
-
\mu\big(X^{t,x}_{\kappa_N(r)}\big)
\big\|^q
\right]
dr
+
\int_t^s
\bE\left[
\big\|
\mu\big(X^{t,x}_{\kappa_N(r)}\big)
-
\mu\big(X^{t,x}_r\big)
\big\|^q
\right]
dr
\right)
\nonumber\\
& \quad
+
4^{q-1}(s-t)^{\frac{q-2}{2}}(4q)^q
\left(
\int_t^s
\bE\left[
\big\|
\sigma\big(\cX^{t,x,N}_{\kappa_N(r)}\big)
-
\sigma\big(X^{t,x}_{\kappa_N(r)}\big)
\big\|_F^q
\right]
dr
+
\int_t^s
\bE\left[
\big\|
\sigma\big(X^{t,x}_{\kappa_N(r)}\big)
-
\sigma\big(X^{t,x}_r\big)
\big\|_F^q
\right]
dr
\right)
\nonumber\\
&
\leq
4^{q-1}T^{\frac{q-2}{2}}L^{q/2}[T^{q/2}+(4q)^q]
\left(
\int_t^s
\bE\bigg[
\sup_{u\in[t,r]}
\big\|
\cX^{t,x,N}_u-X^{t,x}_u
\big\|^q
\bigg]
\,dr
+
\int_t^s
\bE\left[
\big\|
X^{t,x}_{\kappa_N(r)}-X^{t,x}_r
\big\|^q
\right]
dr
\right)
\nonumber\\
&
\leq
4^{q-1}T^{\frac{q-2}{2}}L^{q/2}[T^{q/2}+(4q)^q]
\int_t^s
\bE\bigg[
\sup_{u\in[t,r]}
\big\|
\cX^{t,x,N}_u-X^{t,x}_u
\big\|^q
\bigg]
\,dr
\nonumber\\
& \quad
+
16^{q-1}(TL)^q[T^{q/2}+(4q)^q]^2\Big(\frac{T-t}{N}\Big)^{q/2}
\left[
C_{q,1}e^{\rho_{q,1}(s-t)}(d^p+\|x\|^2)
\right]^{q/2}.
\label{est euler 1}
\end{align}
Moreover, it is well-known (see, e.g., Theorems 4.2 in \cite{Mao2007}, 
and Lemma 1.2 in \cite{GR2011})
for all $(t,x)\in[0,T]\times\bR^d$, $N\in\bN$, and $q\in[2,\infty)$ that
\begin{equation}
\label{q finite X Y N}
\bE\bigg[\sup_{u\in[t,T]}\big\|X^{t,x}_u\big\|^q\bigg]
+
\bE\bigg[\sup_{u\in[t,T]}\big\|\cX^{t,x,N}_u\big\|^q\bigg]
<\infty.
\end{equation}
Then by \eqref{q finite X Y N} and \eqref{est euler 1},
the application of Gr\"onwall's lemma implies
for all $(t,x)\in[0,T]\times\bR^d$,
$s\in[t,T]$, $N\in\bN$, and $q\in[2,\infty)$ that
\begin{align*}
\bE\bigg[
\sup_{u\in[t,s]}
\big\|
\cX^{t,x,N}_u-X^{t,x}_u
\big\|^q
\bigg]
\leq
&
16^{q-1}(LT)^q[T^{q/2}+(4q)^q]^2\Big(\frac{T-t}{N}\Big)^{q/2}
\left[
C_{q,1}e^{\rho_{q,1}(s-t)}(d^p+\|x\|^2)
\right]^{q/2}
\\
&
\cdot
\exp\big\{
4^{q-1}T^{\frac{q-2}{2}}L^{q/2}[T^{q/2}+(4q)^q]T
\big\}.
\end{align*}
This proves \eqref{L q error Euler SDE},
which completes the proof of this lemma.
\end{proof}

\begin{lemma}
Let Assumption \ref{assumption Lip and growth} hold.
For each $(t,x)\in[0,T]\times\bR^d$ and $N\in\bN$, 
let $(X^{t,x}_s)_{s\in[t,T]}$ and $(\cX^{t,x,N}_s)_{s\in[t,T]}$ be the 
stochastic processes defined by \eqref{SDE} and \eqref{Euler 1}, respectively.
Then it holds for all $(t,x)\in[0,T]\times\bR^d$, $s\in[t,T]$, $s'\in[s,T]$,
$N\in\bN$, and $q\in[2,\infty)$ that
\begin{equation}
\label{est X t x s s'}
\bE\left[\big\|X^{t,x}_{s'}-X^{t,x}_s\big\|^q\right]
\leq
4^{q-1}L^{q/2}[T^{q/2}+(4q)^q](s'-s)^{q/2}
\left[
C_{q,1}e^{\rho_{q,1}T}(d^p+\|x\|^2)
\right]^{q/2},
\end{equation}
and
\begin{equation}
\label{est Y t x N s s'}
\bE\left[\big\|\cX^{t,x,N}_{s'}-\cX^{t,x,N}_s\big\|^q\right]
\leq
4^{q-1}L^{q/2}[T^{q/2}+(4q)^q](s'-s)^{q/2}
\left[
C_{q,1}e^{\rho_{q,1}T}(d^p+\|x\|^2)
\right]^{q/2}.
\end{equation}
where $C_{q,1}$ and $\rho_{q,1}$ are the constants defined by \eqref{def C q 1}
and \eqref{def rho q 1}, respectively.
\end{lemma}

\begin{proof}
By \eqref{SDE}, \eqref{assumption Lip mu sigma}, \eqref{assumption growth},
\eqref{q moment est SDE}, \eqref{L q est Euler SDE},  
Burkholder-Davis-Gundy inequality, and Jensen's inequality
we first notice for all $(t,x)\in[0,T]\times\bR^d$, $s\in[t,T]$, $s'\in[s,T]$, 
$N\in\bN$, and $q\in[2,\infty)$ that
\begin{align}
&
\bE\left[\big\|X^{t,x}_{s'}-X^{t,x}_s\big\|^q\right]
\nonumber\\
&
\leq
2^{q-1}\bE\left[
\bigg\|
\int_s^{s'}\mu(X^{t,x}_r)\,dr
\bigg\|^q
\right]
+
2^{q-1}\bE\left[
\bigg\|
\int_s^{s'}\sigma(X^{t,x}_r)\,dW_r
\bigg\|^q
\right]
\nonumber\\
&
\leq
2^{q-1}(s'-s)^{q-1}
\int_s^{s'}
\bE\left[
\|
\mu(X^{t,x}_r)
\|^q
\right]
dr
+
2^{q-1}(4q)^q
\left(
\int_s^{s'}
\bE\left[
\|
\sigma(X^{t,x}_r)
\|_F^2
\right]
dr
\right)^{q/2}
\nonumber\\
&
\leq
2^{q-1}(s'-s)^{q-1}
\int_s^{s'}
\bE\left[
\|
\mu(X^{t,x}_r)
\|^q
\right]
dr
+
2^{q-1}(s'-s)^{\frac{q-2}{2}}(4q)^q
\int_s^{s'}
\bE\left[
\|\sigma(X^{t,x}_r)\|_F^q
\right]dr
\nonumber\\
&
\leq
4^{q-1}(s'-s)^{q-1}
\left(
\int_s^{s'}
\bE\left[
\|
\mu(X^{t,x}_r)-\mu(0)
\|^q
\right]
dr
+
\int_s^{s'}
\|\mu(0)\|^q
\,dr
\right)
\nonumber\\
& \quad
+4^{q-1}(s'-s)^{\frac{q-2}{2}}(4q)^q
\left(
\int_s^{s'}
\bE\left[
\|
\sigma(X^{t,x}_r)-\sigma(0)
\|^q
\right]
dr
+
\int_s^{s'}
\|\sigma(0)\|^q
\,dr
\right)
\nonumber\\
& 
\leq
4^{q-1}(s'-s)^{\frac{q-2}{2}}[T^{q/2}+(4q)^q]L^{q/2}
\int_s^{s'}
\bE\left[\|X^{t,x}_r\|^q\right]
dr
+
4^{q-1}(s'-s)^{q/2}[T^{q/2}+(4q)^q](Ld^p)^{q/2}
\nonumber\\
&
\leq
4^{q-1}L^{q/2}(s'-s)^{q/2}[T^{q/2}+(4q)^q]
\bE\bigg[
\sup_{r\in[s,s']}(d^p+\|X^{t,x}_r\|^2)^{q/2}
\bigg]
\nonumber\\
&
\leq
4^{q-1}L^{q/2}[T^{q/2}+(4q)^q](s'-s)^{q/2}
\left[
C_{q,1}e^{\rho_{q,1}T}(d^p+\|x\|^2)
\right]^{q/2},
\nonumber
\end{align}
and
$$
\bE\left[\big\|\cX^{t,x,N}_{s'}-\cX^{t,x,N}_s\big\|^q\right]
\leq
4^{q-1}L^{q/2}[T^{q/2}+(4q)^q](s'-s)^{q/2}
\left[
C_{q,1}e^{\rho_{q,1}T}(d^p+\|x\|^2)
\right]^{q/2}.
$$
The proof of this lemma is hence completed.
\end{proof}

\begin{corollary}                                         \label{corollary prob 1}
For each $(t,x)\in[0,T]\times\bR^d$ and $s\in[t,T]$,
let $(t_k,s_k,x_k)_{k=1}^\infty$ be a sequence such that 
$(t_k,s_k,x_k)\in\Lambda\times\bR^d$ for all $k\in\bN$ 
and $\lim_{k\to\infty}(t_k,s_k,x_k)=(t,s,x)$.
Then for all $\varepsilon>0$, $(t,x)\in[0,T]\times\bR^d$, $s\in[t,T]$, 
and $N\in\bN$ it holds that
\begin{equation}                                         \label{conv in prob 1}
\lim_{k\to\infty}
\mathbb{P}\left(\left\|X^{t,x}_s-X^{t_k,x_k}_{s_k}\right\|\geq \varepsilon\right)=0,
\end{equation}
and
\begin{equation}                                         \label{conv in prob Y N}
\lim_{k\to\infty}
\mathbb{P}
\left(\left\|\cX^{t,x,N}_s-\cX^{t_k,x_k,N}_{s_k}\right\|\geq \varepsilon\right)
=0.
\end{equation}
\end{corollary}

\begin{proof}
Fix $\varepsilon\in(0,\infty)$, $(t,x)\in[0,T]\times\bR^d$, $s\in[t,T]$ 
and a sequence 
$(t_k,s_k,x_k)_{k=1}^\infty$ satisfying that 
$(t_k,s_k,x_k)\in\Lambda\times\bR^d$ for all $k\in\bN$, 
and $\lim_{k\to\infty}(t_k,s_k,x_k)=(t,s,x)$. By Chebyshev inequality and
the fact that $(a+b)^2\leq2a^2+2b^2$ for all $a,b\in\bR$, we have
\begin{align*}
\bP\left(\left\|X^{t,x}_s-X^{t_k,x_k}_{s_k}\right\|\geq \varepsilon\right)
&
\leq \varepsilon^{-2}\bE\left[\left\|X^{t,x}_s-X^{t_k,x_k}_{s_k}\right\|^2\right]
\\
&
\leq 2\varepsilon^{-2}\bE\left[\left\|X^{t,x}_s-X^{t,x}_{s_k}\right\|^2\right]
+2\varepsilon^{-2}
\bE\left[\left\|X^{t,x}_{s_k}-X^{t_k,x_k}_{s_k}\right\|^2\right].
\end{align*}
Therefore, by \eqref{L q continuity SDE} and \eqref{est X t x s s'} 
we have that
\begin{align}
\bP\left(\left\|X^{t,x}_s-X^{t_k,x_k}_{s_k}\right\|\geq \varepsilon\right) \leq 
&
8\varepsilon^{-2}L(T+64)C_{2,1}e^{\rho_{2,1}T}|s_k-s|(d^p+\|x\|^2)
\nonumber\\
& 
+2\varepsilon^{-2}C_{2,2}
\left[
(1+T)C_{2,1}Ld^p|t_k-t|(d^p+\|x\|^2)+\|x-y\|^2
\right].
\label{conv in prob 0}
\end{align}
Finally, passing limit as $k\to\infty$ 
in \eqref{conv in prob 0} gives \eqref{conv in prob 1}.
Analogously, by \eqref{est Y t x N s s'} and 
e.g., Lemma 3.10 in \cite{NW2022} we obtain \eqref{conv in prob Y N}.
Therefore, the proof of this lemma is completed. 
\end{proof}

\begin{lemma}
Let Assumptions \ref{assumption Lip and growth} and \ref{assumption gradient} hold.
For every $(t,x)\in[0,T]\times\bR^d$, $N\in\bN$, and $k\in\{1,2,\dots,d\}$,
let $\big(\frac{\partial}{\partial x_k} X^{t,x}_s\big)_{s\in[t,T]}$ be the
stochastic process defined in \eqref{SDE Classical Derivative},
and let $\big(\cD_{x_k}\cX^{t,x,N}_s\big)_{s\in[t,T]}$ be the
stochastic process defined in \eqref{def proc cD}.
Then it holds for all $(t,x)\in[0,T]\times\bR^d$, $N\in\bN$, 
$k\in\{1,2,\dots,d\}$, and $q\in[2,\infty)$ that
\begin{equation}                                   \label{moment est partial euler q}
\bE\left[\sup_{s\in[t,T]}\big\|
\cD_{x_k}\cX^{t,x,N}_s\big\|^q\right]
\leq C_{d,q ,0},
\end{equation}
and
\begin{equation}
\label{error D partial q}
\bE\left[\sup_{s\in[t,T]}\Big\|
\cD_{x_k}\cX^{t,x,N}_s
-
\frac{\partial}{\partial x_k}X^{t,x}_s
\Big\|^q\right]
\leq
\big[e^{Ld}K_{d,q,2}N^{-1}(d^p+\|x\|^2)\big]^{q/2},
\end{equation}
where 
$C_{d,q,0}$ is the positive constant defined by \eqref{def C q 0}, and
\begin{align}
K_{d,q,2}:=
&
\exp\Big\{
2^{q-1}\big(T^{\frac{q-1}{q}}+4qT^{\frac{q-2}{2q}}\big)T
\Big\}
\cdot
2T^{q-2}[T+(4q)^q]4^{q-1}T
\nonumber\\
&
\cdot
\Big[
C_{d,2q,0}^{1/2}(L_0K_{2q,1})^{q/2}
+(L_0LTC_{2q,1}e^{\rho_{2q,1}}T)^{q/2}
C_{d,2q,0}^{1/2}2^{2q-1}(T^q+(8q)^{2q})^{1/2}
\nonumber\\
&
+
(Ld)^qT^{q/2}2^{q-1}\big[C_{d,2,0}^{q/2}(T+64)^{q/2}C_{d,q,0}(T^{q/2}+(4q)^q)\big]
\Big],
\label{def K d q 2}
\end{align}
with $C_{2q,1}$ being the constant defined in \eqref{def C q 1},
and $C_{d,2q,0}$, $C_{d,2,0}$ being the constants defined
in \eqref{def C q 0}.
\end{lemma}

\begin{proof}
By Assumptions \ref{assumption Lip and growth} and \ref{assumption gradient},
and analogous arguments as in the proof of lemma \ref{lemma SDE partial est},
we obtain \eqref{moment est partial euler q}.
Throughout the rest of the proof of this lemma, we fix
$(t,x)\in[0,T]\times\bR^d$, $N\in\bN$, $k\in\{1,2,\dots,d\}$, and $q\in[2,\infty)$.
By \eqref{def proc cD}, \eqref{SDE Classical Derivative}, Jensen's inequality,
and Burkholder-Davis-Gundy inequality we have for all $s\in[t,T]$ that
\begin{align}
\left(
\bE\left[
\sup_{u\in[t,s]}
\Big\|
\cD_{x_k}\cX^{t,x,N}_u-\frac{\partial}{\partial x_k}X^{t,x}_u
\Big\|^q
\right]
\right)^{1/q}
\leq
(s-t)^{\frac{q-1}{q}}A_1(s)+4q(s-t)^{\frac{q-2}{2q}}A_2(s),
\label{ineq derivative 0}
\end{align}
where
$$
A_1(s):=
\left(
\int_t^s
\bE\left[
\Big\|
(\nabla\mu)(\cX^{t,x,N}_{\kappa_N(r)})\cD_{x_k}\cX^{t,x,N}_{\kappa_N(r)}
-
(\nabla\mu)(X^{t,x}_r)\frac{\partial}{\partial x_k}X^{t,x}_r
\Big\|^q
\right]
dr
\right)^{1/q},
$$
and
$$
A_2(s):=
\left(
\int_s^t
\left(
\bE\left[
\sum_{j=1}^d
\Big\|
(\nabla\sigma^j)(\cX^{t,x,N}_{\kappa_N(r)})\cD_{x_k}\cX^{t,x,N}_{\kappa_N(r)}
-
(\nabla\sigma^j)(X^{t,x}_r)\frac{\partial}{\partial x_k}X^{t,x}_r
\Big\|^2
\right]
\right)^{q/2}
dr
\right)^{1/q}.
$$
Furthermore, Minkowski inequality ensures for all $s\in[t,T]$ that
\begin{equation}
\label{A 2 s ineq}
A_2(s)\leq\sum_{i=1}^4A_{2,i}(s),
\end{equation}
where
\begin{align*}
&
A_{2,1}(s)
:=
\left(
\int_t^s
\left(
\bE\left[
\sum_{j=1}^d
\Big\|
(\nabla\sigma^j)(\cX^{t,x,N}_{\kappa_N(r)})\cD_{x_k}\cX^{t,x,N}_{\kappa_N(r)}
-
(\nabla\sigma^j)(X^{t,x,N}_{\kappa_N(r)})\cD_{x_k}\cX^{t,x,N}_{\kappa_N(r)}
\Big\|^2
\right]
\right)^{q/2}
dr
\right)^{1/q},
\\
&
A_{2,2}(s)
:=
\left(
\int_t^s
\left(
\bE\left[
\sum_{j=1}^d
\Big\|
(\nabla\sigma^j)(X^{t,x,N}_{\kappa_N(r)})\cD_{x_k}\cX^{t,x,N}_{\kappa_N(r)}
-
(\nabla\sigma^j)(X^{t,x,N}_r)\cD_{x_k}\cX^{t,x,N}_{\kappa_N(r)}
\Big\|^2
\right]
\right)^{q/2}
dr
\right)^{1/q},
\\
&
A_{2,3}(s)
:=
\left(
\int_t^s
\left(
\bE\left[
\sum_{j=1}^d
\Big\|
(\nabla\sigma^j)(X^{t,x,N}_r)\cD_{x_k}\cX^{t,x,N}_{\kappa_N(r)}
-
(\nabla\sigma^j)(X^{t,x,N}_r)\cD_{x_k}\cX^{t,x,N}_r
\Big\|^2
\right]
\right)^{q/2}
dr
\right)^{1/q},
\end{align*}
and
$$
A_{2,4}(s)
:=
\left(
\int_t^s
\left(
\bE\left[
\sum_{j=1}^d
\Big\|
(\nabla\sigma^j)(X^{t,x,N}_r)\cD_{x_k}\cX^{t,x,N}_r
-
(\nabla\sigma^j)(X^{t,x,N}_r)\frac{\partial}{\partial x_k}X^{t,x}_r
\Big\|^2
\right]
\right)^{q/2}
dr
\right)^{1/q}.
$$
By \eqref{gradient mu sigma}, \eqref{L q error Euler SDE}, 
\eqref{moment est partial euler q}, and H\"older's inequality
it holds for all $s\in[t,T]$ that
\begin{align}
A_{2,1}^q(s)
&
\leq
\int_t^s\left(
\bE\left[
\Big(
\sum_{j=1}^d
\big\|
(\nabla\sigma^j)(\cX^{t,x,N}_{\kappa_N(r)})
-
(\nabla\sigma^j)(X^{t,x}_{\kappa_N(r)})
\big\|_F^2
\Big)
\big\|\cD_{x_k}\cX^{t,x,N}_{\kappa_N(r)}\big\|^2
\right]
\right)^{q/2}dr
\nonumber\\
&
\leq
L_0^{q/2}\int_t^s\bE\left[
\big\|\cX^{t,x,N}_{\kappa_N(r)}-X^{t,x}_{\kappa_N(r)}\big\|^q
\big\|\cD_{x_k}\cX^{t,x,N}_{\kappa_N(r)}\big\|^q
\right]dr
\nonumber\\
&
\leq
L_0^{q/2}T
\left(
\bE\left[
\sup_{r\in[t,s]}\big\|\cX^{t,x,N}_r-X^{t,x}_r\big\|^{2q}
\right]
\right)^{1/2}
\left(
\bE\left[
\big\|\cD_{x_k}\cX^{t,x,N}_r\big\|^{2q}
\right]
\right)^{1/2}
\nonumber\\
&
\leq
L_0^{q/2}C_{d,2q,0}^{1/2}T\big[K_{2q,1}N^{-1}(d^p+\|x\|^2)\big]^{q/2}.
\label{A q 2 1 est}
\end{align}
Similarly, by \eqref{gradient mu sigma}, \eqref{est X t x s s'},
\eqref{moment est partial euler q}, and H\"older's inequality
we have for all $s\in[t,T]$ that
\begin{align}
A_{2,2}^q(s)
&
\leq
L_0^{q/2}
\left(
\sup_{r\in[t,s]}
\bE\left[\big\|X^{t,x}_{\kappa_N(r)}-X^{t,x}_r\big\|^{2q}\right]
\right)^{1/2}
\left(
\bE\left[\sup_{t\in[t,s]}\big\|\cD_{x_k}\cX^{t,x,N}_r\big\|^{2q}\right]
\right)^{1/2}
\nonumber\\
&
\leq
L_0^{q/2}C_{d,2q,0}^{1/2}T2^{2q-1}L^{q/2}[T^q+(8q)^{2q}]^{1/2}
\Big(\frac{T-t}{N}\Big)^{q/2}
\big[C_{2q,1}e^{\rho_{2q,1}T}(d^p+\|x\|^2)\big]^{q/2}.
\label{A q 2 2 est}
\end{align}
To obtain an appropriate estimate for $A^q_{2,3}$, by \eqref{bbd partials},
\eqref{def proc cD}, \eqref{moment est partial euler q}, 
Burkholder-Davis-Gundy inequality, and Cauchy-Schwarz inequality 
we have for all $s\in[t,T]$, $s'\in[s,T]$, and $q'\in[2,\infty)$ that
\begin{align}
&
\bE\left[\big\|\cD_{x_k}\cX^{t,x,N}_{s'}-\cD_{x_k}\cX^{t,x,N}_s\big\|^{q'}\right]
\nonumber\\
&
\leq
2^{q'-1}\bE\left[
\bigg\|
\int_s^{s'}(\nabla \mu)(\cX^{t,x,N}_{\kappa_N(r)})\cD_{x_k}\cX^{t,x,N}_{\kappa_N(r)}\,dr
\bigg\|^{q'}
\right]
+
2^{q'-1}\bE\left[
\bigg\|
\sum_{j=1}^d
\int_s^{s'}
(\nabla \sigma^j)(\cX^{t,x,N}_{\kappa_N(r)})\cD_{x_k}\cX^{t,x,N}_{\kappa_N(r)}
\,dW^j_r
\bigg\|^{q'}
\right]
\nonumber\\
&
\leq
[2(s'-s)]^{q'-1}\int_s^{s'}\bE\left[
\big\|(\nabla \mu)(\cX^{t,x,N}_{\kappa_N(r)})\big\|_F^{q'}
\cdot\|\cD_{x_k}\cX^{t,x,N}_{\kappa_N(r)}\|_F^{q'}
\right]dr
\nonumber\\
& \quad
+
2^{q'-1}(4q')^{q'}(s'-s)^{\frac{q'-2}{2}}\int_s^{s'}\bE\bigg[
\bigg(\sum_{j=1}^d
\big\|(\nabla \sigma^j)(\cX^{t,x,N}_{\kappa_N(r)})\big\|_F^{2}\bigg)^{q'/2}
\cdot\|\cD_{x_k}\cX^{t,x,N}_{\kappa_N(r)}\|_F^{q'}
\bigg]\,dr
\nonumber\\
&
\leq
2^{q'-1}T^{q'/2}(Ld)^{q'/2}(s'-s)^{q'/2}
\left(
(s'-s)^{q'/2}
\bE\left[\sup_{r\in[t,T]}\big\|\cD_{x_k}\cX^{t,x,N}_{\kappa_N(r)}\big\|^{q'}\right]
+
(4q')^{q'}
\bE\left[\sup_{r\in[t,T]}\big\|\cD_{x_k}\cX^{t,x,N}_{\kappa_N(r)}\big\|^{q'}\right]
\right)
\nonumber\\
&
\leq
2^{q'-1}(Ld)^{q'/2}C_{d,q',0}[T^{q'/2}+(4q')^{q'}](s'-s)^{q'/2}.
\label{D s s' est}
\end{align}
This together with \eqref{bbd partials} and Cauchy-Schwarz inequality
imply for all $s\in[t,T]$ that
\begin{align}
A_{2,3}^q
&
\leq
\int_t^s\left(
\bE\left[
\bigg(
\sum_{j=1}^d\big\|(\nabla \sigma^j)(X^{t,x}_r)\big\|_F^2
\bigg)
\big\|\cD_{x_k}\cX^{t,x,N}_{\kappa_N(r)}-\cD_{x_k}\cX^{t,x,N}_r\big\|^2
\right]
\right)^{q/2}
dr
\nonumber\\
&
\leq
(Ld)^qT[2C_{d,2,0}(T+64)]^{q/2}\Big(\frac{T-t}{N}\Big)^{q/2}.
\label{A q 2 3 est}
\end{align}
Similarly, by \eqref{bbd partials}, Cauchy-Schwarz inequality, 
and Jensen's inequality we obtain for all $s\in[t,T]$ that
\begin{equation}
\label{A q 2 4 est}
A_{2,4}^q
\leq
(Ld)^{q/2}\int_t^s\bE\left[
\sup_{u\in[t,r]}\Big\|\cD_{x_k}\cX^{t,x,N}_r-\frac{\partial}{\partial x_k}X^{t,x}_r\Big\|^q
\right]dr. 
\end{equation}
Combining \eqref{A 2 s ineq}, \eqref{A q 2 1 est}, \eqref{A q 2 2 est}, 
\eqref{A q 2 3 est}, and \eqref{A q 2 4 est} yields for all $s\in[t,T]$ that
\begin{equation}
\label{A q 2 est}
A_2^q(s)
\leq
c_{d,2}N^{-q/2}(d^p+\|x\|^2)^{q/2}
+
(Ld)^{q/2}\int_t^s
\bE\left[\sup_{u\in[t,r]}
\Big\|\cD_{x_k}\cX^{t,x,N}_r-\frac{\partial}{\partial x_k}X^{t,x}_r\Big\|^q\right]
dr,
\end{equation}
where
\begin{align*}
c_{d,2}:=
&
4^{q-1}T\big[
C_{d,2q,0}^{1/2}(L_0K_{2q,1})^{q/2}
+(L_0LTC_{2q,1}e^{\rho_{2q,1}}T)^{q/2}
C_{d,2q,0}^{1/2}2^{2q-1}(T^q+(8q)^{2q})^{1/2}
\\
&
+
(Ld)^q(2C_{d,2,0}(T+64)T)^{q/2}
\big].
\end{align*}
Analogously, we also have for all $s\in[t,T]$ that
\begin{equation}
\label{A q 1 est}
A_1^q(s)
\leq
c_{d,1}N^{-q/2}(d^p+\|x\|^2)^{q/2}
+
(Ld)^{q/2}\int_t^s
\bE\left[\sup_{u\in[t,r]}
\Big\|\cD_{x_k}\cX^{t,x,N}_r-\frac{\partial}{\partial x_k}X^{t,x}_r\Big\|^q\right]
dr,
\end{equation}
where
\begin{align*}
c_{d,1}:=
&
4^{q-1}T\big[
C_{d,2q,0}^{1/2}(L_0K_{2q,1})^{q/2}
+(L_0LTC_{2q,1}e^{\rho_{2q,1}}T)^{q/2}
C_{d,2q,0}^{1/2}2^{2q-1}(T^q+(8q)^{2q})^{1/2}
\\
&
+
(Ld)^q2^{q-1}C_{d,q,0}T^{q/2}[T^{q/2}+(4q)^q]
\big].
\end{align*}
Then combining \eqref{ineq derivative 0}, \eqref{A q 2 est}, and \eqref{A q 1 est}
yields for all $s\in[t,T]$ that
\begin{align*}
&
\bE\left[
\sup_{u\in[t,s]}\Big\|
\cD_{x_k}\cX^{t,x,N}_u-\frac{\partial}{\partial x_k}X^{t,x}_u
\Big\|^q
\right]
\\
&
\leq
2^{q-1}\big(T^{\frac{q-1}{q}}c_{d,1}
+4qT^{\frac{q-2}{2q}}c_{d,2}\big)
N^{-q/2}(d^p+\|x\|^2)^{q/2}
\\
& \quad
+
2^{q-1}\big(T^{\frac{q-1}{q}}+4qT^{\frac{q-2}{2q}}\big)(Ld)^{q/2}
\int_t^s\bE\left[
\sup_{u\in[t,r]}\Big\|
\cD_{x_k}\cX^{t,x,N}_u-\frac{\partial}{\partial x_k}X^{t,x}_u
\Big\|^q
\right]dr.
\end{align*}
Hence, by \eqref{L q error Euler SDE} and \eqref{moment est partial euler q}
the application of Gr\"onwall's lemma shows that
\begin{align*}
&
\bE\left[
\sup_{u\in[t,T]}\Big\|
\cD_{x_k}\cX^{t,x,N}_u-\frac{\partial}{\partial x_k}X^{t,x}_u
\Big\|^q
\right]
\\
&
\leq
2^{q-1}T^{q-2}[T+(4q)^q]\big(c_{d,1}+c_{d,2}\big)
N^{-q/2}(d^p+\|x\|^2)^{q/2}
\exp\big\{
2^{q-1}\big(T^{\frac{q-1}{q}}+4qT^{\frac{q-2}{2q}}\big)(Ld)^{q/2}(T-t)
\big\}.
\end{align*}
This proves \eqref{error D partial q}.
The proof of this lemma is hence completed.
\end{proof}

\begin{lemma}
Let Assumptions \ref{assumption Lip and growth}, \ref{assumption ellip}, 
and \ref{assumption gradient} hold. 
For each $(t,x)\in[0,T)\times\bR^d$ and $N\in\bN$, 
let $V^{t,x}=(V^{t,x,k})_{k\in\{1,2,\dots,d\}}:[t,T)\times\bR^d\to \bR^d$
be the stochastic process defined in \eqref{def proc V},
and let $\cV^{t,x,N}=(\cV^{t,x,N,k})_{k\in\{1,2,\dots,d\}}:[t,T)\times\bR^d\to \bR^d$
be the stochastic process defined in \eqref{def euler proc V}.
Then it holds for all $(t,x)\in[0,T)\times\bR^d$, $s\in(t,T]$, $N\in\bN$, 
and $q\in[2,\infty)$ that
\begin{equation}
\label{cV L q est 0}
\bE\left[\big\|\cV^{t,x,N}_s\big\|^q\right]
\leq
(4q)^q\big[d\varepsilon_d^{-1}C_{d,2,0}(s-t)^{-1}\big]^{q/2},
\end{equation}
and
\begin{align}
\bE\left[
\big\|\cV^{t,x,N}_s-V^{t,x}_s\big\|^2
\right]
\leq
C_{d,3}e^{Ld}N^{-1}(s-t)^{-1}(d^p+\|x\|^2),
\label{V cV error 0}
\end{align}
with $C_{d,2,0}$ being the positive constant defined by \eqref{def C q 0} with $q=2$,
and
$$
C_{d,3}:=
d\varepsilon_d^{-1}
\big[
Ld^2C_{d,4,0}^{1/2}K_{4,1}
+8L^2d\varepsilon_d^{-1}C_{d,4,0}^{1/2}(T+2^8)C_{4,1}e^{\rho_{4,1}T}
+2LdC_{2,2}(T+64)+K_{d,2,2}
\big].
$$
Here, $C_{d,4,0}$, $C_{4,1}$, $K_{4,1}$, and $\rho_{4,1}$ are the constants
defined by \eqref{def C q 0}, \eqref{def C q 1}, \eqref{def K q 1}, and
\eqref{def rho q 1}, respectively, with $q=4$,
and $C_{2,2}$, $K_{d,2,2}$ are the constants defined by 
\eqref{def C q 2} and \eqref{def K d q 2}, respectively, with $q=2$.
\end{lemma}

\begin{proof}
For notational convenience, we fix $(t,x)\in[0,T)\times\bR^d$, $s\in(t,T]$,
$N\in\bN$, and $q\in[2,\infty)$ throughout the proof of this lemma.
Then by \eqref{sigma inverse est}, \eqref{def euler proc V}, 
\eqref{moment est partial euler q}, and Burkholder-Davis-Gundy inequality
we notice for all $k\in\{1,2,\dots,d\}$ that
\begin{align*}
\bE\left[\big\|\cV^{t,x,N,k}_s\big\|^q\right]
&
\leq
\frac{(4q)^q}{(s-t)^q}
\left(
\bE\left[
\int_t^s
\big\|
\sigma^{-1}(\cX^{t,x,N}_{\kappa_N(r)})\cD_{x_k}\cX^{t,x,N}_{\kappa_N(r)}
\big\|_F^2
\,dr
\right]
\right)^{q/2}
\\
&
\leq
\frac{(4q)^q\varepsilon_d^{-q/2}}{(s-t)^q}
\left(
\int_t^s\bE\left[\big\|\cD_{x_k}\cX^{t,x,N}_{\kappa_N(r)}\big\|^2\right]dr
\right)^{q/2}
\\
&
\leq
(4q)^q\varepsilon_d^{-q/2}C_{d,2,0}^{q/2}(s-t)^{-q/2}.
\end{align*}
This proves \eqref{cV L q est 0}.
Next, by \eqref{def proc V}, \eqref{def euler proc V}, and It\^o's isometry
we have for all $k\in\{1,2,\dots,d\}$ that
\begin{align}
\bE\left[
\big\|\cV^{t,x,N}_s-V^{t,x}_s\big\|^2
\right]
&
=
\frac{1}{(s-t)^2}
\int_t^s\bE\left[
\Big\|
\sigma^{-1}(\cX^{t,x,N,k}_{\kappa_N(r)})\cD_{x_k}\cX^{t,x,N}_{\kappa_N(r)}
-
\sigma^{-1}(X^{t,x}_{r})\frac{\partial}{\partial x_k}X^{t,x}_r
\Big\|^2
\right]dr
\nonumber\\
&
\leq
4\sum_{i=1}^4A_{k,i},
\label{V cV ineq}
\end{align}
where
\begin{align*}
&
A_{k,1}:=
\frac{1}{(s-t)^2}
\int_t^s\bE\left[
\Big\|
\big[
\sigma^{-1}(\cX^{t,x,N,k}_{\kappa_N(r)})-\sigma^{-1}(X^{t,x}_{\kappa_N(r)})
\big]
\cD_{x_k}\cX^{t,x,N}_{\kappa_N(r)}
\Big\|^2
\right]dr,
\\
&
A_{k,2}:=
\frac{1}{(s-t)^2}
\int_t^s\bE\left[
\Big\|
\big[
\sigma^{-1}(X^{t,x}_{\kappa_N(r)})-\sigma^{-1}(X^{t,x}_{r})
\big]
\cD_{x_k}\cX^{t,x,N}_{\kappa_N(r)}
\Big\|^2
\right]dr,
\\
&
A_{k,3}:=
\frac{1}{(s-t)^2}
\int_t^s\bE\left[
\Big\|
\sigma^{-1}(X^{t,x}_{r})
\big[\cD_{x_k}\cX^{t,x,N}_{\kappa_N(r)}-\cD_{x_k}\cX^{t,x,N}_{r}\big]
\Big\|^2
\right]dr,
\\
&
A_{k,4}:=
\frac{1}{(s-t)^2}
\int_t^s\bE\left[
\Big\|
\sigma^{-1}(X^{t,x}_{r})
\Big[
\cD_{x_k}\cX^{t,x,N}_{r}-\frac{\partial}{\partial x_k}X^{t,x}_r
\Big]
\Big\|^2
\right]dr.
\end{align*}
Then by \eqref{bbd partial sigma inv}, \eqref{L q error Euler SDE},
\eqref{moment est partial euler q}, the mean-value theorem, 
and Cauchy Schwarz inequality, we obtain for all $k\in\{1,2,\dots,d\}$ that
\begin{align}
A_{k,1}
&
\leq
\frac{1}{(s-t)^2}\int_t^s\bE\left[
\big\|
\sigma^{-1}(\cX^{t,x,N,k}_{\kappa_N(r)})
-
\sigma^{-1}(X^{t,x}_{\kappa_N(r)})
\big\|_F^2
\cdot
\big\|\cD_{x_k}\cX^{t,x,N}_{\kappa_N(r)}\big\|^2
\right]dr
\nonumber\\
&
\leq
\frac{Ld^3\varepsilon_d^{-2}}{(s-t)^2}
\int_t^s
\left(\bE\left[\big\|
\cX^{t,x,N,k}_{\kappa_N(r)}-X^{t,x}_{\kappa_N(r)}
\big\|^4\right]\right)^{1/2}
\left(\bE\left[\big\|
\cD_{x_k}\cX^{t,x,N}_{\kappa_N(r)}
\big\|^4\right]\right)^{1/2}
dr
\nonumber\\
&
\leq
Ld^3\varepsilon_d^{-2}C_{d,4,0}^{1/2}K_{4,1}N^{-1}(s-t)^{-1}(d^p+\|x\|^2).
\label{cV V A k 1}
\end{align}
Similarly, by \eqref{bbd partial sigma inv}, \eqref{est X t x s s'},
\eqref{moment est partial euler q}, the mean-value theorem,
and Cauchy-Schwarz inequality it holds for all $k\in\{1,2,\dots,d\}$ that
\begin{align}
A_{k,2}
&
\leq
\frac{Ld^3\varepsilon_d^{-2}}{(s-t)^2}
\int_t^s
\left(\bE\left[\big\|
X^{t,x}_{\kappa_N(r)}-X^{t,x}_{r}
\big\|^4\right]\right)^{1/2}
\left(\bE\left[\big\|
\cD_{x_k}\cX^{t,x,N}_{\kappa_N(r)}
\big\|^4\right]\right)^{1/2}
dr
\nonumber\\
&
\leq
8L^2d^3\varepsilon_d^{-2}C_{d,4,0}^{1/2}(T+2^8)TC_{4,1}e^{\rho_{4,1}T}
N^{-1}(s-t)^{-1}(d^p+\|x\|^2).
\label{cV V A k 2}
\end{align}
Furthermore, by \eqref{bbd inverse sigma}, \eqref{D s s' est}, and
Cauchy Schwarz inequality we have for all $k\in\{1,2,\dots,d\}$ that
\begin{align}
A_{k,3}
&
\leq
\frac{d\varepsilon_d^{-1}}{(s-t)^2}
\int_t^s\bE\left[\big\|
\cD_{x_k}\cX^{t,x,N}_{\kappa_N(r)}-\cD_{x_k}\cX^{t,x,N}_r
\big\|^2\right]dr
\nonumber\\
&
\leq
2Ld^2\varepsilon_d^{-1}C_{2,2}(T+64)N^{-1}(s-t)^{-1}(d^p+\|x\|^2).
\label{cV V A k 3}
\end{align}
Analogously, by \eqref{bbd inverse sigma}, \eqref{error D partial q},
and Cauchy Schwarz inequality it holds for all $k\in\{1,2,\dots,d\}$ that
\begin{align}
A_{k,4}
&
\leq
\frac{d\varepsilon_d^{-1}}{(s-t)^2}
\int_t^s
\bE\left[\big\|
\cD_{x_k}\cX^{t,x,N}_r-X^{t,x}_r
\big\|^2\right]dr
\nonumber\\
&
\leq
e^{Ld}d\varepsilon_d^{-1}K_{d,2,2}N^{-1}(s-t)^{-1}(d^p+\|x\|^2).
\label{cV V A k 4}
\end{align}
Then combining \eqref{V cV ineq}, \eqref{cV V A k 1}, \eqref{cV V A k 2},
\eqref{cV V A k 3}, and \eqref{cV V A k 4} yields \eqref{V cV error 0}.
Therefore, the proof of this lemma is completed. 
\end{proof}

\begin{lemma}
Let Assumptions \ref{assumption Lip and growth}, \ref{assumption ellip}, 
and \ref{assumption gradient} hold.
For each $(t,x)\in[0,T)\times\bR^d$ and $N\in\bN$, 
let $\cV^{t,x,N}=(\cV^{t,x,N,k})_{k\in\{1,2,\dots,d\}}:[t,T)\times\bR^d\to \bR^d$
be the stochastic process defined in \eqref{def euler proc V}.
Then there is a constant $\gamma_d$ 
only depending on $d$, $\varepsilon_d$, $L$, $L_0$, $p$, and $T$
satisfying for all $t\in[0,T)$, $t'\in[t,T)$, $s\in(t',T]$, $x,x'\in\bR^d$ 
and $N\in\bN$ that
\begin{align}
\bE\Big[
\big\|\cV^{t,x,N}_s-\cV^{t',x',N}_s\big\|^2
\Big]
\leq
&
\frac{\gamma_d(t'-t)}{(s-t)(s-t')}
+
\frac{\gamma_d\big[(t'-t)(d^p+\|x\|^2)+\|x-x'\|^2\big]}{s-t'}.
\label{L2 cont est proc cV}
\end{align}
\end{lemma}

\begin{proof}
By analogous arguments as in the proof of \eqref{L2 cont est proc V},
we can obtain \eqref{L2 cont est proc cV}.
\end{proof}

\subsection{Some properties of the coefficient functions in PDE \eqref{APIDE}}
\label{section property coe}

\begin{lemma}
Let $d\in\bN$ and $\sigma\in C([0,T]\times\bR^d,\bR^d\times\bR^d)$ satisfy 
Assumptions \ref{assumption Lip and growth}, \ref{assumption ellip},
and \ref{assumption gradient}. 
Then it holds for all $k\in\{1,2,\dots,d\}$ and $x\in\bR^d$ that 
$\frac{\partial}{\partial x_k}\sigma^{-1}(x)$ exists and satisfies
\begin{equation}
\label{bbd partial sigma inv}
\Big\|\frac{\partial}{\partial x_k}\sigma^{-1}(x)\Big\|_F^2
\leq Ld^2\varepsilon_d^{-2}.
\end{equation} 
\end{lemma}

\begin{proof}
We first observe for all $x\in\bR^d$ and $k\in\{1,2,\dots,d\}$ that
\begin{align*}
\frac{\partial}{\partial x_k}\sigma^{-1}(x)
&
=\lim_{\delta\to 0}
\frac{\sigma^{-1}(x+\delta e_k)-\sigma^{-1}(x)}{\delta}
\\
&
=\lim_{\delta\to 0}
\frac
{\sigma^{-1}(x+\delta e_k)\sigma(x)\sigma^{-1}(x)
-\sigma^{-1}(x+\delta e_k)\sigma(x+\delta e_k)\sigma^{-1}(x)}
{\delta}
\\
&
=\lim_{\delta\to 0}
\frac
{\sigma^{-1}(x+\delta e_k)[\sigma(x)-\sigma(x+\delta e_k)]\sigma^{-1}(x)}
{\delta}
\\
&
=
-\sigma^{-1}(x)
\Big[\lim_{\delta\to 0}\frac{\sigma(x+\delta e_k)-\sigma(x)}{\delta}\Big]
\sigma^{-1}(x)
\\
&
=-\sigma^{-1}(x)\Big[\frac{\partial}{\partial x_k}\sigma(x)\Big]\sigma^{-1}(x).
\end{align*}
This together with \eqref{bbd partials} and \eqref{bbd inverse sigma} imply 
for all $x\in\bR^d$ and $k\in\{1,2,\dots,d\}$ that
\begin{align*}
\Big\|\frac{\partial}{\partial x_k}\sigma^{-1}(x)\Big\|_F
\leq 
\|\sigma^{-1}(x)\|_F^2
\Big\|\frac{\partial}{\partial x_k}\sigma(x)\Big\|_F
\leq 
L^{1/2}d\varepsilon_d^{-1}.
\end{align*}
Hence, we obtain \eqref{bbd partial sigma inv}.
\end{proof}

\begin{lemma}                                       \label{lemma sigma n}
Let Assumptions \ref{assumption Lip and growth}, \ref{assumption ellip},
and \ref{assumption gradient} hold.
Then there exists a sequence of matrix-valued functions 
$\sigma^{(n)}:\bR^d\to\bR^d\times\bR^d$, $n\in\bN$, 
such that the following holds.
\begin{enumerate}[(i)]
\item
$\sigma^{(n)}\in C^{3}_b(\bR^d,\bR^{d\times d})$
for all $n\in\bN$.
\item
For each $n\in\bN$, $\sigma^{(n)}(x)=\sigma(x)$ for all $\|x\|\leq n$.
\item
For all $k\in\{1,2,\dots,d\}$ 
and every compact set $\cK\subset\bR^d$
\begin{equation}
\label{conv sigma n}
\lim_{n\to\infty}\sup_{x\in\cK}
\left[
\big\|\sigma^{(n)}(x)-\sigma(x)\big\|_F
+
\Big\|\frac{\partial}{\partial x_k}\sigma^{(n)}(x)
-\frac{\partial}{\partial x_k}\sigma(x)\Big\|_F
\right]
=0.
\end{equation}
\item
There exists a positive constant $C_{(d),1}$ only depending
on $p$, $T$, $L$, $L_0$ , and $d$ satisfying for all 
$n\in\bN$, $k\in\{1,2,\dots,d\}$, and $x,y\in\bR^d$ that
\begin{align}
&                                   
\big\|\sigma^{(n)}(x)-\sigma^{(n)}(y)\big\|_F^2
+
\Big\|\frac{\partial}{\partial x_k}\sigma^{(n)}(x)
-\frac{\partial}{\partial y_k}\sigma^{(n)}(y)\Big\|_F^2
\leq C_{(d),1}\|x-y\|^2,
\label{Lip sigma n}
\\
&
\big\|\sigma^{(n)}(x)\big\|_F^2\leq C_{(d),1}(1+\|x\|^2).
\label{LG sigma n}
\end{align}
\item
For all $n\in\bN$ and $x\in\bR^d$,
$\sigma^{(n)}(x)$ is invertible and satisfies that
\begin{equation}                                           
\label{ellip sigma n}
y^T\sigma^{(n)}(x)\big[\sigma^{(n)}(x)\big]^Ty\geq \varepsilon_d\|y\|^2
\quad \text{for all } y\in\bR^d,
\end{equation}
where $\varepsilon_d$ is the positive constant defined in \eqref{sigma ellip}.
\end{enumerate}

\end{lemma}

\begin{proof}
Throughout this proof, let $\alpha:\bR^d\to[0,1]$ be a smooth function such that
$$
\alpha(x)=
\begin{cases}
0, & \|x\|\leq 1;\\
1, & \|x\|\geq 2.
\end{cases}
$$
For each $n\in\bN$, we define a Borel function $h^{(n)}:\bR^d/\{0\}\to(0,\infty)$ by
$$
h^{(n)}(x):=\left(2n/\|x\|\right)^{\alpha(x/n)},\quad x\in\bR^d/\{0\}.
$$
Then for each $n\in\bN$, define $\sigma^{(n)}$ by
$$
\sigma^{(n)}(x):=
\begin{cases}
\sigma(0), & x=0;\\
\sigma\big(h^n(x)x\big), & x\neq 0.
\end{cases}
$$
Notice that
$$
\sigma^{(n)}(x)=
\begin{cases}
\sigma(x), & \|x\|\leq n;\\
\sigma\Big(2nx/\|x\|\Big), & \|x\|\geq 2n.
\end{cases}
$$
By \eqref{assumption growth}, \eqref{sigma ellip}, 
and the construction of $\sigma^{(n)}$, 
it is easy to see that (ii), (iii), and (v) hold,
and $\sigma^{(n)}\in C^{3}(\bR^d,\bR^d\times\bR^d)$ for all $n\in\N$,
and it holds for all $n\in\N$ and $x,y\in\bR^d$ that 
\begin{equation}
\label{bbd LG sigma n}
\big\|\sigma^{(n)}(x)\big\|_F^2\leq Ld^p\big[1+(4n)^2\big],
\quad
\big\|\sigma^{(n)}(x)\big\|_F^2\leq 4Ld^p(1+\|x\|^2).
\end{equation}
To show that $\sigma^{(n)}$ has bounded derivative up to order $3$, 
we first notice for
all $n\in\N$, $i,k\in\{1,\dots,d\}$, and $x\in\bR^d$ that
\begin{equation}
\label{deri est 1}
\frac{\partial}{\partial x_i}\left(2nx_k/\|x\|\right)
=2n\delta_{ik}/\|x\|-2nx_ix_k/\|x\|^3,
\end{equation}
where $\delta_{ii}=1$ for all $i$, and $\delta_{ij}=0$ for $i\neq j$.
Hence, it holds for all $n\in\N$, $i,k\in\{1,\dots,d\}$, 
and $x\in\bR^d$ with $\|x\|\geq 2n$ that
\begin{equation}
\label{deri est 2}
\Big|
\frac{\partial}{\partial x_i}\left(2nx_k/\|x\|\right)
\Big|
\leq 2n\delta_{ik}/\|x\|+n(x_i^2+x_k^2)/\|x\|^3
\leq 3/2.
\end{equation}
This together with \eqref{bbd partials} and \eqref{deri est 1}
imply for $i\in\{1,\dots,d\}$, $n\in\bN$, 
and $x\in\bR^d$ with $\|x\|\geq 2n$ that
\begin{align}
\Big\|\frac{\partial}{\partial x_i}\sigma^{(n)}(x)\Big\|_F
&
=
\Bigg\|
\sum_{k=1}^d\Big(\frac{\partial}{\partial x_k}\sigma\Big)\left(2nx/\|x\|\right)
\frac{\partial}{\partial x_i}\left(2nx_k/\|x\|\right)
\Bigg\|_F
\nonumber\\
&
=
\Bigg\|
\sum_{k=1}^d\Big(\frac{\partial}{\partial x_k}\sigma\Big)
\left(2nx/\|x\|\right)
\left(2n\delta_{ik}/\|x\|-2nx_ix_k/\|x\|^3\right)
\Bigg\|_F
\label{partial sigma n 1}
\\
&
\leq
\frac{3}{2}
\sum_{k=1}^d
\Bigg\|
\Big(\frac{\partial}{\partial x_k}\sigma\Big)\left(2nx/\|x\|\right)
\Bigg\|_F
\nonumber\\
&
\leq 3L^{1/2}d/2.
\label{partial sigma n 2}
\end{align}
Similarly, \eqref{bbd partials}, \eqref{bbd 2nd partials}, and \eqref{partial sigma n 1} 
imply for all $n\in\bN$, $i,j\in\{1,2,\dots,d\}$, 
and $x\in\bR^d$ with $\|x\|\geq 2n$ that
\begin{align}
&
\Bigg\|
\frac{\partial^2}{\partial x_i\partial x_j}\sigma^{(n)}(x)
\Bigg\|_F
\nonumber\\
&
=
\Bigg\|
\sum_{k=1}^d\sum_{l=1}^d
\Big(\frac{\partial^2}{\partial x_k\partial x_l}\sigma\Big)
\left(2nx/\|x\|\right)
\frac{\partial}{\partial x_j}\left(2nx_l/\|x\|\right)
\left(2n\delta_{ik}/\|x\|-2nx_ix_k/\|x\|^3\right)
\nonumber\\ 
& \quad
+\sum_{k=1}^d\Big(\frac{\partial}{\partial x_k}\sigma\Big)
\left(2nx/\|x\|\right)
\frac{\partial}{\partial x_j}
\left(2n\delta_{ik}/\|x\|-2nx_ix_k/\|x\|^3\right)
\Bigg\|_F
\nonumber\\
&
=\Bigg\|\sum_{k=1}^d\sum_{l=1}^d
\Big(\frac{\partial^2}{\partial x_k\partial x_l}\sigma\Big)
\left(2nx/\|x\|\right)
\left(2n\delta_{jl}/\|x\|-2nx_jx_l/\|x\|^3\right)
\left(2n\delta_{ik}/\|x\|-2nx_ix_k/\|x\|^3\right)
\nonumber\\
& \quad
+
\sum_{k=1}^d\Big(\frac{\partial}{\partial x_k}\sigma\Big)
\left(2nx/\|x\|\right)
\left(
-2n\delta_{ik}x_j/\|x\|^3+6nx_jx_kx_i/\|x\|^5
-2n(\delta_{ij}x_k+x_i\delta_{kj})/\|x\|^3
\right)
\Bigg\|_F
\nonumber\\
&
\leq 
c_1
\sum_{k=1}^d
\sum_{l=1}^d
\left\|
\Big(\frac{\partial^2}{\partial x_k\partial x_l}\sigma\Big)
\left(2nx/\|x\|\right)
\right\|_F
+
c_1
\sum_{k=1}^d
\left\|
\Big(\frac{\partial}{\partial x_k}\sigma\Big)
\left(2nx/\|x\|\right)
\right\|_F
\nonumber\\
&
\leq c_1d(d+1)L_0^{1/2},
\label{partial sigma n 3}
\end{align}
where $c_1$ is a positive constant independent of $n$.
Furthermore, by the construction of $h^{(n)}$ we notice for all $n\in\bN$, 
$x\in\bR^d$ with $n\leq\|x\|\leq 2n$, and $k\in\{1,2,\dots,d\}$ that
\begin{equation}
\label{bbd g n 0}
\big|h^{(n)}(x)\big|\leq 2,
\end{equation}
and
\begin{align}
\Big|\frac{\partial}{\partial x_k}h^{(n)}(x)\Big|
&
=
\Big|
\frac{\partial}{\partial x_k}
\exp\left\{\alpha(x/n)\log^{2n/\|x\|}\right\}
\Big|
\nonumber\\
&
=
\left|
h^{(n)}(x)
\left[
n^{-1}\Big(\frac{\partial}{\partial x_k}\alpha\Big)(x/n)\log^{2n/\|x\|}
+\alpha(x/n)\frac{\|x\|}{2n}\frac{\partial}{\partial x_k}(2n/\|x\|)
\right]
\right|
\nonumber\\
&
=
\left|
h^{(n)}(x)
\left[
n^{-1}\Big(\frac{\partial}{\partial x_k}\alpha\Big)(x/n)\log^{2n/\|x\|}
-\alpha(x/n)\frac{x_k}{\|x\|^2}
\right]
\right|
\label{deri g n}
\\
&
\leq 
h^{(n)}(x)
\left[
n^{-1}\log^2\left(\sup_{y\in\bR^d}\|\nabla \alpha(y)\|\right)+n^{-1}
\right]
\nonumber\\
&
\leq n^{-1}c_2,
\label{bbd g n 1}
\end{align}
where 
$
c_2:=2\left[
\log^2\left(\sup_{y\in\bR^d}\|\nabla \alpha(y)\|\right)+1
\right].
$
Analogously, \eqref{bbd g n 0}--\eqref{bbd g n 1} and the construction of $h^{(n)}$
imply for all $n\in\bN$, $k,l\in\{1,2,\dots,d\}$, 
and $x\in\bR^d$ with $n\leq\|x\|\leq 2n$ that
\begin{align}
\Big|
\frac{\partial^2}{\partial x_k\partial x_l}h^{(n)}(x)
\Big|
&
=
\left|
\frac{\partial}{\partial x_l}
\left(
h^{(n)}(x)
\left[
n^{-1}\Big(\frac{\partial}{\partial x_k}\alpha\Big)(x/n)\log^{2n/\|x\|}
+\alpha(x/n)\frac{x_k}{\|x\|^2}
\right]
\right)
\right|
\nonumber\\
&
\leq
\Big|\frac{\partial}{\partial x_l}h^{(n)}(x)\Big|
\cdot
\left|\left[
n^{-1}\Big(\frac{\partial}{\partial x_k}\alpha\Big)(x/n)\log^{2n/\|x\|}
+\alpha(x/n)\frac{x_k}{\|x\|^2}
\right]\right|
\nonumber\\
& \quad 
+
\big|h^{(n)}(x)\big|
\cdot
\Bigg|\Bigg[
n^{-2}\Big(\frac{\partial^2}{\partial x_k \partial x_l}\alpha\Big)(x/n)
\log^{2n/\|x\|}
+
n^{-1}\Big(\frac{\partial}{\partial x_k}\alpha\Big)(x/n)\frac{x_l}{\|x\|^2}
\nonumber\\
& \quad
+
n^{-1}\Big(\frac{\partial}{\partial x_l}\alpha\Big)(x/n)\frac{x_k}{\|x\|^2}
+
\alpha(x/n)\frac{\delta_{kl}\|x\|^2-2x_kx_l}{\|x\|^4}
\Bigg]\Bigg|
\nonumber\\
&
\leq n^{-2}c_3,
\label{bbd g n 2}
\end{align}
where
$$
c_3:=c_2\left[\log^2\Big(\sup_{y\in\bR^d}\|\nabla\alpha(y)\|\Big)+1\right]^2
+2\left[
\log^2\Big(\sup_{y\in\bR^d}\|\operatorname{Hess}\alpha(y)\|_F\Big)
+2\Big(\sup_{y\in\bR^d}\|\nabla\alpha(y)\|\Big)+3
\right].
$$
Then, by \eqref{bbd partials}, \eqref{bbd g n 0}, \eqref{bbd g n 1}, \eqref{bbd g n 2},
and the construction of $\sigma^{(n)}$ we obtain for all $n\in\bN$, 
$i,j\in\{1,2,\dots,d\}$, and $x\in\bR^d$ with $n\leq\|x\|\leq 2n$ that
\begin{align}
\Big\|\frac{\partial}{\partial x_i}\sigma^{(n)}(x)\Big\|_F
&
=
\left\|
\sum_{k=1}^d\Big(\frac{\partial}{\partial x_k}\sigma\Big)(h^{(n)}(x)x)
\frac{\partial}{\partial x_i}\big(h^{(n)}(x)x_k\big)
\right\|_F
\nonumber\\
&
=
\left\|
\sum_{k=1}^d\Big(\frac{\partial}{\partial x_k}\sigma\Big)(h^{(n)}(x)x)
\left(
x_k\frac{\partial}{\partial x_i}h^{(n)}(x)+h^{(n)}(x)\delta_{ki}
\right)
\right\|_F
\nonumber\\
&
\leq
\sum_{k=1}^d L(|x_k|n^{-1}c_2+2)
\leq 2Ld(c_2+1).
\label{bbd sigma n 1}
\end{align}
Analogously, by \eqref{bbd partials}, \eqref{bbd 2nd partials},
\eqref{bbd g n 0}, \eqref{bbd g n 1}, and \eqref{bbd g n 2} 
it holds for all $n\in\bN$, $i,j\in\{1,2,\dots,d\}$, and $(t,x)\in[0,T]\times\bR^d$
with $n\leq \|x\| \leq 2n$ that
\begin{align}
&
\Big\|\frac{\partial^2}{\partial x_i \partial x_j}
\sigma^{(n)}(x)\Big\|_F
\nonumber\\
&
=
\left\|
\frac{\partial}{\partial x_j}
\left[
\sum_{k=1}^d\Big(\frac{\partial}{\partial x_k}\sigma\Big)(h^{(n)}(x)x)
\Big(
x_k\frac{\partial}{\partial x_i}h^{(n)}(x)+h^{(n)}(x)\delta_{ki}
\Big)
\right]
\right\|_F
\nonumber\\
&
=
\Bigg\|
\sum_{k=1}^d\sum_{l=1}^d
\left[
\Big(\frac{\partial^2}{\partial x_k \partial x_l}\sigma\Big)(h^{(n)}(x)x)
\Big(
x_l\frac{\partial}{\partial x_j}h^{(n)}(x)+h^{(n)}(x)\delta_{lj}
\Big)
\Big(
x_k\frac{\partial}{\partial x_i}h^{(n)}(x)+h^{(n)}(x)\delta_{ki}
\Big)
\right]
\nonumber\\
& \quad
+
\sum_{k=1}^d\Big(\frac{\partial}{\partial x_k}\sigma\Big)(h^{(n)}(x)x)
\Big(
\delta_{kj}\frac{\partial}{\partial x_i}h^{(n)}(x)
+
x_k\frac{\partial^2}{\partial x_i \partial x_j}h^{(n)}(x)
+
\delta_{ki}\frac{\partial}{\partial x_j}h^{(n)}(x)
\Big)
\Bigg\|_F
\nonumber\\
&
\leq 
\sum_{k=1}^d\sum_{l=1}^dL_0(|x_l|n^{-1}c_2+2)(|x_k|n^{-1}c_2+2)
+
\sum_{k=1}^dL_0(2n^{-1}c_2+|x_k|n^{-2}c_3)
\nonumber\\
&
\leq 4d^2L_0(c_2+1)^2+2dL_0(c_2+c_3).
\label{bbd sigma n 2}
\end{align}
In addition, \eqref{bbd partials}, \eqref{bbd 2nd partials}, 
and the construction of $\sigma^{(n)}$ ensure for all
$n\in\bN$, $i,j\in\{1,2,\dots,d\}$, and $x\in\bR^d$ with $\|x\|\leq n$
that
\begin{equation}
\label{bbd sigma n 0}
\Big\|\frac{\partial}{\partial x_i}\sigma^{(n)}(x)\Big\|_F^2\leq L,
\quad
\Big\|\frac{\partial}{\partial x_i \partial x_j}\sigma^{(n)}(x)\Big\|_F^2\leq L_0.
\end{equation}
Combining \eqref{bbd LG sigma n}, \eqref{partial sigma n 2}, \eqref{partial sigma n 3},
\eqref{partial sigma n 1}, \eqref{bbd sigma n 1}, \eqref{bbd sigma n 2},
and \eqref{bbd sigma n 0}, we have that the derivatives of $\sigma^{n}$ up to
order $2$, $n\in\bN$, are bounded.
By analogous calculation, 
we obtain that the third derivatives of $\sigma^{(n)}$, $n\in\bN$, 
are bounded. Hence, it holds for all $n\in\bN$ that 
$\sigma^{(n)}\in C_b^3(\bR^d,\bR^{d\times d})$.
Finally, by \eqref{bbd LG sigma n}, 
\eqref{partial sigma n 2}, \eqref{partial sigma n 3},
\eqref{bbd sigma n 1}, \eqref{bbd sigma n 2}, \eqref{bbd sigma n 0},
and the mean-value theorem we obtain \eqref{Lip sigma n}.
The proof of this lemma is therefore completed.
\end{proof}

\begin{lemma}
\label{lemma mu n}
Let Assumptions \ref{assumption Lip and growth} and \ref{assumption gradient} hold. 
Then there exist functions  
$\mu^{(n)}:\bR^d\to\bR^d$, $n\in\bN$,
such that the following holds.
\begin{enumerate}[(i)]
\item
$\mu^{(n)}\in C^{3}_b(\bR^d,\bR^d)$
for all $n\in\bN$.
\item
For each $n\in\bN$, $\mu^{(n)}(x)=\mu(x)$ for all $\|x\|\leq n$.
\item
For all $k\in\{1,2,\dots,d\}$ 
and every compact set $\cK\subset\bR^d$
\begin{equation}
\label{conv mu n}
\lim_{n\to\infty}\sup_{x\in\cK}
\left[
\big\|\mu^{(n)}(x)-\mu(x)\big\|
+
\Big\|\frac{\partial}{\partial x_k}\mu^{(n)}(x)
-\frac{\partial}{\partial x_k}\mu(x)\Big\|
\right]
=0.
\end{equation}
\item
There exist a positive constant $C_{(d),2}$ only depending on $p$, $T$,
$L$, $L_0$, and $d$
satisfying for all $n\in\bN$, $k\in\{1,2,\dots,d\}$, and $x,y\in\bR^d$ that
\begin{equation}  
\label{Lip mu n}                                
\big\|\mu^{(n)}(x)-\mu^{(n)}(y)\big\|^2
+
\Big\|\frac{\partial}{\partial x_k}\mu^{(n)}(x)
-\frac{\partial}{\partial y_k}\mu^{(n)}(y)\Big\|^2
\leq C_{(d),2}\|x-y\|^2,
\end{equation}
and
\begin{equation}
\label{LG mu n}
\big\|\mu^{(n)}(x)\big\|^2\leq C_{(d),2}(1+\|x\|^2).
\end{equation}
\end{enumerate}
\end{lemma}

\begin{proof}
Using \eqref{assumption Lip mu sigma}, \eqref{assumption growth},
and \eqref{bbd partials}--\eqref{bbd 2nd partials},
one can follow the proof of Lemma \ref{lemma sigma n} step by step
to get the desired results.
\end{proof}

\section{\textbf{Existence of the viscosity solutions of semilinear PDEs}}
\label{appendix existence PDE}
In this section, we study the existence of viscosity solutions 
for semilinear PDEs in the form of \eqref{APIDE}, 
and provide the proof of Proposition \ref{theorem PDE existence}.
We assume the settings in Section \ref{section setting}, 
fix $d$ and $\theta$, and omit the superscript $d$ and $\theta$ in the notations.

To show the construction of the viscosity solution to PDE \eqref{APIDE}.
We will start with the case that 
the functions $\mu$, $\sigma$, $g$, and $f$ are bounded and regular enough,
and prove that there exists an unique classical solution to PDE \eqref{APIDE}
which has a probabilistic representation of Feynman-Kac and Bismut-Elworthy-Li type.
Then we will use an analogous approximation procedure 
as in Section 2.5 in \cite{beck2021nonlinear}
to obtain the existence of the viscosity solution to PDE \eqref{APIDE},
and establish the probabilistic representation of the viscosity solution
under the settings in Section \ref{section setting}
(cf. Proposition \ref{theorem PDE existence}). 

\begin{lemma}                            \label{lemma smooth PDE}
Let Assumptions \ref{assumption Lip and growth}, \ref{assumption ellip},
and \ref{assumption gradient} hold. 
Moreover, assume that $\mu\in C^3_b(\bR^d,\bR^d)$, 
$\sigma\in C^3_b(\bR^d,\bR^d\times\bR^d)$,
$g\in C^3_b(\bR^d,\bR)$, 
and $f(t,\cdot,\cdot,\cdot)\in C^3_b(\bR^d\times\bR\times\bR^d,\bR)$ 
for all $t\in[0,T]$. Then there exists a classical solution
$u\in C^{1,2}([0,T]\times\bR^d,\bR)$ of PDE \eqref{APIDE} satisfying for all
$(t,x)\in[0,T)\times\bR^d$ that
\begin{align}
&
(u(t,x),\nabla_x u(t,x))
\nonumber\\
&
=\bE\left[g(X^{t,x}_T)\left(1,\frac{1}{T-t}
\int_t^T\left(\big[\sigma(X^{t,x}_{r})\big]^{-1}
DX^{t,x}_{r}\right)^T\,dW_r\right)\right]  
\nonumber\\
& \quad    
+\int_t^T\bE\left[
f\big(s,X^{t,x}_s,u(s,X^{t,x}_s),(\nabla_x u)(s,X^{t,x}_s)\big)
\left(1,\frac{1}{s-t}
\int_t^s\left(\big[\sigma(X^{t,x}_{r})\big]^{-1}
DX^{t,x}_{r}\right)^T\,dW_r\right)
\right]ds.
\label{Feynman-Kac BEL smooth}
\end{align}
\end{lemma}

\begin{proof}
First note that Theorem 3.2 in \cite{PP1992} ensures that PDE \eqref{APIDE} 
has a classical solution $u\in C^{1,2}([0,T]\times\bR^d,\bR)$.
Then by Feynman-Kac formula 
(see, e.g., Theorem 7.6 in \cite{KS1991}, and Theorem 17.4.10 in \cite{CE}), 
we obtain for all $(t,x)\in[0,T]\times\bR^d$ that
\begin{equation}                                                   \label{FK smooth}
u(t,x)=\bE\big[g(X^{t,x}_T)\big]+\int_t^T\bE\big
[f\big(s,X^{t,x}_s,u(s,X^{t,x}_s),
(\nabla_x u)(s,X^{t,x}_s)\big)\big]\,ds.
\end{equation}
Furthermore, by Lemma \ref{lemma L2 derivative} 
the application of Bismut-Elworthy-Li formula 
(see, e.g., Theorem 2.1 in \cite{PZ1997}, Theorem 2.1 in \cite{EL1994},
and Proposition 3.2 in \cite{FLLLT1999})
to \eqref{FK smooth} yields for all $(t,x)\in[0,T)\times\bR^d$ that
\begin{align}
\nabla_x u(t,x)
&
=\bE\left[\frac{g(X^{t,x}_T)}{T-t}
\int_t^T\left(\big[\sigma(X^{t,x}_{r})\big]^{-1}DX^{t,x}_{r}
\right)^T\,dW_r\right]
\nonumber\\
& \quad                                        
+\int_t^T \bE\bigg[
f\big(s,X_s^{t,x},u(s,X_s^{t,x}),
(\nabla_x u)(s,X^{t,x}_s)\big)
\frac{1}{s-t}
\int_t^s\left(\big[\sigma(X^{t,x}_{r})\big]^{-1}DX^{t,x}_{r}\right)^T
\,dW_r\bigg]\,ds.
\nonumber
\end{align}
The proof of this lemma is thus completed. 
\end{proof}

Now we present some lemmas which will be used later on for the approximation
of the viscosity solution of PDE \eqref{APIDE}.

\begin{lemma}                                  \label{lemma continuity G}
Let Assumption \ref{assumption Lip and growth} hold. Then the mapping
\begin{equation}
\label{G mapping}
(0,T)\times\bR^d\times\bR\times\bR^d\times\bS^d\ni(t,x,r,y,A)
\mapsto G(t,x,r,y,A)\in\bR
\end{equation}
is continuous.
\end{lemma}
\begin{proof}
The structure of $G$ and the assumptions $\mu\in C(\bR^d,\bR^d)$, 
$\sigma\in C(\bR^d,\bR^d\times\bR^d)$,
and $f\in C([0,T]\times\bR^d,\bR^{2d+1})$ implies
that mapping \eqref{G mapping} is continuous.
\end{proof}
To introduce the next lemmas, 
for every $n\in\bN$ let $g^{(n)}\in C(\bR^d,\bR)$, 
$\mu^{(n)}\in C(\bR^d,\bR^d)$,
and $\sigma^{(n)}\in C(\bR^d,\bR^{d\times d})$.
Then we make the following assumptions.  
\begin{assumption}                            \label{assumption conv compact}
It holds for every non-empty compact set 
$\mathcal{K}\subseteq \bR^d$ 
and every non-empty compact set
$\cK'\subseteq [0,T]\times\bR^d\times\bR\times\bR^d$ that 
\begin{equation}
\label{conv compact}
\lim_{n\to\infty}\left[ 
\sup_{x\in\mathcal{K}}
\Big(\big|g^{(n)}(x)-g(x)\big|
+\big\|\mu^{(n)}(x)-\mu(x)\big\|
+\big\|\sigma^{(n)}(x)-\sigma(x)\big\|_F\Big)
\right]=0,                                             
\end{equation}
and
\begin{equation}
\label{conv compact f}
\lim_{n\to\infty}
\left[
\sup_{(t,x,v,w)\in\cK'}
\big|
f(t,x,v,w)-f^{(n)}(t,x,v,w)
\big|
\right]=0.
\end{equation}
\end{assumption}

\begin{assumption}                           \label{assumption Lip and growth n}
There exists a constant $C_{(d),1}>0$ satisfying 
for all $n\in\bN$, $x,y\in\bR^d$, $t\in[0,T]$, 
$v_1,v_2\in\bR$, and $w_1,w_2\in\bR^d$ that
\begin{align}   
&                        
|g^{(n)}(x)-g^{(n)}(y)|^2+\|\mu^{(n)}(x)-\mu^{(n)}(y)\|^2
+\|\sigma^{(n)}(x)-\sigma^{(n)}(y)\|_F^2
\leq C_{(d),1}\|x-y\|^2,
\label{assumption Lip n}
\\
& 
|f^{(n)}(t,x,v_1,w_1)-f^{(n)}(t,y,v_2,w_2)|\leq C_{(d),1}(\|x-y\|+|v_1-v_2|+\|w_1-w_2\|)
\label{assumption Lip f n}
\end{align}
and
\begin{equation}                                     \label{assumption growth n}
|f^{(n)}(t,x,0,\mathbf{0})|^2+|g^{(n)}(x)|^2
+\|\mu^{(n)}(x)\|^2+\|\sigma^{(n)}(x)\|_F^2
\leq C_{(d),1}(d^p+\|x\|^2).
\end{equation}
\end{assumption}
                                    
Then for each $n\in\bN$, let 
$
G^{(n)}:(0,T)\times \bR^d\times\bR\times\bR^d\times\bS^d
\to \bR
$
be a function defined for all
$(t,x,r,y,A)
\in(0,T)\times\bR^d\times\bR\times\bR^d\times\bS^d$ by 
\begin{align*}
G^{(n)}(t,x,r,y,A):=
&
-\langle y,\mu^{(n)}(x)\rangle
-\frac{1}{2}\operatorname{Trace}
\big(\sigma^{(n)}(x)[\sigma^{(n)}(x)]^TA\big)
-f^{(n)}(t,x,r,y).
\end{align*}

\begin{lemma}                                      \label{lemma conv compact G}
Let Assumptions \ref{assumption Lip and growth}, \ref{assumption conv compact}
and \ref{assumption Lip and growth n} hold.
Then for every compact set 
$\mathcal{K}\subseteq(0,T)\times\bR^d\times\bR\times\bR^d\times\bS^d$ 
it holds that
\begin{equation}                                          \label{conv compact G}
\lim_{n\to\infty}\left(
\sup_{(t,x,r,y,A)\in\mathcal{K}}
\big|G(t,x,r,y,A)-G^{(n)}(t,x,r,y,A)\big|
\right)=0.
\end{equation}
\end{lemma}

\begin{proof}
Assumptions \ref{assumption Lip and growth}, \ref{assumption conv compact},
and \ref{assumption Lip and growth n}, and the structure of $G$
ensure \eqref{conv compact G}.
\end{proof}

\begin{lemma}                                       \label{v solution approximation}
Let Assumptions \ref{assumption Lip and growth}, \ref{assumption conv compact},    
and \ref{assumption Lip and growth n} hold.
For every $n\in\bN$, let $u\in C_{lin}((0,T)\times\bR^d)$, 
and $u^{(n)}\in C_{lin}((0,T)\times\bR^d)$ 
with $u(T,\cdot)=g$ and $u^{(n)}(T,\cdot)=g^{(n)}$.
Assume for all non-empty compact sets 
$\mathcal{K}\subseteq(0,T)\times\bR^d$ that 
\begin{equation}                                        \label{conv compact u}
\lim_{n\to\infty}\Bigg[
\sup_{(t,x)\in\mathcal{K}}
\big|u^{(n)}(t,x)-u(t,x)\big|\Bigg]= 0.                                                                                                                    
\end{equation}
Moreover, assume for all $n\in\bN$ that $u^{(n)}$ 
is a viscosity solution of 
\begin{align}
-\frac{\partial}{\partial t}u^{(n)}(t,x)
+G^{(n)}(t,x,u^{(n)}(t,x),\nabla_x u^{(n)}(t,x),\operatorname{Hess}_x
u^{(n)}(t,x),u^{(n)}(t,\cdot))=0
\;\text{on $(0,T)\times \bR^d$}
\label{APIDE n}            
\end{align}
with $u^{(n)}(T,\cdot)=g^{(n)}$.
Then $u$ is a viscosity solution 
of PDE \eqref{APIDE} with $u(T,\cdot)=g$.
\end{lemma}

\begin{proof}
By Lemma \ref{lemma continuity G} and Lemma \ref{lemma conv compact G},
the application of Corollary 2.20 in \cite{beck2021nonlinear} yields 
that $u$ is a viscosity solution of PDE \eqref{APIDE}.
\end{proof}

\begin{proposition}  
\label{proposition PDE existence bounded}                                            
Let Assumptions \ref{assumption Lip and growth}, \ref{assumption ellip},
and \ref{assumption gradient} hold. 
Moreover, assume that $\mu\in C^3_b(\bR^d,\bR^d)$, 
$\sigma\in C^3_b(\bR^d,\bR^d\times\bR^d)$,
$g\in C_b(\bR^d,\bR)$, 
and $f(t,\cdot,\cdot,\cdot)\in C_b(\bR^d\times\bR\times\bR^d,\bR)$ 
for all $t\in[0,T]$.
Then the following holds:
\begin{enumerate}[(i)]
\item
There exists a unique pair of Borel functions $(u,w)$ such that
$u\in C_b([0,T)\times\bR^d,\bR)$, $w\in C([0,T)\times\bR^d,\bR^d)$, 
and
\begin{equation}
\label{growth FP PIDE bbd}
\sup_{(s,y)\in[0,T)\times\bR^d}
\left(\frac{|u(s,y)|+(T-s)^{1/2}\|w(s,y)\|}{(d^p+\|y\|^2)^{1/2}}
\right)<\infty,
\end{equation}
and
\begin{align}
&
(u(t,x),w(t,x))
\nonumber\\
&
=\bE\left[g(X^{t,x}_T)\left(1,\frac{1}{T-t}
\int_t^T\left(\big[\sigma(X^{t,x}_{r})\big]^{-1}
DX^{t,x}_{r}\right)^T\,dW_r\right)\right]  
\nonumber\\
& \quad    
+\int_t^T\bE\left[
f\big(s,X^{t,x}_s,u(s,X^{t,x}_s),w(s,X^{t,x}_s)\big)
\left(1,\frac{1}{s-t}
\int_t^s\left(\big[\sigma(X^{t,x}_{r})\big]^{-1}
DX^{t,x}_{r}\right)^T\,dW_r\right)
\right]ds
\label{Feynman-Kac BEL bbd PDE}
\end{align}
for all $(t,x)\in[0,T)\times\bR^d$.
\item
The function $u:[0,T]\times\bR^d\to\bR$ defined in \eqref{Feynman-Kac BEL bbd PDE} 
with $u(T,\cdot):=g$
is a viscosity solution of PDE \eqref{APIDE}.
\item
For all $(t,x)\in[0,T)\times\bR^d$ the gradient of $u$ exists and satisfies
$\nabla_x u(t,x)=w(t,x)$.
\end{enumerate}
\end{proposition}

To prove the above proposition, we recall the notion of mollifications
of functions. 
For $\varepsilon>0$ and 
locally integrable functions $\phi:\bR^d\to\bR$ we use the notation 
$\phi^{(\varepsilon)}$ for the mollification of $\phi$, defined by 
\begin{equation}                                       \label{def mollification}
\phi^{(\varepsilon)}(x)
:=\varepsilon^{-d}\int_{\bR^d}\phi(y)k((x-y)/\varepsilon)\,dy
=\int_{\bR^d}\phi(x-\varepsilon z)k(z)\,dz,\quad x\in\bR^d , 
\end{equation}
where $k:\bR^d\to\bR$ is a fixed nonnegative smooth function on $\bR^d$ 
such that $k(x)=0$ for $|x|\geq1$, $k(-x)=k(x)$ 
for $x\in\bR^d$, and $\int_{\bR^d}k(x)\,dx=1$.

\begin{proof}[Proof of Proposition \ref{proposition PDE existence bounded}]
First notice that Corollary \ref{corollary FP} ensures (i).  
Let $\{\varepsilon_n\}_{n=1}^\infty$ be a sequence taking values in $(0,1]$ 
such that $\lim_{n\to\infty}\varepsilon_n=0$.
Then for each $n\in\bN$ and $t\in[0,T]$, 
we denote $g^{(\varepsilon_n)}$ and $f^{(\varepsilon_n)}(t,\cdot,\cdot,\cdot)$
the mollifications of $g$ and $f(t,\cdot,\cdot,\cdot)$, respectively.  
By the well-known properties of mollifications, we observe that
\begin{equation}
\label{bbd smooth g f molli}
g^{(\varepsilon_n)}\in C^{\infty}_b(\bR^d,\bR) \quad  \text{and} \quad 
f^{(\varepsilon_n)}(t,\cdot,\cdot,\cdot)
\in C^{\infty}_b(\bR^d\times\bR\times\bR^d,\bR) 
\;\; \text{for all $n\in\bN$ and $t\in[0,T]$}.
\end{equation}
Moreover by \eqref{assumption Lip mu sigma}, \eqref{def mollification}, 
and Jensen's inequality we have
for all $n\in\bN$, $x\in\bR^d$, $t\in[0,T]$, $v\in\bR$, and $w\in\bR^d$ that
\begin{align}
\big\|g^{(\varepsilon_n)}(x)-g(x)\big\|^2
&
=\left\|
\int_{\bR^d}g(x-\varepsilon_ny)k(y)\,dy
-\int_{\bR^d}g(x)k(y)\,dy
\right\|^2\nonumber\\
&
\leq \int_{\bR^d}\big\|
g(x-\varepsilon_ny)-g(x)
\big\|^2k(y)\,dy\leq\varepsilon_n^2L, 
\nonumber                         
\end{align}
and
\begin{equation*}
\big\|f^{(\varepsilon_n)}(t,x,v,w)-f(t,x,v,w)\big\|^2\leq\varepsilon_n^2L.
\end{equation*}
Hence, it holds that
\begin{equation}
\label{conv compact g f molli}
\lim_{n\to\infty}
\left[
\Big(\sup_{x\in\bR^d}\big|g^{(\varepsilon_n)}(x)-g(x)\big|\Big)
+
\Big(\sup_{(t,x,v,w)\in[0,T]\times\bR^d}
\big|f^{(\varepsilon_n)}(t,x,v,w)-f(t,x,v,w)\big|\Big)
\right]
=0.
\end{equation}
Next, for each $n\in\bN$ we consider the PDE 
\begin{align}
\frac{\partial}{\partial t}u^n(t,x)+\langle\nabla_x u^n(t,x),\mu(x)\rangle
&
+\frac{1}{2}\operatorname{Trace}\big(\sigma\sigma^T(x)
\operatorname{Hess}_xu^n(t,x)\big)
\nonumber\\
&
+f^{(\varepsilon_n)}(t,x,u^n(t,x),\nabla_x u^n(t,x))=0
\quad \text{on} \quad (0,T)\times\bR^d
\label{PDE smooth bbd}
\end{align}
with $u^n(T,x)=g^{(\varepsilon_n)}(x)$ for all $x\in \bR^d$.
Notice that \eqref{bbd smooth g f molli} and Lemma \ref{lemma smooth PDE}
ensures for all $n\in\bN$ that \eqref{PDE smooth bbd} has a unique classical
solution in $u^n\in C^{1,2}([0,T]\times\bR^d)$ satisfying 
for all $n\in\bN$ and $(t,x)\in[0,T)\times\bR^d$ that
\begin{align}
&
(u^n(t,x),\nabla_x u^n(t,x))
\nonumber\\
&
=\bE\left[g^{(\varepsilon_n)}(X^{t,x}_T)\left(1,\frac{1}{T-t}
\int_t^T\left(\sigma^{-1}(X^{t,x}_{r})
DX^{t,x}_{r}\right)^T\,dW_r\right)\right]  
\nonumber\\
& \quad    
+\int_t^T\bE\left[
f^{(\varepsilon_n)}\big(s,X^{t,x}_s,u^n(s,X^{t,x}_s),(\nabla_x u^n)(s,X^{t,x}_s)\big)
\left(1,\frac{1}{s-t}
\int_t^s\left(\sigma^{-1}(X^{t,x}_{r})
DX^{t,x}_{r}\right)^T\,dW_r\right)
\right]ds.
\label{Feynman-Kac BEL smooth bbd PDE n}
\end{align}
For each $n\in\bN$ and $s\in[0,T)$, we define $E_n(s)$ by
\begin{equation}
\label{def E n s}
E_n(s):=\sup_{(r,y)\in[s,T)\times\bR^d}\Big[
\big|u^n(r,y)-u(r,y)\big|+(T-r)^{1/2}\big\|(\nabla_y u^n)(r,y)-w(r,y)\big\|
\Big]
\end{equation}
To show the convergence of $E_n(0)$,
by \eqref{assumption Lip f} and \eqref{assumption Lip mu sigma} 
we first observe for all $(t,x)\in[0,T)\times\bR^d$ and $n\in\bN$ that
\begin{equation}
\label{smooth conv 1}
\left|\bE\big[g^{(\varepsilon_n)}(X^{t,x}_T)\big]-\bE\big[g(X^{t,x}_T)\big]\right|
\leq \sup_{y\in\bR^d}\left[\big|g^{(\varepsilon_n)}(y)-g(y)\big|\right],
\end{equation}
and
\begin{align}
&
\left|
\int_t^T\bE\left[
f^{(\varepsilon_n)}\big(s,X^{t,x}_s,u^n(s,X^{t,x}_s),(\nabla_x u^n)(s,X^{t,x}_s)\big)
-f\big(s,X^{t,x}_s,u(s,X^{t,x}_s),w(s,X^{t,x}_s)\big)
\right]
\right|
\nonumber\\
&
\leq
\left|
\int_t^T\bE\left[
f^{(\varepsilon_n)}\big(s,X^{t,x}_s,u^n(s,X^{t,x}_s),(\nabla_x u^n)(s,X^{t,x}_s)\big)
-f\big(s,X^{t,x}_s,u^n(s,X^{t,x}_s),(\nabla_x u^n)(s,X^{t,x}_s)\big)
\right]
\right|
\nonumber\\
& \quad
+\left|
\int_t^T\bE\left[
f\big(s,X^{t,x}_s,u^n(s,X^{t,x}_s),(\nabla_x u^n)(s,X^{t,x}_s)\big)
-f\big(s,X^{t,x}_s,u(s,X^{t,x}_s),w(s,X^{t,x}_s)\big)
\right]
\right|
\nonumber\\
&
\leq
\sup_{(s,v)\in[0,T]\times\bR^{2d+1}}\left[T\big|
f^{(\varepsilon_n)}(s,v)-f(s,v)
\big|\right]
\nonumber\\
& \quad
+\int_t^T
\frac
{L^{\frac{1}{2}}(1+T^{\frac{1}{2}})\bE\left[
\big|u^n(s,X^{t,x}_s)-u(s,X^{t,x}_s)\big|
+(T-s)^{1/2}\big\|(\nabla_x u^n)(s,X^{t,x}_s)-w(s,X^{t,x}_s)\big\|
\right]
}
{(T-s)^{1/2}}
\,ds
\nonumber\\
&
\leq
\sup_{(s,v)\in[0,T]\times\bR^{2d+1}}\left[T\big|
f^{(\varepsilon_n)}(s,v)-f(s,v)
\big|\right]
+L^{\frac{1}{2}}(1+T^{\frac{1}{2}})\int_t^T(T-s)^{-1/2}
E_n(s)\,ds.
\label{smooth conv 2}
\end{align}
Furthermore, by Cauchy-Schwarz, It\^o's isometry, \eqref{moment est partial},
and \eqref{sigma inverse est} we have for all $n\in\bN$, $k\in\{1,2,\dots,d\}$,
and $(t,x)\in[0,T)\times\bR^d$ that
\begin{align}
&
(T-t)^{1/2}
\left|
\bE\left[
\big(g^{(\varepsilon_n)}(X^{t,x}_T)-g(X^{t,x}_T)\big)
\frac{1}{T-t}\int_t^T\left(
\sigma^{-1}(X^{t,x}_r)\frac{\partial}{\partial x_k}X^{t,x}_r
\right)dW_r
\right]
\right|
\nonumber\\
&
\leq \sup_{y\in\bR^d}\left[\big|g^{(\varepsilon_n)}(y)-g(y)\big|\right]
(T-t)^{-1/2}\left(\bE\left[
\int_t^T
\Big\|\sigma^{-1}(X^{t,x}_r)\frac{\partial}{\partial x_k}X^{t,x}_r\Big\|^2dr
\right]\right)^{1/2}
\nonumber\\
&
\leq 
\sup_{y\in\bR^d}\left[\big|g^{(\varepsilon_n)}(y)-g(y)\big|\right]
(T-t)^{-1/2}
\left(\varepsilon_d^{-1}
\int_t^T\bE\left[\Big\|\frac{\partial}{\partial x_k}X^{t,x}_r\Big\|^2\right]dr
\right)^{1/2}
\nonumber\\
&
\leq C_{d,0}^{1/2}\varepsilon_d^{-1/2}
\sup_{y\in\bR^d}\left[\big|g^{(\varepsilon_n)}(y)-g(y)\big|\right].
\label{smooth conv 3}
\end{align}
In addition, by Cauchy-Schwarz, It\^o's isometry, H\"older inequality,
\eqref{assumption Lip f}, \eqref{sigma inverse est}, and \eqref{moment est partial} it holds for all
$(t,x)\in[0,T)\times\bR^d$, $k\in\{1,2,\dots,d\}$, $n\in\bN$, 
and $\beta\in(0,1)$ that
\begin{align}
&
(T-t)^{\frac{1}{2}}\Bigg|\int_t^T\bE
\Bigg[
\Big(
f^{(\varepsilon_n)}
\big(s,X^{t,x}_s,u^n(s,X^{t,x}_s),(\nabla_x u^n)(s,X^{t,x}_s)\big)
-
f
\big(s,X^{t,x}_s,u(s,X^{t,x}_s),w(s,X^{t,x}_s)\big)
\Big)
\nonumber\\
& \quad \cdot
\frac{1}{s-t}\int_t^s\Big(
\sigma^{-1}(X^{t,x}_r)\frac{\partial}{\partial x_k}X^{t,x}_r
\Big)^T\,dW_r
\Bigg]\,
ds\Bigg|
\nonumber\\
&
\leq 
(T-t)^{\frac{1}{2}}
\sup_{(s,v)\in[0,T]\times\bR^{2d+1}}
\Big[\big|f^{(\varepsilon_n)}(s,v)-f(s,v)\big|\Big]
\cdot\int_t^T\frac{1}{s-t}
\bE\left[\bigg|
\int_t^s\Big(
\sigma^{-1}(X^{t,x}_r)\frac{\partial}{\partial x_k}X^{t,x}_r
\Big)^T\,dW_r\bigg|
\right]ds
\nonumber\\
& \quad
+(T-t)^{\frac{1}{2}}
\int_t^T\bE\Bigg[
\left|
f\big(s,X^{t,x}_s,u^n(s,X^{t,x}_s),(\nabla_x u^n)(s,X^{t,x}_s)\big)
-
f\big(s,X^{t,x}_s,u(s,X^{t,x}_s),w(s,X^{t,x}_s)\big)
\right|
\nonumber\\
& \quad
\cdot
\frac{1}{s-t}
\bigg\|
\int_t^s
\Big(\sigma^{-1}(X^{t,x}_s)\frac{\partial}{\partial x_k}X^{t,x}_r\Big)^T\,dW_r
\bigg\|
\Bigg]\,ds
\nonumber\\
&
\leq
(T-t)^{\frac{1}{2}}
\sup_{(s,v)\in[0,T]\times\bR^{2d+1}}
\Big[\big|f^{(\varepsilon_n)}(s,v)-f(s,v)\big|\Big]
\cdot\int_t^T\frac{1}{s-t}
\Bigg(\bE\Bigg[
\int_t^s\Big|
\sigma^{-1}(X^{t,x}_r)\frac{\partial}{\partial x_k}X^{t,x}_r
\Big|^2\,dr
\Bigg]\Bigg)^{1/2}ds
\nonumber\\
& \quad
+(T-t)^{\frac{1}{2}}
\int_t^T\bE\Bigg[
\frac
{
L^{\frac{1}{2}}(1+T^{\frac{1}{2}})\Big(
\big|u^n(s,X^{t,x}_s)-u(s,X^{t,x}_s)\big|
+(T-s)^{\frac{1}{2}}\big\|(\nabla_x u^n)(s,X^{t,x}_s)-w(s,X^{t,x}_s)\big\|
\Big)
}
{(T-s)^\frac{1}{2}}
\nonumber\\
& \quad
\cdot
\frac{1}{s-t}
\left\|
\int_t^s
\Big(\sigma^{-1}(X^{t,x}_s)\frac{\partial}{\partial x_k}X^{t,x}_r\Big)^T\,dW_r
\right\|
\Bigg]\,ds
\nonumber\\
&
\leq
(T-t)^{\frac{1}{2}}
\sup_{(s,v)\in[0,T]\times\bR^{2d+1}}
\Big[\big|f^{(\varepsilon_n)}(s,v)-f(s,v)\big|\Big]
\cdot\int_t^T\frac{1}{s-t}
\Bigg(\bE\Bigg[
\int_t^s\Big\|
\sigma^{-1}(X^{t,x}_r)\frac{\partial}{\partial x_k}X^{t,x}_r
\Big\|^2\,dr
\Bigg]\Bigg)^{1/2}ds
\nonumber\\
& \quad
+(T-t)^{\frac{1}{2}}L^{\frac{1}{2}}(1+T^{\frac{1}{2}})\int_t^T
\frac
{
E_n(s)
}
{(T-s)^{\frac{1}{2}}}
\cdot
\frac{1}{s-t}
\Bigg(\bE\Bigg[
\int_t^s\Big\|
\sigma^{-1}(X^{t,x}_r)\frac{\partial}{\partial x_k}X^{t,x}_r
\Big\|^2\,dr
\Bigg]\Bigg)^{1/2}\,ds
\nonumber\\
&
\leq
(T-t)^{\frac{1}{2}}
\sup_{(s,v)\in[0,T]\times\bR^{2d+1}}
\Big[\big|f^{(\varepsilon_n)}(s,v)-f(s,v)\big|\Big]
\cdot\int_t^T\frac{1}{s-t}
\Bigg(\varepsilon_d^{-1}
\int_t^s \bE\left[\Big\|
\frac{\partial}{\partial x_k}X^{t,x}_r
\Big\|^2
\right]dr\Bigg)^{1/2}ds
\nonumber\\
& \quad
+(T-t)^{\frac{1}{2}}L^{\frac{1}{2}}(1+T^{\frac{1}{2}})\int_t^T
\frac
{
E_n(s)
}
{(T-s)^{\frac{1}{2}}}
\cdot
\frac{1}{s-t}
\Bigg(\varepsilon_d^{-1}
\int_t^s \bE\left[\Big\|
\frac{\partial}{\partial x_k}X^{t,x}_r
\Big\|^2
\right]dr\Bigg)^{1/2}\,ds
\nonumber\\
&
\leq
(T-t)^{1/2}C_{d,0}^{1/2}\varepsilon_d^{-1/2}
\sup_{(s,v)\in[0,T]\times\bR^{2d+1}}
\Big[\big|f^{(\varepsilon_n)}(s,v)-f(s,v)\big|\Big]
\int_t^T(s-t)^{-1/2}\,ds
\nonumber\\
& \quad
+(T-t)^{1/2}(LC_{d,0})^{1/2}(1+T^{1/2})\varepsilon_d^{-1/2}
\int_t^T
(s-t)^{-\frac{1}{2}}(T-s)^{-\frac{1}{2}}E_n(s)\,ds.
\label{smooth conv 4}
\end{align}
Then combining \eqref{Feynman-Kac BEL smooth bbd PDE n}, 
\eqref{Feynman-Kac BEL bbd PDE}, \eqref{def E n s}, \eqref{smooth conv 1}, 
\eqref{smooth conv 2}, \eqref{smooth conv 3}, and \eqref{smooth conv 4} yields 
for all $t\in[0,T]$ and $n\in\bN$ that
\begin{align*}
E_n(t)
&
\leq
\big[1+T+(1+2T)dC_{d,0}^{1/2}\varepsilon_d^{-1/2}\big]
\nonumber\\
& \quad
\cdot
\Bigg(
\sup_{(s,v)\in[0,T]\times\bR^{2d+1}}
\Big[\big|f^{(\varepsilon_n)}(s,v)-f(s,v)\big|\Big]
+\sup_{y\in\bR^d}\left[\big|g^{(\varepsilon_n)}(y)-g(y)\big|\right]
\Bigg)
\nonumber\\
& \quad
+
L^{1/2}(1+T^{1/2})
\big[1+(TC_{d,0}\varepsilon_d^{-1})^{1/2}\big]
\int_t^T
\left[
(T-s)^{-1/2}+(s-t)^{-1/2}(T-s)^{-1/2}
\right]
E_n(s)\,ds.
\end{align*}
Hence, by \eqref{time integral est 1/2} and Gr\"onwall's lemma 
we have for all $n\in\bN$ that
\begin{equation}
\label{gronwall E n 0}
E_n(0)\leq 
c_1e^{c_2T}
\Bigg(
\sup_{(s,v)\in[0,T]\times\bR^{2d+1}}
\Big[\big|f^{(\varepsilon_n)}(s,v)-f(s,v)\big|\Big]
+\sup_{y\in\bR^d}\left[\big|g^{(\varepsilon_n)}(y)-g(y)\big|\right]
\Bigg),
\end{equation}
where
$$
c_1:=1+T+(1+2T)dC_{d,0}^{1/2}\varepsilon_d^{-1/2},
$$
and
$$
c_2:=
2(2+T^{1/2})(1+T^{1/2})L^{1/2}
\big[1+(TC_{d,0}\varepsilon_d^{-1})^{1/2}\big].
$$
Then notice that \eqref{conv compact g f molli} and \eqref{gronwall E n 0}
ensures that
$
\lim_{n\to\infty}E_n(0)=0.
$
This implies for all every compact set $\cK\in (0,T)\times\bR^d$ that
\begin{equation}
\label{compact conv u w n}
\lim_{n\to\infty}
\sup_{(r,y)\in\cK}\Big[
\big|u^n(r,y)-u(r,y)\big|+\big\|(\nabla_y u^n)(r,y)-w(r,y)\big\|
\Big]
=0.
\end{equation}
Therefore, by the fact that $u^n$ is a viscosity solution of 
PDE \eqref{PDE smooth bbd} for all $n\in\bN$, and
Lemma \ref{v solution approximation} we obtain (ii).
Moreover, \eqref{compact conv u w n} and e.g., 
\cite[Section 16.3.5, Theorem 4]{zorich2016mathematical} ensure (iii).
The proof of this lemma is thus completed.
\end{proof}

\subsection{Proof of Proposition \ref{theorem PDE existence}}
\label{section proof existence PDE}

\begin{proof}[Proof of Proposition \ref{theorem PDE existence}]
Throughout this proof, for $N\in\{d,2d+1\}$ 
let $\chi_N\in C^\infty_c(\bR^N,\bR)$ such that 
$\chi^{(n)}(x)=1$ for $\|x\|\leq 1$, 
$0\leq\chi_N(x)\leq 1$ for $1<\|x\|<2$,
and $\chi_N(x)=0$ for $\|x\|\geq 2$.
Then for each $n\in\bN$, $t\in[0,T]$, $x\in\bR^d$, and $v\in\bR^{2d+1}$ 
define 
$$\chi_N^{(n)}(x):=\chi_N(x/n), \quad N\in\{d,2d+1\},$$ 
and
\begin{align*}
g^{(n)}(x):=g(x)\chi^{(n)}_d(x), \quad f^{(n)}(t,v):=f(t,v)\chi^{(n)}_{2d+1}(v).
\end{align*}
By the construction of $g^{(n)}$ and $f^{(n)}$, it is easy to see that
$g^{(n)}\in C_b(\bR^d,\bR)$ and $f^{(n)}(t,\cdot)\in C_b(\bR^{2d+1},\bR)$
for all $n\in\bN$.
Also notice that we have for all $n\in\bN$, $t\in[0,T]$, 
$x,x'\in\bR^d$ with $\|x\|\wedge\|x'\|\geq 2n$, 
and $v,v'\in\bR^{2d+1}$ with $\|v\|\wedge\|v'\|\geq 2n$
that
\begin{equation}
\label{Lip g n f n 0}
\big\|g^{(n)}(x)-g^{(n)}(x')\big\|
+
\big|f^{(n)}(t,v)-f^{(n)}(t,v')\big|
=0.
\end{equation}
Moreover, by \eqref{assumption Lip mu sigma}, \eqref{assumption growth f g},
and the mean value theorem
it holds for all $n\in\bN$ and $x,x'\in\bR^d$ with 
$\|x'\|\leq 2n$ that  
\begin{align}
&
\big|g^{(n)}(x)-g^{(n)}(x')\big|^2
=\big|g(x)\chi^{(n)}_d(x)-g(x')\chi^{(n)}_d(x')\big|^2
\nonumber\\
&
\leq 2\big|g(x)\chi^{(n)}_d(x)-g(x')\chi^{(n)}_d(x)\big|^2
+2\big|
g(x')(\chi^{(n)}_d(x)-\chi^{(n)}_d(x'))
\big|^2
\nonumber\\
&
\leq
2L\|x-x'\|^2\cdot \sup_{y\in\bR^d}\big(|\chi_d(y)|^2\big)
+2\cdot\sup_{\|y\|\leq 2n}\big(|g(y)|^2\big)\cdot
\sup_{y\in\bR^d}\big(\big\|\nabla\chi^{(n)}_d(y)\big\|^2\big)\cdot\|x-x'\|^2
\nonumber\\
&
\leq
2L\|x-x'\|^2
+\frac{2L(d^p+4n^2)}{n^2}\cdot
\sup_{y\in\bR^d}\big(\big\|\nabla\chi_d(y)\big\|^2\big)\cdot\|x-x'\|^2
\nonumber\\
&
\leq L_1\|x-x'\|^2,
\label{Lip g n}
\end{align}
where
$$
L_1:=2L\left[
1+5d^p\cdot
\sup_{y\in\bR^d}\big(\big\|\nabla\chi_d(y)\big\|^2\big)
\right].
$$
Similarly, \eqref{assumption Lip f}, \eqref{assumption growth f g},
and the mean-value theorem ensures for all 
$n\in\bN$ and $v,v'\in\bR^{2d+1}$ with $\|v'\|\leq 2n$ that
\begin{equation}
\label{Lip f n}
\big|f^{(n)}(v)-f^{(n)}(v')\big|^2
\leq 
L_2\|v-v'\|^2,
\end{equation}
where
$$
L_2:=2L\left[
1+5(1+d^p)\cdot
\sup_{y\in\bR^d}\big(\big\|\nabla\chi_{2d+1}(y)\big\|^2\big)
\right].
$$
Furthermore, by \eqref{assumption growth} and \eqref{assumption Lip f} 
we have for all $n\in\bN$ and
$(t,x)\in[0,T]\times\bR^d$ and $v\in\bR^{2d+1}$ that
\begin{equation}
\label{LG f g n}
\big|g^{(n)}(x)\big|^2\leq \big|g(x)\big|^2\leq L(d^p+\|x\|^2)
\quad \text{and} \quad
\big|f^{(n)}(v)\big|^2\leq L(d^p+\|v\|^2)
\end{equation}
In addition, by \eqref{assumption growth f g} and the construction of $g^{(n)}$
it holds for all $n\in\bN$, $x\in\bR^d$, and $\varepsilon\in(0,1)$ that
\begin{align}
\frac{|g^{(n)}(x)-g(x)|}{(d^p+\|x\|^2)^{\frac{1+\varepsilon}{2}}}
&
=
\frac{\mathbf{1}_{\{\|x\|\geq n\}}\big|g(x)(\chi^{(n)}_d(x)-1)\big|}
{(d^p+\|x\|^2)^{\frac{1+\varepsilon}{2}}}
\leq
\frac{L^{1/2}(d^p+\|x\|^2)^{1/2}}
{(d^p+n^2)^{\frac{\varepsilon}{2}}(d^p+\|x\|^2)^{1/2}}
=\frac{L^{1/2}}{(d^p+n^2)^{\frac{\varepsilon}{2}}}.
\label{est g n g local}
\end{align}
Analogously, \eqref{assumption Lip f}, \eqref{assumption growth f g}, 
and the construction of $f^{(n)}$
implies for all $n\in\bN$, $(t,v)\in [0,T]\times\bR^{2d+1}$, 
and $\varepsilon\in(0,1)$ that
\begin{align}
\frac{|f^{(n)}(t,v)-f(t,v)|}{(d^p+\|v\|^2)^{\frac{1+\varepsilon}{2}}}
&
=
\frac{\mathbf{1}_{\{\|v\|\geq n\}}\big|f(t,v)(\chi^{(n)}_{2d+1}(v)-1)\big|}
{(d^p+\|v\|^2)^{\frac{1+\varepsilon}{2}}}
\leq
\frac{L^{1/2}(d^p+\|v\|^2)^{1/2}}
{(d^p+n^2)^{\frac{\varepsilon}{2}}(d^p+\|v\|^2)^{1/2}}
=\frac{L^{1/2}}{(d^p+n^2)^{\frac{\varepsilon}{2}}}.
\label{est f n f local}
\end{align}
Note that \eqref{est g n g local} and \eqref{est f n f local} ensures for 
every compact set $\cK\in[0,T]\times\bR^d$ and every compact set $\cK'\in\bR^{2d+1}$ that
\begin{equation}
\label{conv g n f n}
\lim_{n\to\infty}
\bigg(
\sup_{x\in\cK}\Big[\big|g^{(n)}(x)-g(x)\big|\Big]
+
\sup_{(t,v)\in\cK'}\Big[\big|f^{(n)}(t,v)-f(t,v)\big|\Big]
\bigg)
=0.
\end{equation}
For each $n\in\bN$ and $(t,x)\in[0,T]\times\bR^d$, 
let $\big(X^{t,x,(n)}_s\big)_{s\in[0,T]}:[t,T]\times\bR^d\to\bR^d$ 
be the continuous stochastic process such that
\begin{equation}
dX^{t,x,(n)}_s
=\mu^{(n)}\big(X^{t,x,(n)}_s\big)\,ds
+\sigma^{(n)}\big(X^{t,x,(n)}_s\big)\,dW_s,
\quad s\in[0,T]
\end{equation}
with $X^{t,x}_0=x$,
where $\mu^{(n)}$ and $\sigma^{(n)}$ are the coefficients taken from
Lemma \ref{lemma mu n} and Lemma \ref{lemma sigma n}, respectively.
Moreover, for every $n\in\bN$, $(t,x)\in[0,T]\times\bR^d$, and $s\in[t,T]$
we use the notation 
$$
DX^{t,x,(n)}_s:=\left(\frac{\partial}{\partial x_1} X^{t,x,(n)}_s,
\frac{\partial}{\partial x_2} X^{t,x,(n)}_s,
\dots,\frac{\partial}{\partial x_d}X^{t,x,(n)}_s\right).
$$
For every $(t,x)\in[0,T]\times\bR^d$, we also use the notations
$\mu^{(0)}=\mu$, $\sigma^{(0)}=\sigma$,
$X^{t,x,(0)}=X^{t,x}$, and $DX^{t,x,(0)}=DX^{t,x}$.
Then by \eqref{Lip sigma n}, \eqref{LG sigma n}, \eqref{Lip mu n}, \eqref{LG mu n},
\eqref{Lip g n f n 0}, \eqref{Lip g n}, \eqref{Lip f n}, and \eqref{LG f g n}
we apply Corollary \ref{corollary FP}
(with $g^d\cal g^{(n)}$, $f^d\cal f^{(n)}$, 
$\mu^d\cal \mu^{(n)}$, $\sigma^d\cal\sigma^{(n)}$,
and $X^{d,0,t,x}\cal X^{t,x,(n)}$
in the notation of Corollary \ref{corollary FP})
to obtain for all $n\in\bN_0$ and $m\in\bN$ that
there exists a unique pair of Borel functions $(u^{n,m},w^{n,m})$
such that
$u^{n,m}\in C_{lin}([0,T)\times\bR^d,\bR)$, 
$w^{n,m}\in C([0,T)\times\bR^d,\bR^d)$, 
and
\begin{align}
u^{n,m}(t,x)
&
=\bE\left[g^{(m)}(X^{t,x,(n)}_T)\right]     
+\int_t^T\bE\left[
f^{(m)}\big(s,X^{t,x,(n)}_s,u^{n,m}(s,X^{t,x,(n)}_s),w^{n,m}(s,X^{t,x,(n)}_s)\big)
\right]ds
\label{BEL n k u}
\end{align}
and
\begin{align}
w^{n,m}(t,x)
&
=\bE\left[
g^{(m)}(X^{t,x,(n)}_T)
\frac{1}{T-t}
\int_t^T\left(\big[\sigma^{(n)}(X^{t,x,(n)}_{r})\big]^{-1}
DX^{t,x,(n)}_{r}\right)^T\,dW_r
\right]  
\nonumber\\
& \quad    
+\int_t^T\bE\Bigg[
f^{(m)}\big(s,X^{t,x,(n)}_s,u^{n,m}(s,X^{t,x,(n)}_s),w^{n,m}(s,X^{t,x,(n)}_s)\big)
\nonumber\\
& \quad
\cdot
\frac{1}{s-t}
\int_t^s\left(\big[\sigma(X^{t,x,(n)}_{r})\big]^{-1}
DX^{t,x,(n)}_{r}\right)^T\,dW_r
\Bigg]\,ds
\label{BEL n k grad}
\end{align}
for all $(t,x)\in[0,T)\times\bR^d$.
Proposition \ref{proposition PDE existence bounded} ensures for all
$n\in\bN$ and $m\in\bN$ that $u^{n,m}$ is a viscosity solution 
of the following PDE
\begin{align}
&
\frac{\partial}{\partial t}u^{n,m}(t,x)
+\langle\nabla_xu^{n,m}(t,x),\mu^{(n)}(t,x)\rangle
+\frac{1}{2}\operatorname{Trace}
\left(\sigma^{(n)}(t,x)[\sigma^{(n)}(t,x)]^T
\operatorname{Hess}_xu^{n,m}(t,x)\right)\nonumber\\
&
+f^{(m)}\big(t,x,u^{n,m}(t,x), \nabla_x u^{n,m}(t,x)\big)=0.       
\label{PDE n m k}
\end{align}
Proposition \ref{proposition PDE existence bounded} also implies 
for all $n,m\in\bN$ and $(t,x)\in[0,T)\times\bR^d$ 
that $\nabla_x u^{n,m}(t,x)$ exists and
coincides with $w^{n,m}(t,x)$.
For each $(t,x)\in[0,T]\times\bR^d$, $k\in\{1,2,\dots,d\}$, and $n\in\bN$
let $\tau^{t,x}_n:\Omega\to[t,T]$ and $\tau^{t,x,k}_n:\Omega\to[t,T]$ 
be stopping times defined by
\begin{equation}                                     \label{def tau n}
\tau^{t,x}_{n}:=\inf
\big\{
s\geq t: \max\big(\big\|X^{t,x}_s\big\|,\big\|X^{t,x,(n)}_s\big\|\big)>n
\big\}
\wedge T,
\end{equation}
and
\begin{equation}                                     \label{def tau k n}
\tau^{t,x,k}_n:=\inf
\Big\{
s\geq t: 
\max\Big(
\big\|X^{t,x}_s\big\|,\big\|X^{t,x,(n)}_s\big\|,
\Big\|\frac{\partial}{\partial x_k} X^{t,x}_s\Big\|,
\Big\|\frac{\partial}{\partial x_k} X^{t,x,(n)}_s\Big\|
\Big)
>n
\Big\}
\wedge T.
\end{equation}
Then by Assumption \ref{assumption Lip and growth}, 
Lemma \ref{lemma sigma n}, and Lemma \ref{lemma mu n}
the application of Lemma \ref{lemma two SDEs} yields for all 
$(t,x)\in[0,T]\times\bR^d$, $k\in\{1,2,\dots,d\}$, 
and $n\in\bN$ that
\begin{equation}                                        \label{prob 1 tau n}
\bP\left(
\mathbf{1}_{\{s\leq\tau^{t,x}_n\}}\big\|X^{t,x}_s-X^{t,x,(n)}_s\big\|=0
\; \text{for} \; \text{all} \; s\in[t,T]
\right)=1,
\end{equation}
and
\begin{equation}                                     \label{prob 1 tau k n}
\bP\left(
\mathbf{1}_{\{s\leq\tau^{t,x,k}_n\}}
\Big\|\frac{\partial}{\partial x_k} X^{t,x}_s
-\frac{\partial}{\partial x_k} X^{t,x,(n)}_s\Big\|=0
\; \text{for} \; \text{all} \; s\in[t,T]
\right)=1.
\end{equation}
Furthermore, by \eqref{SDE moment est}, 
\eqref{Lip sigma n}, \eqref{LG sigma n},
\eqref{Lip mu n}, \eqref{LG mu n}, and Lemma \ref{lemma SDE partial est}
it holds for all $n\in\bN$, $k\in\{1,2,\dots,d\}$, and $(t,x)\in[0,T]\times\bR^d$
that
\begin{equation}                                             \label{est sup X n}
\bE\bigg[\sup_{s\in[t,T]}\big\|X^{t,x,(n)}_s\big\|^2\bigg]
\leq c_{d,1}(d^p+\|x\|^2),
\quad
\bE\Bigg[
\sup_{s\in[t,T]}\Big\|\frac{\partial}{\partial x_k}X^{t,x,(n)}_s\Big\|^2\Bigg]
\leq c_{d,1},
\end{equation}
where $c_{d,1}$ is a positive constant only depending on 
$C_{(d),1}$ and $C_{(d),2}$.
Thus, by \eqref{SDE moment est}, \eqref{moment est partial}, 
\eqref{def tau n}, \eqref{def tau k n}, and Chebyshev's inequality, 
it holds for all $(t,x)\in[0,T]\times\bR^d$, $n\in\bN$, and $k\in\{1,2,\dots,d\}$
that
\begin{align}                                       
\bP\big(\tau^{t,x}_n<T\big)
&
\leq \bP\left(\big\|X^{t,x}_{\tau^{t,x}_n}\big\|
\geq n\right)
+\bP\left(\big\|X^{t,x,(n)}_{\tau^{t,x}_n}\big\|
\geq n\right)
\leq n^{-2}\left(\bE\left[\big\|X^{t,x}_{\tau^{t,x}_n}\big\|^2\right]
+\bE\left[\big\|X^{t,x,(n)}_{\tau^{t,x}_n}\big\|^2\right]\right)
\nonumber\\
&
\leq \frac{(C_{(d,2)}+c_{d,1})(d^p+\|x\|^2)}{n^2},                       
\label{tau t x n est}
\end{align}
and
\begin{align} 
&                                      
\bP\big(\tau^{t,x,k}_n<T\big)
\nonumber\\
&
\leq \bP\left(\big\|X^{t,x}_{\tau^{t,x,k}_n}\big\|
\geq n\right)
+\bP\left(\big\|X^{t,x,(n)}_{\tau^{t,x,k}_n}\big\|
\geq n\right)
+ \bP\left(\Big\|\frac{\partial}{\partial x_k} X^{t,x}_{\tau^{t,x,k}_n}\Big\|
\geq n\right)
+\bP\left(\Big\|\frac{\partial}{\partial x_k} X^{t,x,(n)}_{\tau^{t,x,k}_n}\Big\|
\geq n\right)
\nonumber\\
&
\leq 
n^{-2}\left(
\bE\left[\big\|X^{t,x}_{\tau^{t,x,k}_n}\big\|^2\right]
+
\bE\left[\big\|X^{t,x,(n)}_{\tau^{t,x,k}_n}\big\|^2\right]
+
\bE\left[\Big\|\frac{\partial}{\partial x_k} X^{t,x}_{\tau^{t,x,k}_n}\Big\|^2\right]
+
\bE\left[\Big\|\frac{\partial}{\partial x_k} 
X^{t,x,(n)}_{\tau^{t,x,k}_n}\Big\|^2\right]
\right)
\nonumber\\
&
\leq \frac{(C_{d,0}+c_{d,1})+(C_{(d,2)}+c_{d,1})(d^p+\|x\|^2)}{n^2}.                         
\label{tau t x k n est}
\end{align}
This together with \eqref{prob 1 tau n} implies for all $n\in\bN_0$, $m\in\bN$, 
and $(t,x)\in[0,T]\times\bR^d$ that
\begin{align}
\bE\left[
\big|
g^{(m)}\big(X^{t,x,(n)}_T\big)-g^{(m)}\big(X^{t,x}_T\big)
\big|
\right]
&
=\bE\Big[
\mathbf{1}_{\{\tau^{t,x}_n<T\}}\big|
g^{(m)}\big(X^{t,x,(n)}_T\big)-g^{(m)}\big(X^{t,x}_T\big)
\big|
\Big]
\nonumber\\
&
\leq 2\Big[\sup_{y\in\bR^d}\big|g^{(m)}(y)\big|\Big]
\big[\bP\big(\tau^{t,x}_n<T\big)\big]^{1/2}
\nonumber\\
&
\leq 2\bigg[\sup_{y\in\bR^d}\big|g^{(m)}(y)\big|\bigg]
\frac{\big[(C_{(d,2)}+c_{d,1})(d^p+\|x\|^2)\big]^{1/2}}{n}.
\label{whisky 1}
\end{align}
Next, by the triangle inequality we obtain for all $(t,x)\in[0,T)\times\bR^d$,
$m,n\in\bN$, and $k\in\{1,2,\dots,d\}$ that
\begin{align}
&
(T-t)^{1/2}\bE\Bigg[\bigg|
g^{(m)}(X^{t,x,(n)}_T)
\frac{1}{T-t}
\int_t^T\left(
\big[\sigma^{(n)}(X^{t,x,(n)}_r)\big]^{-1}
\frac{\partial}{\partial x_k}X^{t,x,(n)}_r
\right)^TdW_r
\nonumber\\
&
-
g^{(m)}(X^{t,x}_T)
\frac{1}{T-t}
\int_t^T\left(
\sigma^{-1}(X^{t,x}_r)
\frac{\partial}{\partial x_k}X^{t,x}_r
\right)^TdW_r
\bigg|\Bigg]
\nonumber\\
&
\leq 
\sum_{i=1}^3A^{k,m,n}_i(t,x),
\label{ineq A k m n}
\end{align}
where
\begin{align*}
&
A^{k,m,n}_1(t,x):=
(T-t)^{1/2}
\bE\Bigg[\bigg|
\big(g^{(m)}(X^{t,x,(n)}_T)-g^{(m)}(X^{t,x}_T)\big)
\frac{1}{T-t}
\int_t^T\left[
\sigma^{-1}(X^{t,x}_r)\frac{\partial}{\partial x_k}X^{t,x}_r
\right]^TdW_r
\bigg|\Bigg],
\\
&
A^{k,m,n}_2(t,x):=
(T-t)^{1/2}
\bE\Bigg[\bigg|
g^{(m)}(X^{t,x,(n)}_T)
\frac{1}{T-t}
\int_t^T\left[
\sigma^{-1}(X^{t,x}_r)
\Big(\frac{\partial}{\partial x_k}X^{t,x,(n)}_r
-\frac{\partial}{\partial x_k}X^{t,x}_r\Big)
\right]^TdW_r
\bigg|\Bigg],
\end{align*}
and
$$
A^{k,m,n}_3(t,x):=
(T-t)^{1/2}
\bE\Bigg[\bigg|
\frac
{g^{(m)}(X^{t,x,(n)}_T)}
{T-t}
\int_t^T\left[
\big(\big[\sigma^{(n)}(X^{t,x,(n)}_r)\big]^{-1}-\sigma^{-1}(X^{t,x}_r)\big)
\frac{\partial}{\partial x_k}X^{t,x,(n)}_r
\right]^TdW_r
\bigg|\Bigg].
$$
Then by \eqref{sigma inverse est}, \eqref{SDE moment est},
\eqref{moment est partial}, \eqref{def tau n}, \eqref{prob 1 tau n}, 
\eqref{tau t x n est}, H\"older's inequality, and It\^o's isometry 
we have for all $k\in\{1,2,\dots,d\}$, $m,n\in\bN$, 
and $(t,x)\in[0,T)\times\bR^d$ that
\begin{align}
&
A^{k,m,n}_1(t,x)
\nonumber\\
&
\leq (T-t)^{-1/2}
\left(
\bE\left[
\mathbf{1}_{\{\tau^{t,x}_n<T\}}
\big|g^{(m)}(X^{t,x,(n)}_T)-g^{(m)}(X^{t,x}_T)\big|^2
\right]
\right)^{1/2}
\left(
\bE\int_t^T
\Big\|
\sigma^{-1}(X^{t,x}_r)\frac{\partial}{\partial x_k}X^{t,x}_r
\Big\|^2\,dr
\right)^{1/2}
\nonumber\\
&
\leq (T-t)^{-1/2}2\Big[\sup_{y\in\bR^d}\big|g^{(m)}(y)\big|\Big]
\big[\bP\big(\tau^{t,x}_n<T\big)\big]^{1/2}
\left(
\int_t^T\varepsilon_d^{-1}
\bE\left[\Big\|\frac{\partial}{\partial x_k}X^{t,x}_r\Big\|^2\right]dr
\right)^{1/2}
\nonumber\\
&
\leq
2\big(\varepsilon_d^{-1}C_{d,0}\big)^{1/2}
\bigg[\sup_{y\in\bR^d}\big|g^{(m)}(y)\big|\bigg]
\frac{\big[(C_{(d,2)}+c_{d,1})(d^p+\|x\|^2)\big]^{1/2}}{n}.
\label{est A k m n 1}
\end{align}
Moreover, by \eqref{sigma inverse est}, 
\eqref{moment est partial}, \eqref{est sup X n}, \eqref{def tau k n}, 
\eqref{prob 1 tau k n},
\eqref{tau t x k n est}, Cauchy-Schwarz inequality, and It\^o's isometry
it holds for all $k\in\{1,2,\dots,d\}$, $m,n\in\bN$, and $(t,x)\in[0,T)\times\bR^d$
that
\begin{align}
&
A^{k,m,n}_2(t,x)
\nonumber\\
&
=
(T-t)^{1/2}
\bE\Bigg[\bigg|
\mathbf{1}_{\{\tau^{t,x,k}_n<t\}}
\frac{g^{(m)}(X^{t,x,(n)}_T)}{T-t}
\int_t^T\bigg[
\sigma^{-1}(X^{t,x}_r)
\Big(\frac{\partial}{\partial x_k}X^{t,x,(n)}_r
-\frac{\partial}{\partial x_k}X^{t,x}_r\Big)
\bigg]^T\,dW_r
\bigg|\Bigg]
\nonumber\\
&
\leq 
(T-t)^{-1/2}
\bigg[\sup_{y\in\bR^d}\big|g^{(m)}(y)\big|\bigg]
\big[\bP(\tau^{t,x,k}_n<T)\big]^{1/2}
\left(
\bE\left[
\int_t^T\Big\|
\sigma^{-1}(X^{t,x}_r)
\Big(\frac{\partial}{\partial x_k}X^{t,x,(n)}_r
-\frac{\partial}{\partial x_k}X^{t,x}_r\Big)
\Big\|^2\,dr
\right]
\right)^{1/2}
\nonumber\\
&
\leq
(T-t)^{-1/2}
\bigg[\sup_{y\in\bR^d}\big|g^{(m)}(y)\big|\bigg]
\big[\bP(\tau^{t,x,k}_n<T)\big]^{1/2}
\left(
\int_t^T\varepsilon_d^{-1}
\bE\left[
\Big\|
\frac{\partial}{\partial x_k}X^{t,x,(n)}_r
-\frac{\partial}{\partial x_k}X^{t,x}_r
\Big\|^2
\right]
dr
\right)^{1/2}
\nonumber\\
&
\leq 
\big[2\varepsilon_d^{-1}(c_{d,1}+C_{d,0})\big]^{1/2}
\bigg[\sup_{y\in\bR^d}\big|g^{(m)}(y)\big|\bigg]
\frac{\big[(C_{d,0}+c_{d,1})+(C_{(d,2)}+c_{d,1})(d^p+\|x\|^2)\big]^{1/2}}{n}.
\label{est A k m n 2}
\end{align}
By \eqref{ellip sigma n} and the same argument to obtain 
\eqref{sigma inverse est} we notice for all $n\in\bN$ and $x,y\in\bR^d$ that
\begin{equation}
\label{sigma n inverse est}
y^T\big(\big[\sigma^{(n)}(x)\big]^{-1}\big)^T[\sigma^{(n)}(x)\big]^{-1}y
\leq\varepsilon_d^{-1}\|y\|^2.
\end{equation}
Then by  \eqref{moment est partial}, \eqref{est sup X n}, \eqref{def tau k n}, 
\eqref{prob 1 tau k n}, \eqref{tau t x k n est}, \eqref{sigma n inverse est},
Cauchy-Schwarz inequality, and It\^o's isometry
we have for all $k\in\{1,2,\dots,d\}$, $m,n\in\bN$, and $(t,x)\in[0,T)\times\bR^d$
that
\begin{align}
&
A^{k,m,n}_3(t,x)
\nonumber\\
&
=
(T-t)^{-1/2}
\bE\Bigg[\bigg|
\mathbf{1}_{\{\tau^{t,x,k}_n<T\}}
\frac{g^{(m)}(X^{t,x,(n)}_T)}{T-t}
\int_t^T
\Big[
\big(
\big[\sigma^{(n)}(X^{t,x,(n)}_r)\big]^{-1}-\sigma^{-1}(X^{t,x}_r)
\big)
\frac{\partial}{\partial x_k}X^{t,x,(n)}_r
\Big]^T\,dW_r
\bigg|\Bigg]
\nonumber\\
&
\leq(T-t)^{1/2}
\bigg[\sup_{y\in\bR^d}\big|g^{(m)}(y)\big|\bigg]
\big[\bP(\tau^{t,x,k}_n<T)\big]^{1/2}
\nonumber\\
& \quad
\cdot
\left[
\left(
\bE\left[
\int_t^T
\Big\|
\big[\sigma^{(n)}(X^{t,x,(n)}_r)\big]^{-1}
\frac{\partial}{\partial x_k}X^{t,x,(n)}_r
\Big\|^2\,dr
\right]
\right)^{1/2}
+
\left(
\bE\left[
\int_t^T
\Big\|
\sigma^{-1}(X^{t,x}_r)
\frac{\partial}{\partial x_k}X^{t,x,(n)}_r
\Big\|^2\,dr
\right]
\right)^{1/2}
\right]
\nonumber\\
&
\leq
(T-t)^{-1/2}
\bigg[\sup_{y\in\bR^d}\big|g^{(m)}(y)\big|\bigg]
\big[\bP(\tau^{t,x,k}_n<T)\big]^{1/2}
2\left(
\int_t^T\varepsilon_d^{-1}
\bE\left[\Big\|\frac{\partial}{\partial x_k}X^{t,x,(n)}_r\Big\|^2\right]dr
\right)^{1/2}
\nonumber\\
&
\leq
2\big(c_{d,1}\varepsilon_d^{-1}\big)^{1/2}
\bigg[\sup_{y\in\bR^d}\big|g^{(m)}(y)\big|\bigg]
\frac{\big[(C_{d,0}+c_{d,1})+(C_{(d,2)}+c_{d,1})(d^p+\|x\|^2)\big]^{1/2}}{n}.
\label{est A k m n 3}
\end{align}
Combining \eqref{ineq A k m n}, \eqref{est A k m n 1}, \eqref{est A k m n 2},
and \eqref{est A k m n 3} yields for all $k\in\{1,2,\dots,d\}$, $m,n\in\bN$, 
and $(t,x)\in[0,T)\times\bR^d$ that
\begin{align}
&
(T-t)^{1/2}\bE\Bigg[\bigg\|
g^{(m)}(X^{t,x,(n)}_T)
\frac{1}{T-t}
\int_t^T\left(
\big[\sigma^{(n)}(X^{t,x,(n)}_r)\big]^{-1}
\frac{\partial}{\partial x_k}X^{t,x,(n)}_r
\right)^TdW_r
\nonumber\\
&
-
g^{(m)}(X^{t,x}_T)
\frac{1}{T-t}
\int_t^T\left(
\sigma^{-1}(X^{t,x}_r)
\frac{\partial}{\partial x_k}X^{t,x}_r
\right)^TdW_r
\bigg\|\Bigg]
\nonumber\\
&
\leq 
\mathfrak{c}_{d,1}
\bigg[\sup_{y\in\bR^d}\big|g^{(m)}(y)\big|\bigg]
\frac{(d^p+\|x\|^2)^{1/2}}{n},
\label{whisky 2}
\end{align}
where
$$
\mathfrak{c}_{d,1}:=6\varepsilon_d^{-1/2}(C_{d,0}\vee C_{(d,2)} +c_{d,1}).
$$
Next, by \eqref{prob 1 tau n} and the triangle inequality we notice 
for all $n,m\in\bN$ and $(t,x)\in[0,T)\times\bR^d$ that
\begin{align}
&
\int_t^T
\bE\Big[\big|
f^{(m)}\big(s,X^{t,x,(n)}_s,u^{n,m}(s,X^{t,x,(n)}_s),w^{n,m}(s,X^{t,x,(n)}_s)\big)
\nonumber\\
& \quad
-
f^{(m)}\big(s,X^{t,x}_s,u^{0,m}(s,X^{t,x}_s),w^{0,m}(s,X^{t,x}_s)\big)
\big|\Big]\,ds
\nonumber\\
&
\leq 
B^{n,m}_1(t,x)+B^{n,m}_2(t,x),
\label{ineq B n m}
\end{align}
where
\begin{align*}
B^{n,m}_1(t,x)
:=
\int_t^T
&
\bE\Big[\mathbf{1}_{\{\tau^{t,x}_n<T\}}\big|
f^{(m)}\big(s,X^{t,x,(n)}_s,u^{n,m}(s,X^{t,x,(n)}_s),w^{n,m}(s,X^{t,x,(n)}_s)\big)
\nonumber\\
&
-
f^{(m)}\big(s,X^{t,x}_s,u^{0,m}(s,X^{t,x}_s),w^{0,m}(s,X^{t,x}_s)\big)
\big|\Big]\,ds,
\end{align*}
and
\begin{align*}
B^{n,m}_2(t,x)
:=
\int_t^T
&
\bE\Big[\mathbf{1}_{\{\tau^{t,x}_n\geq T\}}\big|
f^{(m)}\big(s,X^{t,x}_s,u^{n,m}(s,X^{t,x}_s),w^{n,m}(s,X^{t,x}_s)\big)
\nonumber\\
&
-
f^{(m)}\big(s,X^{t,x}_s,u^{0,m}(s,X^{t,x}_s),w^{0,m}(s,X^{t,x}_s)\big)
\big|\Big]\,ds.
\end{align*}
By \eqref{tau t x n est}, we have for all $n,m\in\bN$ 
and $(t,x)\in[0,T)\times\bR^d$ that
\begin{align}
B^{n,m}_1(t,x)
&
\leq 
2T\bigg[\sup_{(s,v)\in[0,T]\times\bR^{2d+1}}\big|f^{(m)}(s,v)\big|\bigg]
\big[\bP(\tau^{t,x}_n<T)\big]^{1/2}
\nonumber\\
&
\leq
2T\bigg[\sup_{(s,v)\in[0,T]\times\bR^{2d+1}}\big|f^{(m)}(s,v)\big|\bigg]
\frac{(C_{(d,2)}+c_{d,1})^{1/2}(d^p+\|x\|^2)^{1/2}}{n}.
\label{est B n m 1}
\end{align}
For each $n,m\in\bN$ and $s\in[0,T)$, we define
\begin{equation}
\label{def E n m s}
E^{n,m}(s):=\sup_{r\in[s,T)}\sup_{y\in\bR^d}
\frac
{|u^{n,m}(r,y)-u^{0,m}(r,y)|+(T-s)^{1/2}\|w^{n,m}(r,y)-w^{0,m}(r,y)\|}
{(d^p+\|y\|^2)^{1/2}}.
\end{equation}
Then by \eqref{SDE moment est}, \eqref{Lip f n}, and H\"older inequality
it holds for all $n,m\in\bN$ and $(t,x)\in[0,T)\times\bR^d$ that
\begin{align}
&
B^{n,m}_2(t,x)
\nonumber\\
&
\leq
\int_t^TL_2^{1/2}
\bE\left[
|u^{n,m}(s,X^{t,x}_s)-u^{0,m}(s,X^{t,x}_s)|
+
\|w^{n,m}(s,X^{t,x}_s)-w^{0,m}(s,X^{t,x}_s)\|
\right]ds
\nonumber\\
&
\leq
L_2^{1/2}(T^{1/2}+1)
\int_t^T\bE
\bigg[
(T-s)^{-1/2}(d^p+\|X^{t,x}_s\|^2)^{1/2}
\nonumber\\
& \quad
\cdot
\frac
{|u^{n,m}(s,X^{t,x}_s)-u^{0,m}(s,X^{t,x}_s)|
+(T-s)^{1/2}\|w^{n,m}(s,X^{t,x}_s)-w^{0,m}(s,X^{t,x}_s)\|}
{(d^p+\|X^{t,x}_s\|^2)^{1/2}}
\bigg]\,ds
\nonumber\\
&
\leq
L_2^{1/2}(T^{1/2}+1)
\int_t^T
(T-s)^{-1/2}\left(\bE\big[d^p+\|X^{t,x}_s\|^2\big]\right)^{1/2}E^{n,m}(s)\,ds
\nonumber\\
&
\leq
L_2^{1/2}(T^{1/2}+1)(C_{(d,2)}+1)^{1/2}(d^p+\|x\|^2)^{1/2}
\int_t^T
(T-s)^{-1/2}E^{n,m}(s)\,ds.
\label{est B n m 2}
\end{align}
Combining \eqref{ineq B n m}, \eqref{est B n m 1}, and  \eqref{est B n m 2}
yields for all $n,m\in\bN$ and $(t,x)\in[0,T)\times\bR^d$ that
\begin{align}
&
\int_t^T
\bE\Big[\big|
f^{(m)}\big(s,X^{t,x,(n)}_s,u^{n,m}(s,X^{t,x,(n)}_s),w^{n,m}(s,X^{t,x,(n)}_s)\big)
\nonumber\\
& \quad
-
f^{(m)}\big(s,X^{t,x}_s,u^{0,m}(s,X^{t,x}_s),w^{0,m}(s,X^{t,x}_s)\big)
\big|\Big]\,ds
\nonumber\\
&
\leq
\mathfrak{c}_{d,2}\bigg[\sup_{(s,v)\in[0,T]\times\bR^{2d+1}}\big|f^{(m)}(s,v)\big|\bigg]
\frac{(d^p+\|x\|^2)^{1/2}}{n}
+
\mathfrak{c}_{d,3}(d^p+\|x\|^2)^{1/2}
\int_t^T(T-s)^{-1/2}E^{n,m}(s)\,ds
,
\label{whisky 3}
\end{align}
where
$$
\mathfrak{c}_{d,2}:=2T(C_{(d,2)}+c_{d,1})^{1/2}
\quad
\text{and}
\quad
\mathfrak{c}_{d,3}:=L_2^{1/2}(T^{1/2}+1)(C_{(d,2)}+1)^{1/2}.
$$
Next, notice for all $n,m\in\bN$, $k\in\{1,2,\dots,d\}$, 
and $(t,x)\in[0,T)\times\bR^d$ that
\begin{align}
&
(T-t)^{1/2}\int_t^T\bE\bigg[\Big|
f^{(m)}\big(s,X^{t,x}_s,u^{0,m}(s,X^{t,x}_s),w^{0,m}(s,X^{t,x}_s)\big)
\frac{1}{s-t}\int_t^s\Big(\big[\sigma(X^{t,x}_r)\big]^{-1}
\frac{\partial}{\partial x_k}X^{t,x}_r\Big)^T\,dW_r
\nonumber\\
&
-
f^{(m)}\big(s,X^{t,x,(n)}_s,u^{n,m}(s,X^{t,x,(n)}_s),w^{n,m}(s,X^{t,x,(n)}_s)\big)
\frac{1}{s-t}\int_t^s\Big(\big[\sigma^{(n)}(X^{t,x,(n)}_r)\big]^{-1}
\frac{\partial}{\partial x_k}X^{t,x,(n)}_r\Big)^T\,dW_r
\Big|\bigg]\,ds
\nonumber\\
&
\leq \sum_{i=1}^3B^{k,m,n}_i(t,x),
\label{ineq B k m n}
\end{align}
where
\begin{align*}
&
B^{k,m,n}_1(t,x)
:=(T-t)^{1/2}\int_t^T
\bE\bigg[\Big|
\big[
f^{(m)}\big(s,X^{t,x,(n)}_s,u^{n,m}(s,X^{t,x,(n)}_s),w^{n,m}(s,X^{t,x,(n)}_s)\big)
\\
& \qquad \qquad \qquad \; \;
-
f^{(m)}\big(s,X^{t,x}_s,u^{0,m}(s,X^{t,x}_s),w^{0,m}(s,X^{t,x}_s)\big)
\big]
\frac{1}{s-t}
\int_t^s\Big(\sigma^{-1}(X^{t,x}_r)
\frac{\partial}{\partial x_k}X^{t,x}_r\Big)^T\,dW_r
\Big|\bigg]\,ds,
\\
&
B^{k,m,n}_2(t,x)
:=(T-t)^{1/2}\int_t^T\bE\bigg[\Big|
f^{(m)}\big(s,X^{t,x,(n)}_s,u^{n,m}(s,X^{t,x,(n)}_s),w^{n,m}(s,X^{t,x,(n)}_s)\big)
\\
& \qquad \qquad \qquad \;\;
\cdot
\frac{1}{s-t}
\int_t^s\Big[\sigma^{-1}(X^{t,x}_r)
\Big(\frac{\partial}{\partial x_k}X^{t,x,(n)}_r-
\frac{\partial}{\partial x_k}X^{t,x}_r\Big)\Big]^T\,dW_r
\Big|\bigg]\,ds,
\end{align*}
and
\begin{align*}
B^{k,m,n}_3(t,x)
:=
&
(T-t)^{1/2}\int_t^T\bE\bigg[\Big|
f^{(m)}\big(s,X^{t,x,(n)}_s,u^{n,m}(s,X^{t,x,(n)}_s),w^{n,m}(s,X^{t,x,(n)}_s)\big)
\\
& \
\cdot
\frac{1}{s-t}
\int_t^s\Big[
\big(\big[\sigma^{(n)}(X^{t,x,(n)}_r)\big]^{-1}-\sigma^{-1}(X^{t,x}_r)\big)
\frac{\partial}{\partial x_k}X^{t,x,(n)}_r\Big]^T\,dW_r
\Big|\bigg]\,ds.
\end{align*}
By \eqref{prob 1 tau n}, we observe for all $m,n\in\bN$, $k\in\{1,2,\dots,d\}$,
and $(t,x)\in[0,T)\times\bR^d$ that
\begin{equation}
\label{ineq B k m n 1}
B^{k,m,n}_1(t,x)\leq B^{k,m,n}_{1,1}(t,x)+B^{k,m,n}_{1,2}(t,x),
\end{equation}
where
\begin{align*}
&
B^{k,m,n}_{1,1}(t,x)
\\
&
:=(T-t)^{1/2}\int_t^T
\bE\bigg[\Big|
\big[
\mathbf{1}_{\{\tau^{t,x}_s<T\}}
f^{(m)}\big(s,X^{t,x,(n)}_s,u^{n,m}(s,X^{t,x,(n)}_s),w^{n,m}(s,X^{t,x,(n)}_s)\big)
\\
& \quad \; \;
-
f^{(m)}\big(s,X^{t,x}_s,u^{0,m}(s,X^{t,x}_s),w^{0,m}(s,X^{t,x}_s)\big)
\big]
\frac{1}{s-t}
\int_t^s\Big(\sigma^{-1}(X^{t,x}_r)
\frac{\partial}{\partial x_k}X^{t,x}_r\Big)^T\,dW_r
\Big|\bigg]\,ds,
\end{align*}
and
\begin{align*}
&
B^{k,m,n}_{1,2}(t,x)
\\
&
:=(T-t)^{1/2}\int_t^T
\bE\bigg[\Big|
\big[
\mathbf{1}_{\{\tau^{t,x}_s\geq T\}}
f^{(m)}\big(s,X^{t,x}_s,u^{n,m}(s,X^{t,x}_s),w^{n,m}(s,X^{t,x}_s)\big)
\\
& \quad \; \;
-
f^{(m)}\big(s,X^{t,x}_s,u^{0,m}(s,X^{t,x}_s),w^{0,m}(s,X^{t,x}_s)\big)
\big]
\frac{1}{s-t}
\int_t^s\Big(\sigma^{-1}(X^{t,x}_r)
\frac{\partial}{\partial x_k}X^{t,x}_r\Big)^T\,dW_r
\Big|\bigg]\,ds.
\end{align*}
By \eqref{sigma inverse est}, \eqref{moment est partial}, \eqref{tau t x n est}, 
Cauchy-Schwarz inequality, and It\^o's isometry we have for all $m,n\in\bN$,
$k\in\{1,2,\dots,d\}$, and $(t,x)\in[0,T)\times\bR^d$ that
\begin{align}
&
B^{k,m,n}_{1,1}(t,x)
\nonumber\\
&
\leq
2(T-t)^{1/2}\bigg[\sup_{(s,v)\in[0,T]\times\bR^{2d+1}}\big|f^{(m)}(s,v)\big|\bigg]
\int_t^T\bE\left[
\frac{\mathbf{1}_{\{\tau^{t,x}_n<T\}}}{s-t}
\Big|
\int_t^s
\Big[\sigma^{-1}(X^{t,x}_r)\frac{\partial}{\partial x_k}X^{t,x}_r\Big]^T\,dW_r
\Big|
\right]ds
\nonumber\\
&
\leq
2(T-t)^{1/2}\bigg[\sup_{(s,v)\in[0,T]\times\bR^{2d+1}}\big|f^{(m)}(s,v)\big|\bigg]
\nonumber\\
& \quad
\cdot
\int_t^T\frac{\big[\bP(\tau^{t,x}_n<T)\big]^{1/2}}{s-t}
\left(
\bE\bigg[
\int_t^s\Big\|\sigma^{-1}(X^{t,x}_r)\frac{\partial}{\partial x_k}X^{t,x}_r\Big\|^2
\,dr
\bigg]
\right)^{1/2}ds
\nonumber\\
&
\leq
2(T-t)^{1/2}\bigg[\sup_{(s,v)\in[0,T]\times\bR^{2d+1}}\big|f^{(m)}(s,v)\big|\bigg]
\int_t^T\frac{\big[\bP(\tau^{t,x}_n<T)\big]^{1/2}}{s-t}
\left(
\int_t^s\bE\left[\Big\|\frac{\partial}{\partial x_k}X^{t,x}_r\Big\|^2\right]
\,dr
\right)^{1/2}ds
\nonumber\\
&
\leq
2(T-t)^{1/2}(\varepsilon_d^{-1}C_{d,0})^{1/2}
\bigg[\sup_{(s,v)\in[0,T]\times\bR^{2d+1}}\big|f^{(m)}(s,v)\big|\bigg]
\big[\bP(\tau^{t,x}_n<T)\big]^{1/2}
\int_t^T(s-t)^{-1/2}\,ds
\nonumber\\
&
\leq
4T(\varepsilon_d^{-1}C_{d,0})^{1/2}
\bigg[\sup_{(s,v)\in[0,T]\times\bR^{2d+1}}\big|f^{(m)}(s,v)\big|\bigg]
\frac{\big[(C_{(d,2)}+c_{d,1})(d^p+\|x\|^2)\big]^{1/2}}{n}.
\label{est B k m n 1 1}
\end{align}
By \eqref{sigma inverse est}, \eqref{moment est partial}, \eqref{time integral est},
\eqref{Lip f n}, It\^o's isometry, and H\"older's inequality
we also notice for all $m,n\in\bN$, $k\in\{1,2,\dots,d\}$, 
and $(t,x)\in[0,T)\times\bR^d$ that
\begin{align}
&
B^{k,m,n}_{1,2}(t,x)
\nonumber\\
&
\leq
(T-t)^{1/2}\int_t^T
\bE\bigg[
L^{1/2}\big(
\big|u^{n,m}(s,X^{t,x}_s)-u^{0,m}(s,X^{t,x}_s)\big|
+
\|w^{n,m}(s,X^{t,x}_s)-w^{0,m}(s,X^{t,x}_s)\|
\big)
\nonumber\\
& \quad
\cdot
\frac{1}{s-t}\Big|
\int_t^s\Big(
\sigma^{-1}(X^{t,x}_r)\frac{\partial}{\partial x_k}X^{t,x}_r
\Big)^T\,dW_r
\Big|
\bigg]
\,ds
\nonumber\\
&
\leq
(T-t)^{1/2}L^{1/2}(1+T^{1/2})\int_t^T
\bE\bigg[
\frac{\big(d^p+\|X^{t,x}_s\|^2\big)^{1/2}}{s-t}\Big|
\int_t^s\Big(
\sigma^{-1}(X^{t,x}_r)\frac{\partial}{\partial x_k}X^{t,x}_r
\Big)^T\,dW_r
\Big|
\nonumber\\
& \quad
\cdot
\frac
{\big|u^{n,m}(s,X^{t,x}_s)-u^{0,m}(s,X^{t,x}_s)\big|
+
(T-s)^{1/2}\|w^{n,m}(s,X^{t,x}_s)-w^{0,m}(s,X^{t,x}_s)\|}
{\big(d^p+\|X^{t,x}_s\|^2\big)^{1/2}}
\bigg]
\,ds
\nonumber\\
&
\leq
(T-t)^{1/2}L^{1/2}(1+T^{1/2})
\nonumber\\
& \quad
\cdot
\int_t^T\frac{E^{n,m}(s)}{(T-s)^{1/2}(s-t)}
\big(\bE\big[d^p+\|X^{t,x}_s\|^2\big]\big)^{\frac{1}{2}}
\bigg(
\bE\bigg[\int_t^s
\Big\|\sigma^{-1}(X^{t,x}_r)\frac{\partial}{\partial x_k}X^{t,x}_r\Big\|^2\bigg]
\,dr
\bigg)^{\frac{1}{2}}\,ds
\nonumber\\
&
\leq
(T-t)^{1/2}[L(1+C_{(d,2)})]^{1/2}(1+T^{1/2})(d^p+\|x\|^2)^{1/2}
\nonumber\\
& \quad
\cdot
\int_t^T
\frac{E^{n,m}(s)}{(T-s)^{1/2}(s-t)}
\left(
\int_t^s
\varepsilon_d^{-1}
\bE\left[\Big\|\frac{\partial}{\partial x_k}X^{t,x}_r\Big\|^2\right]
dr
\right)^{\frac{1}{2}}
ds
\nonumber\\
&
\leq
(T-t)^{1/2}\big[L(1+C_{(d,2)})\varepsilon_d^{-1}C_{d,0}\big]^{1/2}(1+T^{1/2})
(d^p+\|x\|^2)^{1/2}
\int_t^T(T-s)^{-1/2}(s-t)^{-1/2}E^{n,m}(s)\,ds.
\label{est B k m n 1 2}
\end{align}
Furthermore, by \eqref{sigma inverse est}, \eqref{moment est partial},
\eqref{prob 1 tau k n}, \eqref{est sup X n}, \eqref{tau t x k n est},
It\^o's isometry, and Cauchy-Schwarz inequality it holds for all
$m,n\in\bN$, $k\in\{1,2,\dots,d\}$, and $(t,x)\in[0,T)\times\bR^d$ that
\begin{align}
&
B^{k,m,n}_2(t,x)
\nonumber\\
&
\leq(T-t)^{1/2}
\bigg[\sup_{(s,v)\in[0,T]\times\bR^{2d+1}}\big|f^{(m)}(s,v)\big|\bigg]
\nonumber\\
& \quad
\cdot
\int_t^T\frac{1}{s-t}
\bE\left[
\mathbf{1}_{\{\tau^{t,x,k}_n<t\}}
\Big\|
\int_t^s\Big[
\sigma^{-1}(X^{t,x}_r)
\Big(
\frac{\partial}{\partial x_k}X^{t,x,(n)}_r
-
\frac{\partial}{\partial x_k}X^{t,x}_r
\Big)
\Big]^T\,dW_r
\Big\|
\right]ds
\nonumber\\
&
\leq
(T-t)^{1/2}
\bigg[\sup_{(s,v)\in[0,T]\times\bR^{2d+1}}\big|f^{(m)}(s,v)\big|\bigg]
\nonumber\\
& \quad
\cdot
\int_t^T
\frac{\big[\bP(\tau^{t,x,k}_n<T)\big]^{1/2}}{s-t}
\left(
\bE\bigg[
\int_t^s\Big\|
\frac{\partial}{\partial x_k}X^{t,x,(n)}_r
-
\frac{\partial}{\partial x_k}X^{t,x}_r
\Big\|^2\,dr
\bigg]
\right)^{1/2}
ds
\nonumber\\
&
\leq
(T-t)^{1/2}
\bigg[\sup_{(s,v)\in[0,T]\times\bR^{2d+1}}\big|f^{(m)}(s,v)\big|\bigg]
\big[\bP(\tau^{t,x,k}_n<T)\big]^{1/2}
\nonumber\\
& \quad
\cdot
\int_t^T\frac{1}{s-t}
\left[
\bigg(
\int_t^s\varepsilon_d^{-1}
\bE\left[\Big\|\frac{\partial}{\partial x_k}X^{t,x,(n)}_r\Big\|^2\right]dr
\bigg)^{1/2}
+
\bigg(
\int_t^s\varepsilon_d^{-1}
\bE\left[\Big\|\frac{\partial}{\partial x_k}X^{t,x}_r\Big\|^2\right]dr
\bigg)^{1/2}
\right]ds
\nonumber\\
&
\leq
(T-t)^{1/2}
\bigg[\sup_{(s,v)\in[0,T]\times\bR^{2d+1}}\big|f^{(m)}(s,v)\big|\bigg]
\big[\bP(\tau^{t,x,k}_n<T)\big]^{1/2}
\varepsilon_d^{-1/2}(C_{d,0}^{1/2}+c_{d,1}^{1/2})\int_t^T(s-t)^{-1/2}\,ds
\nonumber\\
&
\leq 
2T\varepsilon_d^{-1/2}(C_{d,0}^{1/2}+c_{d,1}^{1/2})
\bigg[\sup_{(s,v)\in[0,T]\times\bR^{2d+1}}\big|f^{(m)}(s,v)\big|\bigg]
\frac{\big[(C_{d,0}+c_{d,1})+(C_{(d,2)}+c_{d,1})(d^p+\|x\|^2)\big]^{1/2}}{n}.
\label{est B k m n 2}
\end{align}
By \eqref{sigma inverse est}, \eqref{prob 1 tau k n}, \eqref{est sup X n}, 
\eqref{tau t x k n est}, \eqref{sigma n inverse est},
It\^o's isometry, and Cauchy-Schwarz inequality we have for all
$m,n\in\bN$, $k\in\{1,2,\dots,d\}$, and $(t,x)\in[0,T)\times\bR^d$ that
\begin{align}
B^{k,m,n}_3(t,x)
&
\leq
(T-t)^{1/2}
\bigg[\sup_{(s,v)\in[0,T]\times\bR^{2d+1}}\big|f^{(m)}(s,v)\big|\bigg]
\nonumber\\
& \quad
\cdot
\int_t^T\frac{1}{s-t}\bE\bigg[\mathbf{1}_{\{\tau^{t,x,k}_n<T\}}
\Big|
\int_t^s
\Big[
\Big(
\big[\sigma^{(n)}(X^{t,x,(n)}_r)\big]^{-1}-\sigma^{-1}(X^{t,x}_r)
\Big)
\frac{\partial}{\partial x_k}X^{t,x,(n)}_r
\Big]^T\,dW_r
\Big|
\bigg]\,ds
\nonumber\\
&
\leq
(T-t)^{1/2}
\bigg[\sup_{(s,v)\in[0,T]\times\bR^{2d+1}}\big|f^{(m)}(s,v)\big|\bigg]
\int_s^T\frac{1}{s-t}\big[\bP(\tau^{t,x,k}_n<T)\big]^{1/2}
\nonumber\\
& \quad
\cdot
\left(
\bE\bigg[
\int_t^s
\Big\|
\Big(
\big[\sigma^{(n)}(X^{t,x,(n)}_r)\big]^{-1}-\sigma^{-1}(X^{t,x}_r)
\Big)
\frac{\partial}{\partial x_k}X^{t,x,(n)}_r
\Big\|^2
\,dr
\bigg]
\right)^{1/2}ds
\nonumber\\
&
\leq
(T-t)^{1/2}
\bigg[\sup_{(s,v)\in[0,T]\times\bR^{2d+1}}\big|f^{(m)}(s,v)\big|\bigg]
\big[\bP(\tau^{t,x,k}_n<T)\big]^{1/2}
\nonumber\\
& \quad
\cdot
\int_t^T\frac{2}{s-t}
\left(
\int_t^s\varepsilon_d^{-1}
\bE\left[\Big\|\frac{\partial}{\partial x_k}X^{t,x,(n)}_r\Big\|^2\right]dr
\right)^{1/2}ds
\nonumber\\
&
\leq
2(T-t)^{1/2}
\bigg[\sup_{(s,v)\in[0,T]\times\bR^{2d+1}}\big|f^{(m)}(s,v)\big|\bigg]
\big[\bP(\tau^{t,x,k}_n<T)\big]^{1/2}(\varepsilon_d^{-1}c_{d,1})^{1/2}
\int_t^T(s-t)^{-1/2}\,ds
\nonumber\\
&
\leq
4T(\varepsilon_d^{-1}c_{d,1})^{1/2}
\bigg[\sup_{(s,v)\in[0,T]\times\bR^{2d+1}}\big|f^{(m)}(s,v)\big|\bigg]
\frac{\big[(C_{d,0}+c_{d,1})+(C_{(d,2)}+c_{d,1})(d^p+\|x\|^2)\big]^{1/2}}{n}.
\label{est B k m n 3}
\end{align}
Then combining \eqref{ineq B k m n}, \eqref{ineq B k m n 1}, 
\eqref{est B k m n 1 1}, \eqref{est B k m n 1 2}, \eqref{est B k m n 2},
and \eqref{est B k m n 3} yields for all 
$m,n\in\bN$, $k\in\{1,2,\dots,d\}$, and $(t,x)\in[0,T)\times\bR^d$ that
\begin{align}
&
(T-t)^{1/2}\int_t^T\bE\bigg[\Big|
f^{(m)}\big(s,X^{t,x}_s,u^{0,m}(s,X^{t,x}_s),w^{0,m}(s,X^{t,x}_s)\big)
\frac{1}{s-t}\int_t^s\Big(\big[\sigma(X^{t,x}_r)\big]^{-1}
\frac{\partial}{\partial x_k}X^{t,x}_r\Big)^T\,dW_r
\nonumber\\
&
-
f^{(m)}\big(s,X^{t,x,(n)}_s,u^{n,m}(s,X^{t,x,(n)}_s),w^{n,m}(s,X^{t,x,(n)}_s)\big)
\frac{1}{s-t}\int_t^s\Big(\big[\sigma^{(n)}(X^{t,x,(n)}_r)\big]^{-1}
\frac{\partial}{\partial x_k}X^{t,x,(n)}_r\Big)^T\,dW_r
\Big|\bigg]\,ds
\nonumber\\
&
\leq 
\mathfrak{c}_{d,4}\bigg[\sup_{(s,v)\in[0,T]\times\bR^{2d+1}}\big|f^{(m)}(s,v)\big|\bigg]
\frac{(d^p+\|x\|^2)^{1/2}}{n}
\nonumber\\
& \quad
+
\mathfrak{c}_{d,5}(d^p+\|x\|^2)^{1/2}
\int_t^T(T-s)^{-1/2}(s-t)^{-1/2}E^{n,m}(s)\,ds,
\label{whisky 4}
\end{align}
where
$$
\mathfrak{c}_{d,4}=16T\varepsilon_d^{-1/2}(C_{d,0}\vee C_{(d,2)}+c_{d,1}),
$$
and
$$
\mathfrak{c}_{d,5}=
T^{1/2}\big[L(1+C_{(d,2)})\varepsilon_d^{-1}C_{d,0}\big]^{1/2}(1+T^{1/2})
\big[L(1+C_{(d,2)})\varepsilon_d^{-1}C_{d,0}\big]^{1/2}(1+T^{1/2}).
$$
By \eqref{whisky 1}, \eqref{whisky 2}, \eqref{whisky 3}, and \eqref{whisky 4}
we establish for all $m,n\in\bN$ and $t\in[0,T]$ that
\begin{align*}
E^{n,m}(t)
\leq 
&
\frac{\mathfrak{c}_{d,6}}{n}
\left(
\bigg[\sup_{y\in\bR^d}\big|g^{(m)}(y)\big|\bigg]
+
\bigg[\sup_{(s,v)\in[0,T]\times\bR^{2d+1}}\big|f^{(m)}(s,v)\big|\bigg]
\right)
\nonumber\\
&
+
(\mathfrak{c}_{d,3}+\mathfrak{c}_{d,5}d)
\int_t^T
\left[
(T-s)^{-1/2}+(T-s)^{-1/2}(s-t)^{-1/2}
\right]
E^{n,m}(s)
\,ds,
\end{align*}
where
$$
\mathfrak{c}_{d,6}=2(C_{d,0}\vee C_{(d,2)}+c_{d,1})+\mathfrak{c}_{d,2}
+(\mathfrak{c}_{d,1}+\mathfrak{c}_{d,4})d.
$$
This together with \eqref{time integral est 1/2} and Gr\"onwall's lemma
imply for all $m,n\in\bN$ and $t\in[0,T]$ that
\begin{align*}
E^{n,m}(t)
\leq 
&
\frac{\mathfrak{c}_{d,6}}{n}
\left(
\bigg[\sup_{y\in\bR^d}\big|g^{(m)}(y)\big|\bigg]
+
\bigg[\sup_{(s,v)\in[0,T]\times\bR^{2d+1}}\big|f^{(m)}(s,v)\big|\bigg]
\right)
\exp\big\{2(\mathfrak{c}_{d,3}+\mathfrak{c}_{d,5}d)(T^{1/2}+2)\big\}.
\end{align*}
This ensures for all $m\in\bN$ and $t\in[0,T]$ that
$
\lim_{n\to\infty}E^{n,m}(t)=0.
$
Therefore, by \eqref{def E n m s} it holds for all $m\in\bN$ 
and every compact set $\cK\subseteq(0,T)\times\bR^d$ that
\begin{equation}
\label{com conv u w n m}
\lim_{n\to\infty}
\sup_{(t,x)\in\cK}
\big[
|u^{n,m}(t,x)-u^{0,m}(t,x)|+\|w^{n,m}(t,x)-w^{0,m}(t,x)\|
\big]=0.
\end{equation}
By \eqref{conv sigma n}, \eqref{Lip sigma n}, \eqref{conv mu n}, \eqref{Lip mu n},
\eqref{Lip g n f n 0}, \eqref{Lip g n}, \eqref{Lip f n}, \eqref{LG f g n}, 
\eqref{PDE n m k}, and \eqref{com conv u w n m},
Lemma \ref{v solution approximation} ensures for all
$m\in\bN$ that $u^{0,m}$ is a viscosity solution of the PDE
\begin{align}
&
\frac{\partial}{\partial t}u^{0,m}(t,x)
+\langle\nabla_x u^{0,m}(t,x),\mu(t,x)\rangle
+\frac{1}{2}\operatorname{Trace}
\left(\sigma(t,x)\sigma^T(t,x)
\operatorname{Hess}_xu^{0,m}(t,x)\right)\nonumber\\
&
+f^{(m)}\big(t,x,u^{0,m}(t,x), \nabla_x u^{0,m}(t,x)\big)=0 .        
\label{PDE m k}
\end{align}
Moreover, by \eqref{com conv u w n m}, the fact that 
$\nabla_x u^{n,m}(t,x)=w^{n,m}(t,x)$ for all $n,m\in\bN$ and $(t,x)\in[0,T)\times\bR^d$,
and e.g., \cite[Section 16.3.5, Theorem 4]{zorich2016mathematical} we have for all
$(t,x)\in[0,T)\times\bR^d$ that $\nabla_x u^{0,m}(t,x)$ exists and coincides with
$w^{0,m}(t,x)$.
Next, the application of Corollary \ref{corollary FP} yields that
there exists a unique pair of Borel functions $(u,w)$ such that
$u\in C_{lin}([0,T)\times\bR^d,\bR)$, 
$w\in C([0,T)\times\bR^d,\bR^d)$, 
and
\begin{align}
u(t,x)
&
=\bE\left[g(X^{t,x}_T)\right]     
+\int_t^T\bE\left[
f\big(s,X^{t,x}_s,u(s,X^{t,x}_s),w(s,X^{t,x}_s)\big)
\right]ds
\label{BEL u}
\end{align}
and
\begin{align}
w(t,x)
&
=\bE\left[
g(X^{t,x}_T)
\frac{1}{T-t}
\int_t^T\left[\sigma^{-1}(X^{t,x}_{r})
DX^{t,x}_{r}\right]^T\,dW_r
\right]  
\nonumber\\
& \quad    
+\int_t^T\bE\Bigg[
f\big(s,X^{t,x}_s,u(s,X^{t,x}_s),w(s,X^{t,x}_s)\big)
\frac{1}{s-t}
\int_t^s\left[\sigma^{-1}(X^{t,x}_{r})
DX^{t,x}_{r}\right]^T\,dW_r
\Bigg]\,ds
\label{BEL grad}
\end{align}
for all $(t,x)\in[0,T)\times\bR^d$.
Moreover, Corollary \ref{corollary FP} also ensures that
\begin{equation}
\label{growth u w}
K_0:=\sup_{(s,y)\in[0,T)\times\bR^d}
\left(\frac{|u(s,y)|+(T-s)^{1/2}\|w(s,y)\|}{(d^p+\|y\|^2)^{1/2}}
\right)<\infty.
\end{equation}
Then by \eqref{SDE moment est}, \eqref{est g n g local}, and H\"older's inequality,
we observe for all $m\in\bN$, $(t,x)\in[0,T)\times\bR^d$, 
and $\epsilon\in(0,1)$ that
\begin{align}
\bE\left[\big|g^{(m)}(X^{t,x}_T)-g(X^{t,x}_T)\big|\right]
&
=
\bE
\left[
\frac{\big|g^{(m)}(X^{t,x}_T)-g(X^{t,x}_T)\big|}
{\big(d^p+\|X^{t,x}_T\|\big)^{\frac{1+\epsilon}{2}}}
\big(d^p+\|X^{t,x}_T\|\big)^{\frac{1+\epsilon}{2}}
\right]
\nonumber\\
&
\leq
\bigg[\sup_{y\in\bR^d}
\frac{|g^{(m)}(y)-g(y)|}{(d^p+\|y\|^2)^{\frac{1+\epsilon}{2}}}\bigg]
\left(\bE\left[d^p+\|X^{t,x}_T\|^2\right]\right)^{\frac{1+\epsilon}{2}}
\nonumber\\
&
\leq 
\frac{L^{1/2}}{(d^p+n^2)^{\frac{\epsilon}{2}}}(C_{(d,2)}+1)^{\frac{1+\epsilon}{2}}
(d^p+\|x\|^2)^{\frac{1+\epsilon}{2}}.
\label{rum 1}
\end{align}
By \eqref{sigma inverse est}, \eqref{SDE moment est}, \eqref{moment est partial},
\eqref{est g n g local}, Burkholder-Davis-Gundy inequality, and H\"older inequality,
we also have for all $m\in\bN$, $k\in\{1,2,\dots,d\}$, $(t,x)\in[0,T)\times\bR^d$,
and $\epsilon\in(0,1)$ that
\begin{align}
&
(T-t)^{1/2}\bE
\left[
\big|g^{(m)}(X^{t,x}_T)-g(X^{t,x}_T)\big|
\cdot
\frac{1}{T-t}
\bigg|\int_t^T\bigg[
\sigma^{-1}(X^{t,x}_r)\frac{\partial}{\partial x_k}X^{t,x}_r
\bigg]^T\,dW_r\bigg|
\right]
\nonumber\\
&
=(T-t)^{-1/2}\bE
\left[
\frac
{\big|g^{(m)}(X^{t,x}_T)-g(X^{t,x}_T)\big|}
{\big(d^p+\|X^{t,x}_T\|^2\big)^{\frac{1+\epsilon}{2}}}
\big(d^p+\|X^{t,x}_T\|^2\big)^{\frac{1+\epsilon}{2}}
\bigg|\int_t^T\bigg[
\sigma^{-1}(X^{t,x}_r)\frac{\partial}{\partial x_k}X^{t,x}_r
\bigg]^T\,dW_r\bigg|
\right]
\nonumber\\
&
\leq
(T-t)^{-1/2}\bigg[\sup_{y\in\bR^d}
\frac{|g^{(m)}(y)-g(y)|}{(d^p+\|y\|^2)^{\frac{1+\epsilon}{2}}}\bigg]
\left(\bE\left[d^p+\|X^{t,x}_T\|^2\right]\right)^{\frac{1+\epsilon}{2}}
\nonumber\\
& \quad
\cdot
\left(
\bE\left[
\bigg|\int_t^T\bigg[
\sigma^{-1}(X^{t,x}_r)\frac{\partial}{\partial x_k}X^{t,x}_r
\bigg]^T\,dW_r\bigg|^{\frac{2}{1-\epsilon}}
\right]
\right)^{\frac{1-\epsilon}{2}}
\nonumber\\
&
\leq
\frac{(T-t)^{-1/2}L^{1/2}}{(d^p+n^2)^{\frac{\epsilon}{2}}}
(1+C_{(d,2)})^{\frac{1+\epsilon}{2}}(d^p+\|x\|^2)^{\frac{1+\epsilon}{2}}
\frac{8}{1-\epsilon}
\left(
\bE\left[
\int_t^T
\Big\|
\sigma^{-1}(X^{t,x}_r)\frac{\partial}{\partial x_k}X^{t,x}_r
\Big\|^2
\,dr
\right]
\right)^{1/2}
\nonumber\\
&
\leq
\frac{(T-t)^{-1/2}L^{1/2}}{(d^p+n^2)^{\frac{\epsilon}{2}}}
\frac{8}{1-\epsilon}(1+C_{(d,2)})^{\frac{1+\epsilon}{2}}
(d^p+\|x\|^2)^{\frac{1+\epsilon}{2}}
\left(\int_t^T\varepsilon_d^{-1}
\bE\left[
\Big\|
\frac{\partial}{\partial x_k}X^{t,x}_r
\Big\|^2
\right]dr\right)^{1/2}
\nonumber\\
&
\leq
c_{\epsilon,1}^{(d)}(d^p+n^2)^{-\frac{\epsilon}{2}}(d^p+\|x\|^2)^{\frac{1+\epsilon}{2}},
\label{rum 2}
\end{align}
where
$$
c_{\epsilon,1}^{(d)}
:=
8(LC_{d,0}\varepsilon_d^{-1})^{1/2}(1-\epsilon)^{-1}
(1+C_{(d,2)})^{\frac{1+\epsilon}{2}}.
$$
Furthermore, we notice for all $m\in\bN$ and $(t,x)\in[0,T)\times\bR^d$ that
\begin{align}
&
\int_t^T
\bE\left[
\big|
f^{(m)}\big(s,X^{t,x}_s,u^{0,m}(s,X^{t,x}_s),w^{0,m}(s,X^{t,x}_s)\big)
-
f\big(s,X^{t,x}_s,u(s,X^{t,x}_s),w(s,X^{t,x}_s)\big)
\big|
\right]ds
\nonumber\\
&
\leq A^m_1(t,x)+A^m_2(t,x),
\label{ineq A m 1 2}
\end{align}
where
$$
A^m_1(t,x):=
\bE\left[
\big|
f^{(m)}\big(s,X^{t,x}_s,u(s,X^{t,x}_s),w(s,X^{t,x}_s)\big)
-
f\big(s,X^{t,x}_s,u(s,X^{t,x}_s),w(s,X^{t,x}_s)\big)
\big|
\right]ds,
$$
and
$$
A^m_2(t,x):=
\bE\left[
\big|
f^{(m)}\big(s,X^{t,x}_s,u^{0,m}(s,X^{t,x}_s),w^{0,m}(s,X^{t,x}_s)\big)
-
f^{(m)}\big(s,X^{t,x}_s,u(s,X^{t,x}_s),w(s,X^{t,x}_s)\big)
\big|
\right]ds.
$$
By \eqref{SDE moment est}, \eqref{est f n f local}, \eqref{growth u w}, 
and H\"older's inequality we obtain
for all $(t,x)\in[0,T)\times\bR^d$, $m\in\bN$, and $\epsilon\in(0,1)$ that
\begin{align}
&
A^m_1(t,x)
\nonumber\\
&
\leq
\int_t^T
\bE\Bigg[
\frac
{
\big|
f^{(m)}\big(s,X^{t,x}_s,u(s,X^{t,x}_s),w(s,X^{t,x}_s)\big)
-
f\big(s,X^{t,x}_s,u(s,X^{t,x}_s),w(s,X^{t,x}_s)\big)
\big|
}
{
\big(d^p+\|X^{t,x}_s\|^2+|u(s,X^{t,x}_s)|^2+\|w(s,X^{t,x}_s)\|^2\big)
^{\frac{1+\epsilon}{2}}
}
\nonumber\\
& \quad
\cdot
\big(d^p+\|X^{t,x}_s\|^2+|u(s,X^{t,x}_s)|^2+\|w(s,X^{t,x}_s)\|^2\big)
^{\frac{1+\epsilon}{2}}
\Bigg]ds
\nonumber\\
&
\leq
\bigg[
\sup_{(s,v)\in[0,T]\times\bR^{2d+1}}
\frac{\big|f^{(m)}(s,v)-f(s,v)\big|}{(d^p+\|v\|^2)^{\frac{1+\epsilon}{2}}}
\bigg]
\nonumber\\
& \quad
\cdot
\int_t^T
\bE\left[
\big(d^p+\|X^{t,x}_s\|^2\big)^{\frac{1+\epsilon}{2}}
+
\big(|u(s,X^{t,x}_s)|^2+\|w(s,X^{t,x}_s)\|^2\big)^{\frac{1+\epsilon}{2}}
\right]
ds
\nonumber\\
&
\leq
\frac{(Ld^p)^{1/2}}{(d^p+n^2)^{\frac{\epsilon}{2}}}
\Bigg(
T(1+C_{(d,2)})^{\frac{1+\epsilon}{2}}(d^p+\|x\|^2)^{\frac{1+\epsilon}{2}}
+(1+T^{1/2})^{1+\epsilon}
\nonumber\\
&
\cdot
\int_t^T
\frac{\big(|u(s,X^{t,x}_s)|+(T-s)^{1/2}\|w(s,X^{t,x}_s)\|\big)^{1+\epsilon}}
{\big(d^p+\|X^{t,x}_s\|^2\big)^{\frac{1+\epsilon}{2}}}
\frac{\big(d^p+\|X^{t,x}_s\|^2\big)^{\frac{1+\epsilon}{2}}}
{(T-s)^{\frac{1+\epsilon}{2}}}\,ds
\Bigg)
\nonumber\\
&
\leq
\frac{L^{1/2}}{(d^p+n^2)^{\frac{\epsilon}{2}}}
\Bigg(
T(1+C_{(d,2)})^{\frac{1+\epsilon}{2}}(d^p+\|x\|^2)^{\frac{1+\epsilon}{2}}
+(1+T^{1/2})^{1+\epsilon}
\nonumber\\
& \quad 
\cdot
\bigg[
\sup_{(s,y)\in[0,T)\times\bR^d}
\frac{|u(s,y)|+(T-s)^{1/2}\|w(s,y)\|}{(d^p+\|y\|^2)^{1/2}}
\bigg]^{1+\epsilon}
\int_t^T
\left(\bE\left[d^p+\|X^{t,x}_s\|^2\right]\right)^{\frac{1+\epsilon}{2}}
(T-s)^{-\frac{1+\epsilon}{2}}
\Bigg)
\nonumber\\
&
\leq
\frac{L^{1/2}}{(d^p+n^2)^{\frac{\epsilon}{2}}}
(1+C_{(d,2)})^{\frac{1+\epsilon}{2}}(d^p+\|x\|^2)^{\frac{1+\epsilon}{2}}
\left[
T+(1+T^{1/2})^{1+\epsilon}K_0^{1+\epsilon}
\int_t^T(T-s)^{-\frac{1+\epsilon}{2}}\,ds
\right]
\nonumber\\
&
=
\frac{L^{1/2}}{(d^p+n^2)^{\frac{\epsilon}{2}}}
(1+C_{(d,2)})^{\frac{1+\epsilon}{2}}(d^p+\|x\|^2)^{\frac{1+\epsilon}{2}}
\left[
T+(1+T^{1/2})^{1+\epsilon}K_0^{1+\epsilon}
\frac{2}{1-\epsilon}(T-t)^{\frac{1-\epsilon}{2}}
\right].
\label{est A m 1}
\end{align}
For every $m\in\bN$, $t\in[0,T)$, and $\epsilon\in(0,1)$ we define
\begin{equation}
\label{def E m epsilon}
E^{m,\epsilon}(t)
:=
\sup_{(s,y)\in[0,T)\times\bR^d}
\frac{|u^{0,m}(s,y)-u(s,y)|+(T-s)^{1/2}\|w^{0,m}(s,y)-w(s,y)\|}
{(d^p+\|y\|^2)^{\frac{1+\epsilon}{2}}}.
\end{equation}
Then by \eqref{SDE moment est}, \eqref{Lip f n}, and H\"older inequality,
it holds for all $m\in\bN$, $(t,x)\in[0,T)\times\bR^d$, and $\epsilon\in(0,1)$ that
\begin{align}
A^m_2(t,x)
&
\leq
\int_t^T
L_2^{1/2}\bE\left[
\big|u^{0,m}(s,X^{t,x}_s)-u(s,X^{t,x}_s)\big|
+
\big\|w^{0,m}(s,X^{t,x}_s)-w(s,X^{t,x}_s)\big\|
\right]
ds
\nonumber\\
&
\leq 
L_2^{\frac{1}{2}}(1+T^{\frac{1}{2}})
\int_t^T
\bigg[
\frac
{
\big|u^{0,m}(s,X^{t,x}_s)-u(s,X^{t,x}_s)\big|
+
(T-s)^{1/2}\big\|w^{0,m}(s,X^{t,x}_s)-w(s,X^{t,x}_s)\big\|
}
{\big(d^p+\|X^{t,x}_s\|^2\big)^{\frac{1+\epsilon}{2}}}
\nonumber\\
& \quad
\cdot
\frac{\big(d^p+\|X^{t,x}_s\|^2\big)^{\frac{1+\epsilon}{2}}}{(T-s)^{1/2}}
\bigg]\,ds
\nonumber\\
& 
\leq
L_2^{\frac{1}{2}}(1+T^{\frac{1}{2}})
\int_t^T
E^{m,\epsilon}(s)
\left(\bE\left[d^p+\|X^{t,x}_s\|^2\right]\right)^{\frac{1+\epsilon}{2}}
(T-s)^{-\frac{1}{2}}\,ds
\nonumber\\
&
\leq
L_2^{\frac{1}{2}}(1+T^{\frac{1}{2}})(1+C_{(d,2)})^{\frac{1+\epsilon}{2}}
(d^p+\|x\|^2)^{\frac{1+\epsilon}{2}}
\int_t^T
(T-s)^{-\frac{1}{2}}
E^{m,\epsilon}(s)
\,ds
\label{est A m 2}
\end{align}
Combining \eqref{ineq A m 1 2}, \eqref{est A m 1}, and \eqref{est A m 2} 
shows for all $m\in\bN$, $(t,x)\in[0,T)\times\bR^d$, and $\epsilon\in(0,1)$ that
\begin{align}
&
\int_t^T
\bE\left[
\big|
f^{(m)}\big(s,X^{t,x}_s,u^{0,m}(s,X^{t,x}_s),w^{0,m}(s,X^{t,x}_s)\big)
-
f\big(s,X^{t,x}_s,u(s,X^{t,x}_s),w(s,X^{t,x}_s)\big)
\big|
\right]ds
\nonumber\\
&
\leq 
c_{\epsilon,2}^{(d)}(d^p+n^2)^{-\frac{\epsilon}{2}}
(d^p+\|x\|^2)^{\frac{1+\epsilon}{2}}
+
\mathfrak{c}_{\epsilon,1}^{(d)}(d^p+\|x\|^2)^{\frac{1+\epsilon}{2}}
\int_t^T(T-s)^{-\frac{1}{2}}E^{m,\epsilon}(s)\,ds,
\label{rum 3}
\end{align}
where
$$
c_{\epsilon,2}^{(d)}
:=
(Ld^p)^{1/2}
(1+C_{(d,2)})^{\frac{1+\epsilon}{2}}
\big[
T+2(1+T^{1/2})^{1+\epsilon}K_0^{1+\epsilon}
(1-\epsilon)^{-1}T^{\frac{1-\epsilon}{2}}
\big],
$$
and
$$
\mathfrak{c}_{\epsilon,1}^{(d)}
:=
L_2^{1/2}(1+T^{1/2})
(1+C_{(d,2)})^{\frac{1+\epsilon}{2}}.
$$
Next, we notice for all $m\in\bN$ and $(t,x)\in[0,T)\times\bR^d$ that
\begin{align}
&
(T-t)^{1/2}\int_t^T\bE\bigg[\Big|
f^{(m)}\big(s,X^{t,x}_s,u^{0,m}(s,X^{t,x}_s),w^{0,m}(s,X^{t,x}_s)\big)
\frac{1}{s-t}\int_t^s\Big[\sigma^{-1}(X^{t,x}_r)
\frac{\partial}{\partial x_k}X^{t,x}_r\Big]^T\,dW_r
\nonumber\\
&
-
f\big(s,X^{t,x}_s,u(s,X^{t,x}_s),w(s,X^{t,x}_s)\big)
\frac{1}{s-t}\int_t^s\Big[\sigma^{-1}(X^{t,x}_r)
\frac{\partial}{\partial x_k}X^{t,x}_r\Big]^T\,dW_r
\Big|\bigg]\,ds
\nonumber\\
&
\leq \mathfrak{B}^{k,m}_1(t,x)+\mathfrak{B}^{k,m}_2(t,x),
\label{ineq B k m}
\end{align}
where
\begin{align*}
&
\mathfrak{B}^{k,m}_1(t,x)
:=(T-t)^{1/2}\int_t^T
\bE\bigg[\Big|
\big[
f^{(m)}\big(s,X^{t,x}_s,u(s,X^{t,x}_s),w(s,X^{t,x}_s)\big)
\\
& \qquad \qquad \qquad \; \;
-
f\big(s,X^{t,x}_s,u(s,X^{t,x}_s),w(s,X^{t,x}_s)\big)
\big]
\frac{1}{s-t}
\int_t^s\Big[\sigma^{-1}(X^{t,x}_r)
\frac{\partial}{\partial x_k}X^{t,x}_r\Big]^T\,dW_r
\Big|\bigg]\,ds,
\end{align*}
and
\begin{align*}
&
\mathfrak{B}^{k,m}_2(t,x)
:=(T-t)^{1/2}\int_t^T
\bE\bigg[\Big|
\big[
f^{(m)}\big(s,X^{t,x,(n)}_s,u^{0,m}(s,X^{t,x}_s),w^{0,m}(s,X^{t,x}_s)\big)
\\
& \qquad \qquad \qquad \; \;
-
f^{(m)}\big(s,X^{t,x}_s,u(s,X^{t,x}_s),w(s,X^{t,x}_s)\big)
\big]
\frac{1}{s-t}
\int_t^s\Big[\sigma^{-1}(X^{t,x}_r)
\frac{\partial}{\partial x_k}X^{t,x}_r\Big]^T\,dW_r
\Big|\bigg]\,ds.
\end{align*}
By \eqref{sigma inverse est}, \eqref{moment est partial},
\eqref{est f n f local}, \eqref{growth u w},
Burkholder-Davis-Gundy inequality, and H\"older's inequality,
it holds for all $(t,x)\in[0,T)\times\bR^d$, $m\in\bN$, and $\epsilon\in(0,1)$
that
\begin{align}
&
\mathfrak{B}^{k,m}_1(t,x)
\nonumber\\
&
\leq
(T-t)^{1/2}
\int_t^T\bE\bigg[
\frac
{
\big|f^{(m)}\big(s,X^{t,x}_s,u(s,X^{t,x}_s),w(s,X^{t,x}_s)\big)
-f\big(s,X^{t,x}_s,u(s,X^{t,x}_s),w(s,X^{t,x}_s)\big)\big|
}
{
\big(d^p+\|X^{t,x}_s\|^2+|u(s,X^{t,x}_s)|^2+\|w(s,X^{t,x}_s)\|^2\big)
^{\frac{1+\epsilon}{2}}
}
\nonumber\\
& \quad
\cdot
\big(d^p+\|X^{t,x}_s\|^2+|u(s,X^{t,x}_s)|^2+\|w(s,X^{t,x}_s)\|^2\big)
^{\frac{1+\epsilon}{2}}
\frac{1}{s-t}
\Big|
\int_t^s
\Big[\sigma^{-1}(X^{t,x}_r)\frac{\partial}{\partial x_k}X^{t,x}_r\Big]^T
\,dW_r
\Big|
\bigg]\,ds
\nonumber\\
&
\leq 
(T-t)^{1/2}
\bigg[\sup_{(s,v)\in[0,T]\times\bR^{2d+1}}
\frac{\big|f^{(m)}(s,v)-f(s,v)\big|}{(d^p+\|v\|^2)^{\frac{1+\epsilon}{2}}}
\bigg]
\int_t^T
\bE\bigg[
\frac{1}{s-t}
\Big|
\int_t^s
\Big[\sigma^{-1}(X^{t,x}_r)\frac{\partial}{\partial x_k}X^{t,x}_r\Big]^T
\,dW_r
\Big|
\nonumber\\
& \quad
\cdot
\Big(
\big(d^p+\|X^{t,x}_s\|^2\big)^{\frac{1+\epsilon}{2}}
+
\big(|u(s,X^{t,x}_s)|^2+\|w(s,X^{t,x}_s)\|^2\big)^{\frac{1+\epsilon}{2}}
\Big)
\bigg]\,ds
\nonumber\\
&
\leq
\frac{(T-t)^{1/2}L^{1/2}}{(d^p+n^2)^{\frac{\epsilon}{2}}}
\bigg(
\int_t^T\frac{1}{s-t}\bE\left[
\big(d^p+\|X^{t,x}_s\|^2\big)^{\frac{1+\epsilon}{2}}
\Big|
\int_t^s
\Big[\sigma^{-1}(X^{t,x}_r)\frac{\partial}{\partial x_k}X^{t,x}_r\Big]^T
\,dW_r
\Big|
\right]ds
\nonumber\\
& \quad
+(1+T^{1/2})^{1+\epsilon}\int_t^T
\bE\bigg[
\frac
{
\big(|u(s,X^{t,x}_s)|+(T-s)^{1/2}\|w(s,X^{t,x}_s)\|\big)^{1+\epsilon}
\big(d^p+\|X^{t,x}_s\|^2\big)^{\frac{1+\epsilon}{2}}
}
{
\big(d^p+\|X^{t,x}_s\|^2\big)^{\frac{1+\epsilon}{2}}(T-s)^{\frac{1+\epsilon}{2}}
}
\nonumber\\
& \quad
\cdot
\frac{1}{s-t}
\Big|
\int_t^s
\Big[\sigma^{-1}(X^{t,x}_r)\frac{\partial}{\partial x_k}X^{t,x}_r\Big]^T
\,dW_r
\Big|
\bigg]\,ds
\bigg)
\nonumber\\
&
\leq
\frac{(T-t)^{1/2}L^{1/2}}{(d^p+n^2)^{\frac{\epsilon}{2}}}
\bigg[
\int_t^T\frac{1}{s-t}\left(\bE
\big[d^p+\|X^{t,x}_s\|^2\big]\right)^{\frac{1+\epsilon}{2}}
\bigg(\bE\bigg[
\Big|
\int_t^s
\Big[\sigma^{-1}(X^{t,x}_r)\frac{\partial}{\partial x_k}X^{t,x}_r\Big]^T
\,dW_r
\Big|^{\frac{2}{1-\epsilon}}
\bigg]\bigg)^{\frac{1-\epsilon}{2}}ds
\nonumber\\
& \quad
+(1+T^{1/2})^{1+\epsilon}
\bigg[\sup_{[0,T)\times\bR^d}
\frac{|u(s,y)|+(T-s)^{1/2}\|w(s,y)\|}{(d^p+\|y\|^2)^{1/2}}
\bigg]^{1+\epsilon}
\nonumber\\
& \quad
\cdot
\int_t^T\frac{1}{(s-t)(T-s)^{\frac{1+\epsilon}{2}}}
\left(\bE
\big[d^p+\|X^{t,x}_s\|^2\big]\right)^{\frac{1+\epsilon}{2}}
\bigg(\bE\bigg[
\Big|
\int_t^s
\Big[\sigma^{-1}(X^{t,x}_r)\frac{\partial}{\partial x_k}X^{t,x}_r\Big]^T
\,dW_r
\Big|^{\frac{2}{1-\epsilon}}
\bigg]\bigg)^{\frac{1-\epsilon}{2}}ds
\nonumber\\
&
\leq
\frac{(T-t)^{1/2}L^{1/2}}{(d^p+n^2)^{\frac{\epsilon}{2}}}
\bigg[
\int_t^T
\frac{(1+C_{(d,2)})^{\frac{1+\epsilon}{2}}(d^p+\|x\|^2)^{\frac{1+\epsilon}{2}}}{s-t}
\cdot 
\frac{8}{1-\epsilon}
\bigg(
\bE\bigg[
\int_t^s
\Big\|\sigma^{-1}(X^{t,x}_r)\frac{\partial}{\partial x_k}X^{t,x}_r\Big\|^2
\,dr
\bigg]
\bigg)^{\frac{1}{2}}
ds
\nonumber\\
& \quad
+(1+T^{1/2})^{1+\epsilon}K_0^{1+\epsilon}
\int_t^T
\frac{(1+C_{(d,2)})^{\frac{1+\epsilon}{2}}(d^p+\|x\|^2)^{\frac{1+\epsilon}{2}}}
{(s-t)(T-s)^{\frac{1+\epsilon}{2}}}
\cdot 
\frac{8}{1-\epsilon}
\bigg(
\bE\bigg[
\int_t^s
\Big\|\sigma^{-1}(X^{t,x}_r)\frac{\partial}{\partial x_k}X^{t,x}_r\Big\|^2
\,dr
\bigg]
\bigg)^{\frac{1}{2}}
ds
\bigg]
\nonumber\\
&
\leq
\frac
{8(T-t)^{1/2}L^{1/2}(1+C_{(d,2)})^{\frac{1+\epsilon}{2}}
(d^p+\|x\|^2)^{\frac{1+\epsilon}{2}}}
{(d^p+n^2)^{\frac{\epsilon}{2}}(1-\epsilon)}
\Bigg[
\int_t^T\frac{1}{s-t}
\bigg(
\int_t^s\varepsilon_d^{-1}
\bE\left[\Big\|\frac{\partial}{\partial x_k}X^{t,x}_r\Big\|^2\right]dr
\bigg)^{1/2}
ds
\nonumber\\
& \quad
+(1+T^{1/2})^{1+\epsilon}K_0^{1+\epsilon}
\int_t^T\frac{1}{(s-t)(T-s)^{\frac{1+\epsilon}{2}}}
\bigg(
\int_t^s\varepsilon_d^{-1}
\bE\left[\Big\|\frac{\partial}{\partial x_k}X^{t,x}_r\Big\|^2\right]dr
\bigg)^{1/2}
ds
\Bigg]
\nonumber\\
&
\leq
\frac
{8(T-t)^{1/2}
\big(LC_{d,0}\varepsilon_d^{-1}\big)^{1/2}(1+C_{(d,2)})^{\frac{1+\epsilon}{2}}
(d^p+\|x\|^2)^{\frac{1+\epsilon}{2}}}
{(d^p+n^2)^{\frac{\epsilon}{2}}(1-\epsilon)}
\left(
\int_t^T
\left[
(s-t)^{-\frac{1}{2}}
+
(s-t)^{-\frac{1}{2}}(T-s)^{-\frac{1+\epsilon}{2}}
\right]
ds
\right)
\nonumber\\
& 
\leq
\frac
{8(T-t)^{1/2}\big(LC_{d,0}\varepsilon_d^{-1}\big)^{1/2}(1+C_{(d,2)})^{\frac{1+\epsilon}{2}}
(d^p+\|x\|^2)^{\frac{1+\epsilon}{2}}}
{(d^p+n^2)^{\frac{\epsilon}{2}}(1-\epsilon)}
\nonumber\\
& \quad
\cdot
\Bigg[
2(T-t)^{1/2}
+
\int_t^{\frac{T+t}{2}}
(s-t)^{-\frac{1}{2}}\Big(\frac{T-t}{2}\Big)^{-\frac{1+\epsilon}{2}}
ds
+
\int_{\frac{T+t}{2}}^{T}
\Big(\frac{T-t}{2}\Big)^{-\frac{1}{2}}(T-s)^{-\frac{1+\epsilon}{2}}
ds
\Bigg]
\nonumber\\
&
=
\frac
{
16\cdot 2^{\frac{\epsilon}{2}}(T-t)^{\frac{1-\epsilon}{2}}(2-\epsilon)
\big(LC_{d,0}\varepsilon_d^{-1}\big)^{1/2}
(1+C_{(d,2)})^{\frac{1+\epsilon}{2}}
(d^p+\|x\|^2)^{\frac{1+\epsilon}{2}}
}
{(d^p+n^2)^{\frac{\epsilon}{2}}(1-\epsilon)^2}.
\label{est B k m 1}
\end{align}
Furthermore, by \eqref{sigma inverse est}, \eqref{moment est partial},
\eqref{time integral est}, \eqref{Lip f n},
Burkholder-Davis-Gundy inequality, and H\"older's inequality we obtain for all
$m\in\bN$, $k\in\{1,2,\dots,d\}$, $(t,x)\in[0,T)\times\bR^d$, $\epsilon\in(0,1)$,
and $\beta\in(0,1)$ that
\begin{align}
\mathfrak{B}^{k,m}_2(t,x)
&
\leq
(T-t)^{1/2}
\int_t^T
\bE\bigg[
L_2^{1/2}\left(
\big|u^{0,m}(s,X^{t,x}_s)-u(s,X^{t,x}_s)\big|
+
\big\|w^{0,m}(s,X^{t,x}_s)-w(s,X^{t,x}_s)\big\|
\right)
\nonumber\\
& \quad
\cdot
\frac{1}{s-t}
\bigg|
\int_t^s
\Big[\sigma^{-1}(X^{t,x}_r)\frac{\partial}{\partial x_k}X^{t,x}_r\Big]^T
\,dW_r
\bigg|
\bigg]\,ds
\nonumber\\
&
\leq
(T-t)^{1/2}L^{1/2}(1+T^{1/2})
\int_t^T\bE\Bigg[
\bigg|
\int_t^s
\Big[\sigma^{-1}(X^{t,x}_r)\frac{\partial}{\partial x_k}X^{t,x}_r\Big]^T
\,dW_r
\bigg|
\frac
{\big(d^p+\|X^{t,x}_s\|^2\big)^\frac{1+\epsilon}{2}}
{(s-t)(T-s)^{1/2}}
\nonumber\\
& \quad
\cdot
\frac
{
\big|u^{0,m}(s,X^{t,x}_s)-u(s,X^{t,x}_s)\big|
+
(T-s)^{1/2}\big\|w^{0,m}(s,X^{t,x}_s)-w(s,X^{t,x}_s)\big\|
}
{\big(d^p+\|X^{t,x}_s\|^2\big)^\frac{1+\epsilon}{2}}
\Bigg]\,ds
\nonumber\\
&
\leq
(T-t)^{1/2}L^{1/2}(1+T^{1/2})
\int_t^T\frac{E^{m,\epsilon}(s)}{(s-t)(T-s)^{1/2}}
\left(\bE\left[d^p+\|X^{t,x}_s\|^2\right]\right)^{\frac{1+\epsilon}{2}}
\nonumber\\
& \quad
\cdot
\left(\bE\left[
\bigg|
\int_t^s
\Big[\sigma^{-1}(X^{t,x}_r)\frac{\partial}{\partial x_k}X^{t,x}_r\Big]^T
\,dW_r
\bigg|^{\frac{2}{1-\epsilon}}
\right]\right)^{\frac{1-\epsilon}{2}}
ds
\nonumber\\
&
\leq
(T-t)^{1/2}L^{1/2}(1+T^{1/2})
\int_t^T\frac{E^{m,\epsilon}(s)}{(s-t)(T-s)^{1/2}}
(1+C_{(d,2)})^{\frac{1+\epsilon}{2}}(d^p+\|x\|^2)^{\frac{1+\epsilon}{2}}
\nonumber\\
& \quad
\cdot
\frac{8}{1-\epsilon}
\left(
\bE\left[
\int_t^s
\Big\|\sigma^{-1}(X^{t,x}_r)\frac{\partial}{\partial x_k}X^{t,x}_r\Big\|^2
\,dr
\right]
\right)^{1/2}ds
\nonumber\\
&
\leq
(T-t)^{1/2}L^{1/2}(1+T^{1/2})
(1+C_{(d,2)})^{\frac{1+\epsilon}{2}}(d^p+\|x\|^2)^{\frac{1+\epsilon}{2}}
\nonumber\\
& \quad
\cdot
\int_t^T\frac{E^{m,\epsilon}(s)}{(s-t)(T-s)^{1/2}}
\frac{8}{1-\epsilon}
\left(
\int_t^s\varepsilon_d^{-1}
\bE\left[\Big\|\frac{\partial}{\partial x_k}X^{t,x}_r\Big\|^2\right]dr
\right)^{1/2}ds
\nonumber\\
&
\leq
8(T-t)^{1/2}\big(LC_{d,0}\varepsilon_d^{-1}\big)^{1/2}(1+T^{1/2})(1-\varepsilon)^{-1}
(1+C_{(d,2)})^{\frac{1+\epsilon}{2}}(d^p+\|x\|^2)^{\frac{1+\epsilon}{2}}
\nonumber\\
& \quad
\cdot
\int_t^T(s-t)^{-1/2}(T-s)^{-1/2}E^{m,\epsilon}(s)\,ds.
\label{est B k m 2}
\end{align}
Combing \eqref{ineq B k m}, \eqref{est B k m 1}, and \eqref{est B k m 2} shows
for all $m\in\bN$, $k\in\{1,2,\dots,d\}$, $(t,x)\in[0,T)\times\bR^d$, 
$\epsilon\in(0,1)$, and $\beta\in(0,1)$ that
\begin{align}
&
(T-t)^{1/2}\int_t^T\bE\bigg[\Big|
f^{(m)}\big(s,X^{t,x}_s,u^{0,m}(s,X^{t,x}_s),w^{0,m}(s,X^{t,x}_s)\big)
\frac{1}{s-t}\int_t^s\Big[\sigma^{-1}(X^{t,x}_r)
\frac{\partial}{\partial x_k}X^{t,x}_r\Big]^T\,dW_r
\nonumber\\
&
-
f\big(s,X^{t,x}_s,u(s,X^{t,x}_s),w(s,X^{t,x}_s)\big)
\frac{1}{s-t}\int_t^s\Big[\sigma^{-1}(X^{t,x}_r)
\frac{\partial}{\partial x_k}X^{t,x}_r\Big]^T\,dW_r
\Big|\bigg]\,ds
\nonumber\\
&
\leq 
c_{\epsilon,3}^{(d)}
(d^p+n^2)^{-\frac{\epsilon}{2}}
(d^p+\|x\|^2)^{\frac{1+\epsilon}{2}}
+
\mathfrak{c}_{\epsilon,2}^{(d)}
(d^p+\|x\|^2)^{\frac{1+\epsilon}{2}}
\int_t^T(s-t)^{-1/2}(T-s)^{-1/2}E^{m,\epsilon}(s)\,ds,
\label{rum 4}
\end{align}
where
$$
c_{\epsilon,3}^{(d)}:=
16\cdot 2^{\frac{\epsilon}{2}}T^{\frac{1-\epsilon}{2}}(2-\epsilon)
\big(LC_{d,0}\varepsilon_d^{-1}\big)^{1/2}
(1+C_{(d,2)})^{\frac{1+\epsilon}{2}}(1-\epsilon)^{-2},
$$
and
$$
\mathfrak{c}_{\epsilon,2}^{(d)}
:=
8\big(LC_{d,0}\varepsilon_d^{-1}\big)^{1/2}(1+T^{1/2})(1-\varepsilon)^{-1}
(1+C_{(d,2)})^{\frac{1+\epsilon}{2}}.
$$
Then by \eqref{rum 1}, \eqref{rum 2}, \eqref{def E m epsilon}
\eqref{rum 3}, and \eqref{rum 4}, we have for all
$m\in\bN$, $t\in[0,T)$, and $\epsilon\in(0,1)$ that
\begin{align*}
E^{m,\epsilon}(t)
\leq
&
\Big[
L^{1/2}(1+C_{(d,2)})^{\frac{1+\epsilon}{2}}
+dc_{\epsilon,1}^{(d)}+c_{\epsilon,2}^{(d)}+dc_{\epsilon,3}^{(d)}
\Big]
(d^p+n^2)^{-\frac{\epsilon}{2}}
\\
&
+
(\mathfrak{c}_{\epsilon,1}^{(d)}+d\mathfrak{c}_{\epsilon,2}^{(d)})
\int_t^T
\left[(T-s)^{-1/2}+(s-t)^{1/2}(T-s)^{-1/2}\right]
E^{m,\epsilon}(s)\,ds.
\end{align*}
This together with Gr\"onwall's lemma and \eqref{time integral est 1/2} 
imply for all $m\in\bN$, $t\in[0,T)$, and $\epsilon\in(0,1)$ that
$$
E^{m,\epsilon}(t)
\leq
\Big[
L^{1/2}(1+C_{(d,2)})^{\frac{1+\epsilon}{2}}
+dc_{\epsilon,1}^{(d)}+c_{\epsilon,2}^{(d)}+dc_{\epsilon,3}^{(d)}
\Big]
(d^p+n^2)^{-\frac{\epsilon}{2}}
\exp\big\{
2(T^{1/2}+2)(\mathfrak{c}_{\epsilon,1}^{(d)}+d\mathfrak{c}_{\epsilon,2}^{(d)})
\big\}.
$$
Hence, it holds for all $t\in[0,T)$ and $\epsilon\in(0,1)$that
$
\lim_{m\to\infty}E^{m,\epsilon}(t)=0.
$
This together with \eqref{def E m epsilon} implies for
every compact set $\cK\subseteq(0,T)\times\bR^d$ that
\begin{equation}
\label{com conv u w m}
\lim_{n\to\infty}
\sup_{(t,x)\in\cK}
\big[
|u^{0,m}(t,x)-u(t,x)|+\|w^{0,m}(t,x)-w^0(t,x)\|
\big]=0.
\end{equation}
By \eqref{Lip g n f n 0}, \eqref{Lip g n}, \eqref{Lip f n}, 
\eqref{conv g n f n}, \eqref{PDE m k}, and \eqref{com conv u w m},
Lemma \ref{v solution approximation} ensures that $u$ is a viscosity solution 
PDE \eqref{APIDE}, which proves (ii). 
In addition, by \eqref{com conv u w m}, the fact that 
$\nabla_x u^{0,m}(t,x)=w^{0,m}(t,x)$ for all $m\in\bN$ and $(t,x)\in[0,T)\times\bR^d$,
and e.g., \cite[Section 16.3.5, Theorem 4]{zorich2016mathematical} 
we obtain (iii).
Therefore, we have completed the proof of this theorem.
\end{proof}

\section{\textbf{Some proofs of the results presented in Section \ref{section general MLP}}}
\label{appendix MLP}

\subsection{Proof of Lemma \ref{lemma MLP property}}

\begin{proof}[Proof of Lemma \ref{lemma MLP property}]
\label{sec proof C1}
First notice that \eqref{def MLP general} and the measurability conditions
of $\bX^\theta$ and $\bV^\theta$ assumed in Setting \ref{MLP setting} establish (i),
and the construction of $U^\theta_{n,M}$ ensures (ii).
Moreover, (ii) and the fact that it holds for all $\theta\in\Theta$ that
$(\xi^{\theta,\upsilon})_{\upsilon\in\Theta}$,
$\big(\bX^{(\theta,\upsilon,t,x)}_s\big)
_{(t,s,x,\upsilon)\in\Delta\times\bR^d\times\Theta}$,
$\big(\bX^{\theta,t,x}_s\big)_{(t,s,x)\in\Delta\times\bR^d}$,
and $\xi^\theta$ are independent prove (iii).
Next, note that (ii), the fact that it holds 
for all $i,j\in \bZ$ and $\theta\in\Theta$ that
$\xi^{(\theta,i,j)}$ and $\big(\bX^{(\theta,i,j,t,x)}_s\big)
_{(t,s,x)\in\Delta\times\bR^d}$ are independent, and the fact that
it holds for all $\theta\in\Theta$ and $i,j,k,l\in\bZ$
with $(i,j)\neq(k,l)$ that
$$
\big(
\xi^{\theta,i,j,\upsilon},\bX^{\theta,i,j,\upsilon,t,x}_s,\bV^{\theta,i,j,\upsilon,t,x}_s
\big)_{(t,s,x,\upsilon)\in\Delta\times\bR^d\times\Theta}
\quad \text{and} \quad
\big(
\xi^{\theta,k,l,\upsilon},\bX^{\theta,k,l,\upsilon,t,x}_s,\bV^{\theta,k,l,\upsilon,t,x}_s
\big)_{(t,s,x,\upsilon)\in\Delta\times\bR^d\times\Theta}
$$
are independent establish (iv).
In addition, \eqref{def MLP general}, (iii), (iv), and the independence conditions 
of $\bX^\theta$, $\bV^\theta$, and $\xi^\theta$ assumed in Setting \ref{MLP setting}
ensure (v). We have therefore completed the proof of this lemma.
\end{proof}

\subsection{Proof of Lemma \ref{Lemma MLP property}}
\label{sec proof C2}

\begin{proof}[Proof of Lemma \ref{Lemma MLP property}]
We first observe that (ii) in Lemma \ref{lemma MLP property} and
the assumption that 
$
\big(\bX^{\theta,t,x}_s,\bV^{\theta,t,x}_s\big)
_{(\theta,t,s,x)\in\Theta\times\Delta\times\bR^d}
$
and
$(\xi^\theta)_{\theta\in\Theta}$ are independent ensure
for all $l\in\bN_0$, and $\eta,\zeta,\upsilon\in\Theta$ with
$\min\{\dim(\eta),\dim(\zeta)\}\geq\dim(\upsilon)$ that
$
\Big(\big(F(U^\eta_{l,M})-\mathbf{1}_{\{l\geq 1\}}F(U^\zeta_{l-1,M})\big)(t,x)\Big)
_{(t,x)\in[0,T)\times\bR^d},
$
$\xi^\upsilon$, and
$
\big(\bX^{\upsilon,t,x}_s,\bV^{\upsilon,t,x}_s\big)
_{(t,s,x)\in\Delta\times\bR^d}
$
are independent.
Hence, by the construction of $(\cR_t^\theta)_{(\theta,t)\in\Theta\times[0,T)}$,
and e.g., Lemma 2.2 in \cite{HJKNW2020} it holds for all $l\in\bN_0$, 
$(t,x)\in[0,T)\times\bR^d$, and $\eta,\zeta,\upsilon\in\Theta$ with
$\min\{\dim(\eta),\dim(\zeta)\}\geq\dim(\upsilon)$ that
\begin{align}
&
\bE\left[
\Big\|
(T-t)\varrho^{-1}\Big(\frac{\cR^{\upsilon}_t-t}{T-t}\Big)
\big(F\big(U^\eta_{l,M}\big)\big)
\big(\cR^{\upsilon}_t,\bX^{\upsilon,t,x}_{\cR^\upsilon_t}\big)
\big(1,\bV^{\upsilon,t,x}_{\cR^\upsilon_t}\big)
\Big\|^2
\right]
\nonumber\\
&
=(T-t)\int_t^T
\bE\left[
\varrho^{-1}\Big(\frac{s-t}{T-t}\Big)
\cdot
\big\|
\big(F\big(U^\eta_{l,M}\big)\big)
\big(s,\bX^{\upsilon,t,x}_s\big)
\big(1,\bV^{\upsilon,t,x}_s\big)
\big\|^2
\right]
ds,
\label{expectation time integral 1}
\end{align}
and
\begin{align}
&
\bE\left[
\Big\|
(T-t)\varrho^{-1}\Big(\frac{\cR_t^{(\upsilon,l,i)}-t}{T-t}\Big)
\Big(F(U^\eta_{l,M})-\mathbf{1}_{\{l \geq 1\}}F(U^\zeta_{l-1,M})\Big)
\Big(\cR^\upsilon_t,\bX^{\upsilon,t,x}_{\cR^\upsilon_t}\Big)
\Big(1,\bV^{\upsilon,t,x}_{\cR^\upsilon_t}\Big)
\Big\|^2
\right]
\nonumber\\
&
=(T-t)\int_t^T\varrho^{-1}\Big(\frac{s-t}{T-t}\Big)
\bE\left[
\Big\|
\big(F(U^\eta_{l,M})-\mathbf{1}_{\{l\geq 1\}}F(U^\zeta_{l-1,M})\big)
\big(s,\bX^{\upsilon,t,x}_s\big)
\big(1,\bV^{\upsilon,t,x}_s\big)
\Big\|^2
\right]
ds
\nonumber\\
&
=A^{\eta,\zeta,\upsilon}_{l,1}+A^{\eta,\zeta,\upsilon}_{l,2},
\label{ineq A l 1 2}
\end{align}
where
$$
A^{\eta,\zeta,\upsilon}_{l,1}
:=(T-t)\int_t^T\varrho^{-1}\Big(\frac{s-t}{T-t}\Big)
\bE\left[
\bE\left[
\big|
\big(F(U^\eta_{l,M})-\mathbf{1}_{\{l\geq 1\}}F(U^\zeta_{l-1,M})\big)
(s,z)
\big|^2
\right]
\Big|_{z=\bX^{\upsilon,t,x}_s}
\right]
ds,
$$
and
$$
A^{\eta,\zeta,\upsilon}_{l,2}
:=(T-t)\int_t^T\varrho^{-1}\Big(\frac{s-t}{T-t}\Big)
\bE\left[
\bE\left[
\big\|
\big(F(U^\eta_{l,M})-\mathbf{1}_{\{l\geq 1\}}F(U^\zeta_{l-1,M})\big)
(s,z)\cdot y
\big\|^2
\right]\Big|_{(z,y)=(\bX^{\upsilon,t,x}_s,\bV^{\upsilon,t,x}_s)}
\right]
ds.
$$
Then by \eqref{MLP Z est}, \eqref{MLP V est}, \eqref{MLP Lip f},
and Cauchy-Schwarz inequality, we obtain for all $l\in\bN_0$, 
$(t,x)\in[0,T)\times\bR^d$, and $\eta,\zeta,\upsilon\in\Theta$ with
$\min\{\dim(\eta),\dim(\zeta)\}\geq\dim(\upsilon)$ that
\begin{align}
A^{\eta,\zeta,\upsilon}_{l,1}(t,x)
&
\leq
(T-t)\int_t^T
\varrho^{-1}\Big(\frac{s-t}{T-t}\Big)
\frac{e^{-\rho s}}{T-s}\bE\left[d^p+\|\bX^{\upsilon,t,x}_s\|^2\right]
\nonumber\\
& \quad
\sup_{(r,z)\in[s,T)\times\bR^d}
\bigg(
\frac{e^{\rho r}(T-r)}{(d^p+\|z\|^2)}
\bE\left[
\big|\big(F(U^\eta_{l,M})
-\mathbf{1}_{\{l\geq 1\}}\bF(U^\zeta_{l-1,M})\big)(r,z)\big|^2
\right]
\bigg)
\,ds
\nonumber\\
&
\leq
ae^{-\rho t}(d^p+\|x\|^2)(T-t)
\int_t^T\varrho^{-1}\Big(\frac{s-t}{T-t}\Big)(T-s)^{-1}
\nonumber\\
& \quad
\sup_{(r,z)\in[s,T)\times\bR^d}
\bigg(
\frac{e^{\rho r}(T-r)}{(d^p+\|z\|^2)}
\bE\left[
\big|\big(F(U^\eta_{l,M})
-\mathbf{1}_{\{l\geq 1\}}\bF(U^\zeta_{l-1,M})\big)(r,z)\big|^2
\right]
\bigg)
\,ds,
\label{A l 1 est}
\end{align}
and
\begin{align}
A^{\eta,\zeta,\upsilon}_{l,2}(t,x)
&
\leq
(T-t)\int_t^T
\varrho^{-1}\Big(\frac{s-t}{T-t}\Big)
\frac{e^{-\rho s}}{T-s}
\bE\left[\big(d^p+\|\bX^{\upsilon,t,x}_s\|^2\big)\|\bV^{\upsilon,t,x}_s\|^2\right]
\nonumber\\
& \quad
\sup_{(r,z)\in[s,T)\times\bR^d}
\bigg(
\frac{e^{\rho r}(T-r)}{(d^p+\|z\|^2)}
\bE\left[
\big|\big(F(U^\eta_{l,M})
-\mathbf{1}_{\{l\geq 1\}}\bF(U^\zeta_{l-1,M})\big)(r,z)\big|^2
\right]
\bigg)
\,ds
\nonumber\\
&
\leq
(T-t)\int_t^T
\varrho^{-1}\Big(\frac{s-t}{T-t}\Big)
\frac{e^{-\rho s}}{T-s}
\left(\bE\left[\big(d^p+\|\bX^{\upsilon,t,x}_s\|^2\big)^2\right]\right)^{1/2}
\left(\bE\left[\|\bV^{\upsilon,t,x}_s\|^4\right]\right)^{1/2}
\nonumber\\
& \quad
\sup_{(r,z)\in[s,T)\times\bR^d}
\bigg(
\frac{e^{\rho r}(T-r)}{(d^p+\|z\|^2)}
\bE\left[
\big|\big(F(U^\eta_{l,M})
-\mathbf{1}_{\{l\geq 1\}}\bF(U^\zeta_{l-1,M})\big)(r,z)\big|^2
\right]
\bigg)
\,ds
\nonumber\\
&
\leq
a_1b_1e^{-\rho t}(d^p+\|x\|^2)(T-t)
\int_t^T
\varrho^{-1}\Big(\frac{s-t}{T-t}\Big)(s-t)^{-1}(T-s)^{-1}
\nonumber\\
& \quad
\cdot
\sup_{(r,z)\in[s,T)\times\bR^d}
\bigg(
\frac{e^{\rho r}(T-r)}{(d^p+\|z\|^2)}
\bE\left[
\big|\big(F(U^\eta_{l,M})
-\mathbf{1}_{\{l\geq 1\}}\bF(U^\zeta_{l-1,M})\big)(r,z)\big|^2
\right]
\bigg)
\,ds.
\label{A l 2 est}
\end{align}
Furthermore, by \eqref{def F MLP}, \eqref{MLP Lip f},
\eqref{ineq A l 1 2}, \eqref{A l 1 est}, and \eqref{A l 2 est}
it holds for all $l\in\bN_0$, $(t,x)\in[0,T)\times\bR^d$,
$\eta,\zeta,\upsilon\in\Theta$ 
with $\min\{\dim(\eta),\dim(\zeta)\}\geq\dim(\upsilon)$ that
\begin{align*}
&
\Bigg(
\sup_{x\in\bR^d}\bigg(
\frac{e^{\rho t}(T-t)^2}{d^p+\|x\|^2}
\bE\bigg[
\Big\|
\varrho^{-1}\Big(\frac{\cR_t^{\upsilon}-t}{T-t}\Big)
\Big(F(U^\eta_{l,M})-\mathbf{1}_{\{l\geq 1\}}F(U^\zeta_{l-1,M})\Big)
\Big(\cR^\upsilon_t,\bX^{\upsilon,t,x}_{\cR^\upsilon_t}\Big)
\Big(1,\bV^{\upsilon,t,x}_{\cR^\upsilon_t}\Big)
\Big\|^2
\bigg]
\bigg)
\Bigg)^{1/2}
\\
&
=
\bigg(
\sup_{x\in\bR^d}
\left[
(d^p+\|x\|^2)^{-1}e^{\rho t}
\big(A^{\eta,\zeta,\upsilon}_{l,1}+A^{\eta,\zeta,\upsilon}_{l,2}\big)
\right]
\bigg)^{1/2}
\\
&
\leq
\bigg(
\sup_{x\in\bR^d}
\bigg[
(a+a_1b_1)(T-t)
\int_t^T
\varrho^{-1}\Big(\frac{s-t}{T-t}\Big)[(s-t)^{-1}+1](T-s)^{-1}
\\
& \quad 
\cdot
\sup_{(r,z)\in[s,T)\times\bR^d}
\bigg(
\frac{e^{\rho r}(T-r)}{(d^p+\|z\|^2)}
\bE\left[
\big|\big(F(U^\eta_{l,M})
-\mathbf{1}_{\{l\geq 1\}}\bF(U^\zeta_{l-1,M})\big)(r,z)\big|^2
\right]
\bigg)
\,ds
\bigg]
\bigg)^{1/2}
\\
&
\leq
[(a+a_1b_1)(T-t)]^{1/2}
\Bigg[
\bigg(\int_t^T
\frac{[(s-t)^{-1}+1]}{(T-s)\varrho(\frac{s-t}{T-t})}
\,ds\bigg)^{1/2}
\sup_{(r,x)\in[t,T]\times\bR^d}
\bigg(
\frac{\mathbf{1}_{\{0\}}(l)[e^{\rho r}(T-r)]^{1/2}}{(d^p+\|x\|^2)^{1/2}}
\big|(F(\mathbf{0}))(r,x)\big|
\bigg)
\\
& \quad
+
\left(
\int_t^T
\frac{[(s-t)^{-1}+1]}{(T-s)\varrho(\frac{s-t}{T-t})}
\sup_{(r,x)\in[s,T)\times\bR^d}
\left(
\frac{\mathbf{1}_{\{l\geq 1\}}Le^{\rho r}(T-r)}{d^p+\|x\|^2}
\bE\left[\big\|(U^\eta_{l,M}-U^\zeta_{l-1,M})(r,x)\big\|^2\right]
\right)
ds
\right)^{1/2}
\Bigg].
\end{align*}
This proves (i). 
Next by \eqref{MLP Z x est}, \eqref{MLP V est}, \eqref{MLP Lip f},
Cauchy-Schwarz inequality, 
and the fact that for all $(t,x)\in[0,T)\times\bR^d$ and $s\in[t,T)$ it holds that
$\bX^{\theta,t,x}_s$, $\theta\in\Theta$, are identically distributed,
and $\bV^{\theta,t,x}_s$, $\theta\in\Theta$, are identically distributed,
we have for all $\theta\in\Theta$ and $(t,x)\in[0,T)\times\bR^d$ that
\begin{align}
&
\bE\left[\big\|
\big[g\big(\bX^{\theta,t,x}_T\big)-g(x)\big]\big(1,\bV^{\theta,t,x}_T\big)
\big\|^2\right]
\nonumber\\
&
\leq
\bE\left[\big|g\big(\bX^{0,t,x}_T\big)-g(x)\big|^2\right]
+
\left(\bE\left[\big|g\big(\bX^{0,t,x}_T\big)-g(x)\big|^4\right]\right)^{1/2}
\left(\bE\left[\big\|\bV^{0,t,x}_T\big\|^4\right]\right)^{1/2}
\nonumber\\
&
\leq
L\bE\left[\big\|\bX^{0,t,x}_T-x\big\|^2\right]
+
Lb_1c(T-t)^{-1}\left(\bE\left[\big\|\bX^{0,t,x}_T-x\big\|^4\right]\right)^{1/2}
\nonumber\\
&
\leq
(d^p+\|x\|^2)e^{\rho(T-t)}L(a_2T+a_3b_1c).
\end{align} 
This implies (ii).
Next, we start to show (iii) and (iv) by induction on $n\in\bN_0$.
By \eqref{def pdf rho}, we notice for all $t\in[0,T)$ that
\begin{align}
&
\int_t^T
\varrho^{-1}\Big(\frac{s-t}{T-t}\Big)[(s-t)^{-1}+1](T-s)^{-1}
\,ds
\nonumber\\
&
=
\frac{\cB(1-\alpha,1-\alpha)}{(T-t)^{2\alpha}}
\int_t^T
(T-s)^{-(1-\alpha)}
\left[
(s-t)^{-(1-\alpha)}+(s-t)^\alpha
\right]
ds
\nonumber\\
&
\leq
\frac{\cB(1-\alpha,1-\alpha)}{(T-t)^{2\alpha}}
\Bigg[
\int_t^{\frac{T+t}{2}}
\Big(\frac{T-t}{2}\Big)^{-(1-\alpha)}(s-t)^{-(1-\alpha)}
\,ds
+
\int_{\frac{T+t}{2}}^T
\Big(\frac{T-t}{2}\Big)^{-(1-\alpha)}(T-s)^{-(1-\alpha)}
\,ds
\nonumber\\
& \quad
+\int_t^T(T-s)^{-(1-\alpha)}(T-t)^\alpha\,ds
\Bigg]
\nonumber\\
&
=
\frac{\cB(1-\alpha,1-\alpha)}{\alpha(T-t)}(4^{-\alpha}+T).
\label{beta integral est 1}
\end{align}
Then by \eqref{def pdf rho}, \eqref{MLP LG}, \eqref{MLP est F l l-1},
and \eqref{beta integral est 1}
it holds for all $\upsilon\in\Theta$ and $t\in[0,T)$ that
\begin{align*}
&
\sup_{x\in\bR^d}
\left(
\frac{e^{\rho t}(T-t)^2}{d^p+\|x\|^2}
\bE\left[
\Big\|
\varrho^{-1}\Big(\frac{\cR^\upsilon_t-t}{T-t}\Big)
(F(\mathbf{0}))\big(\cR^\upsilon_t,\bX^{\upsilon,t,x}_t\big)
\Big\|^2
\right]
\right)^{1/2}
\\
&
\leq
[c(a+a_1b_1)]^{1/2}(T-t)^{1/2}
\left(
\int_t^T
\varrho^{-1}\Big(\frac{s-t}{T-t}\Big)[(s-t)^{-1}+1](T-s)^{-1}
\,ds
\right)^{1/2}
\\
&
\leq
\big[\alpha^{-1} c(a+a_1b_1)\beta(1-\alpha,1-\alpha)(4^{-\alpha}+T)\big]^{1/2}
<\infty.
\end{align*}
This together with the assumption that $U^\theta_{0,M(t,x)}=\mathbf{0}$ 
for all $(t,x)\in[0,T)\times\bR^d$ and $\theta\in\Theta$ establish (iii)
and (iv) for $n=0$.
Then, let $n\in\bN$ and assume for all $t\in[0,T)$, $l\in[0,n-1]\cap\bN_0$,
$\theta\in\Theta$, and $\eta,\zeta,\upsilon\in\Theta$ with
$\min\{\dim(\eta),\dim(\zeta)\}\geq \dim(\upsilon)$ that
\begin{align}
&
\sup_{(t,x)\in[0,T)\times\bR^d}
\left(
\frac{e^{\rho t}(T-t)^2}{d^p+\|x\|^2}
\bE\left[
\Big\|
\varrho^{-1}\Big(\frac{\cR^{\upsilon}_t-t}{T-t}\Big)
\big(
F\big(U^\eta_{l,M}\big)
-
\mathbf{1}_{\{l\geq 1\}}F\big(U^\zeta_{l-1,M}\big)
\big)
\big(\cR^{\upsilon}_t,\bX^{\upsilon,t,x}_{\cR^\upsilon_t}\big)
\big(1,\bV^{\upsilon,t,x}_{\cR^\upsilon_t}\big)
\Big\|^2
\right]
\right)
\nonumber\\
&
<\infty,
\label{induction 1}
\end{align}
and
\begin{equation}
\label{induction 2}
\sup_{(t,x)\in[0,T)\times\bR^d}
\left(
(d^p+\|x\|^2)^{-1}e^{\rho t}
\bE\left[\big\|U^\theta_{l,M}(s,x)\big\|^2\right]
\right)
<\infty.
\end{equation}
By \eqref{def MLP general}, \eqref{MLP est g}, and \eqref{induction 1},
we obtain for all $\theta\in\Theta$ that
\begin{equation}
\label{induction 3}
\sup_{(t,x)\in[0,T)\bR^d}
\left(
(d^p+\|x\|^2)^{-1}e^{\rho t}
\bE\left[\big\|U^\theta_{n,M}(s,x)\big\|^2\right]
\right)
<\infty.
\end{equation}
Furthermore, combining \eqref{MLP Lip f}, \eqref{MLP est F l l-1}, 
\eqref{beta integral est 1}, \eqref{induction 2}, and \eqref{induction 3} yields
for all $t\in[0,T)$ and $\eta,\zeta,\upsilon\in\Theta$ with
$\min\{\dim(\eta),\dim(\zeta)\}\geq \dim(\upsilon)$ that
\begin{align}
&
\sup_{(t,x)\in[0,T)\times\bR^d}
\left(
\frac{e^{\rho t}(T-t)^2}{d^p+\|x\|^2}
\bE\left[
\Big\|
\varrho^{-1}\Big(\frac{\cR^{\upsilon}_t-t}{T-t}\Big)
\big(
F\big(U^\eta_{n,M}\big)
-
\mathbf{1}_{\{n\geq 1\}}F\big(U^\zeta_{n-1,M}\big)
\big)
\big(\cR^{\upsilon}_t,\bX^{\upsilon,t,x}_{\cR^\upsilon_t}\big)
\big(1,\bV^{\upsilon,t,x}_{\cR^\upsilon_t}\big)
\Big\|^2
\right]
\right)
\nonumber\\
&
\leq
(a+a_1b_1)
\sup_{t\in[0,T)}
\bigg[
(T-t)\int_t^T
\varrho^{-1}(\frac{s-t}{T-t})[(s-t)^{-1}+1](T-s)^{-1}
\nonumber\\
& \quad
\cdot
\sup_{(r,x)\in[s,T)\times\bR^d}
\left(
\frac{Le^{\rho r}(T-r)}{d^p+\|x\|^2}
\bE\left[
\big\|
\big(U^\eta_{n,M}-U^{\zeta}_{n-1,M}\big)(r,x)
\big\|^2
\right]
\right)
ds
\bigg]
\nonumber\\
&
\leq
(a+a_1b_1)L\alpha^{-1}(4^{-\alpha}+T)\cB(1-\alpha,1-\alpha)
\sup_{(t,x)\in[0,T)\times\bR^d}
\left(
\frac{e^{\rho t}(T-t)}{d^p+\|x\|^2}
\bE\left[\big\|
\big(U^\eta_{n,M}-U^\zeta_{n-1,M}\big)(t,x)
\big\|^2\right]
\right)
\nonumber\\
&
<\infty.
\label{induction 4}
\end{align}
Then by \eqref{induction 1}--\eqref{induction 4} and induction,
we have proved (iii) and (iv).

Next, by \eqref{MLP F l l-1 finite} and the triangle inequality we obtain
for all $n\in\bN_0$ and $\eta,\zeta\in\Theta$ with $\dim(\eta)\geq\dim(\zeta)$ that
\begin{align*}
&
\sup_{(t,x)\in[0,T)\times\bR^d}
\left(
\frac{e^{\rho t}}{d^p+\|x\|^2}
\bE\left[
\Big\|
(T-t)\varrho^{-1}\Big(\frac{\cR^{\upsilon}_t-t}{T-t}\Big)
\big(F\big(U^\eta_{l,M}\big)\big)
\big(\cR^{\upsilon}_t,\bX^{\upsilon,t,x}_{\cR^\upsilon_t}\big)
\big(1,\bV^{\upsilon,t,x}_{\cR^\upsilon_t}\big)
\Big\|^2
\right]
\right)^{1/2}
\\
&
\leq
\sum_{l=1}^n
\sup_{(t,x)\in[0,T)\times\bR^d}
\bigg(
\frac{e^{\rho t}}{d^p+\|x\|^2}
\bE\Big[
\Big\|
(T-t)\varrho^{-1}\Big(\frac{\cR^{\upsilon}_t-t}{T-t}\Big)
\\
& \quad
\cdot
\big(F\big(U^\eta_{l,M}\big)-\mathbf{1}_{\{l\geq 1\}}F\big(U^\eta_{l-1,M}\big)\big)
\big(\cR^{\upsilon}_t,\bX^{\upsilon,t,x}_{\cR^\upsilon_t}\big)
\big(1,\bV^{\upsilon,t,x}_{\cR^\upsilon_t}\big)
\Big\|^2
\Big]
\bigg)
<\infty.
\end{align*}
This together with \eqref{expectation time integral 1} establish (v).
Therefore, we have completed the proof of this lemma.
\end{proof}

\subsection{Proof of Lemma \ref{lemma iid expectation}}
\label{sec proof C3}

\begin{proof}[Proof of Lemma \ref{lemma iid expectation}]
Note that (ii) in Lemma \ref{lemma MLP property} ensures for all $i\in\bN$,
$l\in\bN_0$, and $M\in\bN_0$ that
$$
\sigma\Big(
\big(U^{(\theta,l,i)}_{l,M}(t,x)\big)_{(t,x)\in[0,T)\times\bR^d}
\Big)
\subseteq
\sigma\Big(
\big(\xi^{(\theta,l,i,\upsilon)}\big)_{\upsilon\in\Theta},
\big(
\bX^{(\theta,l,i,\upsilon,t,x)}_s,\bV^{(\theta,l,i,\upsilon,t,x)}_s
\big)_{(t,s,x,\upsilon)\in\Delta\times\bR^d\times\Theta}
\Big),
$$
and
$$
\sigma\Big(
\big(U^{(\theta,-l,i)}_{l-1,M}(t,x)\big)_{(t,x)\in[0,T)\times\bR^d}
\Big)
\subseteq
\sigma\Big(
\big(\xi^{(\theta,-l,i,\upsilon)}\big)_{\upsilon\in\Theta},
\big(
\bX^{(\theta,-l,i,\upsilon,t,x)}_s,\bV^{(\theta,-l,i,\upsilon,t,x)}_s
\big)_{(t,s,x,\upsilon)\in\Delta\times\bR^d\times\Theta}
\Big).
$$
This together with the fact that $\xi^\upsilon$, $\upsilon\in\Theta$, are
independent, and the fact that 
$\big(\bX^{\upsilon,t,x}_s,\bV^{\upsilon,t,x}_s\big)$, $\upsilon\in\Theta$,
are independent for all $(t,s,x)\in\Delta\times\bR^d$ imply for all
$l\in\bN_0$, $M\in\bN$, and $(t,x)\in[0,T)\times\bR^d$ that
$$
\Big(
F\big(U^{(\theta,l,i)}_{l,M}\big)
-
\mathbf{1}_{\{l\geq 1\}}F\big(U^{(\theta,-l,i)}_{l-1,M}\big)
\Big)
\Big(
\cR^{(\theta,l,i)}_t,\bX^{(\theta,l,i,t,x)}_{\cR^{(\theta,l,i)}_t}
\Big)
\Big(
1,\bV^{(\theta,l,i,t,x)}_{\cR^{(\theta,l,i)}_t}
\Big),
\quad
i\in\bN,
$$
are independent. 
Moreover, the fact that $\xi^\upsilon$, $\upsilon\in\Theta$, are i.i.d.,
and the fact that it holds for all $(t,s,x)\in\Delta\times\bR^d$ that
$\big(\bX^{\upsilon,t,x}_s,\bV^{\upsilon,t,x}_s\big)$, $\upsilon\in\Theta$,
are i.i.d., and the fact that
$\big(\bX^{\upsilon,t,x}_s,\bV^{\upsilon,t,x}_s\big)
_{(t,s,x,\upsilon)\in\Delta\times\bR^d\times\Theta}$
and
$(\xi^\upsilon)_{\upsilon\in\Theta}$ are independent establish that
$$
\Big(
F\big(U^{(\theta,l,i)}_{l,M}\big)
-
\mathbf{1}_{\{l\geq 1\}}F\big(U^{(\theta,-l,i)}_{l-1,M}\big)
\Big)
\Big(
\cR^{(\theta,l,i)}_t,\bX^{(\theta,l,i,t,x)}_{\cR^{(\theta,l,i)}_t}
\Big)
\Big(
1,\bV^{(\theta,l,i,t,x)}_{\cR^{(\theta,l,i)}_t}
\Big),
\quad
i\in\bN,
$$
are identically distributed.
Hence, we obtain (i).

Next, note that (iii) and (v) in Lemma \ref{lemma MLP property},
(iii) in Lemma \ref{Lemma MLP property},
the fact that it holds for all $l\in\bN_0$ and $i\in\bN$ that
$\xi^{(\theta,l,i)}$ and $\xi^\theta$ are identically distributed,
the fact that it holds for all $l\in\bN_0$, $i\in\bN$, and
$(t,s,x)\in\Delta\times\bR^d$ that 
$\big(\bX^{(\theta,l,i,t,x)}_s,\bV^{(\theta,l,i,t,x)}_s\big)$ 
and $\big(\bX^{\theta,t,x}_s,\bV^{\theta,t,x}_s\big)$
are identically distributed, and e.g., Lemma 2.2 in \cite{HJKNW2020} ensure for all
$i,M\in\bN$, $l\in\bN_0$, and $(t,x)\in[0,T)\times\bR^d$ that
\begin{align}
&
\bE\left[
\Big(
F\big(U^{(\theta,l,i)}_{l,M}\big)
-
\mathbf{1}_{\{l\geq 1\}}F\big(U^{(\theta,-l,i)}_{l-1,M}\big)
\Big)
\Big(\cR^{(\theta,l,i)}_t,\bX^{(\theta,l,i,t,x)}_{\cR^{(\theta,l,i)}_t}\Big)
\Big(1,\bV^{(\theta,l,i,t,x)}_{\cR^{(\theta,l,i)}_t}\Big)
\right]
\nonumber\\
&
=
\bE\left[
\Big(F\big(U^{(\theta,l,i)}_{l,M}\big)\Big)
\Big(\cR^{(\theta,l,i)}_t,\bX^{(\theta,l,i,t,x)}_{\cR^{(\theta,l,i)}_t}\Big)
\Big(1,\bV^{(\theta,l,i,t,x)}_{\cR^{(\theta,l,i)}_t}\Big)
\right]
\nonumber\\
& \quad
-\mathbf{1}_{\{l\geq 1\}}
\bE\left[
\Big(F\big(U^{(\theta,-l,i)}_{l-1,M}\big)\Big)
\Big(\cR^{(\theta,l,i)}_t,\bX^{(\theta,l,i,t,x)}_{\cR^{(\theta,l,i)}_t}\Big)
\Big(1,\bV^{(\theta,l,i,t,x)}_{\cR^{(\theta,l,i)}_t}\Big)
\right]
\nonumber\\
&
=
\bE\left[
\Big(F\big(U^{\theta}_{l,M}\big)\Big)
\Big(\cR^{\theta}_t,\bX^{\theta,t,x}_{\cR^{\theta}_t}\Big)
\Big(1,\bV^{\theta,t,x}_{\cR^{\theta}_t}\Big)
\right]
-\mathbf{1}_{\{l\geq 1\}}
\bE\left[
\Big(F\big(U^{\theta}_{l-1,M}\big)\Big)
\Big(\cR^{\theta}_t,\bX^{\theta,t,x}_{\cR^{\theta}_t}\Big)
\Big(1,\bV^{\theta,t,x}_{\cR^{\theta}_t}\Big)
\right].
\nonumber
\end{align}
Thus, by (iii) in Lemma \ref{lemma MLP property}, 
(iii) in Lemma \ref{Lemma MLP property}, the assumption that it holds for all
$(t,s,x)\in\Delta\times\bR^d$ that 
$\big(\bX^{\upsilon,t,x}_s,\bV^{\upsilon,t,x}_s\big)$, $\upsilon\in\Theta$,
are identically distributed, the construction of $\cR^\theta_t$,
and e.g., Lemma 2.2 in \cite{HJKNW2020} we have for all $n,M\in\bN$ and
$(t,x)\in[0,T)\times\bR^d$ that
\begin{align*}
&
\bE\left[U^\theta_{n,M}(t,x)\right]
\\
&
=
\frac{1}{M^n}
\sum_{i=1}^{M^n}\bE\left[g\big(\bX^{(\theta,t,x,0,-i)}_T\big)
\Big(1,\bV^{(\theta,t,x,0,-i)}_T\Big)\right]
+
\sum_{l=0}^{n-1}\frac{T-t}{M^{n-l}}
\Bigg(
\sum_{i=1}^{M^{n-l}}
\bE\bigg[
\varrho^{-1}\Big(\frac{\cR^{(\theta,l,i)}_t-t}{T-t}\Big)
\nonumber\\
& \quad
\cdot
\Big[F\big(U^{(\theta,l,i)}_{l,M}\big)
-\mathbf{1}_{\{l\geq 1\}}F\big(U^{(\theta,-l,i)}_{l-1,M}\big)\Big]
\bigg]
\Big(\cR^{(\theta,l,i)}_t,\bX^{(\theta,t,x,l,i)}_{\cR^{(\theta,l,i)}_t}\Big)
\Big(1,\bV^{(\theta,t,x,l,i)}_{\cR^{(\theta,l,i)}_t}\Big)
\Bigg)
\\
&
=
\bE\left[g\big(\bX^{\theta,t,x}_T\big)\Big(1,\bV^{\theta,t,x}_T\Big)\right]
+(T-t)\sum_{l=0}^{n-1}\bigg(
\bE\left[
\varrho^{-1}\Big(\frac{\cR^\theta_t-t}{T-t}\Big)
\Big(F\big(U^\theta_{l,M}\big)\Big)
\Big(\cR^\theta_t,\bX^{\theta,t,x}_{\cR^\theta_t}\Big)
\Big(1,\bV^{\theta,t,x}_{\cR^\theta_t}\Big)
\right]
\\
& \quad
-\mathbf{1}_{\{l\geq 1\}}
\bE\left[
\varrho^{-1}\Big(\frac{\cR^\theta_t-t}{T-t}\Big)
\Big(F\big(U^\theta_{l-1,M}\big)\Big)
\Big(\cR^\theta_t,\bX^{\theta,t,x}_{\cR^\theta_t}\Big)
\Big(1,\bV^{\theta,t,x}_{\cR^\theta_t}\Big)
\right]
\bigg)
\\
&
=
\bE\left[g\big(\bX^{\theta,t,x}_T\big)\Big(1,\bV^{\theta,t,x}_T\Big)\right]
+
(T-t)\bE\left[
\varrho^{-1}\Big(\frac{\cR^\theta_t-t}{T-t}\Big)
\Big(F\big(U^\theta_{n-1,M}\big)\Big)
\Big(\cR^\theta_t,\bX^{\theta,t,x}_{\cR^\theta_t}\Big)
\Big(1,\bV^{\theta,t,x}_{\cR^\theta_t}\Big)
\right]
\\
&
=
\bE\left[g\big(\bX^{\theta,t,x}_T\big)\Big(1,\bV^{\theta,t,x}_T\Big)\right]
+
\int_t^T\bE\left[
\Big(F\big(U^\theta_{n-1,M}\big)\Big)
\big(s,\bX^{\theta,t,x}_s\big)
\big(1,\bV^{\theta,t,x}_s\big)
\right]ds.
\end{align*}
This proves (ii). Hence, we have completed the proof of this lemma.
\end{proof}

\smallskip
\noindent 

\bibliographystyle{plain}
\bibliography{Reference}

\end{document}